\documentclass[letterpaper]{article}
\usepackage[margin=1.5in]{geometry}

\newif\ifdraft
\draftfalse
\usepackage{amsmath,amsfonts,mathrsfs,amsthm,thmtools}
\usepackage{amssymb}

\usepackage[unicode,bookmarks]{hyperref} %ocgcolorlinks
\usepackage{cleveref}
\usepackage[usenames,dvipsnames]{xcolor}
\hypersetup{colorlinks=true,citecolor=NavyBlue,linkcolor=BrickRed,urlcolor=Black}

\usepackage{expl3}

\ExplSyntaxOn
\prop_new:N \g_cite_map_prop
\tl_new:N \l_citekey_result_tl

\cs_new:Npn \mapcitekey #1#2 {
  \clist_map_inline:nn {#2}
       {  \prop_gput:Nnn  \g_cite_map_prop  {##1} {#1}   }
}

\cs_new:Npn \getcitekey #1 {
   \prop_get:NoN \g_cite_map_prop{#1}  \l_citekey_result_tl
   \quark_if_no_value:NF \l_citekey_result_tl
       {  \tl_set_eq:NN #1  \l_citekey_result_tl  }
}

\cs_new:Npn \showcitekeymaps {\prop_show:N  \g_cite_map_prop }
\ExplSyntaxOff

\usepackage{etoolbox}
\makeatletter
\patchcmd{\@citex}{\if@filesw}{\getcitekey\@citeb \if@filesw}%
    {\typeout{*** SUCCESS ***}}{\typeout{*** FAIL ***}}
\patchcmd{\nocite}{\if@filesw}{\getcitekey\@citeb \if@filesw}%
    {\typeout{*** SUCCESS ***}}{\typeout{*** FAIL ***}}
\makeatother

\usepackage{enumitem}

\newenvironment{aenumerate}{%
	\begin{enumerate}[label=(\alph{*}), ref=(\alph{*})]
}{%
	\end{enumerate}%
}

\usepackage{chngcntr}

\ifdraft
\usepackage[notcite,notref,color]{showkeys}

\definecolor{labelkey}{gray}{0.5}
\fi

\usepackage{tikz}
\usepackage{tikz-cd}
\tikzset{commutative diagrams/arrow style=Latin Modern}

\newcommand{\Lie}[1]{\mathcal{L}_{#1}}

\newcommand{\norm}[1]{\lVert#1\rVert}
\newcommand{\bignorm}[1]{\bigl\lVert#1\bigr\rVert}

\newcommand{\abs}[1]{\lvert #1 \rvert}
\newcommand{\bigabs}[1]{\bigl\lvert #1 \bigr\rvert}

\newcommand{\eps}{\varepsilon}

\newcommand{\tensor}{\otimes}
\newcommand{\del}{\partial}

\newcommand{\dbar}{\bar{\partial}}

\newcommand{\dz}{\mathit{dz}}
\newcommand{\dzb}{d\bar{z}}

\newcommand{\NN}{\mathbb{N}}
\newcommand{\ZZ}{\mathbb{Z}}

\newcommand{\RR}{\mathbb{R}}
\newcommand{\CC}{\mathbb{C}}
\newcommand{\HH}{\mathbb{H}}

\newcommand{\menge}[2]{\bigl\{ \thinspace #1 \thinspace\thinspace \big\vert%
\thinspace\thinspace #2 \thinspace \bigr\}}

\newcommand{\MENGE}[2]{\left\{ \thinspace #1 \thinspace\thinspace \middle\vert%
\thinspace\thinspace #2 \thinspace \right\}}

\DeclareMathOperator{\rk}{rk}

\DeclareMathOperator{\ad}{ad}

\DeclareMathOperator{\tr}{tr}

\DeclareMathOperator{\Res}{Res}

\DeclareMathOperator{\id}{id}

\renewcommand{\Im}{\operatorname{Im}}
\renewcommand{\Re}{\operatorname{Re}}

\DeclareMathOperator{\Sym}{Sym}
\DeclareMathOperator{\gr}{gr}

\DeclareMathOperator{\End}{End}
\DeclareMathOperator{\Hom}{Hom}

\DeclareMathOperator{\GL}{GL}
\DeclareMathOperator{\SL}{SL}

\newcommand{\define}[1]{{\boldmath\textbf{#1}}}

\newcommand{\glie}{\mathfrak{g}}

\newcommand{\sltwo}{\mathfrak{sl}_2(\CC)}

\newcommand{\shf}[1]{\mathscr{#1}}
\newcommand{\OX}{\shf{O}_X}
\newcommand{\OmX}{\Omega_X}

\newcommand{\restr}[1]{\big\vert_{#1}}

\newcommand{\argbl}{-}

\def\overbar#1#2#3{{%
	\setbox0=\hbox{\displaystyle{#1}}%
	\dimen0=\wd0
	\advance\dimen0 by -#2 
	\vbox {\nointerlineskip \moveright #3 \vbox{\hrule height 0.3pt width \dimen0}%
		\nointerlineskip \vskip 1.5pt \box0}%
}}

\newcommand{\dst}{\Delta^{\ast}}

\newcommand{\into}{\hookrightarrow}

\newcommand{\inner}[2]{\langle #1, #2 \rangle}
\newcommand{\biginner}[2]{\bigl\langle #1, #2 \bigr\rangle}

\newcommand{\tl}{t_{\ast}}

\newcommand{\shE}{\shf{E}}

\newcommand{\shO}{\shf{O}}

\counterwithin{equation}{paragraph}
\counterwithin{subsection}{section}
\counterwithin{figure}{section}
\newtheorem{thm}{Theorem}[section]
\newtheorem*{thm*}{Theorem}

\newtheorem*{lem*}{Lemma}
\newtheorem{prop}[thm]{Proposition}
\newtheorem*{prop*}{Proposition}

\newtheorem*{cor*}{Corollary}

\declaretheoremstyle[numbered=yes,headformat=\NAME\NOTE,numberwithin=paragraph,%
	bodyfont=\normalfont\itshape,postheadspace={ },%
	spaceabove=\topsep,spacebelow=\topsep]{par-thm}
\declaretheoremstyle[numbered=yes,headformat=\NAME\NOTE,numberwithin=paragraph,%
bodyfont=\normalfont,postheadspace={ },%
	spaceabove=\topsep,spacebelow=\topsep]{par-def}
\declaretheoremstyle[numbered=yes,headformat=\NAME\NOTE,numberwithin=paragraph,%
	headfont=\normalfont\itshape,bodyfont=\normalfont,postheadspace={ },
	spaceabove=\topsep,spacebelow=\topsep]{par-exa}
\declaretheorem[name=Theorem,style=par-thm,preheadhook={},%
	postheadhook={\leavevmode}]{pthm} 
\declaretheorem[name=Lemma,style=par-thm,preheadhook={},%
	postheadhook={\leavevmode}]{plem} 
\declaretheorem[name=Corollary,style=par-thm,preheadhook={},%
	postheadhook={\leavevmode}]{pcor} 
\declaretheorem[name=Proposition,style=par-thm,preheadhook={},%
	postheadhook={\leavevmode}]{pprop} 
\declaretheorem[name=Definition,style=par-def,preheadhook={},%
	postheadhook={\leavevmode}]{pdfn} 
\declaretheorem[name=Example,style=par-exa,preheadhook={},%
	postheadhook={\leavevmode}]{pexa} 

\theoremstyle{definition}

\newtheorem*{dfn*}{Definition}

\newtheorem*{conj*}{Conjecture}

\theoremstyle{remark}

\newtheorem*{exa*}{Example}
\newtheorem*{problem*}{Problem}
\newtheorem*{note}{Note}

\theoremstyle{plain}

\counterwithout{paragraph}{subsubsection}

\newcommand{\newpar}{\paragraph{\hspace{-1em}.}}
\makeatletter 
\renewcommand\paragraph{\@startsection{paragraph}{4}{\z@}%
	{1.25ex \@plus.2ex \@minus.2ex}{-0.5em}%
	{\normalfont\normalsize\bfseries}} 
\makeatother

\crefformat{paragraph}{#2\S#1#3}
\crefformat{section}{#2chapter~#1#3}
\crefformat{subsection}{#2section~#1#3}
\crefformat{pthm}{#2\S#1\ Theorem#3}
\crefformat{plem}{#2\S#1\ Lemma#3}
\crefformat{pcor}{#2\S#1\ Corollary#3}
\crefformat{pprop}{#2\S#1\ Proposition#3}
\crefformat{pdfn}{#2\S#1\ Definition#3}
\crefformat{pexa}{#2\S#1\ Example#3}

\Crefformat{paragraph}{#2\S#1#3}
\Crefformat{section}{#2Chapter~#1#3}
\Crefformat{subsection}{#2Section~#1#3}
\Crefformat{pthm}{#2\S#1\ Theorem#3}
\Crefformat{plem}{#2\S#1\ Lemma#3}
\Crefformat{pcor}{#2\S#1\ Corollary#3}
\Crefformat{pprop}{#2\S#1\ Proposition#3}
\Crefformat{pdfn}{#2\S#1\ Definition#3}
\Crefformat{pexa}{#2\S#1\ Example#3}

\makeatletter
\let\old@caption\caption
\renewcommand*{\caption}[1]{%
	\setcounter{figure}{\value{equation}}%
	\stepcounter{equation}%
	\old@caption{#1}\relax%
}
\makeatother

\newcounter{intro}

\newtheorem{intro-conjecture}[intro]{Conjecture}
\newtheorem{intro-corollary}[intro]{Corollary}
\newtheorem{intro-theorem}[intro]{Theorem}

\newcommand{\parref}[1]{\hyperref[#1]{\S\ref*{#1}}}
\newcommand{\chapref}[1]{\hyperref[#1]{Chapter~\ref*{#1}}}

\makeatletter
\newcommand*\if@single[3]{%
  \setbox0\hbox{${\mathaccent"0362{#1}}^H$}%
  \setbox2\hbox{${\mathaccent"0362{\kern0pt#1}}^H$}%
  \ifdim\ht0=\ht2 #3\else #2\fi
  }
\newcommand*\rel@kern[1]{\kern#1\dimexpr\macc@kerna}
\newcommand*\widebar[1]{\@ifnextchar^{{\wide@bar{#1}{0}}}{\wide@bar{#1}{1}}}
\newcommand*\wide@bar[2]{\if@single{#1}{\wide@bar@{#1}{#2}{1}}{\wide@bar@{#1}{#2}{2}}}
\newcommand*\wide@bar@[3]{%
  \begingroup
  \def\mathaccent##1##2{%
    \if#32 \let\macc@nucleus\first@char \fi
    \setbox\z@\hbox{$\macc@style{\macc@nucleus}_{}$}%
    \setbox\tw@\hbox{$\macc@style{\macc@nucleus}{}_{}$}%
    \dimen@\wd\tw@
    \advance\dimen@-\wd\z@
    \divide\dimen@ 3
    \@tempdima\wd\tw@
    \advance\@tempdima-\scriptspace
    \divide\@tempdima 10
    \advance\dimen@-\@tempdima
    \ifdim\dimen@>\z@ \dimen@0pt\fi
    \rel@kern{0.6}\kern-\dimen@
    \if#31
      \overline{\rel@kern{-0.6}\kern\dimen@\macc@nucleus\rel@kern{0.4}\kern\dimen@}%
      \advance\dimen@0.4\dimexpr\macc@kerna
      \let\final@kern#2%
      \ifdim\dimen@<\z@ \let\final@kern1\fi
      \if\final@kern1 \kern-\dimen@\fi
    \else
      \overline{\rel@kern{-0.6}\kern\dimen@#1}%
    \fi
  }%
  \macc@depth\@ne
  \let\math@bgroup\@empty \let\math@egroup\macc@set@skewchar
  \mathsurround\z@ \frozen@everymath{\mathgroup\macc@group\relax}%
  \macc@set@skewchar\relax
  \let\mathaccentV\macc@nested@a
  \if#31
    \macc@nested@a\relax111{#1}%
  \else
    \def\gobble@till@marker##1\endmarker{}%
    \futurelet\first@char\gobble@till@marker#1\endmarker
    \ifcat\noexpand\first@char A\else
      \def\first@char{}%
    \fi
    \macc@nested@a\relax111{\first@char}%
  \fi
  \endgroup
}
\makeatother

\newcommand{\wbar}[1]{\mathpalette\dowidebar{#1}}
\newcommand{\dowidebar}[2]{\widebar{#1#2}}

\mapcitekey{Schmid:VHS}{Schmid}
\mapcitekey{Mochizuki:AsymptoticBehaviorI}{Mochizuki}
\mapcitekey{Simpson:HarmonicBundlesCurves}{Simpson}
\mapcitekey{Kashiwara:AsymptoticBehavior}{Kashiwara}
\mapcitekey{Berndtsson:Things}{Berndtsson}
\mapcitekey{Ahlfors:SchwarzLemma}{Ahlfors}
\mapcitekey{Deligne:HodgeII}{Deligne}
\mapcitekey{Cornalba+Griffiths:AnalyticCycles}{CG}
\mapcitekey{Simpson:ConstructingVHS}{SimpsonVHS}
\mapcitekey{Deligne:PositiviteSignes}{Deligne-signs}
\mapcitekey{Deligne:EquationsDifferentielles}{Deligne-eq}
\mapcitekey{Cattani+Kaplan+Schmid:Degeneration}{CKS}
\mapcitekey{Deng:NilpotentOrbit}{Deng}
\mapcitekey{Griffiths+Schmid:RecentDevelopments}{GS}
\mapcitekey{Cattani+Kaplan:Luminy}{CK}
\mapcitekey{Demailly:ComplexGeometry}{Demailly}
\mapcitekey{deCataldo+Migliorini:HodgeTheoryMaps}{dCM}
\mapcitekey{Cattani:MixedLefschetzTheorems}{Cattani}
\mapcitekey{Mochizuki:AsymptoticBehaviorI}{MochizukiI}
\mapcitekey{Mochizuki:AsymptoticBehaviorII}{MochizukiII}
\mapcitekey{Mochizuki:Nilpotent}{Mochizuki-nilpotent}
\mapcitekey{Morgan:AlgebraicTopology}{Morgan}
\mapcitekey{Zucker:DegeneratingCoefficients}{Zucker}

\newcommand{\Vb}{\wbar{V}}

\newcommand{\thetast}{\theta^{\ast}}
\newcommand{\delb}{\wbar{\del}}
\newcommand{\zb}{\wbar{z}}
\newcommand{\Ast}{A^{\ast}}
\newcommand{\Higgst}{\Higg^{\ast}}
\newcommand{\Higg}{A}
\newcommand{\ABS}[1]{\left\lvert #1 \right\rvert}
\newcommand{\Abs}[1]{\bigl\lvert #1 \bigr\rvert}
\DeclareMathOperator{\Aut}{Aut}
\newcommand{\dy}{\, \mathit{dy}}
\newcommand{\half}{\frac{1}{2}}

\newcommand{\Dch}{\check{D}}
\newcommand{\mlie}{\mathfrak{m}}
\newcommand{\dDch}{d_{\Dch}}
\DeclareMathOperator{\Ad}{Ad}
\renewcommand{\dy}{\mathit{dy}}
\newcommand{\alphamax}{\alpha_{\mathrm{max}}}
\newcommand{\alphamin}{\alpha_{\mathrm{min}}}
\newcommand{\Flim}{F_{\mathrm{lim}}}

\newcommand{\Fblim}{\Fb_{\mathrm{lim}}}
\newcommand{\dmu}{\,d\mu}
\newcommand{\vb}{\bar{v}}
\renewcommand{\sltwo}{\mathfrak{sl}_2}

\newcommand{\Ilim}{I_{\mathrm{lim}}}

\newcommand{\Hsl}{\mathsf{H}}
\newcommand{\Xsl}{\mathsf{X}}
\newcommand{\Ysl}{\mathsf{Y}}

\newcommand{\wsl}{\mathsf{w}}
\newcommand{\Fb}{\overline{F}\vphantom{F}}
\renewcommand{\Lie}{\operatorname{Lie}}
\newcommand{\Fsh}{F_{\sharp}}

\newcommand{\Csh}{C_{\sharp}}

\newcommand{\tb}{\bar{t}}
\newcommand{\dtb}{d\tb}
\newcommand{\dt}{\mathit{dt}}
\newcommand{\hphi}{h^{\varphi}}
\newcommand{\Thetaphi}{\Theta^{\varphi}}

\newcommand{\Phinil}{\Phi_{\mathrm{nil}}}
\newcommand{\PhiSH}{\hat{\Phi}_{S,H}}
\newcommand{\PhiH}{\hat{\Phi}_H}
\newcommand{\qlie}{\mathfrak{q}}

\newcommand{\Ib}{\bar{I}}
\newcommand{\Hb}{\bar{H}}
\newcommand{\VR}{V_{\RR}}
\newcommand{\Gammat}{\tilde{\Gamma}}

\begin{document}

%========================================================
\title{Degenerating complex variations of Hodge structure in dimension one}

\author{%
 Claude Sabbah \\
 Centre de Math\'ematiques Laurent Schwartz, \'Ecole Polytechnique \\
 \texttt{claude.sabbah@polytechnique.edu}
 \and 
 Christian Schnell \\
 Department of Mathematics, Stony Brook University \\
 \texttt{christian.schnell@stonybrook.edu}
}

\date{\today}
\maketitle
%========================================================

\section{Introduction}

\subsection{Summary of the paper}

\newpar
In this paper, we analyze the behavior of polarized complex variations
of Hodge structure on the punctured unit disk. For \emph{integral} variations of
Hodge structure, this analysis was carried out by Wilfried Schmid \cite{Schmid} in
his famous article \emph{Variation of Hodge structure: the singularities of the period
mapping}, one of the central works in modern Hodge theory. The restriction to
integral variations of Hodge structure is natural from the point of view of geometry;
it also implies that the eigenvalues of the monodromy transformation
are roots of unity, and this fact makes things technically easier.
Nevertheless, there are two reasons why we want to remove this restriction now. First, we
need the results for complex variations of Hodge structure to set up the theory of
complex mixed Hodge modules, the topic of our ``Mixed Hodge Module Project''. Second,
it turns out that working with complex variations clarifies many aspects of the
theory that were somewhat obscured by the presence of an integral structure,
especially the central role played by the Hodge metric.

\newpar
Analytically, the behavior of a polarized variation of Hodge structure on the
punctured disk turns out to be surprisingly simple. Before getting into any details,
let us therefore quickly highlight the two results that are, conceptually, the most important.
A complex variation of Hodge structure is a smooth vector bundle $E$ with a flat
connection $d$, and a decomposition into smooth subbundles $E^{p,q}$ that interacts
with the connection in a certain way. The polarization defines a smooth metric on the
bundle $E$, called the Hodge metric. In order to compare different fibers of $E$, we
use the vector space $V$ of multi-valued flat sections as a reference frame; together
with the monodromy transformation $T$, it contains the same information as the flat
bundle $(E,d)$. The \emph{first result} is that, in this reference frame, the Hodge
metric degenerates in a very simple manner: near the origin, the Hodge metric of any
nonzero multi-valued flat section grows or decays like a power of $\abs{\log
\abs{t}}$, where $t$ is the coordinate on the disk; the exponents are determined by
the monodromy transformation $T$. This means that, up to a constant, the asymptotic
behavior of the Hodge metric only depends on the underlying flat bundle.

\newpar 
The fact that the Hodge metric grows and decays at different rates of course
prevents the Hodge structures from having any sort of limit as $t \to 0$. We then
take the natural step of ``rescaling'' the Hodge metric, in order to even out these
different rates. The rescaling can be done by moving the Hodge structures by
elements of the symmetry group of the period domain; it involves a choice of
splitting for the filtration by order of growth of the Hodge metric.
The \emph{second result} is that the rescaled Hodge structures converge to a limit, which
is again a polarized Hodge structure of the same type. In particular, the asymptotic
behavior of the Hodge metric also controls the asymptotic behavior of the Hodge
structures in the variation. Without rescaling, one gets a ``limiting'' mixed
Hodge structure on the vector space $V$, whose weight filtration is, up to a shift, the
filtration by order of growth of the Hodge metric. Each subquotient of the weight
filtration is again a polarized Hodge structure; the general principle at work,
coming from analysis, is that one can only expect to get a meaningful limit when
everything has the same order of growth.

\subsection{What is new?}

\newpar
For the benefit of those readers who are already familiar with Schmid's results, we
briefly summarize the new features of our approach. We mentioned already that we
treat arbitrary polarized complex variations of Hodge structure on the punctured
disk, without assuming that the eigenvalues of the monodromy transformation are roots
of unity. In this generality, we give new -- and, we think, more conceptual -- proofs
for all the major results in Schmid's paper, such as the estimates for the rate of
growth of the Hodge norm; the existence of a limiting mixed Hodge structure; the
nilpotent orbit theorem; and a simplified (but still sufficiently powerful) version
of the $\SL(2)$-orbit theorem. The results as such are of course just special cases
of Mochizuki's monumental work on tame harmonic bundles \cite{MochizukiI}, but we
think that it is worthwhile to have a self-contained and simple
treatment.\footnote{In several places, for example \cite[p.~453]{Zucker}, it is claimed
	that Schmid's results carry over to polarized variations of Hodge structure whose
monodromy is not quasi-unipotent; but despite considerable effort, we have not been
able to adapt Schmid's proof of the nilpotent orbit theorem to this setting.}

\newpar
Our starting point is a direct proof for the Hodge norm estimates, based on
calculations with harmonic bundles and on a comparison with certain model variations
of Hodge structure. Technically, the crucial point is Simpson's ``basic estimate''
for the Hodge norm of the Higgs field, which is proved using Ahlfors' lemma. In
Schmid's paper, the Hodge
norm estimates are deduced from the orbit theorems, and therefore appear towards the
end; here, they stand at the beginning.  We feel that this gives a
better conceptual explanation for the appearance of the monodromy weight filtration. 
Next, we give a different proof for the nilpotent orbit theorem, based on curvature
properties of the Hodge metric and $L^2$-extension theorems; analytically, the key point is
that one can make the curvature of the Hodge metric either positive or negative by
multiplying by a suitable power of $\abs{\log \abs{t}}$; this general idea is due to
Cornalba-Griffiths \cite{CG} and Simpson \cite{SimpsonVHS}. The parameter dependence
of various constants, important for extending the theory to more than one dimension,
is handled by using the maximum principle. We then use the Hodge norm estimates and
the nilpotent orbit theorem to prove that the rescaled period mapping has a
well-defined limit inside the period domain. 

\newpar
From the convergence of the rescaled period mapping, we deduce the
existence of a limiting mixed Hodge structure; the argument is a pleasant mix of
linear algebra and representation theory, and unlike in Schmid's paper, does not rely
on the $\SL(2)$-orbit theorem. (This is useful because there is a simplified proof
for the $\SL(2)$-orbit theorem by Cattani-Kaplan-Schmid \cite[\S6]{CKS}, whose input
is the existence of the limiting mixed Hodge structure.) We also show that, after a
choice of splitting, the space of multi-valued flat sections becomes what we call a
``polarized $\sltwo$-Hodge structure''; this is the same kind of structure that one
has on the cohomology of a compact K\"ahler manifold. Finally, we prove a cheap version
of the $\SL(2)$-orbit theorem that applies to an arbitrary splitting of the monodromy
weight filtration (instead of the $\SL(2)$-splitting). The degree of approximation is
slightly worse than in Schmid's version, but in return, the proof is much easier. As
an application, we describe the asymptotic behavior of the Hodge metric in the frame
of multi-valued flat sections, something that is only implicit in Schmid's paper. 

\subsection{Acknowledgements}

\newpar
This paper has its immediate origins in a graduate course about variations of Hodge
structure that Ch.S.~taught at Stony Brook University in Fall 2019. We take this
opportunity to thank all the students in the course for their interest, and for
their patience while the instructor worked through the details of Schmid's paper.
Ch.S.~thanks Wilfried Schmid for answering some questions about
the nilpotent orbit theorem; Takuro Mochizuki for a useful discussion about the
definition of the limiting Hodge filtration; Ding\-xing Zhang for organizing a small
reading seminar
about Schmid's paper back in 2014; Ruijie Yang for helping with the calculations for
harmonic bundles; and Bruno Klingler for several conversations about variations of Hodge
structure and harmonic bundles. He also thanks Ya Deng for a very useful email exchange in
2020 about the nilpotent orbit theorem, and especially for sharing an early version
of a preprint \cite{Deng} that suggested the use of H\"ormander's $L^2$-estimates in
this context.

\newpar
During the preparation of this paper, Ch.S.~was partially supported by NSF grant
DMS-1551677 and by a Simons Fellowship. He thanks the National Science Foundation and
the Simons Foundation for their financial support. Large parts of this paper were
written during a sabbatical stay at the Max-Planck-Institute for Mathematics and at
the Kavli Institute for the Physics and Mathematics of the Universe, and Ch.S.~thanks
both of these institutions for providing him with excellent working conditions.

\section{Overview of the results}

\subsection{Complex variations of Hodge structure}

\newpar \label{par:HS}
Since we are going to be working with complex Hodge structures, let us briefly recall the
definition. For the purposes of this text, a \define{Hodge structure} of weight $n$
on a complex vector space $V$ is a decomposition
\[
	V = \bigoplus_{p+q=n} V^{p,q},
\]
and a \define{polarization} is a hermitian pairing $Q \colon V \tensor_{\CC} \Vb \to
\CC$ such that the decomposition is orthogonal with respect to $Q$, and such that
$(-1)^q Q$ is positive definite on the subspace $V^{p,q}$. The \define{Hodge norm} of
a vector $v \in V$ is 
\[
	\norm{v}^2 = \sum_{p+q=n} (-1)^q \, Q(v^{p,q}, v^{p,q}).
\]
Unlike in the case of real Hodge structures, it takes \emph{two} filtrations to
describe an arbitrary complex Hodge structure: the usual Hodge filtration
\[
	F^p = F^p V = \bigoplus_{i \geq p} V^{i,n-i}
\]
and the conjugate Hodge filtration
\[
	\Fb^q = \Fb^q V = \bigoplus_{j \geq q} V^{n-j,j},
\]
because $V^{p,q} = F^p \cap \Fb^q$. When a polarization is given, the Hodge
structure is still determined by the Hodge filtration $F = F^{\bullet} V$ alone: the
reason is that
\[
	\Fb^q = \menge{v \in V}{\text{$Q(v,x) = 0$ for all $x \in F^{n-q+1}$}}.
\]
We are going to use the notation $\norm{v}_F^2$ whenever we want to emphasize the
dependence of the Hodge norm on the Hodge filtration.

\begin{pexa}
A Hodge structure is called \define{real} if $V = \VR \tensor_{\RR} \CC$ for an
$\RR$-vector space $\VR$, and if $\Fb = \sigma(F)$, where $\sigma \in \End_{\RR}(V)$
is the conjugation operator $\sigma(v \tensor z) = v \tensor \zb$. In the real case,
one has the familiar Hodge symmetry $V^{q,p} = \sigma \bigl( V^{p,q} \bigr)$ for
every $p+q=n$.
\end{pexa}

\newpar
The paper contains a certain amount of analysis, and for that reason, it will be
convenient to describe variations of Hodge structure in the language of smooth vector
bundles. Let $E$ be a smooth vector bundle on a complex manifold $X$. We denote
by $A^k(X, E)$ the space of smooth $k$-forms with coefficients in $E$, and by
$A^{i,j}(X, E)$ the space of smooth $(i,j)$-forms with coefficients in $E$. With this
notation, a \define{variation of Hodge structure} of weight $n$ on $E$ is decomposition
\[
	E = \bigoplus_{p+q=n} E^{p,q}
\]
into smooth subbundles, together with a flat connection $d \colon A^0(X, E) \to
A^1(X, E)$ that maps each $A^0(X, E^{p,q})$ into the direct sum of the subspaces
\[
	A^{1,0}(X, E^{p,q}) \oplus A^{1,0}(X, E^{p-1,q+1}) \oplus 
	A^{0,1}(X, E^{p,q}) \oplus A^{0,1}(X, E^{p+1,q-1}).
\]
In the analysis that follows, the most important component of the connection turns
out to be the \define{Higgs field}, which is the linear operator
\[
	\theta \colon A^0(X, E^{p,q}) \to A^{1,0}(X, E^{p-1,q+1}).
\]
A \define{polarization} is a hermitian pairing 
\[
	Q \colon A^0(X, E) \tensor_{A^0(X)} \wbar{A^0(X, E)} \to A^0(X)
\]
that is flat with respect to $d$, such that the expression
\[
	h(v,w) = \sum_{p+q=n} (-1)^q \, Q(v^{p,q}, w^{p,q}) 
		\quad \text{for $v,w \in A^0(X,E)$}
\]
defines a smooth hermitian metric on the bundle $E$, and the different subbundles
$E^{p,q}$ are orthogonal to each other with respect to $h$. This metric is called the
\define{Hodge metric} on the bundle $E$; for the sake of clarity, we sometimes denote
it by the symbol $h_E$.

\newpar
The relation with the more familiar holomorphic description comes from
decomposing $d = d' + d''$ into its $(1,0)$-component $d' \colon A^0(X, E) \to
A^{1,0}(X, E)$ and its $(0,1)$-component $d'' \colon A^0(X, E) \to A^{0,1}(X, E)$.
Then $d''$ gives $E$ the structure of a holomorphic vector bundle, which we denote by
the symbol $\shE$; and $d'$ defines an integrable holomorphic connection $\nabla
\colon \shE \to \OmX^1 \tensor_{\OX} \shE$ on this bundle. The condition on $d$ is 
saying that the Hodge bundles 
\[
	F^p E = \bigoplus_{i \geq p} E^{i,n-i}
\]
come from holomorphic subbundles $F^p \shE$, and that the connection satisfies Griffiths'
transversality relation $\nabla(F^p \shE) \subseteq \OmX^1 \tensor_{\OX} F^{p-1} \shE$.
From this point of view, the Higgs field is simply the holomorphic operator
\[
	\theta \colon F^p \shE/F^{p+1} \shE \to \OmX^1 \tensor_{\OX} F^{p-1} \shE / F^p \shE
\]
induced by the action of the connection $\nabla$.

\subsection{Multi-valued flat sections and monodromy}

\newpar
Our main concern in this paper is the asymptotic behavior of a polarized variation
of Hodge structure on the punctured disk. Let
\[
	\dst = \menge{t \in \CC}{0 < \abs{t} < 1}
\]
be the punctured unit disk, and let $E$ be a polarized variation of Hodge structure
on $\dst$. In order to compare the Hodge structures on different fibers of the bundle
$E$, we need a common reference frame; the most natural choice is the space of all
multi-valued flat sections of the flat bundle $(E,d)$.

\newpar
Since the fundamental group of $\dst$ is cyclic, a flat bundle $(E,d)$ on the
punctured disk is uniquely determined by the pair $(V,
T)$, where $V$ is the complex vector space of all multi-valued flat sections,
and $T \in \GL(V)$ is the monodromy transformation. 
To define $V$ and $T$ more precisely, recall that the universal covering space of
$\dst$ is naturally the \emph{left} half-plane
\[
	\HH = \menge{z \in \CC}{\Re z < 0},
\]
which maps to $\dst$ via the exponential function $\exp \colon \HH \to \dst$. Let $V$
be the complex vector space of all \define{multi-valued flat sections} of $E$;
concretely, $V$ is defined to be the space of $d$-flat sections of the pullback
bundle $\exp^{\ast} \! E$. The polarization on $E$ induces a nondegenerate hermitian form 
\[
	Q \colon V \tensor_{\CC} \bar{V} \to \CC, 
\]
and translation by $2 \pi i$ induces the \define{monodromy operator}
\[
	T = T_s \cdot T_u = T_s \cdot e^{2 \pi i \, N} \in \GL(V).
\]
Here $T_s$ is diagonalizable and $N$ is nilpotent, and one has
\[
	Q(T_s v, T_s w) = Q(v,w)
	\quad \text{and} \quad
	Q(N v, w) = Q(v, N w).
\]
Note that our definition of the nilpotent operator $N$ differs from Schmid's
definition \cite[(4.6)]{Schmid} by a factor of $2 \pi i$, because $\log T_u = 2 \pi i
N$. The advantage is that this makes $N$ independent of the choice of $i =
\sqrt{-1}$.

\newpar
The fact that the flat bundle $(E,d)$ is part of a polarized variation of Hodge
structure puts the following restriction on the monodromy transformation (see
\Cref{prop:monodromy-theorem}).

\begin{pprop}[Monodromy theorem]  
	If $\lambda \in \CC$ is an eigenvalue of $T$, then $\abs{\lambda} = 1$.
\end{pprop}

For \emph{integral} variations of Hodge structure, it follows (because of Kronecker's
theorem) that all eigenvalues of $T$ are roots of unity. This fact is
traditionally called the ``monodromy theorem''; Schmid \cite[Lem.~4.5]{Schmid}
attributes it to Borel.

\subsection{The Hodge norm estimates}

\newpar
Conceptually, the most important piece of the theory are the Hodge norm estimates,
which describe the rate of growth of the pointwise Hodge norm for multi-valued
flat sections. This result is the starting point of our treatment of the theory; in
Schmid's work, it comes towards the end \cite[Thm.~6.6]{Schmid}. For a multi-valued flat
section $v \in V$, the Hodge metric gives us a smooth function
\[
	h(v,v) = \sum_{p+q=n} (-1)^q \, Q(v^{p,q}, v^{p,q}) \colon \HH \to [0, \infty).
\]
The asymptotic behavior of this function as $\abs{\Re z} \to \infty$ turns out to be
surprisingly simple -- in fact, it is completely controlled by the following elementary
construction. The nilpotent operator $N \in \End(V)$ determines an increasing
filtration $W_{\bullet} = W_{\bullet} V$, called the \define{monodromy weight
filtration}. This filtration is uniquely characterized by the following two properties:
\begin{aenumerate}
\item For every $k \in \ZZ$, one has $N W_k \subseteq W_{k-2}$.
\item For every $k \in \NN$, the induced operator $N^k \colon \gr_k^W \to \gr_{-k}^W$
	is an isomorphism.
\end{aenumerate}
An explicit formula for the monodromy weight filtration is
\[
	W_k = \sum_{j \in \NN} N^j \bigl( \ker N^{k+2j+1} \bigr) \quad
	\text{for $k \in \ZZ$.}
\]
The content of the Hodge norm estimates is that the monodromy weight filtration
precisely captures the asymptotic behavior of the Hodge norm.

\begin{pthm}[Hodge norm estimates]  \label{thm:Hodge-norm-intro}
	For $v \in V$ any multi-valued flat section,
	\[
		v \in W_k \setminus W_{k-1} \quad \Longleftrightarrow \quad
		h(v,v) \sim \abs{\Re z}^k,
	\]
	as long as $\Im z$ stays in a fixed interval.
\end{pthm}

On the punctured disk, the statement is that $h(v,v) \sim \abs{\log{\abs t}}^k$ on
each angular sector.

\newpar
The main ingredient of the proof is the following basic estimate
\cite[\S2]{Simpson} for the pointwise Hodge norm of the Higgs field $\theta$, viewed
as a smooth section of the induced variation of Hodge structure on the bundle
$\End(E)$. It is analogous to the distance-decreasing property of period mappings
used in Schmid's paper \cite[Cor.~3.17]{Schmid}.

\begin{pthm}[Basic estimate] \label{thm:basic-estimate}
	One has the inequality
	\[
		h_{\End(E)} \bigl( \theta_{\partial/\partial t}, \theta_{\partial/\partial_t}
		\bigr) \leq \frac{C_0^2}{\abs{t}^2 (\log \abs{t})^2},
	\]
	where $C_0 = \frac{1}{2} \sqrt{\binom{r+1}{3}}$ and $r = \rk E$.
\end{pthm}

We use some elementary calculations with harmonic bundles, already done by Simpson
\cite[\S2]{Simpson}, together with the fact that $\theta$ is nilpotent, to show that 
the Laplacian of the smooth function $f = h_{\End(E)} \bigl(
\theta_{\partial/\partial t}, \theta_{\partial/\partial_t} \bigr)$ satisfies the
differential inequality
\[
	\Delta \log f \geq 8 f / C_0^2.
\]
The basic estimate then follows by applying Ahlfors' lemma \cite{Ahlfors}. The 
value of the constant $C_0$ is optimal; more important than the exact value,
however, is the fact that $C_0$ depends on nothing but the rank of the bundle $E$.

\newpar
To prove the Hodge norm estimates, we first establish a special case.

\begin{plem}
	If $v \in V$ is a multi-valued flat section with $Tv = \lambda v$ for some $\lambda
	\in \CC$, then the function $\varphi_v = \log h(v,v)$ remains bounded from above
	as $\abs{\Re z} \to \infty$.
\end{plem}

This is indeed a special case of the Hodge norm estimates, because $Tv = \lambda v$
implies $Nv = 0$, and therefore $v \in W_0$. Some further calculations with harmonic
bundles, together with the basic estimate, show that:
\begin{enumerate}
	\item The function $\varphi_v$ is subharmonic on $\HH$.
	\item The first derivatives of $\varphi_v$ are bounded by a constant times
		$\abs{\Re z}^{-1}$.
	\item One has $\varphi_v(z + 2 \pi i) = \varphi_v(z)$ for every $z \in \HH$.
\end{enumerate}
By a simple calculus argument, these three properties are enough to conclude that
$\varphi_v$ is bounded from above as $\abs{\Re z} \to \infty$.

\newpar
This special case is already enough to prove that the asymptotic behavior of
the Hodge norm only depends on the underlying flat bundle $(E,d)$. 

\begin{pthm}[Comparison theorem] 
	Let $E_1$ and $E_2$ be two polarized variations of Hodge structure on the
	punctured disk.  If $(E_1,d_1) \cong (E_2,d_2)$ as flat bundles, then the Hodge
	metrics $h_1$ and $h_2$ are mutually bounded, up to a constant, as $t \to 0$.
\end{pthm}

This follows by considering the induced variation of Hodge structure on the flat
bundle $\Hom(E_1, E_2)$. To prove the Hodge norm estimates in general, we then
proceed as follows. Using representation theory, we construct a ``model''
variation of Hodge structure on the given flat bundle $(E,d)$, whose Hodge norm has
the correct asymptotic behavior; the comparison theorem guarantees that, up to a
constant, the Hodge norm on $E$ has the same behavior. The construction of the models
is done by decomposing $V$ into irreducible representations of the Lie algebra
$\sltwo(\CC)$, and then explicitly constructing a model variation on each irreducible
representation. The construction uses the theory of $\sltwo$-Hodge
structures, which we study in some depth in \Cref{sec:sltwo}. (See also
\Cref{par:sltwo} in this chapter.)

\begin{pexa}
	In the model variation of Hodge structure coming from the standard representation of
	$\sltwo(\CC)$ on $\CC^2$, the Hodge metric is given by the formula
	\[
		\begin{pmatrix}
			\abs{x} + y^2 \abs{x}^{-1} & -iy \abs{x}^{-1} \\
			i y \abs{x}^{-1} & \abs{x}^{-1}
		\end{pmatrix},
	\]
	where $z = x + i y$. Since $(1,0) \in W_1$ and $(0,1) \in W_{-1}$, this is in
	complete agreement with the Hodge norm estimates.
\end{pexa}

\subsection{Period domains and period mappings}

\newpar
Once we understand the behavior of the Hodge metric, the next question is what
happens to the Hodge structures as $t \to 0$. The answer is best stated in the
language of period domains and period mappings. The space of multivalued flat
sections gives us a trivialization of the bundle $\exp^{\ast} E$, and so we can
consider all the Hodge structures in the variation as living on the same vector space
$V$. From our polarized variation of Hodge structure $E$, we then get a holomorphic
\define{period mapping}
\[
	\Phi \colon \HH \to D.
\]
The \define{period domain} $D$ parametrizes all Hodge structures of weight $n$ on
$V$, with fixed Hodge numbers $\dim V^{p,q}$, that are polarized by the pairing $Q$;
a basic fact is that $D$ is a homogeneous space for the real Lie group
\[
	G = \menge{g \in \GL(V)}{\text{$Q(gv,gw) = Q(v,w)$ for all $v,w \in V$}}.
\]
The period domain is an open subset of the so-called \define{compact dual} $\Dch$,
which parametrizes all decreasing filtrations $F^{\bullet}$ on $V$ such that $\dim
F^p = \dim V^{p,q} + \dim V^{p+1,q-1} + \dotsb$ for all $p \in \ZZ$. The compact dual
is a projective complex manifold, and a homogeneous space for the complex Lie group
$\GL(V)$; the complex structure on the period domain $D$ comes from the embedding $D
\subseteq \Dch$.

\newpar
The period mapping is holomorphic, and its differential is represented by the Higgs
field $\theta$; see \Cref{lem:dPhi} for the precise statement. Moreover, the
monodromy transformation $T \in G$ is defined in such a way that 
\[
	\Phi(z + 2 \pi i) = T \cdot \Phi(z).
\]
The eigenvalues of $T = T_s T_u$ lie on the unit circle, and so we can write $T_s =
e^{2 \pi i \, S}$, where $S \in \End(V)$ is a semisimple operator with
real eigenvalues in a fixed half-open interval of length $1$; then $T = e^{2 \pi
i(S+N)}$. One can interpret the operator $S+N$ as the residue of the connection
$\nabla$ on the canonical extension of the holomorphic
vector bundle $\shE$; see \Cref{par:canonical-extension}.

\subsection{Convergence of the rescaled period mapping}

\newpar \label{par:splitting-intro}
Now we come to a crucial point in the argument. The Hodge norm estimates 
\[ 
	v \in W_k \setminus W_{k-1} \quad \Longleftrightarrow \quad
	\norm{v}_{\Phi(z)}^2 \sim \abs{\Re z}^k
\]
suggest ``rescaling'' the period mapping, in order to even out the behavior of the
Hodge norm on different parts of the weight filtration. To do this
efficiently, we choose a \define{splitting} $H \in \End(V)$ for the weight
filtration, with the following three properties: 
\begin{enumerate}
	\item $H$ is semisimple, with integer eigenvalues.
	\item $W_k = E_k(H) \oplus W_{k-1}$ and $[H,N] = -2N$.
	\item $[H,T_s] = 0$ and $Q(H v, w) + Q(v, H w) = 0$ for all $v,w \in V$.
\end{enumerate}
It is not hard to show that such a splitting always exists (see
\Cref{prop:splitting}); in general, it is far from
unique. For a nonzero multi-valued flat section $v \in E_k(H)$, we then have
\[
	\norm{v}_{\Phi(z)}^2 \sim \abs{\Re z}^k
	\qquad \text{and} \qquad
	e^{-\half \log \abs{\Re z} \, H} v = \abs{\Re z}^{-\frac{k}{2}} v
\]
Consequently, the Hodge norm of $e^{-\half \log \abs{\Re z} \, H} v$ stays bounded as
$\abs{\Re z} \to \infty$. We can achieve the same effect by moving the Hodge
structures by the operator $e^{\half \log \abs{\Re z} \, H}$. Either way, the
conclusion is that the \define{rescaled period mapping}
\begin{equation} \label{eq:PhiSH-intro}
	\hat{\Phi}_{S, H} \colon \HH \to D, \quad 
	\hat{\Phi}_{S, H}(z) 
	= e^{\half \log \abs{\Re z} \, H} e^{-\half(z-\zb)(S+N)} \Phi(z),
\end{equation}
is invariant under the substitution $z \mapsto z + 2 \pi i$, and the Hodge norm 
of every multi-valued flat section is uniformly bounded as $\Re z \to -\infty$. The
two exponential factors are elements of the real Lie group $G$; this is the
reason why the rescaled period mapping stays in $D$.

\newpar
The second key result is that, after rescaling, the polarized Hodge structures
now converge to a well-defined limit. As indicated by the notation, the limit
does depend on the choice of splitting: different splittings lead to different limits.

\begin{pthm}[Convergence of the rescaled period mapping]
	\label{thm:convergence-intro}
	For any choice of splitting $H \in \End(V)$ of the monodromy weight filtration as
	above, the limit
	\[
		e^{-N} F_H = \lim_{\Re z \to -\infty} \PhiSH(z) \in D
	\]
	exists in the period domain. Moreover, the filtration $F_H \in \Dch$
	satisfies
	\[
		T_s(F_H^{\bullet}) \subseteq F_H^{\bullet}, \quad
		H(F_H^{\bullet}) \subseteq F_H^{\bullet}, \quad 
		N(F_H^{\bullet}) \subseteq F_H^{\bullet-1}.
	\]
\end{pthm}

Together with the Hodge norm estimates, this is saying that a polarized variation of
Hodge structure on the punctured disk behaves in a very simple way. Namely, as $t \to
0$, the Hodge metric grows (and decays) at several different rates, and this of course
prevents the Hodge structures from approaching any sort of limit. But once we rescale
to even out the behavior of the Hodge metric, \emph{both} the Hodge metric and the
Hodge structures converge.

\subsection{The nilpotent orbit theorem}

\newpar
The most difficult step in proving the convergence of the rescaled period mapping is the
\define{nilpotent orbit theorem}, an important result in its own right. Recall that
the period mapping satisfies
\[
	\Phi(z + 2 \pi i) = T \cdot \Phi(z),
\]
and that we have $T = e^{2 \pi i(S+N)}$, where $S \in \End(V)$ is semisimple with
real eigenvalues in a fixed interval of length $<1$. Consequently, the mapping
\[
	\HH \to \Dch, \quad z \mapsto e^{-z(S+N)} \Phi(z),
\]
is invariant under $z \mapsto z + 2 \pi i$, and therefore descends
to a holomorphic mapping
\[
	\Psi_S \colon \dst \to \Dch, \quad \Psi_S(e^z) = e^{-z(S+N)} \Phi(z),
\]
sometimes called the \define{untwisted period mapping}. The exponential factor
$e^{-z(S+N)}$ is now an element of the complex Lie group $\GL(V)$, and so the
untwisted period mapping takes values in the compact dual $\Dch$.

\newpar
The nilpotent orbit theorem has two parts. The first half is a convergence
statement: it says that $\Psi_S$ extends holomorphically over the origin. 

\begin{pthm}[Nilpotent orbit theorem, convergence]
	\label{thm:nilpotent-convergence-intro}
	The holomorphic mapping
	\[
		\Psi_S \colon \dst \to \Dch, \quad \Psi_S(e^z) = e^{-z(S+N)} \Phi(z),
	\]
	extends holomorphically across the origin, and the limit filtration $\Psi_S(0) \in
	\Dch$ has the property that $(S+N) \Psi_S^{\bullet}(0) \subseteq
	\Psi_S^{\bullet-1}(0)$.
\end{pthm}

The limit $\Psi_S(0)$ in the nilpotent orbit theorem depends on the choice of $S$,
but one can show that the \define{limiting Hodge filtration} 
\[
	\Flim = \lim_{\Re z \to -\infty} e^{-zN} \Phi(z) 
	= \lim_{\Re z \to -\infty} e^{-\abs{\Re z} \, S} \Psi_S(0) \in \Dch
\]
is independent of $S$. It satisfies 
\[
	T_s(\Flim^{\bullet}) \subseteq \Flim^{\bullet} \quad \text{and} \quad
	N(\Flim^{\bullet}) \subseteq \Flim^{\bullet-1}.
\]
The limiting Hodge filtration does not come from a polarized Hodge structure on $V$
(except when $N = 0$); instead, it is part of a \emph{mixed Hodge structure}, whose
weight filtration is given by the monodromy weight filtration $W_{\bullet-n}$ from
above. We are going to discuss this point in more detail in \Cref{sec:MHS-intro}
below.

\newpar \label{par:convergence-proof}
We can now explain why the rescaled period mapping $\PhiSH \colon \HH \to D$
converges to a limit in $D$. The argument needs both the Hodge norm estimates
(in \Cref{thm:Hodge-norm-intro}) and the nilpotent orbit theorem
(in \Cref{thm:nilpotent-convergence-intro}). The identities $[H,N]
= -2N$ and $[H,S] = 0$ can be used
to write
\[
	\PhiSH(z) = e^{-N} \cdot e^{\half \log \abs{\Re z} \, H} \, 
	e^{-\abs{\Re z} \, S} \, \Psi_S(e^z).
\]
The nilpotent orbit theorem implies that $\Psi_S(e^z) \to \Psi_S(0)$ at a rate
proportional to $\abs{e^z} = e^{-\abs{\Re z}}$. By analyzing the effect of the two
exponential factors, and by using the fact that the eigenvalues of $S$ lie in an
interval of length $<1$, one deduces the existence of the limit
\[
	F_H = \lim_{\Re z \to -\infty} e^{\half \log \abs{\Re z} \, H} \, 
	e^{-\abs{\Re z} \, S} \, \Psi_S(e^z) \in \Dch.
\]
The filtration $F_H \in \Dch$ is derived from the limit $\Psi_S(0) \in \Dch$ in the
nilpotent orbit theorem in two stages:
\begin{enumerate}
	\item The effect of the first limit
		\[
			\Flim = \lim_{\Re z \to -\infty} e^{-\abs{\Re z} \, S} \Psi_S(0)
		\]
		is to make the filtration $\Psi_S(0)$ compatible with the eigenspace
		decomposition of $T_s$, by projecting to the subquotients of the filtration by
		decreasing eigenvalues of $S$. (Decreasing because of the minus sign in the
		exponent.)
	\item The effect of the second limit
		\[
			F_H = \lim_{\Re z \to -\infty} e^{\half \log \abs{\Re z} \, H} \Flim
		\]
		is to make the filtration $\Flim$ compatible with the eigenspace decomposition
		of $H$, by projecting to the subquotients of the weight filtration
		$W_{\bullet}$ (which is the filtration by increasing eigenvalues of $H$).
\end{enumerate}
It follows that $\PhiSH(z)$ converges to $e^{-N} F_H \in \Dch$ as $\Re z \to
-\infty$.  At the same time, the Hodge norm estimates tell us that the rescaled
Hodge metric is comparable to a constant metric as $\Re z \to -\infty$. These two
facts together imply quite easily that $e^{-N} F_H \in D$.

\newpar
The nilpotent orbit theorem also has a second half, which says that the original period
mapping is very closely approximated by the \define{nilpotent orbit}
\[
	\Phinil \colon \HH \to \Dch, \quad \Phinil(z) = e^{zN} \Flim.
\]
In fact, the nilpotent orbit is itself the period mapping of a polarized variation of
Hodge structure (on a punctured disk of smaller radius).

\begin{pthm}[Nilpotent orbit theorem, approximation]
	\label{thm:nilpotent-approximation-intro} There are constants $C > 0$, $x_0 < 0$,
	and $m \in \NN$, such that
	\[
		\Phinil(z) \in D \quad \text{and} \quad
		d_D \bigl( \Phi(z), \Phinil(z) \bigr) \leq C \abs{\Re z}^m e^{-\delta(T)
		\abs{\Re z}}
	\]
	for every $z \in \HH$ with $\Re z \leq x_0$. The constants $C, x_0$ only depend on
	the base point $o \in D$ and on the minimal polynomial of $T \in
	\GL(V)$; the integer $m$ only depends on $r = \rk E$.
\end{pthm}

Here $d_D$ is the $G$-invariant distance function on the period domain, defined by
the Hodge metric (see \eqref{eq:D-distance}). The distance between the period mapping
$\Phi(z)$ and the nilpotent orbit $\Phinil(z)$ is exponentially small; the constant
$\delta(T)$ in the exponent is the minimal distance between consecutive eigenvalues
of $T$ (on the unit circle), divided by $2 \pi$. For developing the theory in more
than one variable, it is very important that the values of the constants are
essentially independent of the variation of Hodge structure.

\newpar
When $N = 0$, the nilpotent orbit theorem is asserting that $\Flim \in D$; in other
words, when the weight filtration is trivial, $\Flim$ is the Hodge filtration of a
polarized Hodge structure of weight $n$. When the weight filtration is not trivial,
what we get instead is a limiting mixed Hodge structure (see \Cref{thm:MHS-intro}).

\begin{note}
	This is not true for $\Psi_S(0)$: even when $N = 0$, the filtration
	$\Psi_S(0)$ is not necessarily the Hodge filtration of a polarized Hodge
	structure (see \Cref{ex:PsiS}).
\end{note}

\newpar
The geometric meaning of the nilpotent orbit theorem is the following. Recall that
the operator $S \in \End(V)$ depended on a choice of half-open interval $I$ of length
$1$. The holomorphic vector bundle $\shE$ has a \define{canonical extension} to a
holomorphic vector bundle $\shE_I$ on the disk. Up to isomorphism, it is
characterized by the fact that the holomorphic connection $\nabla$ becomes a
logarithmic connection 
\[
	\nabla \colon \shE_I \to \Omega_{\Delta}^1(\log 0) \tensor_{\shO_{\Delta}} \shE_I
\]
whose residue $\Res_0 \nabla$ has real eigenvalues in the interval $I$. There is a
distinguished trivialization $\shE_I \cong \shO_{\Delta} \tensor_{\CC} V$ in which
\[
	\nabla(v \tensor 1) = \frac{\dt}{t} \tensor (S+N) v.
\]
The content of \Cref{thm:nilpotent-convergence-intro} is that the Hodge bundles $F^p
\shE$ extends to holomorphic subbundles $F^p \shE_I$ of the canonical extension; and
that the Higgs field
\[
	\theta \colon F^p \shE_I / F^{p+1} \shE_I \to \Omega_{\Delta}^1(\log 0) 
	\tensor_{\shO_{\Delta}} F^{p-1} \shE_I / F^p \shE_I
\]
has a logarithmic pole at $t=0$, with residue $S+N$. The content of
\Cref{thm:nilpotent-approximation-intro} is that if we replace the extended Hodge
bundles $F^p \shE_I$ by trivial subbundles with fiber $\Psi_S^p(0) \subseteq V$,
then we still get a polarized variation of Hodge structure of weight $n$, at least on
a punctured disk of a fixed smaller radius.

\newpar
Our proof of the nilpotent orbit theorem is quite different from Schmid's original
argument \cite[\S8]{Schmid}. It uses in a crucial way the curvature properties of the
Hodge metric. A short calculation (see \Cref{prop:curvature}) shows that the
curvature tensor of the Hodge metric $h_E$ on the bundle $E^{p,q}$ is given by
$\Theta = -(\theta \thetast + \thetast \theta)$, where 
\begin{align*}
	\theta &\colon A^0(\dst, E^{p,q}) \to A^{1,0}(\dst, E^{p-1,q+1}) \\
	\thetast &\colon A^0(\dst, E^{p,q}) \to A^{0,1}(\dst, E^{p+1,q-1}) 
\end{align*}
are the Higgs field and its adjoint. (There are similar formulas for the curvature
tensor on the Hodge bundles $F^p E$ and on the quotient bundles $E/F^p E$.)
Consequently, the expression
\[
	h_E \bigl( \Theta_{\partial/\partial t \wedge \partial/\partial \tb} \, u, u \bigr) = 
	h_E \bigl( \theta_{\partial/\partial t} u, \theta_{\partial/\partial t} u \bigr)
	- h_E \bigl( \thetast_{\partial/\partial \tb} u, \thetast_{\partial/\partial \tb} u \bigr)
\]
is in general neither positive nor negative definite. However, because of the basic
estimate (in \Cref{thm:basic-estimate}), any metric of the form
\[
	h_E \cdot \abs{t}^a (-\log \abs{t})^b \quad \text{with $a \in \RR$ and $b \in \ZZ$}
\]
will have positive curvature for $b \gg 0$, and negative curvature for $b \ll 0$.
This insight already appears in Simpson's work \cite[\S10]{SimpsonVHS}, who
attributes it to the important paper by Cornalba and Griffiths \cite[p.~29]{CG}. We
use both of these properties in our proof.

\newpar
Let us now outline the steps in our proof of the nilpotent orbit theorem. 
The nilpotent orbit theorem can also be deduced from Mochizuki's work on harmonic bundles
\cite{Mochizuki-nilpotent}; another analytic proof using $L^2$-estimates has recently
been given by Deng \cite{Deng}.

\begin{enumerate}
	\item The bundle $\End(E)$ inherits a polarized variation of Hodge structure of
		weight $0$. The Higgs field $\theta_{\partial/\partial t}$ is a holomorphic section
		of $\End(E)^{-1,1}$, and the basic estimate tells us that its
		pointwise Hodge norm is bounded by a constant multiple of $\abs{t}^{-2} (\log
		\abs{t})^{-2}$. The first step is to construct a lifting of $t
		\theta_{\partial/\partial t}$ to a holomorphic section $\vartheta$ of the
		Hodge bundle $F^{-1} \End(E)$, whose $L^2$-norm is controlled by an inequality
		of the form
		\[
			\int_{\dst} h_{\End(E)}(\vartheta, \vartheta) 
			\abs{t}^a (-\log \abs{t})^b \dmu \leq C
		\]
		where $a > -2$ and $b \gg 0$. This uses H\"ormander's $L^2$-estimates, for
		which we include an elementary proof (in \Cref{sec:L2-estimates}); the key point
		is that the metric $h_{\End(E)} \cdot \abs{t}^a (-\log \abs{t})^b$ has positive
		curvature for $b \gg 0$. Using $L^2$-estimates to extend holomorphic bundles
		goes back to Cornalba-Griffiths \cite{CG}, and in the setting of harmonic
		bundles, to Simpson \cite{SimpsonVHS}.
	\item Pulling back to $\HH$ gives us a holomorphic mapping $\vartheta \colon \HH
		\to \End(V)$ that satisfies $\vartheta(z + 2 \pi i) = T \vartheta(z) T^{-1}$.
		After untwisting, we get
		\[
			B \colon \dst \to \End(V), \quad
			B(e^z) = e^{-z(S+N)} \vartheta(z) e^{z(S+N)}.
		\]
		Since the eigenvalues of $S$ lie in an interval of length $<1$, the
		$L^2$-estimate from the first step can be used to show that $B$ is
		square-integrable on a punctured neighborhood of the origin, and therefore extends
		holomorphically to the entire disk.
	\item We can now describe the differential of the untwisted period mapping
		$\Psi_S$ using holomorphic data.
		The holomorphic tangent space to the complex manifold $\Dch$ at the point
		$\Phi(z)$ is isomorphic to $\End(V)/F^0 \End(V)_{\Phi(z)}$, and the
		differential of the period mapping $\Phi \colon \HH \to \Dch$ is equal to 
		\[
			\theta_{\partial/\partial z} \mod F^0 \End(V)_{\Phi(z)}.  
		\]
		A short calculation with derivatives shows that the differential of the mapping
		$z \mapsto \Psi_S(e^z)$ is therefore equal to
		\begin{align*}
			e^{-z(S+N)} \theta_{\partial/\partial z} e^{z(S+N)} &- (S+N) \\
			\equiv B(e^z) &- (S+N) \mod F^0 \End(V)_{\Psi_S(e^z)}.
		\end{align*}
		The point is that the operator on the right-hand side is holomorphic.
	\item Let $g \colon \HH \to \GL(V)$ be the unique holomorphic solution of the
		ordinary differential equation
		\[
			g'(z) = \Bigl( B(e^z) - (S+N) \Bigr) g(z),
		\]
		subject to the initial condition $g(-1) = \id$. A short calculation shows that
		$g(z)^{-1} \Psi_S(e^z)$ must be constant, which means that
		\[
			\Psi_S(e^z) = g(z) \cdot \Psi_S(e^{-1}).
		\]
		The differential equation has a \emph{regular} singular point at $t = 0$, and by the
		basic theory of Fuchsian differential equations, one has
		\[
			g(z) = M(e^z) \cdot e^{z A},
		\]
		with $M \colon \dst \to \GL(V)$ meromorphic and $A \in \End(V)$ constant. Since
		the untwisted period mapping $\Psi_S(e^z)$ is single-valued, it follows that 
		\[
			\Psi_S(t) = M(t) \cdot \Psi_S(e^{-1}),
		\]
		and because $\Dch$ is a projective complex manifold, this is enough to conclude
		that $\Psi_S$ extends holomorphically across the origin. (The close
		relationship between the nilpotent orbit theorem and regularity was already
		observed by Griffiths and Schmid, see \cite[\S9a]{GS} and the references cited
		there.)
	\item The remainder of the argument consists in deriving effective estimates
		for the rate of convergence. Let $P_{\lambda}$ denote the projection to the
		eigenspace $E_{\lambda}(T_s)$. The fact that $\Psi_S(t)$ converges as $t \to 0$
		can be used to show that 
		\begin{align*}
			\norm{\theta_{\partial/\partial z} - N^{-1,1}}_{\Phi(z)}^2 + 
			\sum_{k \leq -2} \norm{N^{k,-k}}_{\Phi(z)}^2
			&\leq C \abs{\Re z}^{2m} e^{-2\delta(T) \abs{\Re z}} \\
			\sum_{k \leq -1} \norm{P_{\lambda}^{k,-k}}_{\Phi(z)}^2
			&\leq C \abs{\Re z}^{2m} e^{-2\delta(T) \abs{\Re z}}
		\end{align*}
		for $\abs{\Re z} \gg 0$; the exponent $m$ only depends on $r = \rk E$, but we
		have no control over the value of the constant $C$. We then use the
		fact that the metric $h_{\End(E)} \cdot (-\log \abs{t})^b$ has negative
		curvature for $b \ll
		0$, together with the maximum principle for subharmonic functions, to deduce
		that the above estimates actually hold on any half-plane of the form $\Re z \leq
		x < 0$, with a constant $C$ that only depends on $x$, $r$, and on the
		minimal polynomial of $T \in \GL(V)$. 
	\item Finally, we consider the curve
		\[
			[0, \infty) \to \Dch, \quad x \mapsto e^{xN} \Phi(z-x),
		\]
		joining the two points $\Phi(z)$ and $\Phinil(z)$. Its derivative is
		controlled by the inequalities in the previous step.
		After integration, we obtain a bound for the distance in $\Dch$ between these
		two points, and together with the Hodge norm estimates, this gives us the
		desired bound for their distance in $D$, provided that $\abs{\Re z}$ is
		sufficiently large.
\end{enumerate}

\newpar
The estimates in the fifth step are saying that, modulo $F^0 \End(V)_{\Phi(z)}$, the
Higgs field $\theta_{\partial/\partial z}$ converges to the
nilpotent operator $N$ at a rate that is exponential in $\abs{\Re z}$. Moreover, the
eigenspace decomposition
\[
	V = \bigoplus_{\abs{\lambda} = 1} E_{\lambda}(T_s)
\]
becomes orthogonal in the limit, at a rate that is exponential in $\abs{\Re
z}$. Since the limiting Hodge filtration $\Flim$ is compatible with this
decomposition, both of these facts are of course also contained in the statement of
the nilpotent orbit theorem itself.

\subsection{The limiting mixed Hodge structure}
\label{sec:MHS-intro}

\newpar
Probably the most widely known result from Schmid's paper \cite[\S6]{Schmid} is
that, in the limit, a polarized variation of Hodge structure on the punctured disk
gives rise to a mixed Hodge structure on $V$. The weight filtration of this
``limiting'' mixed Hodge structure is the shifted monodromy weight filtration
$W_{\bullet-n}$, and the Hodge filtration is the limiting Hodge filtration $\Flim$. 
Schmid obtained this result as a consequence of his $\SL(2)$-orbit theorem.

\newpar
We take a somewhat different approach here, and deduce the existence of the limiting
mixed Hodge structure directly from the convergence of the rescaled period mapping
$\PhiSH \colon \HH \to D$. In fact, we prove that the vector space $V$ is an example
of a \define{polarized $\sltwo$-Hodge structure}; this is the same kind of structure
that one finds on the cohomology of a compact K\"ahler manifold. The proof is
elementary and only uses linear algebra and some basic representation theory of the
Lie algebra $\sltwo(\CC)$, but no further analysis.

\newpar
Let us briefly recall how $\sltwo(\CC)$ enters into the picture. Recall that this Lie
algebra is $3$-dimensional, and is generated by the three matrices
\[
	\Hsl = \begin{pmatrix} 1 & 0 \\ 0 & -1 \end{pmatrix}, \quad
	\Xsl = \begin{pmatrix} 0 & 1 \\ 0 & 0 \end{pmatrix}, \quad
	\Ysl = \begin{pmatrix} 0 & 0 \\ 1 & 0 \end{pmatrix},
\]
subject to the relations $[\Hsl, \Xsl] = 2 \Xsl$, $[\Hsl, \Ysl] = -2\Ysl$, and
$[\Xsl, \Ysl] = \Hsl$. In any finite-dimensional representation, the operator $\Hsl$
acts semisimply with integer eigenvalues, and the eigenspace $E_k(\Hsl)$ are called
the \define{weight spaces}. The splitting $H$ that we chose in \cref{par:splitting-intro}
uniquely determines a representation 
\[
	\rho \colon \sltwo(\CC) \to \End(V)
\]
of this Lie algebra such that $\rho(\Hsl) = H$ and $\rho(\Ysl) = -N$; the minus
sign is dictated by our sign conventions for Hodge structures. It lifts to a
representation of the Lie group $\SL_2(\CC)$; in particular, we have the \define{Weil
element} 
\[
	\wsl = \begin{pmatrix} 0 & 1 \\ -1 & 0 \end{pmatrix} \in \SL_2(\CC),
\]
which can also be written as $\wsl = e^{\Xsl} e^{-\Ysl} e^{\Xsl}$. The Weil element
is a sort of linear algebra analogue of the Hodge $\ast$-operator on a compact
K\"ahler manifold: it satisfies $\wsl^2 = (-1)^{\Hsl}$, and induces isomorphisms
between the two weight spaces $E_{\pm k}(\Hsl)$. Using a dagger to denote the adjoint
of an operator with respect to the nondegenerate pairing $Q$, it is easy to see that
$\Hsl^{\dag} = -\Hsl$, $\Ysl^{\dag} = \Ysl$, $\Xsl^{\dag} = \Xsl$, and hence also
$\wsl^{\dag} = \wsl$.

\newpar \label{par:sltwo}
We can now define $\sltwo$-Hodge structures and their polarizations. The 
definition is modeled on the cohomology $H^{n + \bullet}(X, \CC)$ of
an $n$-dimensional compact K\"ahler manifold, but perhaps surprisingly, it is equally 
relevant in our context. An \define{$\sltwo$-Hodge structure} on a finite-dimensional
complex vector space $V$ is a representation of $\sltwo(\CC)$ on $V$ with the
following properties:
\begin{aenumerate}
\item Each weight space $V_k = E_k(\Hsl)$ has a Hodge structure of weight $n+k$;
	the integer $n$ is called the \define{weight} of the $\sltwo$-Hodge structure.
\item The two operators
	\[
		\Xsl \colon V_k \to V_{k+2}(1) \quad \text{and} \quad
		\Ysl \colon V_k \to V_{k-2}(-1)
	\]
	are morphisms of Hodge structures.
\end{aenumerate}
In the literature, this kind of structure is also called a ``Hodge-Lefschetz
structure'' \cite{Cattani} or an ``$\RR$-split mixed Hodge structure'' \cite{CKS},
but we feel that the name ``$\sltwo$-Hodge structure'' is more descriptive.

\newpar
The polarization on the cohomology of a compact K\"ahler manifold is usually
discussed only for the primitive cohomology (which is the kernel of $\Ysl$, in our
notation). The standard definition is that, for every $k \geq 0$, the hermitian
pairing
\[
	Q \circ \bigl( \id \tensor \Xsl^k \bigr) \colon V_{-k} \tensor_{\CC} \wbar{V_{-k}} \to \CC
\]
should polarize the Hodge structure of weight $n-k$ on the primitive subspace $\ker
\Ysl \colon V_{-k} \to V_{-(k+2)}$. Since $\wsl$ acts on the primitive subspace as
$\frac{1}{k!} \Xsl^k$, one can define polarizations more compactly with the help of
the Weil element. Following Deligne \cite{Deligne-signs}, we say that a hermitian
form
\[
	Q \colon V \tensor_{\CC} \Vb \to \CC
\]
is a \define{polarization} of the $\sltwo$-Hodge structure $V$ if it satisfies the four
identities
\begin{align*}
	Q \circ (\Hsl \tensor \id) &= - Q \circ (\id \tensor \Hsl), &
	Q \circ (\Xsl \tensor \id) &= Q \circ (\id \tensor \Xsl), \\
	Q \circ (\Ysl \tensor \id) &= Q \circ (\id \tensor \Ysl), &
	Q \circ (\wsl \tensor \id) &= Q \circ (\id \tensor \wsl),
\end{align*}
and if $Q \circ (\id \tensor \wsl)$ polarizes the Hodge structure of weight $n+k$
on each $V_k$. 

\newpar
Because the Weil element $\wsl$ acts as $\frac{1}{k!} (-\Ysl)^k$ on the ``other''
primitive subspace $\ker \Xsl \colon V_k \to V_{k+2}$ for $k \geq 0$, an equivalent
formulation is that
\[
	Q \circ \bigl( \id \tensor (-\Ysl)^k \bigr) \colon V_k \tensor_{\CC} \wbar{V_k} \to \CC
\]
should polarize the Hodge structure of weight $n+k$ on $\ker \Xsl$. Since $\rho(-\Ysl)
= N$ in our case, this matches Schmid's definition \cite[Lem.~6.4]{Schmid}. 

\newpar
We now describe how the limiting $\sltwo$-Hodge structure on $V$ comes about.
The convergence of the rescaled period mapping (in \Cref{thm:convergence-intro})
gives us a filtration $F_H \in \Dch$ with the following properties: 
\begin{aenumerate}
\item The filtration $e^{-N} F_H \in D$ is the Hodge filtration of a Hodge structure
	of weight $n$, polarized by the pairing $Q$.
\item One has $N(F_H^{\bullet}) \subseteq F_H^{\bullet-1}$, 
	$H(F_H^{\bullet}) \subseteq F_H^{\bullet}$, and $T_s(F_H^{\bullet})
	\subseteq F_H^{\bullet}$.
\end{aenumerate}
From these properties, we deduce by a mix of representation theory and linear algebra
that $F_H$ is the ``total'' Hodge filtration of an $\sltwo$-Hodge structure of weight $n$
on $V$.

\begin{pthm}[Limiting $\sltwo$-Hodge structure] 
	The vector space $V$ has a polarized $\sltwo$-Hodge structure of weight $n$, whose
	Hodge filtration is $F_H$. Concretely, this means the following:
	\begin{enumerate}
		\item Each weight space $V_k = E_k(H)$ has a Hodge structure of
			weight $n + k$, whose Hodge filtration is $V_k \cap F_H$.
		\item The operator $N \colon V_k \to V_{k-2}(-1)$ is a morphism of
			Hodge structures.
		\item The Hodge structure on each primitive subspace $\ker N^{k+1} \colon V_k \to
			V_{-k-2}$ is polarized by the hermitian pairing $Q(\argbl, N^k \argbl)$.
%		\item The Hodge structure on each weight space $V_k$ is polarized by the
%			pairing $Q(\argbl, \wsl \argbl)$, where $\wsl = e^{\Xsl} e^{-\Ysl} e^{\Xsl}$ is
%			the Weil element and $\Ysl = -N$.
		\item The operator $T_s$ is an endomorphism of the polarized $\sltwo$-Hodge
			structure.
	\end{enumerate}
\end{pthm}

\newpar
Recall from \cref{par:convergence-proof} that the filtration $F_H$ is obtained from
the limiting Hodge filtration $\Flim$ by the following procedure: $\Flim$ induces a
filtration on each subquotient $W_k/W_{k-1}$, and $V_k \cap F_H$ is gotten by pulling
back along the isomorphism $V_k \cong W_k / W_{k-1}$. The fact that each
weight space $V_k$ has a Hodge structure of weight $n+k$ therefore means that
$W_{\bullet-n}$ is the weight filtration of a mixed Hodge structure
\cite[Thm.~6.16]{Schmid}.

\begin{pthm}[Limiting mixed Hodge structure]  \label{thm:MHS-intro}
The vector space $V$ has a mixed Hodge structure, with weight filtration
$W_{\bullet-n}$ and Hodge filtration $\Flim$; the associated graded of this mixed
Hodge structure is a polarized $\sltwo$-Hodge structure of weight $n$ with $\Ysl =
-N$ and polarization $Q$. The two operators
\[
	N \colon V \to V(-1) \quad \text{and} \quad T_s \colon V \to V
\]
are morphisms of mixed Hodge structures.
\end{pthm}

The $\sltwo$-Hodge structure depends on the choice of splitting $H$; the limiting
mixed Hodge structure is independent of it. Without choosing a splitting, one gets a
canonical polarized $\sltwo$-Hodge structure on the associated graded
$\gr_{\bullet}^W$ with respect to the monodromy weight filtration. In the terminology
of \cite[Def.~2.26]{CKS}, the limiting mixed Hodge structure is therefore 
``polarized'' by the pairing $Q$ and the nilpotent operator $N$.

\subsection{Asymptotic behavior of the Hodge metric}

\newpar
Finally, we describe the asymptotic behavior of the Hodge metric $h(v,v')$ where
$v,v' \in V$ are two multi-valued flat sections. The Hodge norm estimates cover the
case $v = v'$, but something new happens when $v \neq v'$, namely that the Hodge
metric $h(v,v')$ grows more slowly than expected as $\abs{\Re z} \to \infty$. This
result does not appear as such in Schmid's paper, but it is implicitly contained in
the statement about Siegel sets in \cite[Cor.~5.29]{Schmid}.

\begin{pthm}[Asymptotic behavior of the Hodge metric]  \label{thm:asymptotic-intro}
	For any two multi-valued flat sections $v, v' \in V$, there is a
	constant $C(v,v') > 0$ such that
	\[
		\bigabs{\inner{v}{v'}_{\Phi(z)}} \leq C(v,v') \cdot
		\min \Bigl( \norm{v}_{\Phi(z)}^2, \norm{v'}_{\Phi(z)}^2 \Bigr)
	\]
	for all sufficiently large values of $\abs{\Re z}$.
\end{pthm}

An equivalent formulation is that if $v \in V_k$ and $v' \in V_{k'}$, then
\[
	\inner{v}{v'}_{\Phi(z)} = O \bigl( \abs{\Re z}^{\min(k,k')} \bigr)
\]
which is typically much smaller than the $\abs{\Re z}^{\half (k+k')}$ one
would expect based on the Hodge norm estimates. We can also restate the theorem in
terms of the rescaled period mapping from \eqref{eq:PhiSH-intro}: the result is that
if $v \in V_k$ and $v' \in V_{k'}$ belong to different weight spaces, then
\begin{equation} \label{eq:Hodge-metric-PhiSH}
	\inner{v}{v'}_{\PhiSH(z)} = O \bigl( \abs{\Re z}^{-\half \abs{k-k'}} \bigr).
\end{equation}
Since different weight spaces are orthogonal to each other in the Hodge structure
$\Fsh = e^{-N} F_H \in D$, this is not unexpected; but not only are the weight spaces
becoming orthogonal in the limit, they are doing so at a rate that depends on the
difference between their weights.

\newpar
We end the introduction with a brief sketch of the proof. Up to an error of size
$e^{-\eps \abs{\Re z}}$, which is negligible in this context, the nilpotent orbit
theorem allows us to assume that the period mapping has the form $\Phi(z) = e^{zN}
\Flim$. The rescaled period mapping then looks like
\[
	\PhiH(z) = e^{-N} \cdot e^{\half \log \abs{\Re z} \, H} \Flim.
\]
This expression only depends on $\abs{\Re z}$; since we are interested in its
behavior as $\abs{\Re z} \to \infty$, we make the substitution $u = \abs{\Re
z}^{-\half}$, so that
\[
	\PhiH \bigl( 1/u^2 \bigr) = e^{-N} \cdot e^{-(\log u) H} \Flim.
\]
Recall from \cref{par:convergence-proof} that we have
\begin{equation} \label{eq:FH-limit}
	\lim_{u \to 0} e^{-(\log u) H} \Flim = F_H,
\end{equation}
and hence that $\PhiH \bigl( 1/u^2 \bigr)$ converges to $\Fsh = e^{-N} F_H$
as $u \to 0$.

\newpar
To get a handle on the rate of convergence, we prove the following cheap
version of the \define{$\SL(2)$-orbit theorem}. Schmid's original result
\cite[Thm.~5.13]{Schmid} has two functions: it provides a distinguished splitting
$H$ for the weight filtration, called the ``$\SL(2)$-splitting'' (or ``canonical
splitting''); and it gives a very precise description, using power series with values
in the real group $G$, for the rate of convergence of \eqref{eq:FH-limit}. Here we
focus on the second aspect, using an arbitrary splitting $H$ chosen as in
\Cref{par:splitting}.

\begin{pthm}[Cheap $\SL(2)$-orbit theorem]  \label{thm:SL2-intro}
	There is a convergent power series
	\[
		g \colon (-\eps, \eps) \to G, \quad
		g(u) = \id + u g_1 + u^2 g_2 + \dotsb,
	\]
	defined for some $\eps > 0$, such that for $0 < \abs{u} < \eps$, one has
	\[
		e^{-N} \cdot e^{-(\log u) H} \Flim = g(u)^{-1} \cdot e^{-N} F_H.
	\]
	The coefficients $g_n \in \End(V)$ of the series are $(S_n \oplus S_{n-2} \oplus
	\dotsb)$-isotypical with respect to the action of $\sltwo(\CC)$ on $\End(V)$.
\end{pthm}

The point is that we can then write the rescaled period mapping
\[
	\PhiH \bigl( 1/u^2 \bigr) = g(u)^{-1} \cdot \Fsh
\]
in terms of the real group $G$. Now suppose that $v \in V_k$ and $v' \in V_{k'}$.
Since different weight spaces are orthogonal in the Hodge structure $\Fsh$, and since
\Cref{thm:SL2-intro} gives us some control on how much the coefficients $g_n$ can
change the weight of a given vector, it is an easy matter to deduce that
\begin{align*}
	\inner{v}{v'}_{\PhiH(1/u^2)}
	&= \bigl\langle g(u) \cdot v, \, g(u) \cdot v'\bigr\rangle_{\Fsh}  \\
	&= \sum_{m+n \geq \abs{k-k'}} u^{m+n} \,
	\bigl\langle g_m v, \, g_n v' \bigr\rangle_{\Fsh}
	= O \bigl( u^{\abs{k-k'}} \bigr),
\end{align*}
which is equivalent to \eqref{eq:Hodge-metric-PhiSH}.

\newpar
The original proof of the $\SL(2)$-orbit theorem in \cite[\S9]{Schmid} is very difficult;
and even the simplified argument in \cite[\S6]{CKS} (assuming the limiting mixed Hodge
structure) is quite involved. By contrast, our proof of the cheap $\SL(2)$-orbit
theorem is short and elementary. The crucial point is that 
\[
	\Flim = h \cdot F_H,
\]
where $h = \id + h_1 + h_2 + \dotsb \in \GL(V)$ has the property that
\[
	h_n \in E_{-n}(\ad H) \cap \ker(\ad N).
\]
Such an element $h$ can be constructed by comparing Deligne's decomposition of the
limiting mixed Hodge structure (whose Hodge filtration is $\Flim$) and the Hodge
decomposition of the $\sltwo$-Hodge structure (whose Hodge filtration in $F_H)$, and
using the fact that, up to a Tate twist, $N$ is a morphism in both structures.
\Cref{thm:SL2-intro} follows from this by a formal argument. The crucial fact from
the representation theory of the Lie algebra $\sltwo(\CC)$ is that the tensor product
\[
	S_k \tensor S_{\ell} \cong \bigoplus_{i=0}^{\min(k,\ell)} S_{k+\ell-2i}
\]
of two irreducible representation decomposes in a very simple way.

\section{Variations of Hodge structure and harmonic bundles}
\label{sec:VHS}

\subsection{Harmonic bundles and harmonic metrics}

\newpar
Since we are going to do a certain amount of analysis, it will be convenient to
describe variations of Hodge structure in terms of smooth
vector bundles. Let $E$ be a smooth vector bundle on a complex manifold $X$. We denote
by $A^k(X, E)$ the space of smooth $k$-forms with coefficients in $E$, and by
$A^{i,j}(X, E)$ the space of smooth $(i,j)$-forms with coefficients in $E$. 

\begin{pdfn} \label{def:VHS}
	Let $E$ be a smooth vector bundle on a complex manifold $X$.
A \define{variation of Hodge structure} of weight $n$ on $E$ is a decomposition
\[
	E = \bigoplus_{p+q=n} E^{p,q}
\]
into smooth subbundles, together with a flat connection $d \colon A^0(X, E) \to
A^1(X, E)$ that maps each $A^0(X, E^{p,q})$ into the direct sum of the subspaces
\[
	A^{1,0}(X, E^{p,q}) \oplus A^{1,0}(X, E^{p-1,q+1}) \oplus 
	A^{0,1}(X, E^{p,q}) \oplus A^{0,1}(X, E^{p+1,q-1}).
\]
\end{pdfn}

\newpar
The definition gives us a decomposition $d = \del + \theta + \delb + \thetast$ into
four operators
\begin{align*}
	\del &\colon A^0(X, E^{p,q}) \to A^{1,0}(X, E^{p,q}) \\
	\theta &\colon A^0(X, E^{p,q}) \to A^{1,0}(X, E^{p-1,q+1}) \\
	\delb &\colon A^0(X, E^{p,q}) \to A^{0,1}(X, E^{p,q}) \\
	\thetast &\colon A^0(X, E^{p,q}) \to A^{1,0}(X, E^{p+1,q-1}).
\end{align*}
It is easy to see that $\del$ and $\delb$ are again connections of type $(1,0)$ and
$(0,1)$, whereas $\theta$ and $\thetast$ are linear over smooth functions. As usual,
we can extend all four operators to the space of all $E$-valued forms
\[
	A^{\ast}(X, E) = \bigoplus_{i,j,p,q} A^{i,j}(X, E^{p,q})
\]
by enforcing the Leibniz rule. With respect to the $\ZZ^2 \tensor \ZZ^2$-grading on
this space, $\partial$ is an operator of bidegree $(1,0) \tensor (0, 0)$, because it maps
$A^{i,j}(X, E^{p,q})$ into $A^{i+1,j}(X, E^{p,q})$; similarly, $\theta$ has bidegree
$(1,0) \tensor (-1,1)$, $\delb$ has bidegree $(0,1) \tensor
(0,0)$, and $\thetast$ has bidegree $(0,1) \tensor (1,-1)$. By decomposing the
identity $d^2 = 0$ according to bidegree, one can deduce the following set of
identities for these four operators.

\begin{plem} \label{lem:VHS-identities}
	The four operators $\partial$, $\theta$, $\delb$, and $\thetast$ satisfy the
	following identities:
	\begin{align*}
		\partial^2 = \theta^2 = \delb^2 = (\thetast)^2 &= 0 \\
		\partial \theta + \theta \partial = \delb \thetast + \thetast \delb &= 0 \\
		\delb \theta + \theta \delb = \partial \thetast + \thetast \partial &= 0 \\
		\partial \delb + \delb \partial + \theta \thetast + \thetast \theta &= 0
	\end{align*}
\end{plem}

\newpar
In order to define polarizations, we consider a hermitian pairing
\[
	Q \colon A^0(X, E) \tensor_{A^0(X)} \wbar{A^0(X, E)} \to A^0(X).
\]
It takes as input two smooth sections of the bundle $E$, and outputs a smooth
function on $X$. We require that $Q$ is hermitian symmetric and
$A^0(X)$-linear in its first argument, meaning that for $v,w \in A^0(X, E)$
and $f \in A^0(X)$, one has
\[
	Q(w,v) = \wbar{Q(v,w)} \quad \text{and} \quad
	Q(f \cdot v, w) = f \cdot Q(v,w).
\]
Of course, it follows that $Q$ is also conjugate $A^0(X)$-linear in its second argument.
We say that $Q$ is \define{flat} with respect to the connection $d$ if 
\[
	d Q(v,w) = Q(dv, w) + Q(v,dw) \quad \text{for $v,w \in A^0(X, E)$.}
\]

\begin{pdfn}
	Let $E$ be a variation of Hodge structure of weight $n$ on a complex manifold $X$.
	A \define{polarization} is a flat hermitian pairing
	\[
		Q \colon A^0(X, E) \tensor_{A^0(X)} \wbar{A^0(X, E)} \to A^0(X),
	\]
	such that the formula
	\[
		h(v,w) = \sum_{p+q=n} (-1)^q \, Q(v^{p,q}, w^{p,q})
		\quad \text{for $v,w \in A^0(X,E)$}
	\]
	defines a smooth hermitian metric on the bundle $E$, and the subbundles
	$E^{p,q}$ are orthogonal to each other with respect to this metric.
\end{pdfn}

The hermitian metric in the definition is called the \define{Hodge metric}; it is the
most important object in the theory. For the sake of clarity, we sometimes denote the
Hodge metric on $E$ by the symbol $h_E$. 

\newpar
The fact that $Q$ is flat with respect to $d$ gives us the following additional
information about the operators $\theta$ and $\thetast$.

\begin{plem} \label{lem:VHS-adjoints}
	With respect to the Hodge metric on $E$, the operator $\thetast$ is the adjoint of
	the operator $\theta$, in the sense that
	\[
		h(\theta v, w) = h(v, \thetast w) \quad \text{for $v,w \in A^0(X, E)$.}
	\]
\end{plem}

\begin{proof}
	Let $v \in A^0(X, E^{p,q})$ and $w \in A^0(X, E^{p-1,q+1})$ be two arbitrary sections.
	Then $Q(v,w) = 0$, and therefore
	\[
		0 = d Q(v,w) = Q(dv, w) + Q(v, d w) 
		= Q(\theta v, w) + Q(v, \thetast w)
	\]
	because all other terms have the wrong type. But then
	\[
		h(\theta v, w) = (-1)^{q+1} Q(\theta v, w) = (-1)^q Q(v, \thetast w) = h(v,
		\thetast w).
	\]
	Since the decomposition of $E$ is orthogonal with respect to $h$, the result
	follows.
\end{proof}

\newpar
Let us briefly recall the relation with the (more familiar) holomorphic
description of variations of Hodge structure. As usual, we decompose the flat
connection into $d = d' + d''$, where $d' \colon A^0(X, E) \to
A^{1,0}(X, E)$ is a connection of type $(1,0)$, and $d'' \colon A^0(X, E) \to
A^{0,1}(X, E)$ is a connection of type $(0,1)$. The identity $d^2 = 0$ then becomes
$(d')^2 = d' d'' + d'' d' = (d'')^2 = 0$. The operator $d''$ gives $E$ the structure
of a holomorphic vector bundle, which we denote by the symbol $\shE$, and the
operator $d'$ defines a flat holomorphic connection
\[
	\nabla \colon \shE \to \OmX^1 \tensor_{\OX} \shE.
\]
Moreover, for each $p \in \ZZ$, the Hodge bundle
\[
	F^p E = E^{p,q} \oplus E^{p+1,q-1} \oplus E^{p+2,q-2} \oplus \dotsb
\]
is preserved by the operator $d''$, and therefore defines a holomorphic subbundle
$F^p \shE$. Since $d' = \partial + \theta$, we also get Griffiths' transversality condition
\[
	\nabla(F^p \shE) \subseteq \OmX^1 \tensor_{\OX} F^{p-1} \shE.
\]
From this perspective, we can see the other operators
by looking at the quotient bundles $F^p \shE / F^{p+1} \shE$. The operator $\dbar$
gives each smooth bundle $E^{p,q}$ the structure of a holomorphic vector bundle
$\shE^{p,q}$, and because $d'' = \dbar + \thetast$,
\[
	\shE^{p,q} \cong F^p \shE / F^{p+1} \shE
\]
are isomorphic as holomorphic vector bundles. Since $\dbar \theta + \theta \dbar =
0$, the operator $\theta$ defines a morphism of holomorphic vector bundles
\[
	\theta \colon \shE^{p,q} \to \OmX^1 \tensor_{\OX} \shE^{p-1,q+1},
\]
and one checks that the following diagram is commutative:
\[
	\begin{tikzcd}
		\shE^{p,q} \rar{\theta} \dar{\cong} & 
		\OmX^1 \tensor_{\OX} \shE^{p-1,q+1} \dar{\cong} \\
		F^p \shE / F^{p+1} \shE \rar{\nabla} & 
		\OmX^1 \tensor_{\OX} F^{p-1} \shE / F^p \shE
	\end{tikzcd}
\]
The two remaining operators $\thetast$ and $\partial$ in the decomposition of $d$
have the following interpretation: $\thetast$ is the adjoint of the Higgs field with
respect to the Hodge metric (\Cref{lem:VHS-adjoints}), and $\partial$ is the Chern
connection on the bundle $E^{p,q}$.

\newpar \label{par:harmonic}
We are now going to do some computations with the Hodge metric; these become
easier if we use the language of harmonic bundles. In fact, a polarized variation of Hodge
structure is an example of a harmonic bundle, with the harmonic metric being the
Hodge metric. We briefly recall the general definition. Let $E$ be a smooth vector
bundle on a complex manifold $X$, let $d =
d' + d''$ be a flat connection on $E$, and let $h$ be a hermitian metric on $E$. As
in the construction of the Chern connection, there are unique operators
\begin{align*}
	\delta' &\colon A^0(X, E) \to A^{1,0}(X, E) \\
	\delta'' &\colon A^0(X, E) \to A^{0,1}(X, E)
\end{align*}
with the property that $d' + \delta''$ and $\delta' + d''$ are metric connections on
$E$. Then
\[
	\del = \half(d' + \delta') \quad \text{and} \quad
	\delb = \half(d'' + \delta'')
\]
are again connections of type $(1,0)$ and $(0,1)$, and 
\[
	\theta = \half(d' - \delta') \quad \text{and} \quad
	\thetast = \half(d'' - \delta'')
\]
are linear over smooth functions. The metric $h$ is called a \define{harmonic
metric}, and the triple $(E,d,h)$ is called a \define{harmonic bundle}, if the
identity $(\delb + \theta)^2 = 0$ holds. This is equivalent to the identities in
\Cref{lem:VHS-identities}. The associated Higgs bundle is $E$, with the
holomorphic structure defined by $\delb$, and with Higgs field $\theta$. As suggested
by the notation, in the case of a polarized variation of Hodge structure, $\del$, $\delb$,
$\theta$, and $\thetast$ agree with the operators that we had defined above.

\begin{plem} 
	Let $E$ be a variation of Hodge structure with polarization $Q$. Then
	\[
		d' = \del + \theta, \quad
		d'' = \delb + \thetast, \quad
		\delta' = \del - \theta, \quad
		\delta'' = \delb - \thetast,
	\]
	and so $E$ is a harmonic bundle and the Hodge metric $h$ is a harmonic metric. The
	associated Higgs bundle is the graded holomorphic vector bundle 
	\[
		\bigoplus_{p+q=n} \shE^{p,q}
	\]
	with Higgs field $\theta \colon \shE^{p,q} \to \OmX^1 \tensor \shE^{p-1,q+1}$.
\end{plem}

\begin{proof}
	The claim is that the two operators $(\partial + \dbar) \pm (\theta - \thetast)$
	are metric connections. For any two sections $v, w \in A^0(X, E^{p,q})$, we have
	\[
		d h(v,w) = (-1)^q Q(d v, w) + (-1)^q Q(v, dw) 
		= h \bigl( (\partial + \dbar) v, w \bigr) + h \bigl( v, (\partial + \dbar) w
		\bigr),
	\]
	because all other terms have the wrong type. Therefore $\partial + \dbar$ is a
	metric connection on $E$. Since $\thetast$ is the adjoint of $\theta$ with respect
	to $h$, it follows that $(\partial + \dbar) \pm (\theta - \thetast)$ are also
	metric connections. The identities in \Cref{lem:VHS-identities} imply that
	$(\dbar + \theta)^2 = 0$, and so $E$ is a harmonic bundle with harmonic metric
	$h$. The remaining assertions are clear from the discussion in the previous
	paragraph.
\end{proof}

\newpar
The Hodge metric $h$ gives us a hermitian metric on the holomorphic vector bundle
$\shE$ and on the subbundles $F^p \shE$. It is not hard to derive a formula for the
curvature of this metric, using the identities among the operators $\partial$,
$\dbar$, $\theta$, and $\thetast$; Schmid does this, in a more cumbersome
way, in \cite[\S7]{Schmid}. The curvature of the Hodge bundles will
be important later, in \Cref{sec:nilpotent}, when we prove the nilpotent orbit
theorem. We need one more piece of notation to state the result: we denote by
\[
	\pi^{p,q} \colon E \to E^{p,q}
\]
the projection operator to the direct summand $E^{p,q}$; of course,
$\pi^{p,q}$ is also orthogonal projection with respect to the Hodge metric.

\begin{pprop} \label{prop:curvature}
	Let $E$ be a variation of Hodge structure with polarization $Q$.
	\begin{aenumerate}
	\item On the holomorphic vector bundle $\shE$, the curvature operator of the Hodge
		metric $h$ is equal to $-2(\theta \thetast + \thetast \theta)$.
	\item On the holomorphic vector bundle $\shE^{p,q}$, the curvature operator of the Hodge
		metric $h$ is equal to $-(\theta \thetast + \thetast \theta)$.
	\item On the holomorphic quotient bundle $\shE/F^p \shE$, the curvature
		operator of the induced metric is equal to $-2(\theta \thetast +
		\thetast \theta) + \theta \thetast \pi^{p-1,q+1}$.
	\end{aenumerate}
\end{pprop}

\begin{proof}
	Since the holomorphic structure on $\shE$ is defined by the operator $d''$, the
	Chern connection of $h$ is nothing but the metric connection $\delta' + d'' =
	\partial + \dbar - \theta + \thetast$. A short computation shows that the
	curvature operator is
	\[
		(\delta' + d'')^2 = (\partial + \dbar - \theta + \thetast)^2 
		= (\partial \dbar + \dbar \partial) - (\theta \thetast + \thetast \theta)
		= -2(\theta \thetast + \thetast \theta),
	\]
	using the identities in \Cref{lem:VHS-identities}. This proves (a). The proof
	of (b) is very similar. The holomorphic structure on $\shE^{p,q}$ is defined by
	the operator $\dbar$, and so the Chern connection of $h$ is the metric connection
	$\partial + \dbar$, and the curvature operator is
	\[
		(\partial + \dbar)^2 = (\partial \dbar + \dbar \partial) = 
		-(\theta \thetast + \thetast \theta).
	\]
	The derivation of (c) is slightly more complicated, because the operator $d'' =
	\dbar + \thetast$ does not preserve the individual summands $E^{p,q}$. We can
	correct this with the help of the projection $\pi^{p,q} \colon E \to E^{p,q}$. 
	Note that $\pi^{p,q}$ commutes with $\partial$ and $\dbar$, whereas
	\[
		\pi^{p,q} \theta = \theta \pi^{p+1,q-1} \quad \text{and} \quad
		\pi^{p,q} \thetast = \thetast \pi^{p-1,q+1}.
	\]
	Now for the proof. The Hodge decomposition is orthogonal with respect to the Hodge
	metric, and so, as smooth vector bundles with metric, 
	\begin{equation} \label{eq:bundle}
		E/F^p E \cong E^{p-1,q+1} \oplus E^{p-2,q+2} \oplus E^{p-3,q+3} \oplus \dotsb
	\end{equation}
	Under this isomorphism, the holomorphic structure on $E/F^p E$ is represented by
	the operator $d'' - \thetast \pi^{p-1,q+1}$, which preserves the
	bundle on the right-hand side. A short computation shows that the modified
	operator
	\[
		\delta' + d'' - \thetast \pi^{p-1,q+1}
		= \partial + \dbar - \theta + \thetast - \thetast \pi^{p-1,q+1}
	\]
	is a metric connection, and therefore equal to the Chern connection.
	It follows that the curvature operator of the induced metric is
	\begin{align*}
		\bigl( \delta' + d'' &- \thetast \pi^{p-1,q+1} \bigr)^2 \\
		&= (\delta' + d'')^2 - (\partial + \dbar - \theta + \thetast) \thetast
		\pi^{p-1,q+1} - \thetast \pi^{p-1,q+1} (\partial + \dbar - \theta + \thetast) \\
		&= -2(\theta \thetast + \thetast \theta) + \theta \thetast \pi^{p-1,q+1},
	\end{align*}
	using the identities in \Cref{lem:VHS-identities} and the fact that
	$\pi^{p-1,q+1} \theta = \theta \pi^{p,q} = 0$ on the vector bundle in
	\eqref{eq:bundle}.
\end{proof}

\subsection{Multivalued flat sections and the Hodge metric}

\newpar
Now we are ready to start doing some computations specifically for polarized
variations of Hodge structure over the punctured disk. As usual, let
\[
	\dst = \menge{t \in \CC}{0 < \abs{t} < 1}
\]
be the punctured unit disk, with coordinate $t$. Let $E$ be a polarized variation of
Hodge structure of weight $n$ on $\dst$. We are going to denote by $V$ the space of
all multi-valued flat sections of $(E, d)$. More precisely, we use the letter
\[
	\HH = \menge{z \in \CC}{\Re z < 0}
\]
for the \emph{left} half-plane; the exponential function $\exp \colon \HH \to \dst$ makes it
into the universal covering space of $\dst$, independently of the choice of $i =
\sqrt{-1}$. To keep the notation simple, let us denote the pullbacks of $E$, $d$,
and $h$ by the same symbols; then $V \subseteq A^0(\HH, E)$ is exactly the kernel of the
operator $d$. 

\newpar
If $v \in V$ is a nontrivial multi-valued flat section, the pointwise Hodge norm
$h(v,v)$ is a smooth function on $\HH$ that is everywhere nonzero. The first
observation is that the Higgs field $\theta_{\partial/\partial z}$ controls the first
derivatives of $h(v,v)$.

\begin{plem} \label{lem:metric-derivative}
	For every nonzero $v \in V$, one has $\partial h(v,v) = -2 h(\theta v, v)$.
\end{plem}

\begin{proof}
	This is true for arbitrary harmonic bundles. By assumption, we have $d' v = d'' v
	= 0$, and therefore $\partial v = -\theta v$. Since $\delta' + d''$ is a metric
	connection, 
	\[
		\partial \, h(v,v) = h(\delta' v, v) + h(v, d'' v) = h(\partial v - \theta v, v) 
		= -2h(\theta v, v),
	\]
	as claimed.
\end{proof}

\newpar
Another important property of the function $h(v,v)$, discovered by
Griffiths and Schmid \cite[Lem~7.19]{Schmid}, is that the logarithm of $h(v,v)$ is
subharmonic.

\begin{plem} \label{lem:metric-subharmonic}
	If $v \in V$ is nonzero, the function $\log h(v,v)$ is subharmonic
	on $\HH$.
\end{plem}

\begin{proof}
	This is again true for arbitrary harmonic bundles. From $d' v = d'' v = 0$, we get
	$\partial v = -\theta v$ and $\delb v = - \thetast v$. For the same reason as in
	\Cref{lem:metric-derivative}, we have
	\[
		\delb \, h(v,v) = -2 h(v, \theta v).
	\]
	Since $\partial + \dbar$ is a metric connection, we then get
	\begin{align*}
		\del\delb \, h(v,v) = -2 h(\partial v, \theta v) - 2 h(v, \dbar \theta v)
		&= 2 h(\theta v, \theta v) + 2 h(v, \theta \dbar v) \\
		&= 2 h(\theta v, \theta v) - 2 h(\thetast v, \thetast v),
	\end{align*}
	using the identity $\dbar \theta + \theta \dbar = 0$ and
	\Cref{lem:VHS-adjoints}. To shorten the next couple of formulas, we introduce the
	two operators 
	\[
		\Higg = \theta_{\partial/\partial z} \quad \text{and} \quad
		\Higgst = \thetast_{\partial/\partial \zb}, 
	\]
	which are both smooth sections of the bundle $\End(E)$. The notation is justified
	because $\Higg$ and $\Higgst$ are adjoints under the Hodge metric
	$h$, according to \Cref{lem:VHS-adjoints}. We then get
	\[
		\frac{\partial}{\partial z} \frac{\partial}{\partial \zb} \, h(v,v)
		= 2 \Bigl( h(\Higg v, \Higg v) + h(\Higgst v, \Higgst v) \Bigr).
	\]
	We also know from \Cref{lem:metric-derivative} that
	\[
		\frac{\partial}{\partial z} \, h(v,v) = -2h(v, \Higg v) = -2h(\Higgst v,v).
	\]
	An application of the Cauchy-Schwarz inequality now gives
	\[
		\ABS{\frac{\partial}{\partial z} \, h(v,v)}^2 \leq
		\half \Bigl( 4 h(v,v) h(\Higg v, \Higg v) 
		+ 4 h(v,v) h(\Higgst v, \Higgst v) \Bigr) 
		= h(v,v) \cdot \frac{\partial}{\partial z} \frac{\partial}{\partial \zb} \,
		h(v,v).
	\]
	Since the Laplacian of $\varphi = \log h(v,v)$ satisfies
	\[
		\frac{1}{4} \Delta \varphi 
		= \frac{\partial}{\partial z} \frac{\partial}{\partial \zb} \log h(v,v)
		= \frac{1}{h(v,v)} \cdot \frac{\partial}{\partial z} \frac{\partial}{\partial \zb} \,
		h(v,v) - \frac{1}{h(v,v)^2} \ABS{\frac{\partial}{\partial z} \, h(v,v)}^2 
	\]
	this inequality is what we need to conclude that $\Delta \varphi \geq 0$.
\end{proof}

\newpar
The argument above also shows that the first derivatives of $\log h_E(v,v)$
are controlled by the operator $\Higg = \theta_{\partial/\partial z}$, in the
following sense. From \Cref{lem:metric-derivative} and the triangle inequality, we get
\begin{align*}
	\ABS{\frac{\partial}{\partial z} \, h_E(v,v)} 
	= 2 \Abs{h_E(Av,v)}
	&\leq 2 h_E(v,v)^{\half} h_E(\Higg v, \Higg v)^{\half}  \\
	&\leq 2 h_E(v,v) \, h_{\End(E)}(\Higg, \Higg)^{\half},
\end{align*}
where $h_{\End(E)}(\Higg, \Higg)$ is the Hodge norm of $\Higg$ under the induced
hermitian metric on the bundle $\End(E)$ (which is an upper bound for the operator
norm of $\Higg$ with respect to the Hodge metric $h_E$). After dividing by
$h_E(v,v)$, this says that
\begin{equation} \label{eq:derivatives}
	\ABS{\frac{\partial}{\partial z} \log h_E(v,v)} 
	= \ABS{\frac{\partial}{\partial \zb} \log h_E(v,v)}
	\leq 2 h_{\End(E)} \bigl( \theta_{\partial/\partial z}, 
		\theta_{\partial/\partial z} \bigr)^{\half}.
\end{equation}
The Hodge norm of the Higgs field therefore gives us an upper bound for the derivatives of
the function $\log h_E(v,v)$.

\subsection{A universal upper bound for the Higgs field}

\newpar
Clearly, our next task is to bound the quantity
$h_{\End(E)}(\theta_{\partial/\partial z}, \theta_{\partial/\partial z})$. This
is analogous to the distance-decreasing property of period mappings, and the proof
naturally uses \define{Ahlfors' lemma}, a generalization of
the Schwarz-Pick lemma from complex analysis \cite{Ahlfors}. The result is usually
stated for the open unit disk; for the sake of completeness, we include a proof that
covers both cases.

\begin{plem}[Ahlfors] \label{lem:Ahlfors}
	Let $f \colon \HH \to (0,+\infty)$ be a positive smooth function such that
	\[
		\frac{\partial^2}{\partial z \partial \zb} \log f \geq C f
	\]
	for a positive constant $C > 0$. Then
	\[
		f(z) \leq \frac{(2C)^{-1}}{\abs{\Re z}^2} \quad \text{for all $z \in \HH$.}
	\]
\end{plem}

\begin{proof}
	We start by proving the result on the open unit disk $\Delta$, where the technique
	is easier to understand. Here the statement of Ahlfors' lemma is the following:
	for any positive smooth function $g \colon \Delta \to (0,+\infty)$, we have the
	implication
	\[
		\frac{\partial^2}{\partial t \partial \tb} \log g \geq g
		\quad \Longrightarrow \quad
		g(t) \leq \frac{2}{(1 - \abs{t}^2)^2} \quad \text{for $t \in \Delta$.}
	\]
	Let $\Delta_R$ denote the open disk of radius $R > 0$. We are going to argue that
	\begin{equation} \label{eq:Ahlfors}
		g(t) \leq \frac{2R^2}{(R^2 - \abs{t}^2)^2} \quad \text{for $t \in \Delta_R$}
	\end{equation}
	for every $R < 1$; this is enough, because we can then let $R \to 1$ to get the
	result. Define an auxiliary function $u \colon \Delta_r \to (0,+\infty)$ by the formula
	\[
		g = u \cdot \frac{2R^2}{(R^2 - \abs{t}^2)^2}.
	\]
	We observe that $u$ goes to zero near the boundary of $\Delta_R$, because 
	$g$ is bounded on $\Delta_R$, whereas $2R^2/(R^2-\abs{t}^2)^2$ goes to
	infinity near the boundary. Therefore $u$ must achieve its maximum at some
	interior point $t_0 \in \Delta_R$. In particular, the function $\log u$ 
	has a local maximum at $t_0$, and by looking at the second derivatives, we get
	\[
		\frac{\partial^2}{\partial t \partial \tb} \log u
		= \frac{1}{4} \left( \frac{\partial^2}{\partial x \partial x} +
		\frac{\partial^2}{\partial y \partial y} \right) \log u \leq 0
	\]
	for $t = t_0$, where $t = x + iy$. It follows that
	\[
		0 \geq \frac{\partial^2}{\partial t \partial \tb} \log u 
		= \frac{\partial^2}{\partial t \partial \tb} \log g 
		- \frac{\partial^2}{\partial t \partial \tb} \left( \log
		\frac{2r^2}{(r^2-\abs{t}^2)^2} \right)
		\geq g - \frac{2r^2}{(r^2-\abs{t}^2)^2}
	\]
	for $t = t_0$, which says exactly that $u(t_0) \leq 1$. But $u$ had a maximum at
	$t_0$, and therefore $u \leq 1$ on the entire disk $\Delta_R$, which is the
	content of \eqref{eq:Ahlfors}.

	It remains to deduce the analogous result for a smooth function $f \colon \HH \to
	(0, +\infty)$. Let us write the hypothesis in the symbolic form $\partial \delb
	\log f \geq f \dz \wedge \dzb$. After pulling back both $2$-forms along the isomorphism
	\[
		\Delta \to \HH, \quad t \mapsto \frac{t+1}{t-1},
	\]
	this becomes the condition that
	\[
		\partial \delb \log f \left( \frac{t+1}{t-1} \right) 
		\geq \frac{4}{\abs{t-1}^4} \cdot f \left( \frac{t+1}{t-1} \right) \dt \wedge
		\dtb.
	\]
	If we now set $g(t) = 4 \abs{t-1}^{-4} \cdot f \bigl( (t+1)/(t-1) \bigr)$, then we
	get
	\[
		\frac{\partial^2}{\partial t \partial \tb} \log g \geq g,
	\]
	and therefore, according to the discussion above, the inequality
	\[
		g(t) \leq \frac{2}{(1 - \abs{t}^2)^2} \quad \text{for $t \in \Delta$.}
	\]
	If we convert this back into a statement about $f$, the result follows.
\end{proof}

\newpar
The following lemma establishes the fundamental inequality that we need in order to
apply Ahlfors' lemma to our setting \cite[Thm.~1]{Simpson}.

\begin{plem} \label{lem:A-inequality}
	With the notation $\Higg = \theta_{\partial/\partial z}$ and $\Higgst =
	\thetast_{\partial/\partial \zb}$, we have
	\[
		\frac{\partial^2}{\partial z \partial \zb} \log h_{\End(E)}(\Higg, \Higg) \geq 
		\frac{h_{\End(E)} \bigl( [\Higgst, \Higg], [\Higgst, \Higg]
		\bigr)}{h_{\End(E)}(\Higg, \Higg)}.
	\]
\end{plem}

\begin{proof}
	It is best to view this as a special case of a more general result. Let $(E, d,
	h)$ be a harmonic bundle on $\HH$, and let $s$ be a smooth section of $E$ that
	satisfies $\delb s = 0$ and $\theta s = 0$; in other words, $s$ is a holomorphic
	section of the Higgs bundle $E$ that lies in the kernel of the Higgs field. Then
	we claim that
	\begin{equation} \label{eq:Higgs-inequality}
		\frac{\partial^2}{\partial z \partial \zb} \log h(s,s) 
		\geq \frac{h(\Higgst s, \Higgst s)}{h(s,s)},
	\end{equation}
	where again $\Higgst = \thetast_{\partial/\partial \zb}$. The calculations are very
	similar to those in the proof of \Cref{lem:metric-subharmonic}. The identity
	$\partial \dbar + \dbar \partial + \theta \thetast + \thetast \theta = 0$ in
	\Cref{lem:VHS-identities} gives us $\dbar \partial s = - \theta \thetast s$.
	Since $\partial + \dbar$ is a metric connection, we have
	\[
		\dbar \, h(s,s) = h(\dbar s, s) + h(s, \partial s) = h(s, \partial s).
	\]
	Applying the same reasoning again, we get
	\[
		\partial \dbar \, h(s,s) = h(\partial s, \partial s) + h(s, \dbar \partial s)
		= h(\partial s, \partial s) - h(s, \theta \thetast s)
		= h(\partial s, \partial s) - h(\thetast s, \thetast s).
	\]
	As before, we can now use the Cauchy-Schwarz inequality to conclude that
	\[
		\frac{\partial}{\partial z} \frac{\partial}{\partial \zb}
		\log h(s,s) \geq \frac{h(\Higgst s, \Higgst s)}{h(s,s)},
	\]
	which is of course just \eqref{eq:Higgs-inequality}.

	Now let us return to our specific situation. The vector bundle $\End(E)$, with the
	flat connection induced by $d$ and the metric induced by $h$, is again a harmonic
	bundle; moreover, it is easy to see that $\theta_{\partial/\partial z}$ acts on
	$\End(E)$ as the commutator with $\Higg$, and $\thetast_{\partial/\partial \zb}$ acts
	as the commutator with $\Higgst$. If we consider $\Higg$ as a section of
	$\End(E)$, it is holomorphic and trivially commutes with $\Higg$, hence belongs to the
	kernel of both $\delb$ and $\theta$. As such, the calculation from above
	applies to it, and so
	\[
		\frac{\partial^2}{\partial z \partial \zb} 
		\log h(\Higg,\Higg) \geq \frac{h \bigl( [\Higgst, \Higg],
		[\Higgst, \Higg] \bigr)}{h(\Higg, \Higg)}.
	\]
	This is the desired inequality.
\end{proof}

\newpar
At each point of $\HH$, the fiber of $E$ is a vector space of dimension $r = \rk
E$ with a hermitian inner product, and $\theta_{\partial/\partial z}$ is a nilpotent
endomorphism. The following lemma bounds the pointwise norm
of the commutator.

\begin{plem} \label{lem:nilpotent}
	Let $V$ be a complex vector space of finite dimension $r \geq 1$, with
	a hermitian inner product. If $A \in \End(V)$ is a nilpotent endomorphism, and if
	$\Ast \in \End(V)$ denotes its adjoint with respect to the inner product, then 
	\[
		\bigl\lVert [\Ast, A] \bigr\rVert^2 \leq 
		2 \norm{A}^4 \leq \binom{r+1}{3} \bigl\lVert [\Ast, A] \bigr\rVert^2.
	\]
\end{plem}

\begin{proof}
	Choose an orthonormal basis $e_1, \dotsc, e_r \in V$ such that the matrix
	representing $A$ is upper triangular; the matrix representing $\Ast$ is
	the conjugate transpose, hence lower triangular. Setting
	$a_{i,j} = \inner{Ae_i}{e_j}$, we get
	\[
		\norm{A}^2 = \sum_{i < j} \abs{a_{i,j}}^2.
	\]
	The $k$-th diagonal entry of the matrix representing $[\Ast, A]$ is easily
	seen to be
	\[
		d_{k,k} = \bigl( \abs{a_{1,k}}^2 + \dotsb + \abs{a_{k-1,k}}^2 \bigr)
		- \bigl( \abs{a_{k,k+1}}^2 + \dotsb + \abs{a_{k,r}}^2 \bigr),
	\]
	and this gives us the rather simple-minded inequality
	\[
		\bigl\lVert [\Ast, A] \bigr\rVert 
		\geq \sqrt{\abs{d_{1,1}}^2 + \dotsb + \abs{d_{r,r}}^2}.
	\]
	Let $x \in \RR$ be arbitrary. It is easy to see that
	\[
		\sum_{k=1}^r (x + k) d_{k,k} 
		= \sum_{i<j} \bigl( (x+j) - (x+i) \bigr) \abs{a_{i,j}}^2
		= \sum_{i<j} (j-i) \abs{a_{i,j}}^2,
	\]
	and from the trivial lower bound $j-i \geq 1$, we therefore get
	\[
		\sum_{i<j} \abs{a_{i,j}}^2 \leq 
		\sum_{k=1}^r (x+k) \norm{d_{k,k}} \leq 
		\left( \sum_{k=1}^r (x+k)^2 \right)^{\frac{1}{2}}
		\left( \sum_{k=1}^r \abs{d_{k,k}}^2 \right)^{\frac{1}{2}}.
	\]
	The right-hand side becomes minimal for $x = -\half (r+1)$, with the result that
	\[
		\norm{A}^2 \leq \sqrt{\frac{r(r+1)(r-1)}{12}} 
		\bigl\lVert [\Ast, A] \bigr\rVert.
	\]
	This inequality is actually sharp: equality is achieved for the $r \times r$-matrix with
	\[
		\abs{a_{k,k+1}}^2 = k(r-k) \quad \text{for $k=1, \dotsc, r-1$,}
	\]
	and all other entries zero. Interestingly, this is exactly the case where $V$ is
	an irreducible representation of the Lie algebra $\sltwo(\CC)$, with $\rho(\Ysl) =
	A$, $\rho(\Xsl) = \Ast$, and $\rho(\Hsl) = [\Ast, A]$. We leave the verification
	of this claim to the reader.

	The proof of the other inequality is easier. Without loss of generality, we can
	assume that $\norm{A} = \norm{\Ast} = 1$. The operator $[\Ast, A]$ is
	self-adjoint; by the spectral theorem, there is an orthonormal basis $e_1',
	\dotsc, e_r'$ such that $[\Ast, A] e_i' = \lambda_i e_i'$ with $\lambda_1,
	\dotsc, \lambda_n \in \RR$. Therefore
	\[
		\bigl\lVert [\Ast, A] \bigr\rVert^2 = \sum_{i=1}^r \lambda_i^2
		= \sum_{i=1}^r \bigl( \norm{A e_i'}^2 - \norm{\Ast e_i'}^2 \bigr)^2
		\leq \sum_{i=1}^r \norm{A e_i'}^4 + \sum_{i=1}^r \norm{\Ast e_i'}^4,
	\]
	because $\lambda_i = \biginner{(\Ast A - A \Ast) e_i'}{e_i'} = \norm{A e_i'}^2 -
	\norm{\Ast e_i'}^2$. At the same time,
	\[
		1 = \norm{A}^2 = \sum_{i=1}^r \norm{A e_i'}^2,
	\]
	and so both sums on the right-hand side are $\leq 1$.
\end{proof}

\newpar
We can now apply the version of Ahlfors' lemma in \Cref{lem:Ahlfors}, and
deduce the following very striking upper bound for the Hodge norm of the Higgs field.

\begin{pcor} \label{cor:Higgs-bound}
	Let $E$ be a polarized variation of Hodge structure of rank $r \geq 1$ on the left
	half-plane $\HH$. Then the Higgs field $\theta_{\partial/\partial z}$ satisfies
	the inequality
	\[
		h_{\End(E)} \bigl( \theta_{\partial/\partial z}, \theta_{\partial/\partial z}
		\bigr) \leq \frac{C_0^2}{(\Re z)^2},
	\]
	where $C_0 = \frac{1}{2} \sqrt{\binom{r+1}{3}}$.
\end{pcor}

\begin{proof}
	Since the Higgs field is trivial when $r = 1$, we can assume that $r \geq 2$.
	Let us again use the abbreviations $A = \theta_{\partial/\partial z}$ and $\Ast =
	\theta_{\partial/\partial \zb}$. If we put the inequalities in
	\Cref{lem:A-inequality} and \Cref{lem:nilpotent} together, we get
	\[
		\frac{\partial^2}{\partial z \partial \zb} \log h_{\End(E)}(A, A)
		\geq \frac{12}{r(r+1)(r-1)} h_{\End(E)}(A,A).
	\]
	Now apply \Cref{lem:Ahlfors} to reach the desired conclusion.
\end{proof}

From an analytic point of view, this is really the central fact about
polarized variations of Hodge structure. The amazing thing is that the constant
depends on nothing but the rank of the variation of Hodge structure. This fact turns
out to be especially useful for the theory in several variables.

\newpar
The induced variation of Hodge structure on the bundle $\End(E)$ has rank $r^2$, but
in that case, there is a bound on the Higgs field that is much better than simply
replacing $r$ by $r^2$ is the result above.

\begin{pcor} \label{cor:Higgs-bound-End}
	Let $E$ be a polarized variation of Hodge structure of rank $r \geq 1$ on the left
	half-plane $\HH$. Then the Higgs field $\theta_{\partial/\partial z}$ of the
	induced variation of Hodge structure on $\End(E)$ satisfies the inequality
	\[
		h_{\End(\End(E))} \bigl( \theta_{\partial/\partial z}, \theta_{\partial/\partial z}
		\bigr) \leq \frac{2r C_0^2}{(\Re z)^2},
	\]
	where again $C_0 = \frac{1}{2} \sqrt{\binom{r+1}{3}}$.
\end{pcor}

\begin{proof}
	Let $A = \theta_{\partial/\partial z}$ be the Higgs field acting on $V$.
	The induced Higgs field on $\End(V)$ is then equal to the commutator $(\ad A)(X) =
	[A,X]$. It is therefore enough to prove that for any nilpotent matrix $A \in
	\End(V)$, one has
	\[
		\norm{\ad A}^2 \leq 2r \norm{A}^2.
	\]
	Let $e_1, \dotsc, e_r \in V$ be an orthonormal basis such that the matrix for $A$
	is strictly upper triangular. Then the matrices $E_{i,j}$, whose only nonzero entry
	is a $1$ in the position $(i,j)$, form an orthonormal basis for $\End(V)$, and
	a short calculation gives
	\[
		(\ad A)(E_{i,j}) = \sum_{k<i} a_{k,i} E_{k,j} - \sum_{j < k} a_{j,k} E_{i,k}.
	\]
	We therefore get
	\begin{align*}
		\norm{\ad A}^2 &= \sum_{i,j} \abs{(\ad A)(E_{i,j})}^2 \\
			&= \sum_{i,j} \Bigl( \abs{a_{1,i}}^2 + \dotsb + \abs{a_{i-1,i}}^2 +
			\abs{a_{j,j+1}}^2 + \dotsb + \abs{a_{j,r}}^2 \Bigr) 
			\leq 2r \norm{A}^2,
	\end{align*}
	as claimed.
\end{proof}

\subsection{Proof of the monodromy theorem}

\newpar
Let us summarize the results so far. We have shown that if $v \in V$ is a nontrivial
multi-valued flat section of $E$, then the function $\varphi = \log h(v,v)$ is smooth
and subharmonic on $\HH$, and its first derivatives are bounded by
\begin{equation} \label{eq:derivatives-final}
	\ABS{\frac{\partial \varphi}{\partial z}}
	= \ABS{\frac{\partial \varphi}{\partial \zb}}
	\leq \frac{2C_0}{\abs{\Re z}},
\end{equation}
where $C_0 = \half \sqrt{\binom{r+1}{3}}$ and $r = \rk E$. Everything up to this point,
maybe except for the optimal value of the constant, was known to Simpson
\cite[\S2]{Simpson} back in the 1990s. Now we add a small -- but crucial! -- new
insight, namely that the uniform bound on the first derivatives can be used directly
to control the behavior of the function $\varphi$ for $\Re z \ll 0$.

\newpar
Before we can explain how this works, we first need to introduce the \define{monodromy
operator} $T \in \Aut(V)$. Following Schmid \cite[(4.4)]{Schmid}, we define this by
the formula
\[
	(Tv)(z) = v(z-2 \pi i),
\]
where $i = \sqrt{-1}$. That is to say, for any $v \in V$, the function $z \mapsto
v(z-2 \pi i)$ is again a flat section of the bundle $E$ on $\HH$, and we define $Tv
\in V$ to be this section. The flat pairing on $E$ induces a hermitian pairing 
\[
	Q \colon V \tensor_{\CC} \Vb \to \CC,
\]
and one checks easily that $T$ preserves $Q$, in the sense that $Q(Tv, Tw) =
Q(v,w)$ for every $v,w \in V$. We write the Jordan decomposition of $T$ in the
form
\[
	T = T_s \cdot e^{2 \pi i N},
\]
with $T_s \in \Aut(V)$ semisimple and $N \in \End(V)$ nilpotent. Then
\[
	Q(T_s v, T_s w) = Q(v, w) \quad \text{and} \quad
	Q(Nv, w) = Q(v,Nw),
\]
for every $v,w \in V$, because $Q$ is conjugate-linear in the second argument. Note
that $T$ depends on the choice of $i = \sqrt{-1}$, but $N$ is independent of it.

\newpar
The space $V$ of multi-valued flat sections gives us a trivialization of the vector
bundle $E$ on the left half-plane $\HH$. We can therefore consider the polarized
variation of Hodge structure as a family of Hodge structures on $V$; as usual, we
denote by the letter $\Phi(z)$ the Hodge filtration of the Hodge structure at the
point $z \in \HH$. The definition of the monodromy operator has the following
consequence.

\begin{plem} \label{lem:Phi-monodromy}
	For every $z \in \HH$, we have the equality
	\[
		\Phi(z+2 \pi i) = T \, \Phi(z),
	\]
	where both sides are filtrations on $V$. 
\end{plem}

\begin{proof}
	This may be a bit confusing, so let us write out the argument. The
	trivialization of the pullback of the bundle $E$ along $\exp \colon \HH \to \dst$
	gives us, for every point $z \in \HH$, an isomorphism of complex vector spaces
	\[
		\phi_z \colon V \to (\exp^{\ast} E) \restr{z} \cong E_t, \quad
		\phi_z(v) = v(z).
	\]
	Here $t = e^z$. Since $v(z+2 \pi i) = (T^{-1}v)(z)$, the diagram
	\[
		\begin{tikzcd}[row sep=tiny,column sep=huge]
			V \arrow{dd}{T^{-1}} \drar[bend left=15]{\phi_{z+2 \pi i}} \\
			& E_t \\
			V \urar[swap,bend right=15]{\phi_z}
		\end{tikzcd}
	\]
	commutes. The way the period mapping is constructed, we have
	\[
		F_{\Phi(z+2 \pi i)}^p = \phi_{z+2 \pi i}^{-1} \bigl( F^p E_t \bigr)
		= (T^{-1})^{-1} \phi_z^{-1} \bigl( F^p E_t \bigr) = T F_{\Phi(z)}^p,
	\]
	which is what we wanted to prove.
\end{proof}

\newpar
Given a multi-valued flat section $v \in V$, we can write
the value of the function $h(v,v)$ at the point $z \in \HH$ in the equivalent form
\[
	h(v,v)(z) = \norm{v}_{\Phi(z)}^2,
\]
using the notation for the Hodge norm from \Cref{par:HS}. Since $T
\in \Aut(V, Q)$, we have
\begin{equation} \label{eq:translation}
	h(v,v)(z+2 \pi i) = \norm{v}_{T \Phi(z)}^2 = \norm{T^{-1} v}_{\Phi(z)}^2
	= h(T^{-1} v, T^{-1} v)(z).
\end{equation}
It is this relation that connects the asymptotic behavior of the Hodge norm to the
monodromy transformation $T$.

\newpar
To illustrate the power of the derivative bound in \eqref{eq:derivatives-final}, we
are now going to prove the monodromy theorem \cite[Lem.~4.5]{Schmid}.

\begin{pprop} \label{prop:monodromy-theorem}
	If $\lambda \in \CC$ is an eigenvalue of $T$, then $\abs{\lambda} = 1$.
\end{pprop}

\begin{proof}
	Let $v \in V$ be a nonzero eigenvector with $T v = \lambda v$, and consider the
	smooth function $\varphi = \log h(v,v)$. The relation in \eqref{eq:translation}
	implies that
	\[
		\varphi(z + 2 \pi i) = \log \abs{\lambda}^{-2} + \varphi(z)
	\]
	for every $z \in \HH$. On the other hand, we clearly have
	\[
		\varphi(z + 2 \pi i) - \varphi(z) = 
		\int_0^{2 \pi} \left( i \frac{\partial \varphi}{\partial z}(z + i y) 
		- i \frac{\partial \varphi}{\partial \zb}(z + i y) \right) \dy,
	\]
	and on account of \eqref{eq:derivatives-final}, this leads to the inequality
	\[
		\Abs{\varphi(z + 2 \pi i) - \varphi(z)} \leq \frac{8 \pi C_0}{\abs{\Re z}},
	\]
	where $C_0 = \frac{1}{2} \sqrt{\binom{r+1}{3}}$. Letting $\abs{\Re z} \to
	\infty$, we conclude that $\log \abs{\lambda} = 0$, whence $\abs{\lambda} = 1$.
\end{proof}

\subsection{Effective bounds for the monodromy transformation}

\newpar
We can use the derivative bound in \eqref{eq:derivatives-final} to control the
Hodge norms of the nilpotent operator $N$ and of the projections $P_{\lambda} \colon
V \to E_{\lambda}(T_s)$ to the eigenspaces of $T_s$. The fact that $N$ and
$\theta_{\partial/\partial z}$ are both of order $\abs{\Re z}^{-1}$ is no accident:
during the proof of the nilpotent orbit theorem, we will see that the difference
between these two operators, modulo $F^0 \End(V)_{\Phi(z)}$, is of order $e^{-\eps
\abs{\Re z}}$. (The precise statement is in \Cref{thm:effective-estimates}.)

\begin{pprop} \label{prop:N-norm}
	There is a constant $C > 0$, whose value only depends on $\dim V$ and on the
	minimal polynomial of the monodromy transformation $T \in \GL(V)$, such that:
	\begin{aenumerate}
	\item For every $v \in V$ and every $z \in \HH$ with $\Re z \leq -1$, one has
		\[
			\norm{Nv}_{\Phi(z)} \leq \frac{C}{\abs{\Re z}} \norm{v}_{\Phi(z)}.
		\]
	\item For every $v \in V$ and every $z \in \HH$ with $\Re z \leq -1$, one has
		\[
			\norm{P_{\lambda} v}_{\Phi(z)} \leq C \norm{v}_{\Phi(z)}.
		\]
	\end{aenumerate}
\end{pprop}

\newpar 
Given the bound on the derivative of the function $\log \norm{v}_{\Phi(z)}^2$, 
the proof is just linear algebra. By the same argument as in the proof of
\Cref{prop:monodromy-theorem}, we have
	\[
		\Bigl\lvert \log \norm{Tv}_{\Phi(z)}^2 - \log \norm{v}_{\Phi(z)}^2 \Bigr\rvert 
		= \Bigl\lvert \log \norm{v}_{\Phi(z - 2 \pi i)}^2 
		- \log \norm{v}_{\Phi(z)}^2 \Bigr\rvert \leq \frac{8 \pi C_0}{\abs{\Re z}};
	\]
	after exponentiating, this becomes
	\[
		e^{-8 \pi C_0 / \abs{\Re z}} \leq
			\frac{\norm{Tv}_{\Phi(z)}^2}{\norm{v}_{\Phi(z)}^2}
			\leq e^{8 \pi C_0 / \abs{\Re z}}.
	\]
	If we restrict to $\Re z \leq -1$ for simplicity, it follows that
	\begin{equation} \label{eq:inequality-T}
		\Bigl\lvert \norm{Tv}_{\Phi(z)}^2 - \norm{v}_{\Phi(z)}^2 \Bigr\rvert
		\leq \frac{C'}{\abs{\Re z}} \norm{v}_{\Phi(z)}^2,
	\end{equation}
	where $C' = (e^{8 \pi C_0}-1)/(8 \pi C_0)$. 

\newpar
The following elementary lemma, applied to the inner products
$\inner{\argbl}{\argbl}_{\Phi(z)}$, turns the inequality in \eqref{eq:inequality-T}
into a bound for the operator norm of the nilpotent operator $N$.

\begin{plem} \label{lem:unitary-nilpotent}
	Let $V$ be a finite-dimensional complex vector space with an inner product
	$\inner{\argbl}{\argbl}$. Let $T \in \GL(V)$ be an endomorphism whose eigenvalues
	have absolute value $1$. Suppose that $T$ is close to unitary, in the sense that
	there is a constant $\eps > 0$ such that
	\[
		\Bigl\lvert \norm{Tv}^2 - \norm{v}^2 \Bigr\rvert \leq \eps \norm{v}^2
		\quad \text{for all $v \in V$.}
	\]
	Then there are two constants $C > 0$ and $m \in \NN$, whose values only depends on
	the integer $\dim V$ and on the minimal polynomial of the operator $T$, such that
	\[
		\norm{Nv} \leq C \eps (1 + \eps^m) \cdot \norm{v} \quad \text{for all $v \in V$,}
	\]
	where $N = (2 \pi i)^{-1} \log T_u$ is the logarithm of the unipotent part of $T$.
\end{plem}

\begin{proof}
	Set $r = \dim V$, and choose an orthonormal basis $e_1, \dotsc, e_r \in V$ such
	that the matrix representing $T$ is upper triangular. Write $T = U(\id + B)$,
	where $U$ is a diagonal matrix and $B$ is strictly upper triangular. The entries
	of $U$ are the eigenvalues of $T$, hence of absolute value $1$, and so $U$ is
	unitary. Pick any two integers $1 \leq i < j \leq r$, and let $b_{i,j} = \inner{B
	e_j}{e_i}$ denote the corresponding entry of the matrix representing $B$. By
	applying the hypothesis to vectors of the form $v = e_i + z e_j$ with $\abs{z} =
	1$, we get
	\[
		2 \Re(z b_{i,j}) = 2 \Re \biginner{B(e_i + z e_j)}{e_i + z e_j}
		\leq \eps \norm{e_i + z e_j}^2 = 2 \eps,
	\]
	and by choosing $z$ appropriately, this gives $\abs{b_{i,j}} \leq \eps$.
	Consequently, the $L^2$-norm of the operator $B \in \End(V)$ is bounded by
	$\norm{B}^2 \leq \binom{r}{2} \eps^2$, and therefore
	\[
		\norm{Bv} \leq r \eps \norm{v} \quad \text{for all $v \in V$.}
	\]
	Now it is a basic fact from linear algebra that the nilpotent operator $N = (2 \pi
	i)^{-1} \log T_u$ can be written as a polynomial in $T$, of the form
	\[
		N = P(T) \cdot \prod_{\lambda} (T - \lambda \id),
	\]
	where the product runs over the distinct eigenvalues of $T$; the polynomial $P(t) \in
	\CC[t]$ is uniquely determined by the minimal polynomial of $T$. Consequently,
	\[
		N = P(T) \cdot \prod_{\lambda} \bigl( (U - \lambda \id) + UB \bigr),
	\]
	and because the operator norm of $P(T)$ is easily bounded using the inequality
	$\norm{Tv}^2 \leq (1 + \eps) \norm{v}^2$, it suffices to estimate the
	operator norm of the product. Given our bound $r \eps$ for the operator norm of
	$B$, this comes down to 
	\[
		\prod_{\lambda} (U - \lambda \id) = 0,
	\]
	which holds because $U$ is a diagonal matrix.
\end{proof}

\newpar
The same argument also proves that the projection operators
\[
	P_{\lambda} \colon V \to E_{\lambda}(T_s)
\]
are uniformly bounded, meaning that there is a constant $C > 0$, whose value again
only depends on the integer $\dim V$ and on the minimal polynomial of $T \in \GL(V)$,
such that
\begin{equation} 
	\norm{P_{\lambda} v}_{\Phi(z)} \leq C \norm{v}_{\Phi(z)}
	\quad \text{for $v \in V$ and $\Re z \leq -1$.}
\end{equation}
The reason is that $P_{\lambda}$ can again be expressed as a certain
polynomial in $T$, whose operator norm can then be bounded using
\eqref{eq:inequality-T}. This finishes the proof of \Cref{prop:N-norm}.

\section{Period domains and period mappings}

\newpar
In this chapter, we briefly review the basic properties of period domains and period
mappings, in the setting of \emph{complex} Hodge structures. The theory is actually
simpler than in the case of rational Hodge structures \cite[\S3]{Schmid}, because the
Lie groups that are involved are just real forms of the general linear group
$\GL(V)$.

\subsection{The period domain and its compact dual}

\newpar
Fix a finite-dimensional complex vector space $V$ and a hermitian
pairing $Q \colon V \tensor_{\CC} \Vb \to \CC$. Let $D$ be the \define{period domain}
parametrizing Hodge structures of the type we are interested in: each point $o \in D$
corresponds to a Hodge structure 
\[
	V = \bigoplus_{p+q=n} V_o^{p,q}
\]
on the vector space $V$, with fixed Hodge numbers $\dim V_o^{p,q}$, and polarized by
the hermitian pairing $Q$. We assume that there is at least one such Hodge
structure; in particular, $Q$ is nondegenerate and of the correct signature. In this
setting, the real Lie group 
\[
	G = \menge{g \in \GL(V)}{\text{$Q(gv,gw) = Q(v,w)$ for all $v,w \in V$}}
\]
acts transitively on $D$, according to the rule $V_{g \cdot o}^{p,q} = g(V_o^{p,q})$.
(This is easily proved by chosing an orthonormal basis adapted to the Hodge
decomposition, and then mapping one such orthonormal basis to another.)
The inner products at the two points $o$ and $g \cdot o$ are related by the
formula
\[
	\inner{gv}{gw}_{g \cdot o} = \inner{v}{w}_o,
\]
which follows from the fact that $g \in G$. 

\newpar \label{par:glie}
The Lie algebra of $G$ is easily seen to be
\[
	\glie = \menge{A \in \End(V)}{\text{$Q(Av,w) + Q(v,Aw) = 0$ 
			for all $v,w \in V$}}.
\]
Note that $\glie \tensor_{\RR} \CC = \End(V)$, which means that $G$ is a
non-compact real form of the complex Lie group $\GL(V)$. This also gives us a real
structure on $\End(V)$; the ``complex conjugate'' of an endomorphism $A \in
\End(V)$ is $-A^{\dagger}$, where the adjoint $A^{\dagger}$ with respect to the
hermitian pairing $Q$ is defined by the formula
\[
	Q(Av,w) = Q(v,A^{\dagger} w) \quad \text{for all $v,w \in V$.}
\]
This makes sense because $Q$, being a polarization, is nondegenerate.

\newpar
The stabilizer of a point $o \in D$ is the compact subgroup
\[
	H_o = \menge{g \in G}{\text{$g(V_o^{p,q}) = V_o^{p,q}$ for all $p,q$}}
	\subseteq G,
\]
and this gives us an isomorphism $D \cong G/H_o$. We can use it to
describe the real tangent space $T_o D$ in terms of Hodge structures. The Hodge
structure on $V$ induces a Hodge structure of weight $0$ on $\End(V)$, by setting
\[
	\End(V)_o^{j,-j} = \menge{A \in \End(V)}%
		{\text{$A(V_o^{p,q}) \subseteq V_o^{p+j,q-j}$ for all $p,q$}}.
\]
It is an easy exercise to check that the decomposition
\begin{equation} \label{eq:HS-End}
	\End(V) = \bigoplus_{j \in \ZZ} \End(V)_o^{j,-j}
\end{equation}
is an $\RR$-Hodge structure (with respect to the real structure $\glie$), and that
the induced polarization on $\End(V)$ is equal to the trace pairing
\[
	(A, B) \mapsto \tr(A B^{\dagger}).
\]
Since the conjugate of $B$ is $-B^{\dagger}$, we can view the polarization on
$\End(V)$ as coming from the $\RR$-valued bilinear pairing $(A, B) \mapsto -\tr(AB)$
on $\glie$. 

\newpar
Here is the proof that the trace pairing is indeed a polarization for the induced Hodge
structure on $\End(V)$. The same argument is needed in \Cref{sec:sltwo}, and so we
give some details.

\begin{plem} \label{lem:tr-polarization}
	The pairing
	\[
		\End(V) \tensor_{\CC} \wbar{\End(V)} \to \CC, \quad
		(A, B) \mapsto \tr(A B^{\dagger}),
	\]
	polarizes the Hodge structure on $\End(V)$.
\end{plem}

\begin{proof}
	It is easy to see that $\tr A^{\dagger} = \wbar{\tr A}$ for every $A \in \End(V)$.
	Therefore
	\[
		\tr(B A^{\dagger}) = \tr \bigl( (A B^{\dagger})^{\dagger} \bigr)
		= \wbar{\tr(A B^{\dagger})},
	\]
	and so the pairing is hermitian symmetric. The Hodge decomposition is 
	orthogonal with respect to the trace pairing. Indeed, if $A \in \End(V)_o^{i,-i}$
	and $B \in \End(V)_o^{j,-j}$, with $i \neq j$, then $A B^{\dagger} \in
	\End(V)_o^{i-j,j-i}$ is a nilpotent endomorphism, and so $\tr(A B^{\dagger}) = 0$. 

	Finally, we need to explain why $(-1)^j \tr(A A^{\dagger}) > 0$ if $A \in
	\End(V)_o^{j,-j}$ is nonzero. The point is that $(-1)^j A^{\dagger}$
	is exactly the adjoint of $A$ with respect to the inner product on $V$.
	Indeed, for $v \in V^{p,q}$ and $w \in V^{p+j,q-j}$, we have
	\[
		\inner{Av}{w}_o = (-1)^{q-j} Q(Av,w) = (-1)^{q-j} Q(v, A^{\dagger} w)
		= (-1)^j \inner{v}{A^{\dagger} w}_o.
	\]
	Consequently, the endomorphism $(-1)^j A^{\dagger} A$ is self-adjoint with respect to
	the inner product, and also positive definite, because
	\[
		\bigl\langle (-1)^j A^{\dagger} A v, v \bigr\rangle_o
		= \inner{Av}{Av}_o = \norm{Av}_o^2.
	\]
	This clearly implies that $(-1)^j \tr(A A^{\dagger}) > 0$. In fact, by choosing an
	orthonormal basis in $V$, one can easily see that this expression is nothing but the
	operator norm of $A$ with respect to the inner product.
\end{proof}

\newpar
Using the notation from \eqref{eq:HS-End}, the Lie algebra of the stabilizer subgroup
$H_o$ is exactly
the intersection $\glie_o^{0,0} = \glie \cap \End(V)_o^{0,0}$, and so we conclude that
the real tangent space to the period domain at a point $o \in D$ is
\begin{equation} \label{eq:D-tangent-space}
	T_o D \cong \glie / \glie_o^{0,0}.
\end{equation}
From the polarized Hodge structure on $\End(V)$, we get a positive-definite inner
product on $\glie$. It induces an inner product on the quotient
$\glie/\glie_o^{0,0}$, hence on each tangent space $T_o D$. More precisely, define
\[
	\mlie_o = \glie \cap \bigoplus_{k \neq 0} \End(V)_o^{k,-k},
\]
so that $\glie = \mlie_o \oplus \glie_o^{0,0}$. Note that $\mlie_o$ is just a linear
subspace of $\glie$, not a Lie subalgebra. Then $T_o D \cong \mlie_o$,
and the inner product on $T_o D$ is simply the inner product on $\mlie_o$ coming 
from the Hodge structure on $\End(V)$. In this way, we obtain a Riemannian metric on
the period domain $D$; one shows without trouble that this metric is
$G$-invariant. We denote by 
\begin{equation} \label{eq:D-distance}
	d_D \colon D \times D \to [0, +\infty]
\end{equation}
the resulting $G$-invariant distance function on the period domain. (The value
$+\infty$ is included because $D$ does not have to be connected.)

\begin{pexa} \label{ex:period-domain}
	Consider Hodge structures of the form
	\[
		V = V^{1,0} \oplus V^{0,1}
	\]
	on the vector space $V = \CC^2$ that are polarized by the indefinite hermitian pairing 
	\[
		Q = \begin{pmatrix}
			1 & 0 \\
			0 & -1 
		\end{pmatrix}.
	\]
	Writing $V^{1,0} = \CC(1,t)$, we must have $1 - \abs{t}^2 > 0$, hence $t
	\in \Delta$. The period domain is therefore the unit disk $\Delta$; the Hodge
	structure corresponding to the point $t \in \Delta$ is
	\[
		\CC^2 = \CC(1,t) \oplus \CC(\tb, 1).
	\]
	The $G$-invariant Riemannian metric on the period domain is the usual Poincar\'e
	metric on $\Delta$. Indeed, $G$ is just the automorphism group of $\Delta$, and so
	it is enough to compute the metric on the tangent space $T_0 \Delta$. In the Hodge
	structure at the point $0 \in \Delta$, the two standard basis vectors form an
	orthonormal basis. A short computation shows that
	\[
		\mlie_0 = \MENGE{\begin{pmatrix} 0 & \bar{w} \\ w & 0 \end{pmatrix}}{w \in \CC},
	\]
	and that the $\RR$-linear mapping $\mlie_0 \to T_0 \Delta$ takes the matrix
	$\begin{pmatrix} 0 & \bar{w} \\ w & 0 \end{pmatrix}$ to the complex number $w$.
	The inner product on the $\RR$-vector space $\CC = T_0 \Delta$
	is therefore just the usual Euclidean inner product.
\end{pexa}

\newpar
If two points $o,p \in D$ are close to each other, then of course the corresponding
inner products $\inner{v}{w}_o$ and $\inner{v}{w}_p$ must also be close. The
following lemma makes this precise.

\begin{plem} \label{lem:distance-inner-products}
	There are two constants $\delta > 0$ and $C > 0$, such that if $o,p \in D$ are two
	points with $d_D(p,o) < \delta$, then one has
	\[
		\bigl\lvert \inner{v}{w}_p - \inner{v}{w}_o \bigr\rvert
		\leq C \norm{v}_o \norm{w}_o \cdot d_D(p,o)
	\]
	for every $v,w \in V$.
\end{plem}

\begin{proof}
	Because of the $G$-invariance of the metric on $D$, it suffices to prove this
	when $o \in D$ is a fixed choice of base point. The exponential mapping
	\[
		\mlie_o \to D, \quad A \mapsto e^A \cdot o,
	\]
	is a real-analytic local diffeomorphism. It follows that there is a constant
	$\delta > 0$ such that every point $p \in D$ with $d_D(p,o) < \delta$ can be
	uniquely written in the form $p = e^A \cdot o$, where $A \in \mlie_o$ satisfies
	$\norm{A}_o \leq 2 d_D(p,o)$. Clearly,
	\[
		\bigl\lvert \inner{v}{w}_p - \inner{v}{w}_o \bigr\rvert
		= \bigl\lvert \inner{e^{-A} v}{e^{-A} w}_o - \inner{v}{w}_o \bigr\rvert
		\leq C' \norm{v}_o \norm{w}_o \cdot \norm{A}_o
	\]
	for some constant $C' > 0$ that is independent of $v,w \in V$. This gives the
	desired result (with $C = 2C'$).
\end{proof}

\newpar
The period domain $D$ is an open subset of the \define{compact dual} $\Dch$, which is
a closed algebraic subvariety of a product of Grassmannians, and therefore a
projective complex manifold. The points of $\Dch$ parametrize all decreasing
filtrations $F = F^{\bullet} V$ on the vector space $V$, subject only to the
constraint that 
\[
	\dim F^p = \dim V^{p,q} + \dim V^{p+1,q-1} + \dim V^{p+2,q-2} + \dotsb \quad
	\text{for all $p \in \ZZ$.}
\]
Since we are working with complex Hodge structures, there are no further
restrictions on $F$, unlike in Schmid's paper \cite[\S3]{Schmid}. The complex
Lie group $\GL(V)$ acts transitively on $\Dch$, by the rule $(gF)^p = g(F^p)$. The
holomorphic tangent space at a point $F \in \Dch$ is therefore 
\begin{equation} \label{eq:T-Dch}
	T_F^{1,0} \Dch \cong \End(V) / F^0 \End(V),
\end{equation}
where we are using the notation
\[
	F^0 \End(V) = \menge{A \in \End(V)}%
		{\text{$A(F^p) \subseteq F^p$ for all $p \in \ZZ$}}.
\]
If we go to a point $o \in D$, and denote by $F_o$ the Hodge filtration in the Hodge
structure on $V$, the natural isomorphism of $\RR$-vector spaces
\[
	T_o D \cong T_{F_o}^{1,0} \Dch
\]
becomes, in terms of Hodge structures, the isomorphism
\begin{equation} \label{eq:iso-tangent-spaces}
	\mlie_o \cong \glie / \glie_o^{0,0} \cong \End(V) / F_o^0 \End(V),
\end{equation}
which is of course valid in any $\RR$-Hodge structure of weight $0$. Via the
embedding $D \subseteq \Dch$, the period domain $D$ inherits the structure of a
complex manifold.

\newpar
We can also define a hermitian metric on the compact dual $\Dch$, which is
however not $\GL(V)$-invariant. For that, choose a base point $o \in D$, and
denote by $\inner{v}{w} = \inner{v}{w}_o$ the resulting inner product on the vector space
$V$. It induces an inner product on $\End(V)$, and therefore on each holomorphic
tangent space
\[
	T_F^{1,0} \Dch \cong \End(V) / F^0 \End(V).
\]
This gives us a hermitian metric on the compact dual $\Dch$. Note that, unlike the
$G$-invariant metric on the period domain $D$, this metric is \emph{not} invariant
under the full symmetry group $\GL(V)$, only under the (much smaller) unitary group
$U_o \subseteq \GL(V)$ of the reference inner product $\inner{v}{w}$. The resulting
distance function is
\[
	\dDch \colon \Dch \times \Dch \to [0, +\infty);
\]
note that $\Dch$ has finite diameter, because it is compact and connected. It is
easy to see that the unitary group $U_o$ acts transitively on $\Dch$: choose a
basis adapted to a given filtration, and convert it into an orthonormal basis adapted
to the filtration with the help of the Gram-Schmid process.  

\newpar \label{par:local-coordinates-Dch}
We can use the exponential mapping to get local coordinates on $\Dch$. 

\begin{plem} \label{lem:local-coordinates-Dch}
	There is a constant $\delta > 0$ such that for every point $p \in \Dch$, the set
	\[
		B_{\delta}(p) = \menge{x \in \Dch}{\dDch(p, x) < \delta}
	\]
	is biholomorphic to an open set in Euclidean space, and the function
	$\dDch(p, \argbl)$ is comparable, on $B_{\delta}(p)$, to the Euclidean distance, up
	to a factor of $2$.  
\end{plem}

\begin{proof}
	Fix a base point $o \in D$. Since $\Dch$ is a homogeneous space for the unitary
	group $U_o \subseteq \GL(V)$, and since the distance function $\dDch$ is
	$U_o$-invariant, it suffices to prove the statement when $p = o$. From the induced
	Hodge structure on $\End(V)$, we have a direct sum decomposition $\End(V) = W
	\oplus F_o^0 \End(V)$, where 
	\[
		W = \bigoplus_{j < 0} \End(V)_o^{j,-j},
	\]
	By general theory, the exponential mapping 
	\[
		W \to \Dch, \quad A \mapsto e^A \cdot o,
	\]
	is a local diffeomorphism; it is also holomorphic, and therefore a local isomorphism
	of complex manifolds. It follows that there is a constant $\delta > 0$ such that
	the exponential mapping induces an isomorphism between $B_{\delta}(o) \subseteq
	\Dch$ and an open neighborhood of the origin in the vector space $W$.
	Since the distance function $\dDch$ is continuous, we can arrange moreover that
	\[
		\frac{1}{2} \norm{A}_o \leq \dDch \bigl( e^A \cdot o, o \bigr) \leq 2 \norm{A}_o
	\]
	on the open ball in question. 
\end{proof}

\newpar \label{par:Adg}
We record one additional fact about the distance function $\dDch$ that we are going to
need towards the end of the proof of the nilpotent orbit theorem (in
\Cref{lem:D-criterion}). Fix an operator $g \in \GL(V)$, and consider the mapping
\[
	g \colon \Dch \to \Dch, \quad F \mapsto gF.
\]
What is the effect of this on tangent vectors? Let $\Ad g \colon \End(V) \to
\End(V)$ be the linear mapping $A \mapsto g A g^{-1}$.  Then one checks that the diagram 
\[
	\begin{tikzcd}
		\End(V) \dar \rar{\Ad g} & \End(V) \dar \\
		T_F^{1,0} \Dch \rar{dg} & T_{gF}^{1,0} \Dch
	\end{tikzcd}
\]
is commutative. In particular, left translation by $g$ changes the length of
holomorphic tangent vectors, computed using the hermitian metric on $\Dch$, at most
by the operator norm of $\Ad g$, relative to our fixed inner product on $\End(V)$.
After integrating this, we arrive at the useful formula
\begin{equation} \label{eq:Dch-translation}
	\dDch(g \cdot p, g \cdot q) \leq \max_{\norm{A} = 1} \norm{(\Ad g) A} \cdot \dDch(p, q),
\end{equation}
valid for any two points $p,q \in \Dch$.

\subsection{The period mapping}

\newpar
We return to our variation of Hodge structure $E$ on the punctured disk. At each
point $z \in \HH$, we have a polarized Hodge structure on $V$, whose Hodge filtration
we denoted by the symbol $\Phi(z)$. This gives us the \define{period mapping}
\[
	\Phi \colon \HH \to D;
\]
we shall argue in a moment that it is holomorphic (with respect to the complex
structure on the period domain induced by $D \subseteq \Dch$.) The monodromy
transformation satisfies $T \in G$, hence also $T_s \in G$ and $2 \pi i \,
N \in \glie$, and we have
\[
	\Phi(z + 2 \pi i) = T \cdot \Phi(z).
\]
From now on, we use the subscript $\Phi(z)$ to refer to the Hodge
structure at the point $z \in \HH$, to avoid confusion. So the Hodge decomposition is
\[
	V = \bigoplus_{p+q=n} V_{\Phi(z)}^{p,q}, 
\]
the Hodge norm is $\norm{v}_{\Phi(z)}$, and so on.

\begin{plem} \label{lem:dPhi}
	The derivative of the period mapping is given by
	\[
		d \Phi \big\vert_z = (\theta + \thetast) \big\vert_z + \glie_{\Phi(z)}^{0,0},
	\]
	where we consider both sides as $\RR$-linear mappings from $T_z \HH$ to
	the quotient $\glie/\glie_{\Phi(z)}^{0,0}$.
\end{plem}

\begin{proof}
	Fix a point $z_0 \in \HH$. Choose an orthonormal frame for the bundle $E$,
	consisting of smooth sections $e_1, \dotsc, e_r \in A^0(\HH, E)$ with the property that
	\[
		\inner{e_i}{e_j}_{\Phi(z)} = \delta_{i,j} \quad \text{for all $z \in \HH$,}
	\]
	and such that $e_i \in A^0(\HH, E^{p_i,q_i})$ for certain $p_i,q_i \in \ZZ$. Denote by
	$v_1, \dotsc, v_r \in V$ the unique multi-valued flat sections such that $v_i(z_0)
	= e_i(z_0)$. We can define a smooth function $g \colon \HH \to \GL(V)$ by the
	condition that
	\[
		e_j(z) = g(z) \cdot v_j \quad \text{for $j=1, \dotsc, r$ and $z \in \HH$}.
	\]
	Clearly, $g(z_0) = \id$, and since
	\[
		(-1)^{q_i} Q \Bigl( g(z) v_i, g(z) v_j \Bigr) =
		\inner{e_i}{e_j}_{\Phi(z)} = \inner{v_i}{v_j}_{\Phi(z_0)} = (-1)^{q_i} Q(v_i,v_j),
	\]
	we have $g(z) \in G$ for all $z \in \HH$. The function $g \colon \HH \to
	G$ is a smooth lifting of the period mapping, which can now be described very
	concretely as
	\[
		\Phi(z) = g(z) \cdot o \quad \text{for $z \in \HH$.}
	\]
	The derivative $d\Phi$ is easily computed from this formula. Writing
	\[
		e_j = \sum_{i=1}^r g_{i,j} v_i
	\]
	with smooth functions $g_{i,j} \colon \HH \to \CC$, our matrix-valued function $g
	\colon \HH \to G$ is represented by the $r \times r$-matrix with entries
	$g_{i,j}$. At the point $z_0$, the derivative
	\[
		dg \big\vert_{z_0} \colon T_{z_0} \HH \to \glie
	\]
	is a lifting for the derivative of the period mapping
	\[
		d\Phi \big\vert_{z_0} \colon T_{z_0} \HH \to T_{\Phi(z_0)} D.
	\]
	It is represented by the $r \times r$-matrix with entries $d g_{i,j}$. Recall that
	the flat connection on $E$ is $d = \partial + \dbar + \theta + \thetast$.
	Since $v_1, \dotsc, v_r$ are flat sections, we get
	\[
		(\partial + \dbar + \theta + \thetast) e_j = d e_j 
		= \sum_{i=1}^r d g_{i,j} \tensor v_i.
	\]
	Now the operator $\partial + \dbar$ preserves the subbundle $E^{p,q}$, and so this
	term disappears when we project into $\glie/\glie_{\Phi(z_0)}^{0,0}$. The result
	is that 
	\[
		dg \big\vert_{z_0} \equiv (\theta + \thetast) \big\vert_{z_0} 
			\mod \glie_{\Phi(z_0)}^{0,0},
	\]
	and so the right-hand side indeed represents the derivative of the period mapping
	at the point $z_0 \in \HH$. Note that
	\[
		\theta \big\vert_{z_0} \in \End(V)_{\Phi(z_0)}^{-1,1},
		\quad \text{and} \quad
		\thetast \big\vert_{z_0} \in \End(V)_{\Phi(z_0)}^{1,-1},
	\]
	which is the horizontality condition for the period mapping.
\end{proof}

\newpar
We can now deduce quite easily that the period mapping is holomorphic.

\begin{pcor} \label{cor:Phi-holomorphic}
	The period mapping $\Phi \colon \HH \to D$ is holomorphic.
\end{pcor}

\begin{proof}
	It is enough to prove that $\Phi \colon \HH \to \Dch$ is holomorphic. Recall from
	\eqref{eq:T-Dch} that the holomorphic tangent space at $\Phi(z) \in \Dch$ is
	\[
		T_{\Phi(z)}^{1,0} \Dch \cong \End(V) / F^0 \End(V)_{\Phi(z)}.
	\]
	Using \Cref{lem:dPhi}, the differential of $\Phi$ is therefore equal to
	\[
		\theta \big\vert_z \colon T_z^{1,0} \HH \to \End(V) / F^0 \End(V)_{\Phi(z)}, 
	\]
	and since this is $\CC$-linear, $\Phi$ is indeed holomorphic.
\end{proof}

\section{\boldmath Results about $\sltwo$-Hodge structures}
\label{sec:sltwo}

\newpar
This section contains some background on $\sltwo$-Hodge structures and their
polarizations. Roughly speaking, an $\sltwo$-Hodge structure is a representation of
the Lie algebra $\sltwo(\CC)$, but in the category of Hodge structures. The
cohomology of a compact K\"ahler manifold is a typical example, but $\sltwo$-Hodge
structures also appear naturally in the study of degenerating variations of Hodge
structure.

\subsection{The Weil element and polarizations} 

\newpar
Following one of several competing conventions, we will denote the three generators
of the Lie algebra $\sltwo(\CC)$ by the letters $\Hsl, \Xsl, \Ysl$.  Concretely,
\[
	\Hsl = \begin{pmatrix} 1 & 0 \\ 0 & -1 \end{pmatrix}, \quad
	\Xsl = \begin{pmatrix} 0 & 1 \\ 0 & 0 \end{pmatrix}, \quad
	\Ysl = \begin{pmatrix} 0 & 0 \\ 1 & 0 \end{pmatrix},
\]
and the relations among the three generators are
\[
	[\Hsl,\Xsl] = 2\Xsl, \quad [\Hsl,\Ysl] = -2\Ysl, \quad [\Xsl,\Ysl] = \Hsl.
\]
In any finite-dimensional representation $V$ of $\sltwo(\CC)$, the element $\Hsl$ acts
semisimply with integer eigenvalues, and so the underlying vector space is the direct
sum of the eigenspaces $V_k = E_k(\Hsl)$, for $k \in \ZZ$.

\newpar
The following definition is of central importance in Hodge theory.

\begin{pdfn}
	An \define{$\sltwo$-Hodge structure} on a finite-dimensional complex vector space
	$V$ is a representation of $\sltwo(\CC)$ on $V$ with the following
	properties:
	\begin{aenumerate}
	\item Each weight space $V_k = E_k(\Hsl)$ has a Hodge structure of weight $n+k$;
		the integer $n$ is called the \define{weight} of the $\sltwo$-Hodge
		structure.
	\item The two operators
		\[
			\Xsl \colon V_k \to V_{k+2}(1) \quad \text{and} \quad
			\Ysl \colon V_k \to V_{k-2}(-1)
		\]
		are morphisms of Hodge structures.
	\end{aenumerate}
\end{pdfn}

\begin{pexa} Here are two useful examples.
	\begin{enumerate}
		\item The prototypical example of a $\sltwo$-Hodge structure is the total
			cohomology 
			\[
				\bigoplus_{k \in \ZZ} H^{n+k}(X, \CC)
			\]
			of an $n$-dimensional compact K\"ahler manifold $(X, \omega)$; in this case,
			$\Xsl = 2 \pi i \, L_{\omega}$ and $\Ysl = (2 \pi i)^{-1} \Lambda_{\omega}$,
			and the weight is $n = \dim X$.
		\item Any Hodge structure $V$ of weight $n$ may be viewed as an $\sltwo$-Hodge
			structure of weight $n$ by letting the Lie algebra $\sltwo(\CC)$ act in a
			trivial way. 
	\end{enumerate}
\end{pexa}

\newpar
Equivalently, an $\sltwo$-Hodge structure of weight $n$ is a bigraded vector space
\[
	V = \bigoplus_{i,j \in \ZZ} V^{i,j}
\]
that is simultaneously a representation of $\sltwo(\CC)$, in a way that is
compatible with the bigrading. This means that 
\[
	\Xsl \colon V^{i,j} \to V^{i+1,j+1} \quad \text{and} \quad
	\Ysl \colon V^{i,j} \to V^{i-1,j-1},
\]
and that $\Hsl$ acts on the subspace $V^{i,j}$ as multiplication by the integer $(i+j)-n$.
This makes each of the weight spaces
\[
	V_k = \bigoplus_{i+j=n+k} V^{i,j}
\]
into a Hodge structure of weight $n+k$. In what follows, we are always going to use
the notation $V_k^{i,j}$ for the individual subspaces, to avoid confusion with the
Hodge decomposition in a polarized Hodge structure on $V$.

\newpar
Let $V$ be an $\sltwo$-Hodge structure. From general principles, it follows that
\[
	\Xsl^k \colon V_{-k} \to V_k(k) \quad \text{and} \quad
	\Ysl^k \colon V_k \to V_{-k}(-k)
\]
are isomorphisms of Hodge structures. Another way to understand this symmetry is
through the action of the Lie group $\SL_2(\CC)$. Every finite-dimensional
representation of the Lie algebra $\sltwo(\CC)$ lifts to a representation of the Lie
group $\SL_2(\CC)$, in a way that is compatible with the exponential map:
for a matrix $M \in \sltwo(\CC)$, the element
\[
	\exp(M) = \sum_{n=0}^{\infty} \frac{M^n}{n!} \in \SL_2(\CC),
\]
acts on the representation as the exponential series $\id + M + \half M^2 + \dotsb$.
Now consider the \define{Weil element}
\[
	\wsl = \begin{pmatrix} 0 & 1 \\ -1 & 0 \end{pmatrix} \in \SL_2(\CC),
\]
which plays a similar role as the Hodge $\ast$-operator in classical Hodge theory. A brief
computation shows that $\wsl = \exp(\Xsl) \exp(-\Ysl) \exp(\Xsl)$, and so one can use the
exponential series to see how $\wsl$ acts on representations of $\sltwo(\CC)$.

\newpar
The point of introducing $\wsl$ is that it explains the symmetry between the weight
spaces $V_k$ and $V_{-k}$. Namely, the Lie group $\SL_2(\CC)$ acts on its Lie algebra
$\sltwo(\CC)$ by conjugation, and under this action, one has
\[
	\wsl \Xsl \wsl^{-1} = -\Ysl, \quad
	\wsl \Ysl \wsl^{-1} = -\Xsl, \quad
	\wsl \Hsl \wsl^{-1} = -\Hsl.
\]
This can be checked by direct computation: for example,
\[
	\begin{pmatrix} 0 & 1 \\ -1 & 0 \end{pmatrix}
	\begin{pmatrix} 0 & 1 \\ 0 & 0 \end{pmatrix}
	\begin{pmatrix} 0 & -1 \\ 1 & 0 \end{pmatrix}
	= \begin{pmatrix} 0 & 0 \\ -1 & 0 \end{pmatrix}.
\]
The identity $\wsl \Hsl \wsl^{-1} = -\Hsl$ means that the action of $\wsl$
interchanges the two weight spaces $V_k$ and $V_{-k}$, which must therefore be of the
same dimension. 

\begin{note}
	The square of the Weil element
	\[
		\wsl^2 = \begin{pmatrix} -1 & 0 \\ 0 & -1 \end{pmatrix} \in \SL_2(\CC)
	\]
	acts on the weight space $V_k$ as $(-1)^k$, and \emph{not} as multiplication by
	$-1$. This is easiest to remember with the help of the symbolic identity $\wsl^2 =
	(-1)^{\Hsl}$, whose right-hand side is a convenient abbreviation for the power series
	$e^{i \pi \Hsl}$. 
\end{note}

\newpar
We know that $\wsl \colon V_k \to V_{-k}$ is an isomorphism for every $k \in \ZZ$.
But since $\Xsl$ and $\Ysl$ are only morphisms up to a Tate twist, each term in the series
expansion of $\wsl = e^{\Xsl} e^{-\Ysl} e^{\Xsl}$ needs a different Tate twist, and
so it is not immediately clear that $\wsl$ is a morphism of Hodge structures. 

\begin{plem} \label{lem:w}
If $V$ is an $\sltwo$-Hodge structure, then $\wsl \colon V_k \to V_{-k}(-k)$ is an
isomorphism of Hodge structures (of weight $n+k$).
\end{plem}

\begin{proof}
We first prove an auxiliary formula. Suppose that $b \in V_{-\ell}$ is primitive, in
the sense that $\Ysl b = 0$ (and therefore $\ell \geq 0$). From $\wsl e^{-\Xsl} =
e^{\Xsl} e^{-\Ysl}$, we get $\wsl e^{-\Xsl} b = e^{\Xsl} b$, and after expanding and
comparing terms in degree $\ell-2j$, also
\begin{equation} \label{eq:primitive}
	\wsl \frac{\Xsl^j}{j!} b = (-1)^j \frac{\Xsl^{\ell-j}}{(\ell-j)!} b.
\end{equation}
Now any $a \in V_k$ has a unique Lefschetz decomposition
\[
	a = \sum_{j \geq \max(k,0)} \frac{\Xsl^j}{j!} a_j,
\]
where $a_j \in V_{k-2j}$ satisfies $Ya_j = 0$. Here we only need to consider $j \geq
k$ in the sum because $\Xsl^{2j-k+1} a_j = 0$, which implies that $\Xsl^j a_j = 0$ for $j <
k$. Suppose further that $a \in V^{p,q}$, where $p+q = n+k$. Then $\Xsl^i a_j \in
V^{p+i,q+i}$, and by descending induction on $j \geq \max(k,0)$, we deduce that 
\[
	a_j \in V^{p-j,q-j}. 
\]
In other words, the Lefschetz decomposition holds in the category of Hodge
structures. We can now check what happens when we apply $w$. Using
\eqref{eq:primitive}, 
\[
	\wsl a = \sum_{j \geq \max(k,0)} w \frac{\Xsl^j}{j!} a_j
		= \sum_{j \geq \max(k,0)} (-1)^j \frac{\Xsl^{j-k}}{(j-k)!} a_j
		\in V^{p-k,q-k},
\]
and so $\wsl$ is a morphism of Hodge structures. Since $\wsl$ is bijective, it must be an
isomorphism of Hodge structures, as claimed.
\end{proof}

\newpar
Using the Weil element, one can give a very concise definition of polarizations for
$\sltwo$-Hodge structures. This idea is due to Deligne \cite{Deligne-signs}.

\begin{pdfn} \label{def:sltwo-polarization}
A \define{polarization} of an $\sltwo$-Hodge structure $V$ is a hermitian form
\[
	Q \colon V \tensor_{\CC} \Vb \to \CC
\]
that satisfies the four identities $H^{\dagger} = -H$, $X^{\dagger} = X$,
$Y^{\dagger} = Y$, and $w^{\dagger} = w$, such that $Q \circ (\id \tensor \wsl)$
polarizes the Hodge structure of weight $n+k$ on each $V_k$. 
\end{pdfn}

Here the dagger again means the adjoint of an operator with respect to the
nondegenerate hermitian pairing $Q$.

\newpar
The relation $Q \circ (\Hsl \tensor \id) = -Q \circ (\id \tensor \Hsl)$ implies that 
\[
	Q \restr{V_k \tensor_{\CC} \wbar{V_{\ell}}} = 
	\begin{cases} 
		Q_k & \text{if $\ell = -k$,} \\
		0 & \text{otherwise}.
	\end{cases}
\]
and so $Q$ is really a family of sesquilinear pairings $Q_k \colon V_k
\tensor_{\CC} \wbar{V_{-k}} \to \CC$. The second condition in
\Cref{def:sltwo-polarization} is then saying that the hermitian pairing
\[
	Q_k \circ (\id \tensor \wsl) \colon V_k \tensor_{\CC} \wbar{V_k} \to \CC
\]
polarizes the Hodge structure of weight $n+k$ on $V_k$. In particular, this means
that the Hodge structure on the \define{primitive} subspace $V_{-k} \cap \ker \Ysl$
is polarized by the hermitian pairing $Q_k \circ (\id \tensor \Xsl^k)$; this is the
usual way to describe the polarization on the cohomology of a compact K\"ahler
manifold. Deligne's definition has the advantage of giving a concise formula for the
polarization on all of $V$, not just on the primitive subspaces.

\newpar
	With the exception of positivity, all the conditions in the definition have
	a nice functorial interpretation. The conjugate complex vector space $\Vb$
	is again an $\sltwo$-Hodge structure of weight $n$: the action of $\Hsl$
	is unchanged, but $\Xsl$ and $\Ysl$ act with an extra minus sign. This sign change is
	dictated by the geometric case, where $\Xsl = 2 \pi i \, L_{\omega}$ and $\Ysl = (2 \pi
	i)^{-1} \Lambda_{\omega}$. Likewise, if $V'$ and $V''$ are $\sltwo$-Hodge
	structures of weights $n'$ and $n''$, then the tensor product $V'
	\tensor_{\CC} V''$ is naturally an $\sltwo$-Hodge structure of weight $n' + n''$:
	to be precise, 
	\[
		\bigl( V' \tensor_{\CC} V'' \bigr)_k = \bigoplus_{i+j=k} V'_i \tensor_{\CC} V''_j,
	\]
	and the $\sltwo(\CC)$-action is given by the usual formulas
	\begin{align*}
		\Xsl(v' \tensor v'') &= \Xsl v' \tensor v'' + v' \tensor \Xsl v'', \\
		\Ysl(v' \tensor v'') &= \Ysl v' \tensor v'' + v' \tensor \Ysl v'', \\
		\Hsl(v' \tensor v'') &= \Hsl v' \tensor v'' + v' \tensor \Hsl v''.
	\end{align*}
	Lastly, we can turn $\CC(-n)$ into an $\sltwo$-Hodge structure of weight $2n$ by
	letting $\sltwo(\CC)$ act trivially. Then all the identities in
	\Cref{def:sltwo-polarization} can be summarized in one line by saying
	that the hermitian form
	\[
		Q \colon V \tensor_{\CC} \Vb \to \CC(-n)
	\]
	is a morphism of $\sltwo$-Hodge structures of weight $2n$. 

\newpar
From a polarization $Q$, we obtain a hermitian inner product on the vector
space $V$ as follows. The individual Hodge structures on the weight spaces $V_k$ give
us a direct sum decomposition
\[
	V = \bigoplus_{p,q \in \ZZ} V_{p+q-n}^{p,q},
\]
and we denote by $v = \sum_{p,q} v^{p,q}$ the components of a given vector. The
formula
\begin{equation} \label{eq:inner-product-sltwo}
	\inner{v}{w} = \sum_{p,q \in \ZZ} (-1)^q Q \bigl( v^{p,q}, \wsl(w^{p,q}) \bigr)
\end{equation}
then defines a positive definite hermitian inner product on $V$. By construction, the
above decomposition is orthogonal with respect to this inner product. We will
see in \Cref{thm:sltwo-orbit} how this inner product relates to polarized Hodge structures.

\subsection{The irreducible representations}

\newpar
Recall that $\sltwo(\CC)$ has, up to isomorphism, a unique irreducible representation
of dimension $m+1$. We denote this representation by $S_m$. Clearly, $S_0$
is the trivial representation, and $S_1 = \CC^2$ is the standard representation; for
$m \geq 2$, one has $S_m = \Sym^m(S_1)$. Since finite-dimensional
representations of $\sltwo(\CC)$ are completely reducible, one has by Schur's lemma a
canonical decomposition
\begin{equation} \label{eq:decomposition}
	V \cong \bigoplus_{m \in \NN} 
	S_m \tensor_{\CC} \Hom_{\CC}(S_m, V)^{\sltwo(\CC)}.
\end{equation}
We can turn this into a decomposition in the category of $\sltwo$-Hodge structures,
by the following construction.

\newpar
First, we argue that each irreducible representation $S_m$ comes with a canonical
$\sltwo$-Hodge structure of weight $m$. For the trivial representation $S_0 =
\CC$, this is the trivial Hodge structure of type $(0,0)$. The \define{standard
representation} $S_1$ also has a canonical $\sltwo$-Hodge structure. Denote by
$a = (1,0)$ and $b = (0,1)$ the two standard basis vectors, so that $\Ysl a = b$ and $\Xsl b =
a$. We get an $\sltwo$-Hodge structure of weight $1$ by considering
\[
	E_1(\Hsl) = \CC a \quad \text{and} \quad E_{-1}(\Hsl) = \CC b
\]
as Hodge structures of type $(1,1)$ and $(0,0)$, respectively. It is polarized by the
unique hermitian pairing with $Q(a,b) = 1$ and $Q(a,a) = Q(b,b) = 0$; indeed, 
\begin{align*}
	(-1)^1 Q(a,\wsl a) &= -Q(a,-b) = Q(a,b) = 1, \\
	(-1)^0 Q(b, \wsl b) &= Q(b,a) = 1,
\end{align*}
since $\wsl a = -b$ and $\wsl b = a$. In particular, we have
\[
	\inner{a}{a} = \inner{b}{b} = 1 \quad \text{and} \quad
	\inner{a}{b} = 0,
\]
which means that $a,b \in S_1$ form an orthonormal basis with respect to
\eqref{eq:inner-product-sltwo}.

\newpar \label{par:Sm}
Since everything is compatible with taking symmetric powers, the irreducible
representation $S_m = \Sym^m(S_1)$ inherits an $\sltwo$-Hodge structure of 
weight $m$. We can describe this concretely as follows. The $m+1$ elements 
\[
	v_0 = a^m, v_1 = a^{m-1} b, \dotsc, v_k = a^{m-k} b^k, \dotsc, v_m = b^m 
\]
form a basis for the vector space $S_m$. It is easy to see that
\[
	\Hsl v_k = (m-2k) v_k, \quad \Xsl v_k = k v_{k-1}, \quad
	\Ysl v_k = (m-k) v_{k+1},
\]
if we declare that $v_{-1} = v_{m+1} = 0$. They are morphisms up to the required Tate
twists if we consider each weight space
\[
	E_{m-2k}(\Hsl) = \CC v_k
\]
as a Hodge structure of type $(m-k,m-k)$. The induced hermitian pairing $Q$ on $S_m$
has the property that $Q(v_j, v_k) = 0$ if $k+j \neq m$, while
\[
	Q(v_k, v_{m-k}) = \frac{k!(m-k)!}{m!}.
\]
Since $\wsl v_k = (\wsl a)^{m-k} (\wsl b)^k = (-1)^{m-k} v_{m-k}$, we get
\[
	(-1)^{m-k} Q(v_k, \wsl v_k) = Q(v_k, v_{m-k}) = \frac{k!(m-k)!}{m!} > 0,
\]
as required by \Cref{def:sltwo-polarization}. With respect to the hermitian
inner product in \eqref{eq:inner-product-sltwo}, the vectors $v_0, \dotsc, v_m \in
S_m$ are again orthogonal, and
\[
	\inner{v_k}{v_k} = \frac{k!(m-k)!}{m!}.
\]

\begin{note}
	Recall that a hermitian pairing $Q \colon V \tensor_{\CC} \Vb \to \CC$ induces a 
	hermitian pairing on the symmetric product $\Sym^m V$ by the formula
	\[
		Q \bigl( v_1 \dotsm v_m, w_1 \dotsm w_m \bigr)
		= \frac{1}{m!} \sum_{\sigma \in \mathfrak{S}_m}
			\prod_{j=1}^m Q \bigl( v_j, w_{\sigma(j)} \bigr),
	\]
	where $\mathfrak{S}_m$ is the group of permutations of the set $\{1, \dotsc,
	m\}$. This comes about by viewing the symmetric product as a subspace of
	$V^{\tensor m}$, via the embedding
	\[
		\Sym^m V \into V^{\tensor m}, \quad 
		v_1 \dotsm v_m \mapsto \frac{1}{m!} \sum_{\sigma \in \mathfrak{S}_m}
		v_{\sigma(1)} \tensor \dotsb \tensor v_{\sigma(m)}.
	\]
	This explains the term $m!$ in the formulas above.
\end{note}

\newpar
We continue working on the decomposition in \eqref{eq:decomposition}. Next, we note
that in any $\sltwo$-Hodge structure $V$ of weight $n$, the subspace of
$\sltwo(\CC)$-invariants
\[
	V^{\sltwo(\CC)} = \menge{v \in V_0}{\Xsl v = \Ysl v = 0} = 
		\ker \bigl( \Ysl \colon V_0 \to V_{-2}(-1) \bigr)
\]
is a sub-Hodge structure of $V_0$, and hence a Hodge structure of weight $n$. If $V$
is polarized by $Q$, then $V^{\sltwo(\CC)}$ is polarized by the restriction of $Q$.

\newpar
Now let $V$ be an $\sltwo$-Hodge structure of weight $n$. For each $m \in
\NN$, the vector space $\Hom_{\CC}(S_m, V)$ inherits an $\sltwo$-Hodge structure of
weight $n-m$. According to the preceding paragraph, the subspace of
$\sltwo(\CC)$-invariants
\[
	\Hom_{\CC}(S_m, V)^{\sltwo(\CC)}
\]
is therefore a Hodge structure of weight $n-m$. If we view it as an $\sltwo$-Hodge
structure by letting $\sltwo(\CC)$ act trivially, the evaluation mapping
\[
	S_m \tensor_{\CC} \Hom_{\CC}(S_m, V)^{\sltwo(\CC)} \to V
\]
becomes a morphism of $\sltwo$-Hodge structures of weight $n$, and the
isomorphism in \eqref{eq:decomposition} becomes an isomorphism of $\sltwo$-Hodge
structures.

\newpar
The following lemma shows that the decomposition in \eqref{eq:decomposition} is also
compatible with polarizations.

\begin{plem} \label{lem:sltwo-polarizations}
	Let $V$ be an $\sltwo$-Hodge structure of weight $n$. Then $\Hom_{\CC}(S_m, V)$
	in\-herits an $\sltwo$-Hodge structure of weight $n-m$,
	such that the evaluation morphism
	\[
		S_m \tensor_{\CC} \Hom_{\CC}(S_m, V)^{\sltwo(\CC)} \to V, \quad
		s \tensor f \mapsto f(s),
	\]
	is a morphism of $\sltwo$-Hodge structures. If $V$ is polarized, then
	$\Hom_{\CC}(S_m, V)$ is also polarized, and the evaluation morphism respects the
	polarizations.
\end{plem}

\begin{proof}
	The first assertion is obvious, and so we focus on the second one. For clarity,
	let us write $Q_V$ for the polarization on $V$, and $Q_{S_m}$ for the polarization
	on $S_m$. Given $f \colon S_m \to V$, let $f^{\dagger} \colon V \to S_m$ be
	the adjoint, defined by the rule
	\[
		Q_V \bigl( v, f(s) \bigr) = Q_{S_m} \bigl( f^{\dagger}(v), s \bigr).
	\]
	A short calculation, similar to \Cref{lem:tr-polarization}, shows that the trace
	pairing
	\[
		(f, g) \mapsto \frac{\tr(g^{\dagger} \circ f)}{\dim S_m}
	\]
	polarizes the $\sltwo$-Hodge structure on $\Hom_{\CC}(S_m, V)$. The induced
	polarization on the tensor product is therefore
	\[
		(f \tensor s, g \tensor t) \mapsto 
		Q_{S_m}(s,t) \cdot \frac{\tr(g^{\dagger} \circ f)}{\dim S_m}.
	\]
	Now suppose that $f,g \in \Hom_{\CC}(S_m, V)^{\sltwo(\CC)}$. By Schur's lemma, we have
	\[
		g^{\dagger} \circ f = c(f,g) \id_{S_m}
	\]
	for a constant $c(f,g) \in \CC$. Consequently,
	\[
		Q_{S_m}(s,t) \cdot \frac{\tr(g^{\dagger} \circ f)}{\dim S_m}
		= c(f,g) Q_{S_m}(s, t)
		= Q_{S_m} \bigl( g^{\dagger}(f(s)), t \bigr)
		= Q_V \bigl( f(s), g(t) \bigr),
	\]
	and so the induced polarization on the tensor product is indeed compatible with
	the given polarization on $V$.
\end{proof}

\subsection{The associated Hodge structure}

\newpar
One justification for the name ``$\sltwo$-Hodge structure'' is that 
every (polarized) $\sltwo$-Hodge structure has canonically associated to it a
(polarized) Hodge structure of the same weight. We now describe this construction.
Suppose $V$ is an $\sltwo$-Hodge structure of weight $n$. We denote by
\[
	V = \bigoplus_{i,j \in \ZZ} V_{i+j-n}^{i,j}
\]
the ``total'' Hodge decomposition, and by 
\[
	F^p = \bigoplus_{i \geq p,j} V_{i+j-n}^{i,j} \quad \text{and} \quad
	\Fb^q = \bigoplus_{j \geq q, i} V_{i+j-n}^{i,j}
\]
the total Hodge filtration and its conjugate. Clearly, $\Xsl(F^p) \subseteq F^{p+1}$
and $\Ysl(F^p) \subseteq F^{p-1}$, and likewise for the other filtration
$\Fb$. If $Q$ is a polarization of the $\sltwo$-Hodge structure, one can recover
the conjugate Hodge filtration $\Fb$ from the Hodge filtration $F$ because
\begin{equation} \label{eq:Fb-sltwo}
	\Fb^q = \menge{v \in V}{\text{$Q(v,x) = 0$ for all $x \in F^{n-q+1}$}}.
\end{equation}

\begin{note}
In fact, a polarized $\sltwo$-Hodge structure can be described completely by
specifying the representation $V$, the weight $n$, the Hodge filtration $F$,
and the hermitian pairing $Q$. This data determines the conjugate Hodge filtration
$\Fb$ according to \eqref{eq:Fb-sltwo}, and therefore the subspaces
\[
	V_k^{i,j} = E_k(\Hsl) \cap F^i \cap \Fb^j
\]
in the Hodge decomposition of each weight space. Note that the weight is \emph{not}
determined by $F$ and $Q$ alone.
\end{note}

\newpar
The following theorem explains how to pass from a (polarized) $\sltwo$-Hodge
structure to a (polarized) Hodge structure of the same weight. Later in this chapter,
we are going to prove a sort of converse to this result, which is the key to the
existence of limiting mixed Hodge structures.

\begin{pthm} \label{thm:sltwo-orbit}
	Let $V$ be an $\sltwo$-Hodge structure of weight $n$, and let $Q$
	be a polarization. Then the following is true:
	\begin{aenumerate}
	\item $e^{\Ysl} F$ and $e^{-\Ysl} \Fb$ are the two Hodge filtrations of a polarized Hodge
		structure of weight $n$, with polarization $Q$.  
	\item The inner product determined by this polarized Hodge structure is equal to
		the one determined by the $\sltwo$-Hodge structure in \eqref{eq:inner-product-sltwo}.
	\item We have $\Hsl^{\ast} = \Hsl$ and $\Ysl^{\ast} = \Xsl$, where $\ast$ denotes
		the adjoint with respect to the inner product determined by the polarized Hodge
		structure.
	\end{aenumerate}
\end{pthm}

\begin{proof}
	The decomposition in \eqref{eq:decomposition} implies that it is enough to prove
	this for the irreducible representations $S_m$. Furthermore, everything is
	compatible with taking symmetric powers, and so we only need to consider the
	standard representation $S_1$. Here $F^1 = \Fb^1 = \CC a$, and consequently 
	\[
		e^{\Ysl} F^1 = \CC(a+b) \quad \text{and} \quad e^{-\Ysl} \Fb^1 = \CC(a-b).
	\]
	Since $Q(a+b,a-b) = 0$, we obtain a Hodge structure $V = V^{1,0} \oplus V^{0,1}$
	with
	\[
		V^{1,0} = \CC(a+b) \quad \text{and} \quad V^{0,1} = \CC(a-b).
	\]
	Also, $Q$ is a polarization because $Q(a+b,a+b) = 2 = -Q(a-b,a-b)$. The Hodge
	decompositions of the two standard basis vectors are
	\[
		a = \half(a+b) + \half(a-b) \quad \text{and} \quad
		b = \half(a+b) - \half(a-b),
	\]
	from which it follows that $a$ and $b$ form an orthonormal basis with respect to the
	inner product coming from the Hodge structure. Since $\Hsl a = a$ and $\Hsl b = -b$, we
	see that $\Hsl$ is self-adjoint with respect to the inner product; similarly,
	$\Ysl a = b$ and $\Xsl b = a$ imply that $\Ysl^{\ast} = \Xsl$. Recalling the
	definition of the inner product in \eqref{eq:inner-product-sltwo}, we obviously have
	\[
		\inner{a}{a} = \inner{b}{b} = 1 \quad \text{and} \quad
		\inner{a}{b} = 0,
	\]
	and so the two inner products agree.
\end{proof}

\newpar
The analogy with the cohomology of a compact K\"ahler manifold is striking:
$V$ has a hermitian inner product; $\Ysl = \Xsl^{\ast}$ is the adjoint of the operator
$\Xsl$ with respect to this inner product; $\Hsl = [\Xsl, \Xsl^{\ast}]$ is the
commutator of $\Xsl$ and its adjoint.

\newpar
The next result relates the Weil operator in the Hodge structure $e^{\Ysl} F$
to representation theory. Using the total Hodge decomposition
\[
	V = \bigoplus_{i,j} V_{i+j-n}^{i,j},
\]
we define an involution $C \in \GL(V)$ by the rule
\[
	Cv = (-1)^j v \quad \text{for $v \in V_{i+j-n}^{i,j}$.}
\]
This is the analogue of the Weil operator for an $\sltwo$-Hodge structure.

\begin{pprop}
	Let $\Csh \in \End(V)$ be the Weil operator of the Hodge structure whose Hodge
	filtration is $\Fsh = e^{\Ysl} F$. Then $\Csh = \wsl \circ C$.
\end{pprop}

\begin{proof}
	As before, it suffices to prove this in the case of the standard $\sltwo$-Hodge
	structure on $S_1$. If we again denote the two basis vectors by $a$ and $b$,
	then
	\begin{align*}
		C a = -a, \quad C b = b, \quad \wsl a = -b, \quad \wsl b = a, \quad
		\Csh a = b, \quad \Csh b = a,
	\end{align*}
	from which it follows that $\Csh = \wsl \circ C$.
\end{proof}

\newpar
For later use, let us also record the Hodge decompositions of the three operators
$\Hsl, \Xsl, \Ysl \in \End(V)$ that constitute the $\sltwo(\CC)$-representation. Let
\[
	\End(V) = \bigoplus_{j \in \ZZ} \End(V)^{j,-j}
\]
denote the Hodge structure of weight $0$ on $\End(V)$, induced by the Hodge structure
$e^{\Ysl} F$ on $V$. For any $A \in \End(V)$, we write the Hodge decomposition in the
form
\[
	A = \sum_{j \in \ZZ} A_j,
\]
with $A_j \in \End(V)^{j,-j}$.

\begin{plem} \label{lem:Hodge-HXY}
	The Hodge decompositions of $\Ysl, \Xsl, \Hsl \in \End(V)$ look like
	\[
		\Ysl = \Ysl_{-1} + \Ysl_0 + \Ysl_1, \quad
		\Xsl = -\Ysl_{-1} + \Ysl_0 - \Ysl_1, \quad
		\Hsl = -2 \Ysl_{-1} + 2 \Ysl_1,
	\]
	and conversely, we have
	\[
		4 \Ysl_{-1} = \Ysl - \Hsl - \Xsl, \quad
		2\Ysl_0 = \Ysl + \Xsl, \quad
		4 \Ysl_1 = \Ysl + \Hsl - \Xsl
	\]
\end{plem}

\begin{proof}
	It is again enough to prove this for the standard $\sltwo$-Hodge structure on
	$S_1$. Denote the two basis vectors by $a$ and $b$; then the Hodge decomposition
	is
	\[
		S_1 = \CC(a+b) \oplus \CC(a-b).
	\]
	With respect to the basis $a,b$, we have
	\[
		\Ysl - \Hsl - \Xsl = \begin{pmatrix} -1 & -1 \\ 1 & 1 \end{pmatrix}, \quad
		\Ysl + \Xsl = \begin{pmatrix} 0 & 1 \\ 1 & 0 \end{pmatrix}, \quad
		\Ysl + \Hsl - \Xsl = \begin{pmatrix} 1 & -1 \\ 1 & -1 \end{pmatrix},
	\]
	and it is easy to see that these three matrices have Hodge type $(-1,1)$, $(0,0)$,
	and $(1,-1)$, respectively. This suffices to conclude the proof.
\end{proof}

\newpar \label{par:Casimir}
One nice consequence of these identities is that the isotypical components of the 
$\sltwo(\CC)$-representation are actually sub-Hodge structures. This is most easily
proved with the help of the \define{Casimir operator}
\[
	\Omega = 2(\Xsl \Ysl + \Ysl \Xsl) + \Hsl^2 + \id 
	= 4 \Xsl \Ysl + (\Hsl + \id)^2 \in \End(V).
\]
By general theory, $\Omega$ commutes with $\Xsl,\Hsl, \Ysl$, and acts on the $m$-th summand
in the decomposition \eqref{eq:decomposition} as multiplication by
$(m+1)^2$. A brief computation using the identities in \Cref{lem:Hodge-HXY} gives
\begin{align*}
	\Omega &= 2(-\Ysl_{-1} + \Ysl_0 - \Ysl_1)(\Ysl_{-1} + \Ysl_0 + \Ysl_1) \\
			 &\hphantom{=2(-\Ysl_{-1}\,\,} + 2(\Ysl_{-1} + \Ysl_0 
	+ \Ysl_1)(-\Ysl_{-1} + \Ysl_0 - \Ysl_1) 
	+ 4(\Ysl_1 - \Ysl_{-1})^2 + \id \\
			 &= (2 \Ysl_0 + \id)^2 - 16 \Ysl_1 \Ysl_{-1}.
\end{align*}
This means that $\Omega$ is an endomorphism of the Hodge structure on $V$, and
therefore that each isotypical component $S_m \tensor_{\CC} \Hom_{\CC}(S_m,
V)^{\sltwo(\CC)}$ in the decomposition \eqref{eq:decomposition} is automatically a
sub-Hodge structure.

\subsection{The associated variation of Hodge structure}
\label{subsec:sltwo-VHS}

\newpar
In fact, a polarized $\sltwo$-Hodge structure of weight $n$ determines a polarized
variation of Hodge structure of weight $n$ on the punctured disk $\dst$. These
``model variations'' are important, both as examples of the general theory, and
because they are used to prove the Hodge norm estimates.

\newpar
Let $V$ be an $\sltwo$-Hodge structure of weight $n$, and let $Q$ be a polarization.
From the previous section, we keep the notation
\[
	F^p = \bigoplus_{i \geq p, j} V_{i+j-n}^{i,j}
\]
for the total Hodge filtration. By construction, $\Hsl(F^p) \subseteq F^p$, 
$\Ysl(F^p) \subseteq F^{p-1}$, and $\Xsl(F^p) \subseteq F^{p+1}$. We now define a
holomorphic mapping
\[
	\Phi \colon \HH \to \Dch, \quad \Phi(z) = e^{-z \Ysl} F.
\]
At the point $z = -1$, the filtration $e^{\Ysl} F$ is the Hodge filtration of a
polarized Hodge structure of weight $n$ (by \Cref{thm:sltwo-orbit}), and 
therefore $\Phi(-1) \in D$. Now write $z = x + iy$, with $x < 0$. From the relation
$[\Hsl, \Ysl] = -2\Ysl$, we get
\[
	e^{-z \Ysl} = e^{-iy \Ysl} e^{\abs{x} \Ysl} 
	= e^{-iy \Ysl} e^{-\half \log \abs{x} \, \Hsl} e^{\Ysl} e^{\half \log \abs{x} \, \Hsl},
\]
and therefore
\begin{equation} \label{eq:Phi-sltwo}
	\Phi(z) = e^{-iy \Ysl} e^{-\half \log \abs{x} \, \Hsl} \cdot e^{\Ysl} F
	= e^{-iy \Ysl} e^{-\half \log \abs{x} \, \Hsl} \cdot \Phi(-1).
\end{equation}
Both exponential factors are elements of the real group $G = \Aut(V,Q)$, because
$\Hsl^{\dagger} = -\Hsl$ and $\Ysl^{\dagger} = \Ysl$, and so we conclude that
$\Phi(z) \in D$ for every $z \in \HH$.

\newpar
We are going to show that $\Phi$ is the period mapping of a polarized variation of Hodge
structure. Let $E = \HH \times V$ be the trivial smooth vector bundle with fiber $V$,
and let $d \colon A^0(\HH, E) \to A^1(\HH, E)$ be the flat connection induced by
differentiation. The Hodge structures $\Phi(z)$ give us a decomposition 
\[
	E = \bigoplus_{p+q=n} E^{p,q}
\]
into smooth subbundles. This can be described concretely as follows. Let
\[
	V = \bigoplus_{p+q=n} V^{p,q}
\]
denote the Hodge decomposition in the Hodge structure $\Phi(-1)$. Due to
\eqref{eq:Phi-sltwo}, the Hodge decomposition in the Hodge structure $\Phi(z)$ is
then equal to
\[
	V = \bigoplus_{p+q=n} V_{\Phi(z)}^{p,q} 
	= \bigoplus_{p+q=n} e^{-iy \Ysl} e^{-\half \log\abs{x} \, \Hsl} V^{p,q},
\]
and so the subbundle $E^{p,q} \subseteq E$ is the image of the smooth mapping
\[
	\HH \times V^{p,q} \to \HH \times V, \quad
	(z, v) \mapsto e^{-iy \Ysl} e^{-\half \log\abs{x} \, \Hsl} v.
\]

\newpar
It is now a simple matter to verify the conditions in \Cref{def:VHS}. Every
smooth section in $A^0(\HH, E^{p,q})$ can be written in the form
\[
	e^{-iy \Ysl} e^{-\half \log\abs{x} \, \Hsl} f,
\]
where $f \colon \HH \to V^{p,q}$ is a smooth function. We have
\begin{align*}
	d \Bigl( e^{-iy \Ysl} &e^{-\half \log\abs{x} \, \Hsl} f \Bigr) \\
	&= e^{-iy \Ysl} e^{-\half \log\abs{x} \, \Hsl} \biggl( 
		-i e^{\half \log\abs{x} \, \Hsl} \Ysl e^{-\half \log\abs{x} \, \Hsl} f \tensor dy
		+ \frac{1}{2 \abs{x}} \Hsl f \tensor dx + df \biggr) \\
	&= e^{-i y \Ysl} e^{-\half \log\abs{x} \, \Hsl} \biggl( 
		- \frac{i}{\abs{x}} \Ysl f \tensor dy
		+ \frac{1}{2 \abs{x}} \Hsl f \tensor dx + df \biggr)
\end{align*}
Remembering the identities (from \Cref{lem:Hodge-HXY})
\[
	\Ysl = \Ysl_{-1} + \Ysl_0 + \Ysl_1 \quad \text{and} \quad
	\Hsl = - 2 \Ysl_{-1} + 2 \Ysl_1, 
\]
we can rewrite the factor in parentheses as
\[
	- \frac{i}{\abs{x}} \Ysl f \tensor dy
	+ \frac{1}{2 \abs{x}} \Hsl f \tensor dx + df
	= df - \frac{1}{\abs{x}} \Ysl_{-1} f \tensor \dz + \frac{1}{\abs{x}} \Ysl_1 f
	\tensor \dzb - \frac{1}{2 \abs{x}} \Ysl_0 f \tensor (\dz - \dzb)
\]
It follows that $d = \partial + \theta + \delb + \thetast$, just as in
\Cref{def:VHS}, with
\begin{align*}
	g^{-1} \partial g &= \frac{\partial}{\partial z} - \frac{1}{2 \abs{x}} \Ysl_0
	\tensor \dz, &
	g^{-1} \theta g &= - \frac{1}{\abs{x}} \Ysl_{-1} \tensor \dz, \\
	g^{-1} \delb g &= \frac{\partial}{\partial \zb} + \frac{1}{2 \abs{x}} \Ysl_0
	\tensor \dzb,  &
	g^{-1} \thetast g &= \frac{1}{\abs{x}} \Ysl_1 \tensor \dz,
\end{align*}
where $g = e^{-iy \Ysl} e^{-\half \log\abs{x} \, \Hsl} \in G$. Evidently, these
operators have the required bidegrees $(1,0) \tensor (0,0)$, $(0,1) \tensor (0,0)$, $(1,0)
\tensor (-1,1)$, and $(0,1) \tensor (1,-1)$, respectively. Therefore $\Phi$ is indeed
the period mapping of a polarized variation of Hodge structure on $\HH$. From the
identity
\[
	\Phi(z + 2 \pi i) = e^{-2 \pi i \Ysl} \Phi(z), 
\]
we see that the variation of Hodge structure is pulled back from the punctured disk
$\dst$; the monodromy operator is the unipotent operator
\[
	T = e^{-2 \pi i \Ysl} \in G,
\]
and in our usual notation, we therefore have $N = -\Ysl$.

\newpar \label{par:model-monodromy}
In fact, we can easily turn this into an example with non-unipotent monodromy as
well. Let $S \in \End(V)$ be an operator with the property that $S^{\dagger} = S$,
and assume that $S$ is an endomorphism of the $\sltwo$-Hodge
structure on $V$. Concretely, this means that $[\Hsl, S] = [\Ysl, S] = 0$ and
that $S(F^p) \subseteq F^p$ for every $p \in \ZZ$. It follows that $S$ is
self-adjoint with respect to the inner product in \eqref{eq:inner-product-sltwo}, and
so $S$ is automatically semisimple with real eigenvalues. We still have
\[
	\Phi(z + 2 \pi i) = e^{2 \pi i (S - \Ysl)} \Phi(z),
\]
and so we can also consider $\Phi$ as the period mapping of a polarized variation of
Hodge structure on $\dst$ with monodromy operator
\[
	T = T_s T_u = e^{2 \pi i S} e^{-2 \pi i \Ysl} \in G.
\]
This is the Jordan decomposition of $T$.

\newpar
As an aside, let us verify the result in \Cref{cor:Higgs-bound} about the
pointwise Hodge norm of the Higgs field in this example. Since $\Phi(z) = g \cdot
\Phi(-1)$, we have
\[
	\bignorm{\theta_{\partial/\partial z}}_{\Phi(z)}^2
	= \bignorm{g^{-1} \theta_{\partial/\partial z} g}_{\Phi(-1)}^2
	= \frac{1}{\abs{x}^2} \bignorm{Y_{-1}}_{\Phi(-1)}^2,
\]
and so the pointwise Hodge norm of the Higgs field is exactly a constant multiple of
$\abs{x}^{-1}$ in this case. (It is a pleasant exercise to check that when $V =
S_m$ is the irreducible representation of dimension $m+1$, and hence $\rk E = m+1$,
one has $\bignorm{Y_{-1}}_{\Phi(-1)}^2 = \frac{1}{4} \binom{m+2}{3}$, and so equality
is achieved in \Cref{cor:Higgs-bound}; this explains what we found during
the proof of \Cref{lem:nilpotent}.)

\newpar
We are now going to derive a simple formula for the Hodge metric in our ``model''
variation of Hodge structure. To simplify the notation, let us write $\norm{v}$ and
$\inner{v}{w}$ for the inner product in \eqref{eq:inner-product-sltwo}; according to
\Cref{thm:sltwo-orbit}, we have
\[
	\norm{v} = \norm{v}_{\Phi(-1)} \quad \text{and} \quad
	\inner{v}{w} = \inner{v}{w}_{\Phi(-1)},
\]
and so this agrees with our usual notation. From \eqref{eq:Phi-sltwo}, we get
\[
	\biginner{v}{w}_{\Phi(z)} 
	= \Bigl\langle e^{\half \log{\abs x} \, \Hsl} e^{iy \Ysl} v, \,
	e^{\half \log \abs{x} \, \Hsl} e^{iy \Ysl} w \Bigr\rangle
\]
and since $\Hsl^{\ast} = \Hsl$ and $\Ysl^{\ast} = \Xsl$ by
\Cref{thm:sltwo-orbit}, we can rewrite this as
\[
	\biginner{v}{w}_{\Phi(z)} 
	= \Bigl\langle e^{\log{\abs x} \, \Hsl} e^{iy \Ysl} v, \, e^{iy \Ysl} w \Bigr\rangle
	= \Bigl\langle e^{-iy \Xsl} e^{\log{\abs x} \, \Hsl} e^{iy \Ysl} v, \, w \Bigr\rangle.
\]

\newpar \label{par:model-VHS}
The conclusion that we can draw from this is that the weight filtration
\[
	W_{\ell} = E_{\ell}(\Hsl) \oplus E_{\ell-1}(\Hsl) \oplus E_{\ell-2}(\Hsl) \oplus \dotsb
\]
controls the rate of growth of the Hodge norm in these examples. Indeed, suppose that
we have a vector $v \in V$ with $\Hsl v = \ell v$. Then
\[
	e^{-iy \Xsl} e^{\log\abs{x} \, \Hsl} e^{iy \Ysl} v 
	= e^{-iy \Xsl} \sum_{k=0}^{\infty} \frac{(iy)^k}{k!} \abs{x}^{\ell-2k} \Ysl^k v
	= \sum_{j,k=0}^{\infty} (-1)^j \frac{(iy)^{j+k}}{j! k!} \abs{x}^{\ell-2k}
	\Xsl^j \Ysl^k v,
\]
and since different weight spaces are orthogonal under the inner product, we get
\[
	\norm{v}_{\Phi(z)}^2 
	= \sum_{k=0}^{\infty} (-1)^k \frac{(iy)^{2k}}{(k!)^2} \abs{x}^{\ell-2k}
	\inner{\Xsl^k \Ysl^k v}{v}
	= \sum_{k=0}^{\infty} \frac{y^{2k}}{(k!)^2} \abs{x}^{\ell-2k}
	\norm{\Ysl^k v}^2.
\]
As long as $y = \Im z$ remains bounded, this expression grows like $\abs{x}^{\ell} =
\abs{\Re z}^{\ell}$; in fact, the leading term is exactly $\abs{\Re z}^{\ell}
\norm{v}^2$.

\subsection{Recognizing polarized $\sltwo$-Hodge structures}

\newpar
We end our discussion of $\sltwo$-Hodge structures by proving a converse to
\Cref{thm:sltwo-orbit}. The result below gives us a way to recognize polarized
$\sltwo$-Hodge structures, and is going to be the key step in the
construction of the limiting mixed Hodge structure. Let $V$ be a finite-dimensional
representation of the Lie algebra $\sltwo(\CC)$, and let $Q \colon V
\tensor_{\CC} \Vb \to \CC$ be a nondegenerate hermitian pairing such that
$\Xsl^{\dagger} = \Xsl$, $\Ysl^{\dagger} = \Ysl$, and $\Hsl^{\dagger} = -\Hsl$, where
the dagger always means the adjoint of an operator with respect to $Q$.

\newpar
Now suppose that we have a decreasing filtration $F$ on the vector
space $V$, with the following two properties:
\begin{aenumerate}
\item One has $\Ysl(F^{\bullet}) \subseteq F^{\bullet-1}$ and $\Hsl(F^{\bullet}) =
	F^{\bullet}$.
\item	The filtration $e^{\Ysl} F$ is the Hodge filtration of a polarized 
	Hodge structure of weight $n$ on $V$, polarized by the pairing $Q$.
\end{aenumerate}
It will be convenient to have a name for this kind of object, so let us agree
to call such a filtration $F$ an \define{$\sltwo$-Hodge filtration} (of weight $n$). 

\begin{pthm} \label{thm:sltwo-filtration}
	If $F$ is an $\sltwo$-Hodge filtration on the vector space $V$, of weight $n$,
	then $V$ has a unique $\sltwo$-Hodge structure of weight $n$ such that
	\[
		F^p = \bigoplus_{i \geq p,j} V_{i+j-n}^{i,j} \quad \text{for all $p \in \ZZ$,}
	\]
	and this $\sltwo$-Hodge structure is polarized by the pairing $Q$.
\end{pthm}

\newpar
The proof is basically just linear algebra, with a little bit of representation
theory mixed in. We note the following useful corollary.

\begin{pcor} \label{cor:sltwo-endomorphism}
	If $S \in \End(V)$ satisfies $[\Hsl, S] = [\Ysl, S] = 0$ and $S^{\dagger} = S$,
	and if $S(F^{\bullet}) \subseteq F^{\bullet}$, then $S$ is an endomorphism
	of the $\sltwo$-Hodge structure that also respects the polarization.
\end{pcor}

\newpar
The proof of \Cref{thm:sltwo-filtration} uses the decomposition
\[
	V \cong \bigoplus_{m \in \NN} 
	S_m \tensor_{\CC} \Hom_{\CC}(S_m, V)^{\sltwo(\CC)}
\]
into isotypical components, coming from Schur's lemma. Recall that each $S_m$ has a
canonical polarized $\sltwo$-Hodge structure of weight $m$. According to
\Cref{thm:sltwo-orbit}, the total Hodge filtration $F^{\bullet} S_m$ is an
$\sltwo$-Hodge filtration in the above sense. We are going to argue that the
$\sltwo$-Hodge filtrations on $S_m$ and $V$ give each vector space $\Hom_{\CC}(S_m,
V)^{\sltwo(\CC)}$ a polarized Hodge structure of weight $n-m$, in a way that is
compatible with the above decomposition. The result we want will then follow from the
fact that each irreducible representation $S_m$ is a polarized $\sltwo$-Hodge
structure of weight $m$.

\newpar
The first thing to do is to analyze the subspace of $\sltwo(\CC)$-invariants;
everything else is going to follow from this case by functoriality.

\begin{pprop} \label{prop:invariants}
	Let $F$ be an $\sltwo$-Hodge filtration. Then
	\[
		V^{\sltwo(\CC)} = \menge{v \in V}{\Hsl v = \Ysl v = 0}
	\]
	has a Hodge structure of weight $n$, polarized by the restriction of $Q$, whose 
	Hodge filtration is $F \cap V^{\sltwo(\CC)}$. This Hodge structure is
	compatible with the polarized Hodge structure on $V$ defined by the filtration
	$e^{\Ysl} F$. 
\end{pprop}

\newpar
The idea of the proof is to analyze the Hodge decompositions of the two operators
$\Hsl, \Ysl \in \End(V)$. Note that $\Hsl$ is ``real'', in the sense that it belongs
to the Lie algebra $\glie = \Lie G$; the other two operators $\Xsl, \Ysl \in i
\glie$ are ``purely imaginary'', due to the fact that the pairing $Q$ is hermitian.
(The same thing happens for the $\sltwo(\CC)$-representation on the
cohomology of a compact K\"ahler manifold $(X, \omega)$, because $\Xsl = 2
\pi i \, L_{\omega}$ and $\Ysl = (2 \pi i)^{-1} \Lambda_{\omega}$ in that case.)

\newpar
To simplify the notation, let us from now on denote by
\[
	V = \bigoplus_{p+q=n} V^{p,q}
\]
the Hodge structure of weight $n$ whose Hodge filtration is $e^{\Ysl} F$. Since $Q$
is a polarization, it is easy to see that the conjugate Hodge filtration is
$e^{-\Ysl} \Fb$, where
\[
	\Fb_q = \menge{v \in V}{\text{$Q(v,x) = 0$ for all $x \in F^{n-q+1}$}}.
\]
The Hodge structure on $V$ induces a Hodge structure of weight $0$ on the vector space
$\End(V)$. Recall that the Lie algebra
\[
	\glie = \menge{A \in \End(V)}{A^{\dagger} = -A}
\]
defines a real structure; the induced polarization is $(A,B) \mapsto \tr(A \circ
B^{\dagger})$, where $B^{\dagger}$ means the adjoint of $B$ relative to the
nondegenerate hermitian pairing $Q$. For an operator $A \in \End(V)$, we denote by
\[
	A = \sum_{j \in \ZZ} A_j
\]
its Hodge decomposition in the Hodge structure on $\End(V)$, with $A_j \in
\End(V)^{j,-j}$. Concretely, this means that $A_j(V^{p,q}) \subseteq V^{p+j,q-j}$ for
all $p,q \in \ZZ$. 

\newpar
If we write the Hodge decompositions of $\Ysl, \Hsl \in \End(V)$ as
\[
	\Ysl = \sum_{j \in \ZZ} \Ysl_j \quad \text{and} \quad
	\Hsl = \sum_{j \in \ZZ} \Hsl_j,
\]
then our assumptions on $\Ysl$ and $\Hsl$ can be expressed as follows.

\begin{plem}
	We have
	\[
		\Ysl = \Ysl_{-1} + \Ysl_0 + \Ysl_1 \quad \text{and} \quad
		\Hsl = -2\Ysl_{-1} + \Hsl_0 + 2\Ysl_1,
	\]
	and these operators satisfy $\Ysl_{-1}^{\dagger} = \Ysl_1$, $\Ysl_0^{\dagger} =
	\Ysl_0$, and $\Hsl_0^{\dagger} = - \Hsl_0$. 
\end{plem}

\begin{proof}
	From the fact that $\Ysl^{\dagger} = \Ysl$, we get $(\Ysl_j)^{\dagger} =
	\Ysl_{-j}$ for every $j \in \ZZ$. The condition $\Ysl(F^p) \subseteq F^{p-1}$
	implies that $\Ysl(e^{\Ysl} F^p) \subseteq e^{\Ysl} F^{p-1}$, which
	means that $\Ysl_j = 0$ for $j \leq -2$.  But then also $\Ysl_j = 0$
	for $j \geq 2$, and so actually 
	\[
		\Ysl = \Ysl_{-1} + \Ysl_0 + \Ysl_1.
	\]
	Similarly, we have $\Hsl^{\dagger} = -\Hsl$, hence $(\Hsl_j)^{\dagger} =
	-\Hsl_{-j}$ for every $j \in \ZZ$. The condition $\Hsl(F^p) \subseteq F^p$
	implies that 
	\[
		(\Hsl + 2 \Ysl) e^{\Ysl} F^p = 
		e^{\Ysl} \Hsl e^{-\Ysl} \cdot e^{\Ysl} F^p 
		= e^{\Ysl} \Hsl(F^p) \subseteq e^{\Ysl} F^p,
	\]
	and so $\Hsl_j + 2\Ysl_j = 0$ for $j \leq -1$. In particular,
	$\Hsl_{-1} = -2\Ysl_{-1}$ and $\Hsl_j = 0$ for all $j \leq -2$. But then
	$\Hsl_j = 0$ for $j \geq 2$, and $\Hsl_1 = -(\Hsl_{-1})^{\dagger} =
	2(\Ysl_{-1})^{\dagger} = 2 \Ysl_1$. 
\end{proof}

\newpar
The relation $[\Hsl, \Ysl] = -2\Ysl$ gives us the following additional identity.

\begin{plem} \label{lem:identity}
	We have $2\Ysl_0 = 2[\Ysl_{-1},\Ysl_1] + [\Ysl_0,\Hsl_0]$.
\end{plem}

\begin{proof}
	Consider the component of $2\Ysl = [\Ysl,\Hsl]$ in the subspace $\End(V)^{0,0}$.
	From the preceding lemma, we get
	\[
		2 \Ysl_0 = [\Ysl_1, \Hsl_{-1}] + [\Ysl_0, \Hsl_0] + [\Ysl_{-1}, \Hsl_1]
		= -2[\Ysl_1, \Ysl_{-1}] + [\Ysl_0, \Hsl_0] +2 [\Ysl_{-1}, \Ysl_1],
	\]
	which simplifies to the desired identity.
\end{proof}

\newpar
We can now prove \Cref{prop:invariants}. On $V$, we have a polarized Hodge
structure of weight $n$ with Hodge filtration $e^{\Ysl} F$. It is enough to show that
$V^{\sltwo(\CC)}$ is a sub-Hodge structure; the remaining assertions then follow
because $e^{\Ysl} F$ and $F$ induce the same filtration on $V^{\sltwo(\CC)}$, due to
the fact that $\Ysl$ acts trivially. 

\newpar
To say that $V^{\sltwo(\CC)}$ is a sub-Hodge structure means that whenever we take
a vector $v \in V^{\sltwo(\CC)}$, and write its Hodge decomposition as
\[
	v = \sum_p v_p,
\]
with $v_p \in V^{p,q}$, then each $v_p \in V^{\sltwo(\CC)}$. This is trivially
satisfied if $v = 0$. If $v \neq 0$, let $p \in \ZZ$ be the least integer such
that $v_p \neq 0$. It is clearly enough to prove that $v_p \in V^{\sltwo(\CC)}$,
because we can then repeat the same argument for $v-v_p$. Our goal is therefore to
show that $\Ysl v_p = \Hsl v_p = 0$. 

\newpar
From the fact that $\Ysl v = 0$ and the Hodge decomposition, we deduce that
\[
	\Ysl_{-1} v_p = 0 \quad \text{and} \quad
	\Ysl_0 v_p + \Ysl_{-1} v_{p+1} = 0.
\]
From the fact that $\Hsl v = 0$, we get the additional piece of information that
\[
	\Hsl_0 v_p - 2\Ysl_{-1} v_{p+1} = 0.
\]
In particular, $\Hsl_0 v_p = -2 \Ysl_0 v_p$, which is interesting, because $\Ysl_0$
and $\Hsl_0$ behave very differently under taking adjoints. We can use this different
behavior to show that $Q(\Ysl_0 v_p, v_p) = 0$. Since $(\Ysl_0)^{\dagger} = \Ysl_0$ and
$(\Hsl_0)^{\dagger} = -\Hsl_0$, we have
\[
	Q(2\Ysl_0 v_p, v_p) = -Q(\Hsl_0 v_p, v_p) = Q(v_p, \Hsl_0 v_p) 
	= -Q(v_p, 2\Ysl_0 v_p) = -Q(2\Ysl_0 v_p, v_p),
\]
and therefore $Q(2\Ysl_0 v_p, v_p) = 0$. Now we combine this with the identity in
\Cref{lem:identity}. Because we already know that $\Ysl_{-1} v_p = 0$, this gives
\begin{align*}
	0 = Q(2 \Ysl_0 v_p, v_p) &= Q(2 \Ysl_{-1} \Ysl_1 v_p, v_p) 
	+ Q(\Ysl_0 \Hsl_0 v_p, v_p) - Q(\Hsl_0 \Ysl_0 v_p, v_p) \\
	&= 2 Q(\Ysl_1 v_p, \Ysl_1 v_p) + Q(\Hsl_0 v_p, \Ysl_0 v_p) + Q(\Ysl_0 v_p, \Hsl_0 v_p) \\
	&= 2 Q(\Ysl_1 v_p, \Ysl_1 v_p) - 4 Q(\Ysl_0 v_p, \Ysl_0 v_p).
\end{align*}
Here we used the fact that $(\Ysl_{-1})^{\dagger} = \Ysl_1$ and $(\Hsl_0)^{\dagger} =
-\Hsl_0$, and also $\Hsl_0 v_p = -2\Ysl_0 v_p$. Now $\Ysl_0 v_p \in V^{p,q}$, and
since $Q$ is a polarization, 
\[
	Q(\Ysl_0 v_p, \Ysl_0 v_p) = (-1)^q \norm{\Ysl_0 v_p}^2,
\]
where $\norm{\argbl}$ means the Hodge norm in the polarized Hodge structure $e^{\Ysl}
F$. Likewise, $\Ysl_1 v_p \in V^{p+1,q-1}$, and so
\[
	Q(\Ysl_1 v_p, \Ysl_1 v_p) = (-1)^{q-1} \norm{\Ysl_1 v_p}^2.
\]
Putting everything together, we find that 
\[
	0 = \norm{\Ysl_1 v_p}^2 + 2 \norm{\Ysl_0 v_p}^2,
\]
which clearly implies that $\Ysl_1 v_p = 0$ and $\Ysl_0 v_p = 0$. But then also $\Hsl_0
v_p = 0$, and so we have proved that $\Ysl v_p = \Hsl v_p = 0$, hence $v_p \in
V^{\sltwo(\CC)}$. This finishes the proof of \Cref{prop:invariants}.

\newpar
The next step is to establish some basic functoriality. For the sake of clarity, let
us denote the $\sltwo$-Hodge filtrations on $S_m$ and $V$ by the symbols
$F^{\bullet} S_m$ and $F^{\bullet} V$; the resulting Hodge structures have weight $m$
and $n$. We get an induced filtration on $\Hom_{\CC}(S_m, V)$ according to the rule
\[
	F^k \Hom_{\CC}(S_m, V) = \menge{f \colon S_m \to V}
	{\text{$f(F^p S_m) \subseteq F^{p+k} V$ for all $p \in \ZZ$}}.
\]
The induced $\sltwo(\CC)$-representation is easy to describe: for $f \colon
S_m \to V$, one has
\[
	(\Hsl f)(v) = \Hsl f(v) - f(\Hsl v), \quad
	(\Xsl f)(v) = \Xsl f(v) - f(\Xsl v), \quad
	(\Ysl f)(v) = \Ysl f(v) - f(\Ysl v).
\]
Observe that $\sltwo(\CC)$ acts trivially on a linear mapping $f \colon S_m \to V$ exactly
when $f$ is a morphism of $\sltwo(\CC)$-representations.

\newpar
Since it is important, let us check very carefully that $F^{\bullet} \Hom_{\CC}(S_m, V)$
is indeed an $\sltwo$-Hodge filtration (of weight $n-m$).

\begin{plem} \label{lem:sltwo-filtration-Hom}
	Suppose that $F^{\bullet} V$ is an $\sltwo$-Hodge filtration of weight $n$. Then
	the filtration $F^{\bullet} \Hom_{\CC}(S_m, V)$ is an $\sltwo$-Hodge filtration of
	weight $n-m$.
\end{plem}

\begin{proof}
	We need to check that the filtration on $\Hom_{\CC}(S_m, V)$ satisfies the two
	conditions in the definition. We can again use the modified trace pairing on
	$\Hom_{\CC}(S_m, V)$ as a polarization. Given a linear mapping $f \colon S_m \to
	V$, we denote by $f^{\dagger} \colon V \to S_m$ the adjoint with respect to the
	nondegenerate pairings $Q_{S_m}$ and $Q_V$; to be precise, 
	\[
		Q_V \bigl( f(v), s \bigr) = Q_{S_m} \bigl( v, f^{\dagger}(s) \bigr) \quad
		\text{for all $v \in V$ and $s \in S_m$.}
	\]
	On $\Hom_{\CC}(S_m, V)$, we have the hermitian pairing
	\begin{equation} \label{eq:Hom-trace}
		\Hom_{\CC}(S_m, V) \tensor_{\CC} \overline{\Hom_{\CC}(S_m, V)} \to \CC, \quad
		(f, g) \mapsto \frac{\tr(g^{\dagger} \circ f)}{\dim S_m}.
	\end{equation}
	Let $f \in F^k \Hom_{\CC}(S_m, V)$ be arbitrary. For any $s \in F^p S_m$, we have $f(s)
	\in F^{p+k} V$, and therefore
	\[
		(\Hsl f)(s) = \Hsl f(s) - f(\Hsl s) \in \Hsl(F^{p+k} V) + f(F^p S_m)
		\subseteq F^{p+k} V,
	\]
	which proves that $\Hsl f \in F^k \Hom_{\CC}(S_m, V)$. Similarly, $\Ysl f \in
	F^{k-1} \Hom_{\CC}(S_m, V)$. 

	It remains to show that the filtration $e^{\Ysl} F^{\bullet} \Hom_{\CC}(S_m, V)$
	defines a Hodge structure of weight $n-m$ on $\Hom_{\CC}(S_m, V)$, polarized by
	the pairing in \eqref{eq:Hom-trace}. Since
	\[
		(e^{\Ysl} f)(v) = e^{\Ysl} f \bigl( e^{-\Ysl} v \bigr),
	\]
	it is not hard to see that
	\[
		e^{\Ysl} F^k \Hom_{\CC}(S_m, V) = \menge{f \colon S_m \to V}
		{\text{$f \bigl( e^{\Ysl} F^p S_m \bigr) \subseteq e^{\Ysl} F^{p+k} V$
		for all $p \in \ZZ$}};
	\]
	but the right-hand side is exactly the Hodge filtration of the induced Hodge
	structure of weight $n-m$ on $\Hom_{\CC}(S_m, V)$. The proof that the
	pairing in \eqref{eq:Hom-trace} polarizes this Hodge structure is similar to the
	proof of \Cref{lem:tr-polarization}.
\end{proof}

\newpar
We can now prove \Cref{thm:sltwo-filtration}. Let $F^{\bullet} V$ be an
$\sltwo$-Hodge filtration on $V$, of weight $n$. Our starting point is the
decomposition in \eqref{eq:decomposition}. Fix some $m \in \NN$. According to
\Cref{lem:sltwo-filtration-Hom}, the induced filtration
\[
	F^k \Hom_{\CC}(S_m, V) = \menge{f \colon S_m \to V}
	{\text{$f(F^p S_m) \subseteq F^{p+k} V$ for all $p \in \ZZ$}}
\]
is an $\sltwo$-Hodge filtration of weight $n-m$, and the modified trace pairing
\[
	\Hom_{\CC}(S_m, V) \tensor_{\CC} \overline{\Hom_{\CC}(S_m, V)} \to \CC, \quad
	(f, g) \mapsto \frac{\tr(g^{\dagger} \circ f)}{\dim S_m}
\]
is a polarization. \Cref{prop:invariants} tells us that the subspace
\[
	H_m = \Hom_{\CC}(S_m, V)^{\sltwo(\CC)} \subseteq \Hom_{\CC}(S_m, V)
\]
has a Hodge structure of weight $n-m$, with Hodge filtration
\[
	F^k H_m = F^k \Hom_{\CC}(S_m, V) \cap H_m.
\]
Moreover, we know that this Hodge structure is polarized by the restriction of the
modified trace pairing. But for $f, g \in H_m$, the composition $g^{\dagger} \circ f$
is an endomorphism of $S_m$ as an $\sltwo(\CC)$-representation, hence (by Schur's
lemma) a multiple of the identity.  Thus $g^{\dagger} \circ f = c(f,g) \id$ for some
constant $c(f,g) \in \CC$, and because of how we defined the modified trace pairing,
this says exactly that
\[
	c \colon H_m \tensor_{\CC} \wbar{H_m} \to \CC
\]
polarizes the Hodge structure on $H_m = \Hom_{\CC}(S_m, V)^{\sltwo(\CC)}$.

\begin{plem}
	The evaluation morphism
	\[
		\bigoplus_{m \in \NN} S_m \tensor_{\CC} H_m \to V
	\]
	induces an isomorphism between the $\sltwo$-Hodge filtrations on both sides.
\end{plem}

\begin{proof}
	We know that the mapping is an isomorphism of $\sltwo(\CC)$-representations. By
	the same argument as in \Cref{lem:sltwo-polarizations}, one shows that the
	isomorphism is compatible with the hermitian pairings on both sides. Let
	\[
		S_m = \bigoplus_{p+q=m} S_m^{p,q} \quad \text{and} \quad
		H_m = \bigoplus_{p+q=n-m} H_m^{p,q}
	\]
	denote the Hodge decompositions in the Hodge structures on $S_m$ and $W$. By
	construction, $H_m$ is a sub-Hodge structure of $\Hom_{\CC}(S_m, V)$, and so
	\[
		H_m^{p,q} = \menge{f \colon S_m \to V}
		{\text{$f(S_m^{j,m-j}) \subseteq V^{j+p,m-j+q}$ for all $j \in \ZZ$}}.
	\]
	We are going to argue that the evaluation morphism is a morphism of Hodge
	structures of weight $n$. Let $p \in \ZZ$ be an integer. As with any tensor
	product, the $(p,n-p)$-subspace in the Hodge decomposition of the left-hand side is
	\[
		\bigoplus_{m \in \NN} \bigoplus_{j \in \ZZ} 
		S_m^{j,m-j} \tensor_{\CC} H_m^{p-j,n-m-p+j};
	\]
	the evaluation morphism takes this into the subspace $V^{p,n-p}$, and is therefore
	a morphism of Hodge structures of weight $n$. Since morphisms of Hodge structures
	strictly respect the Hodge filtration, we get isomorphisms
	\[
		\bigoplus_{m \in \NN} \bigoplus_{j \in \ZZ} 
		e^{\Ysl} F^j S_m \tensor_{\CC} F^{p-j} H_m \cong e^{\Ysl} F^p V
	\]
	between the Hodge filtrations on both sides. After multiplying by $e^{-\Ysl}$,
	this gives us the desired isomorphisms
	\[
		\bigoplus_{m \in \NN} \bigoplus_{j \in \ZZ} 
		F^j S_m \tensor_{\CC} F^{p-j} H_m \cong F^p V. \qedhere
	\]
\end{proof}

\newpar
By construction, the $\sltwo$-Hodge filtration on each $S_m$ comes from a polarized
$\sltwo$-Hodge structure of weight $m$. If we let $\sltwo(\CC)$ act trivially on
the Hodge structure $\Hom_{\CC}(S_m, V)^{\sltwo(\CC)}$, the direct sum
\[
	\bigoplus_{m \in \NN} S_m \tensor_{\CC} \Hom_{\CC}(S_m, V)^{\sltwo(\CC)}
\]
inherits a polarized $\sltwo$-Hodge structure of weight $m+(n-m) = n$. This puts an
$\sltwo$-Hodge structure of weight $n$ on $V$, and the lemma above says that the
$\sltwo$-Hodge filtration $F^{\bullet} V$ agrees with the total Hodge filtration of
this $\sltwo$-Hodge structure, and that the polarizations are compatible.
\Cref{thm:sltwo-filtration} is therefore proved in full.

\begin{note}
	A posteriori, we can deduce from \Cref{lem:Hodge-HXY} that $\Hsl_0 = 0$; but this
	fact does not play a role during the proof.
\end{note}

\newpar
In fact, one can improve \Cref{thm:sltwo-filtration} by taking out the
assumption that $V$ is a representation of $\sltwo(\CC)$. This gives a minimal set of
conditions by which one can recognize a polarized $\sltwo$-Hodge structure.
Interestingly, a version of this result is used in the proof of the decomposition
theorem \cite{dCM}, in order to establish the relative Hard Lefschetz theorem. This
is another place where the global theory (cohomology with
coefficients in a polarized variation of Hodge structure) and the local theory
(degenerating variations of Hodge structure) meet rather unexpectedly.

\newpar
Here is the minimal set of conditions. This time, we only assume that
\[
	V = \bigoplus_{k \in \ZZ} V_k
\]
is a finite-dimensional graded complex vector space; we let $\Hsl \in \End(V)$ be the
semisimple operator that acts as multiplication by $k$ on the subspace $V_k$. Let
$\Ysl \in \End(V)$ be a nilpotent operator such that $[\Hsl, \Ysl] = -2\Ysl$; note
that we are \emph{not} assuming that $\Ysl^k \colon V_k \to V_{-k}$ is an
isomorphism. Further, let $Q \colon V \tensor_{\CC} \Vb \to \CC$ be a 
hermitian pairing such that $\Ysl^{\dagger} = \Ysl$ and $\Hsl^{\dagger} = -\Hsl$,
where the dagger means the adjoint with respect to $Q$. Lastly, suppose that we have
a decreasing filtration $F$ on the vector space $V$, with the following two properties:
\begin{aenumerate}
\item One has $\Ysl(F^{\bullet}) \subseteq F^{\bullet-1}$ and $\Hsl(F^{\bullet}) =
	F^{\bullet}$.
\item	The filtration $e^{\Ysl} F$ is the Hodge filtration of a polarized 
	Hodge structure of weight $n$ on $V$, polarized by the pairing $Q$ (which is
	therefore nondegenerate).
\end{aenumerate}

\begin{pcor} \label{cor:sltwo-filtration}
	Under these assumptions, $F$ is an $\sltwo$-Hodge structure of weight $n$;
	consequently, $V$ is an $\sltwo$-Hodge structure of weight $n$,
	polarized by the hermitian pairing $Q$.
\end{pcor}

\begin{proof}
	We are going to argue that $\Ysl^k \colon V_k \to V_{-k}$ is an isomorphism
	for every $k \in \NN$. This implies that $\Hsl$ and $\Ysl$ determine a
	representation of $\sltwo(\CC)$ on $V$, and together with the other conditions
	above, is enough for applying \Cref{thm:sltwo-filtration}. An equivalent
	formulation is that the increasing filtration
	\[
		W_k = V_k \oplus V_{k-1} \oplus V_{k-2} \oplus \dotsb
	\]
	is the monodromy weight filtration of the nilpotent operator $\Ysl$. By
	assumption, $e^{\Ysl} F \in D$. Let us show that $e^{-z \Ysl} F \in D$ for
	every $z \in \HH$. The two operators $e^{-\half \log \abs{\Re z} \, \Hsl}$ and
	$e^{i \Im z \, \Ysl}$ belong to the real Lie group $G = \Aut(V,Q)$, due to the fact
	that $\Hsl^{\dagger} = -\Hsl$ and $\Ysl^{\dagger} = \Ysl$.  From the relation
	$[\Hsl, \Ysl] = -2\Ysl$ and the fact that $\Hsl(F^{\bullet}) \subseteq
	F^{\bullet}$, we thus get
	\[
		e^{-z \Ysl} F = e^{-i \Im z \Ysl} e^{\abs{\Re z} \Ysl} F
		= e^{-i \Im z \, \Ysl} e^{-\half \log \abs{\Re z} \, \Hsl} \cdot e^{\Ysl} F \in D.
	\]
	Because $\Ysl(F^{\bullet}) \subseteq F^{\bullet-1}$, the mapping
	\[
		\Phi \colon \HH \to D, \quad \Phi(z) = e^{-z\Ysl} F,
	\]
	is the period mapping of a polarized variation of Hodge structure of weight $n$,
	with monodromy operator $T = e^{-2 \pi i \Ysl}$. In our usual notation, we
	therefore have $N = -\Ysl$. 

	Now comes the crucial point: because of the Hodge norm estimates, we can recognize the
	monodromy weight filtration of $N$ -- and hence that of $\Ysl$ -- by the order of
	growth of the Hodge norm. Let $v \in V_k$ be any nonzero vector. As long as $\Im z$
	remains bounded, we have
	\[
		\norm{v}_{\Phi(z)}^2 
		= \bignorm{e^{\half \log \abs{\Re z} \, \Hsl} e^{i \Im z \, \Ysl} v}_{e^{\Ysl} F}^2
		\sim \abs{\Re z}^k \cdot \norm{v}_{e^{\Ysl} F}^2,
	\]
	and so according to \Cref{thm:Hodge-norm-intro}, the filtration $W$ we defined
	above must be the monodromy weight filtration of $\Ysl$. By construction, $\gr_k^W
	\cong V_k$, and so $\Ysl^k \colon V_k \to V_{-k}$ is an isomorphism for every $k
	\in \NN$. This concludes the proof.
\end{proof}

\section{The Hodge norm estimates}

\newpar
In this chapter, we use the results about harmonic bundles that we derived in
\Cref{sec:VHS} to prove the Hodge norm estimates. The point of these estimates is to
control the order of growth of the Hodge norm of a multi-valued flat section in terms
of the monodromy weight filtration of the nilpotent operator $N$.

\begin{note}
In fact, exactly the same argument proves the Hodge norm estimates for an arbitrary
harmonic bundle on the punctured disk whose Higgs field $\theta_{\partial/\partial
t}$ is nilpotent.
\end{note}

\subsection{Boundedness of the Hodge norm for flat sections}

\newpar
Let us start by showing that, because of the derivative bound in
\eqref{eq:derivatives-final}, the Hodge norm of any multi-valued flat section can grow
at most like a power of $\abs{\Re z}$.

\begin{pprop} \label{prop:polynomial-bound}
	Let $v \in V$ be a multi-valued flat section. Then
	\begin{align*}
		\norm{v}_{\Phi(z)} &\leq e^{2C_0 \pi} 
		\max \Bigl( \abs{\Re z}^{2C_0}, \abs{\Re z}^{-2C_0} \Bigr)
		\norm{v}_{\Phi(-1)} \\
		\norm{v}_{\Phi(-1)} &\leq e^{2C_0 \pi} 
		\max \Bigl( \abs{\Re z}^{2C_0}, \abs{\Re z}^{-2C_0} \Bigr)
		\norm{v}_{\Phi(z)} 
	\end{align*}
	on the horizontal strip $\abs{\Im z} \leq \pi$, where $C_0 = \frac{1}{2}
	\sqrt{\binom{r+1}{3}}$ and $r = \rk E$.
\end{pprop}

\begin{proof}
	Assuming that $v \neq 0$, we consider the function $\varphi = \log
	\norm{v}_{\Phi(z)}^2$. Set $z = x + iy$. Then
	\[
		\bigabs{\varphi(x+iy) - \varphi(-1)} \leq
		\bigabs{\varphi(x+iy) - \varphi(-1+iy)} + \bigabs{\varphi(-1+iy) -
		\varphi(-1)}.
	\]
	The derivative bound in \eqref{eq:derivatives-final} implies that
	\[
		\ABS{\frac{\partial \varphi}{\partial x}}
		= \ABS{\frac{\partial \varphi}{\partial z} + \frac{\partial \varphi}{\partial
		\zb}} \leq \frac{4C_0}{\abs{\Re z}},
	\]
	with a similar inequality for $\partial \varphi/\partial y$. Therefore
	\[
		\bigabs{\varphi(x+iy) - \varphi(-1+iy)} 
		\leq \ABS{\int_{-1}^x \frac{4C_0}{\abs{t}} \dt} = 4C_0 \bigabs{\log \abs{x}};
	\]
	in the same manner, we have for $\abs{y} \leq \pi$ the inequality
	\[
		\bigabs{\varphi(-1+iy) - \varphi(-1)} \leq 4C_0 \abs{y} \leq 4C_0 \pi.
	\]
	Putting everything together, we get
	\[
		-4C_0 \pi - 4C_0 \bigabs{\log \abs{\Re z}} \leq
		\log \norm{v}_{\Phi(z)}^2 - \log \norm{v}_{\Phi(-1)}^2 
		\leq 4C_0 \pi + 4C_0 \bigabs{\log \abs{\Re z}},
	\]
	from which the desired inequalities follow by exponentiation.
\end{proof}

\newpar
The following result is a first step towards more precise estimates for Hodge norms
of multi-valued flat sections. The proof is disarmingly simple.

\begin{pprop} \label{prop:kerN}
	Let $v \in V$ be a multi-valued flat section such that $T v = \lambda v$
	for some $\lambda \in \CC$. Then for every $x_0 < 0$, we have the inequality
	\[
		\norm{v}_{\Phi(z)} \leq e^{4C_0/\abs{x_0}} \norm{v}_{\Phi(x_0)}
	\]
	for all $z \in \HH$ with $\Re z \leq x_0$; here $C_0 = \frac{1}{2}
	\sqrt{\binom{r+1}{3}}$ and $r = \rk E$.
\end{pprop}

Using the Jordan decomposition $T = T_s e^{2 \pi i \, N}$, the condition $T v =
\lambda v$ is equivalent to $T_s v = \lambda v$ and $Nv = 0$. The inequality in the
proposition is therefore a special case of the Hodge norm estimates for multi-valued
flat sections. Indeed, $Nv = 0$ implies that $v \in W_0$, and so we expect the Hodge
norm of $v$ to remain bounded as $\Re z \to -\infty$. It turns out that this special
case is all that is needed to prove the Hodge norm estimates in general: the power of
the statement comes from the fact that it applies to \emph{all} polarized variations
of Hodge structure on $\dst$. The universal character of the
inequality also makes it very useful for studying variations of Hodge structure in
several variables, as we plan to show in a sequel to this paper.

\newpar
Now for the proof of the proposition. The semisimple part of the monodromy
transformation $T = T_s e^{2 \pi i N}$ give rise to a decomposition
\[
	V = \bigoplus_{\abs{\lambda} = 1} E_{\lambda}(T_s),
\]
and each eigenspace $E_{\lambda}(T_s)$ is preserved by $N$. Let us consider a fixed nonzero
vector $v \in E_{\lambda}(T_s) \cap \ker N$. As in the proof of the monodromy
theorem (in \Cref{prop:monodromy-theorem}), these conditions imply that the function
\[
	\varphi = \log h(v,v)
\]
is invariant under the substitution $z \mapsto z + 2 \pi i$. In addition, we have
already proved that $\varphi$ is smooth and subharmonic
(\Cref{lem:metric-subharmonic}), and that 
\[
	\ABS{\frac{\partial \varphi}{\partial z}}
	= \ABS{\frac{\partial \varphi}{\partial \zb}}
	\leq \frac{2C_0}{\abs{\Re z}},
\]
Recall that we derived this inequality in \eqref{eq:derivatives-final}; here $C_0 =
\half \sqrt{\binom{r+1}{3}}$ and $r = \rk E$.

\begin{pexa}
To get a feeling for these conditions, let us consider a toy example: a
smooth function $f \colon (-\infty, 0) \to \RR$ with the property that $f'' \geq 0$
and $\abs{f'(x)} \leq C \abs{x}^{-1}$. Since $f'' \geq 0$, the function
$f'$ must be increasing, and therefore
\[
	0 = \lim_{x \to -\infty} f'(x) \leq f'.
\]
But this means that $f$ is itself increasing, and so we get
\[
	f(x) \leq f(x_0)
\]
for every $x \leq x_0 < 0$. In particular, $f$ is bounded above as $x \to
-\infty$. A similar convexity argument also shows up in Mochizuki's work on the
asymptotic behavior of tame harmonic bundles \cite[Lem.~2.23]{Mochizuki}.
\end{pexa}

\newpar
In order to apply the same reasoning as in the example, we need a function that
only depends on $x = \Re z$. We therefore consider the vertical averages
\[
	f(x) = \frac{1}{2 \pi} \int_0^{2 \pi} \varphi(x + iy) \dy.
\]
Differentiation under the integral sign gives
\[
	f''(x) = \frac{1}{2 \pi} \int_0^{2 \pi} \varphi_{x,x}(x + iy) \dy
	\geq - \frac{1}{2 \pi} \int_0^{2 \pi} \varphi_{y,y}(x + iy) \dy,
\]
since $\Delta \varphi = \varphi_{x,x} + \varphi_{y,y} \geq 0$. The integral evaluates to
\[
	\frac{1}{2 \pi} \int_0^{2 \pi} \varphi_{y,y}(x + iy) \dy
	= \frac{1}{2 \pi} \Bigl( \varphi_y(x + 2 \pi i) - \varphi_y(x) \Bigr) = 0,
\]
due to the fact that $\varphi(x+iy)$ is periodic in $y$ of period $2\pi$.
Therefore $f(x)$ is a convex function of $x$. Moreover, we know from
\eqref{eq:derivatives-final} that
\[
	\abs{\varphi_x(x+iy)} \leq \frac{4C_0}{\abs{x}},
\]
where $C_0 = \frac{1}{2} \sqrt{\binom{n+1}{3}}$. Plugging this into the integral from
above gives
\[
	\abs{f'(x)} \leq \frac{1}{2 \pi} \int_0^{2 \pi} \abs{\varphi_x(x + iy)} \dy
	\leq \frac{4C_0}{\abs{x}}.
\]
As in the example, it follows that $f(x) \leq f(x_0)$ for every $x \leq x_0 < 0$.

\newpar
All that is left is to relate the behavior of $\varphi(x+iy)$ to that of $f(x)$.
We have
\[
	-\frac{4C_0}{\abs{x}} \leq \varphi_y(x+iy) \leq \frac{4C_0}{\abs{x}},
\]
which can be integrated to give 
\[
	-\frac{4C_0}{\abs{x}} u \leq \varphi(x+iy+iu) - \varphi(x+iy) \leq
	\frac{4C_0}{\abs{x}} u 
\]
for all $u \geq 0$. If we now average over $u \in [0,2\pi]$, we get
\[
	-\frac{4C_0}{\abs{x}} \leq f(x) - \varphi(x+iy) \leq \frac{4C_0}{\abs{x}}.
\]
For $x \leq x_0 < 0$, we have $f(x) \leq f(x_0)$, and therefore
\[
	\varphi(x+iy) \leq f(x) + \frac{4C_0}{\abs{x}} \leq f(x_0) + \frac{4C_0}{\abs{x}} 
	\leq \varphi(x_0) + \frac{8C_0}{\abs{x_0}}.
\]
This gives the desired upper bound for $\norm{v}_{\Phi(z)}$ after exponentiation.

\subsection{The comparison theorem}

\newpar
One striking consequence of \Cref{prop:kerN} is that the behavior of the
Hodge norm as $t \to 0$ only depends on the underlying flat vector bundle $(E,d)$.
Simpson proves a similar result for tame harmonic bundles, but with conditions that
appear to be much more restrictive \cite[Cor.~4.3]{Simpson}.

\begin{pthm} \label{thm:comparison}
	Let $E_1, E_2$ be polarized variations of Hodge structure over $\dst$. If
	$(E_1, d_1)$ and $(E_2, d_2)$ are isomorphic as smooth vector bundles with
	connection, then for every $0 < r < 1$, there is a constant $C_r > 0$ such that
	\[
		C_r^{-1} \cdot h_{E_1} \leq h_{E_2} \leq C_r \cdot h_{E_1}
	\]
	at all points $t \in \dst$ with $\abs{t} \leq r$.
\end{pthm}

\begin{proof}
	Let $f \colon E_1 \to E_2$ be an isomorphism between the underlying smooth vector
	bundles that preserves the connections. We consider $f$ as a single-valued flat
	section of the bundle $\Hom(E_1, E_2)$. Since this bundle underlies a polarized
	variation of Hodge structure, \Cref{prop:kerN} applies to it. Taking
	$x_0 = \log r$ in the statement, we get a bound for the pointwise Hodge norm of
	$f$, of the form
	\[
		h_{\Hom(E_1, E_2)}(f,f) \leq C_r,
	\]
	valid for every $t \in \dst$ with $\abs{t} \leq r$. Here the constant $C_r$ equals
	$e^{8C_0/\abs{\log r}}$ times the value of the Hodge norm at the point $t = r$.
	Since the Hodge norm is an upper bound for the pointwise operator norm of $f$,
	this gives us the inequality
	\[
		h_{E_2} \leq C_r \cdot h_{E_1}.
	\]
	The reverse inequality follows by applying the same reasoning to the inverse
	isomorphism $f^{-1} \colon E_2 \to E_1$. 
\end{proof}

\subsection{Order of growth and the weight filtration}

\newpar
Let $W = W_{\bullet} V$ be the monodromy weight filtration of the nilpotent operator
$N$. Our next goal is to prove that the weight filtration controls the behavior of
the function $h(v,v)$ as $\Re z \to -\infty$, at least on horizontal strips of
bounded height. The idea is to construct another variation of Hodge structure on the
same underlying flat vector bundle $(E, d)$, using representation theory.
\Cref{thm:comparison} guarantees that the Hodge metrics of the two variations
are equivalent up to a constant, and this implies the Hodge norm estimates.

\newpar \label{par:splitting}
To get started, we need to upgrade the nilpotent operator $N \in \End(V)$
into a representation of the Lie algebra $\sltwo(\CC)$. Let us denote the standard basis
elements by $\Hsl, \Xsl, \Ysl \in \sltwo(\CC)$; then
\[
	[\Hsl,\Xsl] = 2\Xsl, \quad [\Hsl,\Ysl] = -2\Ysl, \quad [\Xsl,\Ysl] = \Hsl.
\]
As $N \in \End(V)$ is nilpotent, one can find a semisimple endomorphism $H \in
\End(V)$ with integral eigenvalues, such that $[H,N] = -2N$. In fact, one can make a
more careful choice of $H$; the additional properties are going to be useful
for us later on.

\begin{pprop} \label{prop:splitting}
	One can find $H \in \End(V)$ with the following properties:
	\begin{aenumerate}
	\item \label{en:splitting-a}
		$H$ is semisimple with integral eigenvalues.
	\item \label{en:splitting-b}
		One has $[H,N] = -2N$ and $W_k = E_k(H) \oplus W_{k-1}$ for every $k \in \ZZ$.
	\item \label{en:splitting-c}
		One has $Q(Hv,w) + Q(v,Hw) = 0$ for every $v,w \in V$.  
	\item \label{en:splitting-d}
		$H$ commutes with $T_s$. 
	\end{aenumerate}
\end{pprop}

\begin{proof}
	The proof is easy, and so we only give a sketch. Consider an
	arbitrary nilpotent endomorphism $N \in \End(V)$ of a finite-dimensional vector
	space $V$. Denote by $\Sigma(N) \subseteq \End(V)$ the set of all semisimple
	endomorphisms $H$ with integer eigenvalues, such that $[H, N] = -2N$
	and $W_j = E_j(H) \oplus W_{j-1}$ for every $j \in \ZZ$. It is easy to see that
	$\Sigma(N) \neq \emptyset$, for example by choosing a basis that
	puts $N$ into Jordan canonical form. This already shows that splittings satisfying
	(a) and (b) exist, a special case of the Jacobson-Morozov theorem.

	Let us now compare two arbitrary elements $H, H' \in \Sigma(N)$. The difference
	$H'-H$ commutes with $N$ and satisfies $(H'-H)(W_{\bullet}) \subseteq
	W_{\bullet-1}$. Denote by $\ad H$ the semisimple operator on $\End(V)$, defined as
	$(\ad H)(A) = [H, A]$; then $H'-H$ belongs to the direct sum of the eigenspaces
	$E_k(\ad H)$ with $k \leq -1$. From this, one deduces that 
	\[
		H' = e^B H e^{-B}
	\]
	for a unique endomorphism $B \in \End(V)$ such that $[B,N] = 0$ and
	$B(W_{\bullet}) \subseteq W_{\bullet-1}$. Conversely, for any $B \in \End(V)$ with
	these two properties, one has $e^B H e^{-B} \in \Sigma(N)$. 

	Now we describe how to adjust a given splitting $H_0 \in \Sigma(N)$ so that 
	(c) holds. Suppose that $V$ comes with a nondegenerate hermitian pairing 
	$Q \colon V \tensor_{\CC} \Vb \to \CC$ such that $Q(Nv,w) = Q(v,Nw)$ for all $v,w
	\in V$. For $A \in \End(V)$, denote by $A^{\dagger} \in \End(V)$ the adjoint with
	respect to $Q$; thus $N^{\dagger} = N$. It is easy to see that $-H_0^{\dagger} \in
	\Sigma(N)$, and so by the above, one has
	\[
		-(H_0)^{\dagger} = e^B H_0 e^{-B}
	\]
	for a unique $B \in \End(V)$ with $[N,B] = 0$ and $B(W_{\bullet}) \subseteq
	W_{\bullet-1}$; by uniqueness, $B^{\dagger} = B$. Consequently, the new splitting
	\[
		H = e^{\half B} H_0 e^{-\half B} \in \Sigma(N)
	\]
	satisfies $H = -H^{\dagger}$, which is just a different way of writing (c). 
	If we suppose in addition that $S \in \End(V)$ is semisimple, commutes with $N$,
	and satisfies $S^{\dagger} = S$, then we can easily arrange that moreover $[H,S] = 0$
	(by considering each eigenspace of $S$
	separately). This shows that splittings with all four properties exist.
\end{proof}

\newpar \label{par:representation}
We now define a representation 
\[
	\rho \colon \sltwo(\CC) \to \End(V), \quad
		\rho(\Hsl) = H, \, \rho(\Ysl) = -N.
\]
The reason for using $-N$ (instead of the seemingly more natural $N$) has to do with
the sign conventions for $\sltwo$-Hodge structures. With this choice, each eigenspace
$E_{\lambda}(T_s)$ is a representation of the Lie algebra $\sltwo(\CC)$, and 
\[
	V = \bigoplus_{\abs{\lambda} = 1} E_{\lambda}(T_s).
\]
Moreover, the eigenspaces $E_k(H)$ and $E_{\ell}(H)$ are orthogonal with respect to $Q$
unless $k = -\ell$.

\subsection{Proof of the Hodge norm estimates}

\newpar
We can now prove the Hodge norm estimates in general.

\begin{pthm} \label{thm:Hodge-norm-estimates}
	Let $E$ be a polarized variation of Hodge structure on $\dst$.  If $v \in V$ is a 
	multi-valued flat section such that $v \in W_k$ and $v \not\in W_{k-1}$, then
	the function
	\[
		\abs{\Re z}^{-k} \cdot \norm{v}_{\Phi(z)}^2
	\]
	is uniformly bounded on every region of the form $\Re z \leq x_0 < 0$ and $\abs{\Im
	z} \leq y_0$.
\end{pthm}

\begin{proof}
	As in \Cref{par:representation}, we choose a representation of $\sltwo(\CC)$ on the
	vector space
	\[
		V = \bigoplus_{\abs{\lambda} = 1} E_{\lambda}(T_s)
	\]
	that commutes with $T_s$, is compatible with the pairing $Q$, and such that $\Ysl \in
	\sltwo(\CC)$ acts as the nilpotent operator $-N$. The representation is completely
	reducible, hence
	\[
		V \cong \bigoplus_{\abs{\lambda} = 1} \bigoplus_{m \in \NN} 
		S_m \tensor_{\CC} \Hom_{\CC} \bigl( S_m, E_{\lambda}(T_s) \bigr)^{\sltwo(\CC)}.
	\]
	Each $S_m$ has a canonical $\sltwo$-Hodge structure of weight $m$, and by putting
	suitable Hodge structures on the vector spaces in the above
	decomposition, $V$
	becomes an $\sltwo$-Hodge structure of weight $n$, polarized by the pairing $Q$.
	By the construction in \Cref{par:model-monodromy}, we can arrange that the
	monodromy transformation of the resulting variation of Hodge structure on $\dst$
	is equal to $T = T_s e^{2 \pi i N}$. In this way, we obtain another polarized
	variation of Hodge structure on $\dst$ whose underlying flat vector bundle is
	isomorphic to $(E,d)$.  \Cref{thm:comparison} shows that the two Hodge
	metrics are mutually bounded up to a constant. Since we already know the Hodge
	norm estimates for variations of Hodge structure coming from $\sltwo$-Hodge
	structures (by \Cref{par:model-VHS}), this finishes the proof of the Hodge norm
	estimates in general.
\end{proof}

\newpar
One can sharpen the Hodge norm estimates by using the decomposition
\[
	V = \bigoplus_{k \in \ZZ} E_k(H)
\]
coming from the $\sltwo(\CC)$-representation. If $v \in E_k(H)$ is a nonzero in
the $k$-th weight space, then the Hodge norm estimates show that
\[
	\norm{v}_{\Phi(z)}^2 \cdot \abs{\Re z}^{-k}
\]
is bounded from above and below by positive constants (as long as $\Im z$ lies in a
bounded interval). This suggests rescaling by the operator $e^{-\half \log \abs{\Re z} \,
H}$, because
\[
	e^{-\half \log \abs{\Re z} \, H} v = \abs{\Re z}^{-\frac{k}{2}} v,
\]
The point of insisting that $H^{\dagger} = -H$ is that $e^{-\half \log\abs{\Re z} \,
H} \in G$. Furthermore, we can remove the restriction on the imaginary part
by choosing a logarithm of $T_s$: a semisimple operator $S \in \End(V)$, with real
eigenvalues in a fixed half-open interval of length $1$, such that $T_s = e^{2 \pi i
\, S}$ and $[S,N] = 0$. The relation
\[
	\norm{v}_{\Phi(z+2 \pi i)}^2 = \norm{T^{-1}v}_{\Phi(z)}^2
\]
means that the rescaled expression
\begin{equation} \label{eq:rescaling}
	\bigl\lVert e^{\half(z-\zb) (S+N)} e^{-\half \log \abs{\Re z} \, H} v
	\bigr\rVert_{\Phi(z)}^2
\end{equation}
now depends only on $t = e^z$, and therefore descends to a smooth function on $\dst$
that is bounded from above and below by a positive constant as $t \to 0$. In
\Cref{chap:rescaled}, we are going to argue that \eqref{eq:rescaling} actually
converges to a positive definite hermitian inner product on $V$ (as a consequence of
the nilpotent orbit theorem).

\subsection{A more direct proof}

\newpar
We end this chapter by sketching another proof, without representation theory, for
the fact that the monodromy weight filtration governs the rate of growth of the Hodge
norm. In this section only, let us denote the monodromy weight filtration of the
nilpotent operator $N \in \End(V)$ by the symbol $M_{\bullet}$. Let us also define
the growth order filtration
\[
	W_k = \menge{v \in V}{%
	\text{$\norm{v}_{\Phi(z)}^2 = O \bigl( \abs{\Re z}^k \bigr)$ as $\abs{\Re z} \to \infty$}},
\]
again assuming that $\Im z$ stays in a bounded interval. We are going to prove
directly that $W_k = M_k$ for every $k \in \ZZ$. 

\newpar
We know from \Cref{prop:polynomial-bound} that $W_k = 0$ for $k \leq -2C_0$, and $W_k = V$
for $k \geq 2 C_0$. This means that the growth order filtration has finite length,
just like the monodromy weight filtration. Another simple observation is that 
\begin{equation} \label{eq:Q-weights}
	Q(W_k, W_{\ell}) = 0 \quad \text{if $k + \ell \leq -1$.}
\end{equation}
Indeed, for $v \in W_k$ and $w \in W_{\ell}$, we have
\[
	\abs{Q(v,w)}^2 \leq \norm{v}_{\Phi(z)}^2 \norm{w}_{\Phi(z)}^2
	= O \bigl( \abs{\Re z}^{k+\ell} \bigr).
\]
As long as $k + \ell \leq -1$, the right-hand side is going to zero as $\abs{\Re z}
\to \infty$; because $Q(v,w)$ is constant, it follows that $Q(v,w) = 0$. 

\newpar
We can use the basic estimate (in \Cref{cor:Higgs-bound}) to show that
$N(W_{\bullet}) \subseteq W_{\bullet-2}$; recall that this is one of the two
conditions that characterize the monodromy weight filtration.

\begin{plem} \label{lem:N-weight}
	We have $N(W_k) \subseteq W_{k-2}$ for all $k \in \ZZ$.
\end{plem}

\begin{proof}
	This is an immediate consequence of \Cref{prop:N-norm}. In fact, we proved that
	there is a constant $C > 0$ such that
	\[
		\norm{Nv}_{\Phi(z)} \leq \frac{C}{\abs{\Re z}} \norm{v}_{\Phi(z)}
		\quad \text{for $v \in V$ and $\Re z \leq -1$,}
	\]
	and so $v \in W_k$ implies that $Nv \in W_{k-2}$.
\end{proof}

\newpar
The monodromy weight filtration has the property that $\ker N^{k+1} \subseteq M_k$.
So far, we only know (from \Cref{prop:kerN}) that $\ker N \subseteq W_0$. The
following lemma generalizes this to arbitrary powers of $N$.

\begin{plem} \label{lem:kerNk}
	We have $\ker N^{k+1} \subseteq W_k$ for $k \geq 0$.
\end{plem}

\begin{proof}
	Together with the triangle inequality, \Cref{prop:kerN} implies that $\ker N
	\subseteq W_0$. Now consider any vector $v \in \ker N^{k+1}$. As in \Cref{par:Sm}, let
	$S_k$ denote the irreducible representation of $\sltwo(\CC)$ with $\dim S_k =
	k+1$. Define a linear mapping $f \colon S_k \to V$ by setting
	\[
		f(v_j) = \frac{j!(k-j)!}{k!} (-N)^j v \quad \text{for $j = 0, 1, \dotsc, k$;}
	\]
	then $f(v_0) = v$ and $f \circ \Ysl = -N \circ f$. According to
	\Cref{subsec:sltwo-VHS}, $S_k$ is the space of multi-valued flat sections of a
	polarized variation of Hodge structure of weight $k$, whose monodromy operator is $T =
	e^{-2 \pi i \Ysl}$; we also know from \Cref{par:model-VHS} that the Hodge norm of
	any vector in $S_k$ can
	grow at most like $\abs{\Re z}^k$. Consequently, $\Hom_{\CC}(S_k, V)$ is the space
	of multi-valued flat sections of a polarized variation of Hodge structure of weight
	$n-k$, and $f \in \Hom_{\CC}(S_k, V)$ satisfies $N f = 0$. By \Cref{prop:kerN},
	the Hodge norm of $f$ is bounded; since the Hodge norm is an upper bound for the
	operator norm, it follows that $\norm{v}_{\Phi(z)}^2 = O \bigl( \abs{\Re z}^k
	\bigr)$.
\end{proof}

\newpar
We recall two facts about the monodromy weight filtration; as we are only sketching
the proof, we omit the details. First,
\[
	M_k = \sum_{j \in \NN} N^j \bigl( \ker N^{k+2j+1} \bigr)
	\quad \text{for $k \in \ZZ$.}
\]
Second, since $N^{\dag} = N$ is self-adjoint with respect to the hermitian pairing
$Q$, one has
\[
	M_k^{\perp} = \menge{v \in V}{\text{$Q(v,x) = 0$ for all $x \in M_k$}} = M_{-k-1}
	\quad \text{for $k \in \ZZ$.}
\]
We can use this to prove that $W_k = M_k$. By \Cref{lem:kerNk}, we have
\[
	M_k = \sum_{j \in \NN} N^j \bigl( \ker N^{k+2j+1} \bigr)
	\subseteq \sum_{j \in \NN} N^j \bigl( W_{k+2j} \bigr)
	\subseteq W_k,
\]
where the last inclusion is due to \Cref{lem:N-weight}. Dually, we have
\[
	W_k \subseteq W_{-k-1}^{\perp} \subseteq M_{-k-1}^{\perp} = M_k,
\]
where we used \eqref{eq:Q-weights} for the first inclusion. It follows that $W_k =
M_k$ for all $k \in \ZZ$, and so the filtration by order of growth of the Hodge
metric is indeed equal to the monodromy weight filtration.

\section{The nilpotent orbit theorem}
\label{sec:nilpotent}

\newpar
This chapter is devoted to the proof of the nilpotent orbit theorem. The theorem has
several parts, and it is going to take us some time to prove all of them. Our first 
goal is to prove that the ``untwisted period mapping''
\[
	\Psi_S \colon \dst \to \Dch, \quad \Psi_S(e^z) = e^{-z(S+N)} \Phi(z),
\]
extends holomorphically to the entire disk $\Delta$. The general idea is to construct
a holomorphic extension of the Higgs field $\theta_{\partial/\partial t}$ with the
help of H\"ormander's $L^2$-estimates, and then to use this extension to prove that
$\Psi_S$ is meromorphic at $t=0$. This is enough because $\Dch$ is a projective complex
manifold.

\subsection{\boldmath H\"ormander's $L^2$-estimates}
\label{sec:L2-estimates} 

\newpar 
We need only the most basic version of H\"ormander's $L^2$-estimates for vector
bundles, in one complex dimension. The technique, in the special case of the
trivial bundle, is explained very nicely in \cite[Lecture~1]{Berndtsson}.
We include an elementary proof for arbitrary vector bundles here, both in order to
keep the paper self-contained, and to show the reader that the $L^2$-estimates in one
dimension are not difficult -- in fact, most of the argument is just integration by
parts. 

\newpar
Let $X \subseteq \CC$ be a domain in $\CC$, with the usual Euclidean metric $\omega =
\frac{i}{2} \dz \wedge \dzb$. Let $E$ be a smooth vector bundle on $X$ with a
hermitian metric $h$, and suppose that $E$ has the structure of a holomorphic vector
bundle; this corresponds to a connection $d'' \colon A^0(X, E) \to A^{0,1}(X, E)$ of
type $(0,1)$; the integrability condition $(d'')^2 = 0$ is automatically satisfied in
dimension one. One technique for constructing holomorphic sections of $E$ is to solve
the $\dbar$-equation
\[
	d'' u = f \dzb
\]
for a given section $f \in A^0(X, E)$. The general idea is that this can be done,
provided $h$ has positive curvature. Let us write the Chern connection
in the form $\delta' + d''$, where $\delta' \colon A^0(X, E) \to A^{1,0}(X, E)$ is a
connection of type $(1,0)$. Let
\[
	\Theta = (\delta' + d'')^2 \in A^{1,1} \bigl( X, \End(E) \bigr)
\]
be the curvature operator of the metric $h$. The main result is as follows
\cite[VIII.\S6]{Demailly}.

\begin{pthm}[H\"ormander] \label{thm:Hoermander}
	Suppose that there is a positive function $\rho \colon X \to \RR$ such
	that 
	\[
		\int_X h \bigl( \Theta_{\partial/\partial z \wedge \partial/\partial \zb} \,
			\alpha, \alpha \bigr) \dmu \geq 
			\int_X \rho^2 h(\alpha, \alpha) \dmu \quad \text{for all $\alpha \in A_c^0(X, E)$.}
	\]
	Let $f \in A^0(X, E)$ be an arbitrary smooth section. Then the equation $d'' u = f
	\dzb$ has a solution $u \in A^0(X, E)$ that moreover satisfies the $L^2$-estimate
	\[
		\int_X h(u,u) \dmu \leq \int_X \frac{1}{\rho^2} h(f,f) \dmu,
	\]
	provided that the integral on the right-hand side is finite.
\end{pthm}

\newpar
The proof uses one nontrivial (but well-known) fact, namely the regularity of the
$\dbar$-equation in one complex dimension. Let $f \colon X \to \CC$ be a smooth
function. Suppose that a measurable function $u \colon X \to \CC$ is a \define{weak
solution} to the equation
\begin{equation} \label{eq:dbar}
	\frac{\partial u}{\partial \zb} = f,
\end{equation}
meaning that for every compactly supported smooth function $\varphi \in
A_c^0(X)$, one has
\[
	\int_X f \varphi \dmu = - \int_X u \frac{\partial \varphi}{\partial \zb} \dmu.
\]
Then after modifying $u$ on a set of measure zero, if necessary, the function $u$ is
smooth and solves \eqref{eq:dbar} in the usual sense.

\newpar
Now we turn to the proof of \Cref{thm:Hoermander}. We start by using
integration by parts to relate the two operators $d''_{\partial/\partial \zb}$ and
$\delta'_{\partial/\partial z}$. 

\begin{plem} \label{lem:by-parts}
	Let $u \in A^0(X, E)$ be a smooth section, and let $\alpha \in A_c^0(X, E)$ be a smooth
	section with compact support. Then
	\[
		\int_X h \bigl( d''_{\partial/\partial \zb} u, \alpha \bigr) \dmu
		= -\int_X h \bigl( u, \delta'_{\partial/\partial z} \alpha \bigr) \dmu.
	\]
\end{plem}

\begin{proof}
	Since $\delta' + d''$ is a metric connection, we have
	\[
		\frac{\partial}{\partial \zb} h(u, \alpha) 
		= h \bigl( d''_{\partial/\partial \zb} u, \alpha \bigr)
		+ h \bigl( u, \delta'_{\partial/\partial z} \alpha \bigr).
	\]
	Now integrate over $X$ and use Stokes' theorem to get the result.	
\end{proof}

\newpar
The general idea, which should be familiar from the proof of Hodge's theorem about
harmonic forms, is to construct a ``weak solution'' to the equation by Hilbert space
techniques. We let $L^2(X, E)$ be the Hilbert space of all measurable
sections of the bundle $E$ with finite $L^2$-norm 
\[
	\norm{u}_h^2 = \int_X h(u,u) \dmu.
\]
We say that $u \in L^2(X, E)$ is a \define{weak solution} of the equation $d'' u = f
\dzb$ if
\[
	\int_X h(f, \alpha) \dmu 
	= - \int_X h(u, \delta'_{\partial/\partial z} \alpha) \dmu
\]
for every smooth section $\alpha \in A_c^0(X, E)$ with compact support. The
regularity theory of the $\dbar$-operator implies that weak solutions are
actually smooth.

\begin{pprop} \label{prop:weak-solution}
	Let $f \in A^0(X, E)$. If $u \in L^2(X, E)$ is a weak solution to the equation
	$d'' u = f \dzb$, then $u \in A^0(X, E)$ and $d'' u = f \dzb$ in the usual sense.
\end{pprop}

\begin{proof}
	This is a local problem; after replacing $X$ by an open neighborhood of a
	given point, we may assume that $E$ is a trivial holomorphic bundle of rank $r$. Let
	$s_1, \dotsc, s_r \in A^0(X, E)$ be a holomorphic frame, and define the smooth
	functions
	\[
		h_{i,j} = h(s_i, s_j).
	\]
	The $r \times r$-matrix with entries $h_{i,j}$ is hermitian and positive definite;
	let $h^{i,j}$ be the entries of the inverse matrix. A simple calculation shows
	that
	\[
		\delta'_{\partial/\partial z} s_i 
		= \sum_{j,k} \frac{\partial h_{i,j}}{\partial z} h^{j,k} s_k
	\]
	Now let $\alpha = \sum_i \alpha_i s_i \in A_c^0(X, E)$ be an arbitrary smooth
	section of $E$ with compact support. Since $\delta'$ is a connection of type
	$(1,0)$, we have
	\begin{align*}
		\delta'_{\partial/\partial z} \alpha &=
		\sum_i \frac{\partial \alpha_i}{\partial z} s_i + \sum_i \alpha_i 
		\delta'_{\partial/\partial z} s_i 
		= \sum_{i,j,k} \Bigl( \frac{\partial \alpha_i}{\partial z} h_{i,j} h^{j,k} + \alpha_i
		\frac{\partial h_{i,j}}{\partial z} h^{j,k} \Bigr) s_k \\
		&= \sum_{i,j,k} \frac{\partial (\alpha_i h_{i,j})}{\partial z} h^{j,k} s_k.
	\end{align*}
	Write the given section $u \in L^2(X, E)$ as $u = \sum_i u_i s_i$, with
	measurable functions $u_i \colon X \to \CC$. Since $h$ is
	conjugate-linear in the second argument, we get
	\[
		h(u, \delta'_{\partial/\partial z} \alpha) = \sum_{i,j,k, \ell} u_{\ell}
		h_{\ell,k} \frac{\partial (\wbar{\alpha}_i h_{j,i})}{\partial \zb} h^{k,j}
		= \sum_{i,j} u_j \frac{\partial(\wbar{\alpha}_i h_{j,i})}{\partial \zb}.
	\]
	Because $u$ is a weak solution to $d'' u = f \dzb$, we therefore have
	\begin{align*}
		\int_X \sum_{i,j} h_{j,i} f_j \wbar{\alpha}_i \dmu 
		&= \int_X h(f, \alpha) \dmu \\
		&= -\int_X h \bigl( u, \delta'_{\partial/\partial z} \alpha \bigr) \dmu 
		= - \int_X \sum_{i,j} u_j \frac{\partial(\wbar{\alpha}_i h_{j,i})}{\partial \zb} \dmu,
	\end{align*}
	where $f = \sum_j f_j s_j$. Now let $\varphi_1, \dotsc, \varphi_r \colon X \to \CC$ be
	arbitrary compactly supported smooth functions, and set
	\[
		\alpha = \sum_i \alpha_i s_i = \sum_{i,j} \wbar{\varphi}_j h^{j,i} s_i;
	\]
	with this choice, we have $\varphi_j = \sum_i h_{j,i} \wbar{\alpha}_i$, and
	so the identity from above becomes
	\[
		\int_X \sum_j f_j \varphi_j \dmu
		= - \int_X \sum_j u_j \frac{\partial \varphi_j}{\partial \zb} \dmu.
	\]
	This shows that each coefficient function $u_j \colon X \to \CC$ by itself is a
	weak solution of the ordinary $\dbar$-equation $\partial u_j / \partial \zb =
	f_j$. By standard regularity theory, we can modify each $u_j$ on a set of measure
	zero and make it smooth; then $u \in A^0(X, E)$. We can now use integration by
	parts, as in \Cref{lem:by-parts}, and deduce that
	\[
		\int_X h(f, \alpha) \dmu 
		= - \int_X h \bigl( u, \delta'_{\partial/\partial z} \alpha \bigr) \dmu
		= \int_X h \bigl( d''_{\partial/\partial \zb} u, \alpha \bigr) \dmu 
	\]
	for every $\alpha \in A_c^0(X, E)$. But this means exactly that $d'' u = f \dzb$.
\end{proof}

\newpar
To find the desired weak solution, we use the Riesz representation theorem. The image
of the linear mapping
\begin{equation} \label{eq:linear-mapping-delta}
	\delta'_{\partial/\partial z} \colon A_c^0(X, E) \to L^2(X, E)
\end{equation}
is a linear subspace of $L^2(X, E)$, in general not closed. On this subspace, we can
define a conjugate-linear functional by the formula
\[
	\delta'_{\partial/\partial z} \alpha \mapsto - \int_X h(f, \alpha) \dmu.
\]
The following lemma shows that this is well-defined and bounded.

\begin{plem} \label{lem:functional}
	For every $\alpha \in A_c^0(X, E)$, we have
	\[
		\left\lvert \int_X h(f, \alpha) \dmu \right\rvert \leq
		\left( \int_X \frac{1}{\rho^2} h(f,f) \dmu \right)^{1/2} 
		\left( \int_X h \bigl( \delta'_{\partial/\partial z} \alpha,
		\delta'_{\partial/\partial z} \alpha \bigr) \dmu \right)^{1/2},
	\]
	where $\rho \colon X \to \RR$ is the positive function from \Cref{thm:Hoermander}.
\end{plem}

\begin{proof}
	The Cauchy-Schwarz inequality gives
	\[
		\left\lvert \int_X h(f, \alpha) \dmu \right\rvert \leq
		\left( \int_X \frac{1}{\rho^2} h(f,f) \dmu \right)^{1/2} 
		\left( \int_X \rho^2 h(\alpha, \alpha) \dmu \right)^{1/2},
	\]
	and in view of the inequality in \Cref{thm:Hoermander}, it is therefore
	enough to prove that
	\[
		\int_X h \bigl( \Theta_{\partial/\partial z \wedge \partial/\partial \zb}
			\, \alpha, \alpha \bigr) \dmu  \leq
		\int_X h \bigl( \delta'_{\partial/\partial z} \alpha,
		\delta'_{\partial/\partial z} \alpha \bigr) \dmu.
	\]
	The curvature operator is $\Theta = \delta' d'' + d'' \delta'$, and so 
	\[
		h \bigl( \Theta_{\partial/\partial z \wedge \partial/\partial \zb}
			\, \alpha, \alpha \bigr)
			= h \bigl( \delta'_{\partial/\partial z} 
				d''_{\partial/\partial \zb} \alpha, \alpha \bigr)
			- h \bigl( d''_{\partial/\partial \zb} 
				\delta'_{\partial/\partial z} \alpha, \alpha \bigr).
	\]
	If we integrate over $X$ and apply \Cref{lem:by-parts}, we get
	\[
		\int_X h \bigl( \Theta_{\partial/\partial z \wedge \partial/\partial \zb}
			\, \alpha, \alpha \bigr) \dmu  
		= \int_X h \bigl( \delta'_{\partial/\partial z} \alpha, 
				\delta'_{\partial/\partial z} \alpha \bigr) \dmu
			- \int_X h \bigl( d''_{\partial/\partial \zb} \alpha, 
				d''_{\partial/\partial \zb} \alpha \bigr) \dmu.
	\]
	Since the second term is negative, this gives us the inequality we need.
\end{proof}

\newpar
The inequality in \Cref{lem:functional} says that the conjugate-linear functional
\[
	\delta'_{\partial/\partial z} \alpha \mapsto -\int_X h(f, \alpha) \dmu
\]
is well-defined, bounded, and of norm at most $\norm{f/\rho}_h$. By continuity, it
extends uniquely to a conjugate-linear functional 
\[
	L^2(X, E) \to \CC
\]
that is identically zero on the orthogonal complement of the image of
\eqref{eq:linear-mapping-delta}; its norm is of course still at most
$\norm{f/\rho}_h$. By the Riesz representation theorem, this conjugate-linear
functional is in turn represented by a unique element $u \in L^2(X, E)$. By
construction, we have
\[
	\int_X h \bigl( u, \delta'_{\partial/\partial z} \alpha \bigr) \dmu
	= -\int_X h(f, \alpha) \dmu
\]
for every $\alpha \in A_c^0(X, E)$, and so $u$ is a weak solution to the equation
$d'' u = f \dzb$. This weak solution also satisfies the $L^2$-estimate because
\[
	\int_X h(u, u) \dmu = \norm{u}_h^2 \leq \norm{f/\rho}_h^2 
	= \int_X \frac{1}{\rho^2} h(f,f) \dmu.
\]
Now an application of \Cref{prop:weak-solution} finishes the proof of
\Cref{thm:Hoermander}.

\newpar
	If $v \in L^2(X, E)$ is orthogonal to the image of
	\eqref{eq:linear-mapping-delta}, then 
	\[
		\int_X h(d_{\partial/\partial \zb}'' v, \alpha) \dmu
		= -\int_X h(v, \delta_{\partial/\partial z}' \alpha) \dmu = 0
	\]
	for every $\alpha \in A_c^0(X, E)$, and so $v \in A^0(X, E)$ is smooth and
	satisfies $d'' v = 0$, hence is holomorphic. This means that the solution $u \in
	A^0(X, E)$ constructed above is orthogonal to the space of holomorphic sections of
	$E$, and so it is the (unique) solution of the equation $d'' u = f \dzb$ with
	minimal $L^2$-norm.

\subsection{Hodge bundles and metrics with positive curvature}

\newpar
Now let $E$ be a polarized variation of Hodge structure on $\dst$. In order to apply
H\"ormander's theory to the Hodge bundles, we need a metric with positive curvature.
According to \Cref{prop:curvature}, the curvature operator of the Hodge
metric $h$ on the holomorphic vector bundle $\shE^{p,q}$ is
\[
	\Theta = -(\theta \thetast + \thetast \theta).
\]
For any smooth section $u \in A^0(\dst, E^{p,q})$, we therefore have
\[
	h \bigl( \Theta_{\partial/\partial t \wedge \partial/\partial \tb} \, u, u \bigr)
	= h \bigl( \theta_{\partial/\partial t} u, \theta_{\partial/\partial t} u \bigr) 
	- h \bigl( \thetast_{\partial/\partial \tb} u, 
		\thetast_{\partial/\partial \tb} u \bigr).
\]
The right-hand side is unfortunately not positive, but at least we know that the
troublesome second term can be no bigger than
\[
	h \bigl( \thetast_{\partial/\partial \tb} u, 
		\thetast_{\partial/\partial \tb} u \bigr)
		\leq h \bigl( \thetast_{\partial/\partial \tb}, \thetast_{\partial/\partial
		\tb} \bigr) h(u, u)
		\leq \frac{1}{4} \binom{r+1}{3} \frac{1}{\abs{t}^2 (\log \abs{t})^2} h(u,u),
\]
using the universal bound for the Higgs field in \Cref{cor:Higgs-bound}. (Here $r = \rk
E$.)

\newpar \label{par:hphi}
We can try to fix this problem by multiplying the Hodge metric by a factor of the
form $e^{-\varphi}$, where $\varphi \colon \dst \to \RR$ is a suitable weight function.
Let us denote the new hermitian metric temporarily by the symbol $\hphi = h
e^{-\varphi}$. Since
\begin{align*}
	\dbar \hphi(v,w)
	&= \dbar h(v,w) \cdot e^{-\varphi} + h(v,w) e^{-\varphi} (-\dbar \varphi)  \\
	&= \hphi(d'' v, w) + \hphi(u, \delta' v) - \hphi(u, \partial \varphi),
\end{align*}
the new Chern connection is $\delta' + d'' - \partial \varphi$, and
so the new curvature operator is
\[
	\Thetaphi = (\delta' + d'' - \partial \varphi)^2 = \Theta + \partial \dbar \varphi.
\]
Our inequality from above now changes into
\[
	\hphi \bigl( \Thetaphi_{\partial/\partial t \wedge \partial/\partial \tb} \, u, u \bigr)
	= \hphi \bigl( \theta_{\partial/\partial t} u, \theta_{\partial/\partial t} u \bigr) 
	- \hphi \bigl( \thetast_{\partial/\partial \tb} u, 
		\thetast_{\partial/\partial \tb} u \bigr)
		+ \frac{\partial^2 \varphi}{\partial t \partial \tb} \cdot \hphi(u, u).
\]
The right-hand side is greater or equal to
\[
	\hphi(u,u) \cdot 
	\left( \frac{\partial^2 \varphi}{\partial t \partial \tb} 
	- \frac{1}{4} \binom{r+1}{3} \frac{1}{\abs{t}^2 (\log \abs{t})^2} \right),
\]
and as long as the term in parentheses is positive, the metric $\hphi$ has positive
curvature. The following lemma gives us a family of suitable weight functions.

\begin{plem} \label{lem:weight-functions}
	For $a \in \RR$ and $b \in \ZZ$, the function $e^{-\varphi} = \abs{t}^a
	(-\log \abs{t})^b$ satisfies
	\[
		\frac{\partial^2 \varphi}{\partial t \partial \tb} 
		= \frac{b}{4 \abs{t}^2 (\log \abs{t})^2}
	\]
\end{plem}

\begin{proof}
	A small calculation, which can be made easier by pulling back to $\HH$.
\end{proof}

\newpar
This gives us a nice family of hermitian metrics on the Hodge bundles to which we can
apply H\"ormander's $L^2$-estimates.

\begin{pcor} \label{cor:Hoermander}
	Let $a \in \RR$ and $b \in \NN$, subject to the condition $b \geq
	\binom{r+1}{3}+1$, where $r = \rk E$. Let $f \in A^0(\dst, E^{p,q})$ be a smooth
	section with the property that
	\[
		\int_{\dst} h(f,f) \, \abs{t}^{a+2} (-\log \abs{t})^{b+2} \dmu < +\infty.
	\]
	Then there is a smooth section $u \in A^0(\dst, E^{p,q})$ that solves the
	$\dbar$-equation $\dbar u = f \dtb$ and also satisfies the $L^2$-estimate
	\[
		\int_{\dst} h(u,u) \, \abs{t}^a (-\log \abs{t})^b \dmu
		\leq 4 \int_{\dst} h(f,f) \, \abs{t}^{a+2} (-\log \abs{t})^{b+2} \dmu.
	\]
\end{pcor}

\begin{proof}
	We apply \Cref{thm:Hoermander} to the Hodge bundle $E^{p,q}$,
	with the holomorphic structure defined by the operator $\dbar$ and with the
	hermitian metric 
	\[
		\hphi = h \cdot \abs{t}^a (-\log \abs{t})^b
	\]
	For $b \geq \binom{r+1}{3} + 1$, this metric has positive curvature; in fact, we
	now have
	\begin{equation} \label{eq:curvature-inequality}
		\hphi \Bigl( \Thetaphi_{\partial/\partial t \wedge \partial/\partial \tb} \, u,
		u \Bigr) \geq \hphi(u, u) \cdot \frac{1}{4 \abs{t}^2 (\log \abs{t})^2}
	\end{equation}
	for every smooth section $u \in A^0(\dst, E^{p,q})$. The assumptions of the theorem
	are therefore met with $1/\rho^2 = 4 \abs{t}^2 (\log \abs{t})^2$.
\end{proof}

\newpar
In what follows, we are actually going to apply the $L^2$-estimates to the induced
variation of Hodge structure on the bundle $\End(E)$. Because of the better bound in
\Cref{cor:Higgs-bound-End}, the result takes the following form.

\begin{pcor} \label{cor:Hoermander-End}
	Let $a \in \RR$ and $b \in \NN$, subject to the condition $b \geq 2r
	\binom{r+1}{3}+1$. Let $f \in A^0 \bigl( \dst, \End(E)^{p,q} \bigr)$ be a smooth
	section with the property that
	\[
		\int_{\dst} h_{\End(E)}(f,f) \, \abs{t}^{a+2} (-\log \abs{t})^{b+2} \dmu < +\infty.
	\]
	Then there is a smooth section $u \in A^0 \bigl( \dst, \End(E)^{p,q} \bigr)$ that
	solves the equation $\dbar u = f \dtb$ and also satisfies the $L^2$-estimate
	\[
		\int_{\dst} h_{\End(E)}(u,u) \, \abs{t}^a (-\log \abs{t})^b \dmu
		\leq 4 \int_{\dst} h_{\End(E)}(f,f) \, \abs{t}^{a+2} (-\log \abs{t})^{b+2} \dmu.
	\]
\end{pcor}

\subsection{Constructing a holomorphic lifting of the Higgs field}

\newpar
After these preliminaries, we now turn to the actual proof of the nilpotent orbit
theorem. Fix a half-open interval $I \subseteq \RR$ of length $1$, and let $S \in
\End(V)$ be the unique semisimple operator such that $T_s = e^{2 \pi i S}$ and such
that the eigenvalues of $S$ are real and contained in $I$. If we denote by
$P_{\lambda} \in \End(V)$ the projection to the eigenspace $E_{\lambda}(T_s)$, we can
write $S$ compactly as
\[
	S = \sum_{\alpha \in I} \alpha P_{e^{2 \pi i \alpha}}.
\]
The differential of the holomorphic mapping $z \mapsto \Psi_S(e^z) = e^{-z(S+N)}
\Phi(z)$ at the point $z \in \HH$ is equal to
\[
	e^{-z(S+N)} \theta_{\partial/\partial z} e^{z(S+N)} - (S+N)
	\mod F^0 \End(V)_{\Psi_S(e^z)}.
\]
The differential is holomorphic, modulo the indicated subspace, but the operator on
the left-hand side is \emph{not} holomorphic as a section of the bundle $\End(V)$.
The idea of the proof is to lift this operator to a holomorphic section of the bundle
$\End(V)$ with the help of H\"ormander's $L^2$-estimates, and then to use the
properties of this lifting to show that $\Psi_S$ extends.

\newpar
Let us begin by looking at the Higgs field itself. By construction, the operator
$\theta_{\partial/\partial z} \colon \HH \to \End(V)$ is the pullback of $t
\theta_{\partial/\partial t}$, and as such, it satisfies
\[
	\theta_{\partial/\partial z}(z + 2 \pi i) 
	= T \, \theta_{\partial/\partial z}(z) \, T^{-1}.
\]
Recall that we have an induced variation of Hodge structure of
weight $0$ on the bundle $\End(E)$. The Higgs field $t \theta_{\partial/\partial t}$
is a smooth section of the Hodge bundle $F^{-1} \End(E)$ on $\dst$. Now the
Hodge bundle is a holomorphic vector bundle, whose holomorphic structure is defined
by the operator $d''_{\End(E)} = [d'', \argbl]$, but the Higgs field is not a
holomorphic section.
Instead, because of the identity $\dbar \theta + \theta \dbar = 0$, only its projection
to the quotient bundle
\[
	\End(E)^{-1,1} \cong F^{-1} \End(E) / F^0 \End(E)
\]
is holomorphic, for the holomorphic structure defined by $\dbar_{\End(E)} = [\dbar, \argbl]$. 

\newpar
The $L^2$-estimates allow us to lift $t \theta_{\partial/\partial t}$ to a
holomorphic section of the bundle $F^{-1} \End(E)$ in a controlled manner. We will
make a specific choice of the two parameters $a \in \RR$ and $b \in \ZZ$ a little later.

\begin{pprop} \label{prop:lifting}
	Set $r = \rk E$. For every $a > -2$ and every $b \geq 2r \binom{r+1}{3}+1$, there
	is a holomorphic section $\vartheta$ of the Hodge bundle $F^{-1} \End(E)$ such that
	\[
		\vartheta \equiv t \theta_{\partial/\partial t} \mod F^0 \End(E),
	\]
	and which also satisfies the $L^2$-estimate
	\[
		\int_{\dst} h_{\End(E)}(\vartheta, \vartheta) \, \abs{t}^a (-\log \abs{t})^b \dmu
		\leq C,
	\]
	with a constant $C > 0$ whose exact value only depends on $r$, $a$, and $b$.
\end{pprop}

\begin{proof}
	Since the Higgs field is trivial when $r = 1$, we may assume without loss of
	generality that $r \geq 2$. Using the identities in \Cref{lem:VHS-identities},
	we compute that
	\[
		d''_{\End(E)} \bigl( t \theta_{\partial/\partial t} \bigr)
		= \dbar_{\End(E)} \bigl( t \theta_{\partial/\partial t} \bigr) 
		+ \thetast_{\End(E)} \bigl( t \theta_{\partial/\partial t} \bigr)
		= t \lbrack \thetast_{\partial/\partial \tb}, \theta_{\partial/\partial t} \rbrack
		\dtb.
	\]
	Temporarily define $f_0 = t \lbrack \thetast_{\partial/\partial \tb},
	\theta_{\partial/\partial t} \rbrack$; this is a smooth section of the Hodge bundle
	$\End(E)^{0,0}$. The universal bound on the Higgs field in
	\Cref{cor:Higgs-bound} gives
	\[
		h(f_0,f_0) \leq 2 \abs{t}^2 \cdot h \bigl( \theta_{\partial/\partial t},
		\theta_{\partial/\partial t} \bigr)^2
		\leq \frac{2 C_0^4}{\abs{t}^2 (\log \abs{t})^4}.
	\]
	Provided that $a > -2$ and $b \geq 2$, the integral
	\[
		\int_{\dst} h(f_0,f_0) \, \abs{t}^{a+2} (-\log \abs{t})^{b+2} \dmu
		\leq 2C_0^4 \int_{\dst} \abs{t}^a (\log \abs{t})^{b-2} \dmu
	\]
	is finite. Once $b \geq 2r\binom{r+1}{3} + 1$, the version of the $L^2$-estimates
	in \Cref{cor:Hoermander} gives us a solution $u_0 \in A^0 \bigl( \dst,
	\End(E)^{0,0} \bigr)$ to the equation $\dbar_{\End(E)} u_0 + f_0 \dtb =
	0$ that satisfies the $L^2$-estimate
	\[
		\int_{\dst} h(u_0,u_0) \, \abs{t}^a (-\log \abs{t})^b \dmu
		\leq 8C_0^4 \int_{\dst} \abs{t}^a (-\log \abs{t})^{b-2} \dmu
	\]
	Continuing in this way, we produce smooth sections $u_k, f_k \in
	A^0 \bigl( \dst, \End(E)^{k,-k} \bigr)$, indexed by $k \in \NN$, such that
	\[
		f_{k+1} = [\thetast_{\partial/\partial \tb}, u_k] \quad \text{and} \quad
		\dbar_{\End(E)} u_{k+1} + f_{k+1} \dtb = 0,
	\]
	subject to the $L^2$-estimate
	\begin{align*}
		\int_{\dst} h(u_{k+1},u_{k+1}) \, \abs{t}^a (-\log \abs{t})^b \dmu
		&\leq  4 \int_{\dst} h(f_{k+1},f_{k+1}) \, \abs{t}^{a+2} (-\log \abs{t})^{b+2} \dmu \\
		&\leq 8C_0^2 \int_{\dst} h(u_k, u_k) \abs{t}^a (-\log \abs{t})^b \dmu
	\end{align*}
	Let $p \in \NN$ be the biggest integer such that $\End(E)^{p,-p} \neq 0$. By
	construction, 
	\[
		\vartheta = t \theta_{\partial/\partial t} + \sum_{k=0}^p u_k 
		\in A^0 \bigl( \dst, F^{-1} \End(E) \bigr)
	\]
	satisfies $d_{\End(E)}''(\vartheta) = 0$, hence is a holomorphic section
	of the Hodge bundle $F^{-1} \End(E)$, with the same image in $\End(E)^{-1,1} \cong
	F^{-1} \End(E) / F^0 \End(E)$ as the Higgs field $t \theta_{\partial/\partial t}$.
	We also get another valuable piece of information about $\vartheta$ from the
	$L^2$-estimates for the individual solutions $u_k$. On the one hand, 
	\[
		\int_{\dst} h(u_k,u_k) \, \abs{t}^a (-\log \abs{t})^b \dmu
		\leq (8 C_0^2)^{k+1} \cdot C_0^2 \int_{\dst} \abs{t}^a (-\log \abs{t})^{b-2} \dmu
	\]
	by combining the inequalities above. On the other hand, 
	\[
		h \bigl( t \theta_{\partial/\partial t}, t \theta_{\partial/\partial t} \bigr)
		\leq \frac{C_0^2}{(\log \abs{t})^2},
	\]
	and so we have a similar inequality
	\[
		\int_{\dst} h \bigl( t \theta_{\partial/\partial t}, t \theta_{\partial/\partial
		t} \bigr) \, \abs{t}^a (-\log \abs{t})^b \dmu
		\leq C_0^2 \int_{\dst} \abs{t}^a (-\log \abs{t})^{b-2} \dmu.
	\]
	If we put everything together and remember that the Hodge decomposition on $\End(E)$ is
	orthogonal with respect to the Hodge metric, the result is that
	\begin{align*}
		\int_{\dst} h(\vartheta, \vartheta) \, \abs{t}^a (-\log \abs{t})^b \dmu
		&\leq \sum_{k=0}^{p+1} (8 C_0^2)^k \cdot 
		C_0^2 \int_{\dst} \abs{t}^a (-\log \abs{t})^{b-2} \dmu \\
		&\leq C_0^2 \sum_{k=0}^{p+1} (8 C_0^2)^k \cdot \frac{2 \pi (b-2)!}{(a+2)^{b-1}},
	\end{align*}
	after evaluating the integral.
\end{proof}

\newpar
Recall that the pullback of the flat bundle $(E,d)$ along $\exp \colon \HH \to \dst$
has a canonical trivialization by the space $V$ of multi-valued flat sections. Let us
use the same letter
\[
	\vartheta \colon \HH \to \End(V)
\]
also for the pullback of $\vartheta \in A^0 \bigl( \dst, F^{-1} \End(E) \bigr)$ to
the half-plane $\HH$; this is now a holomorphic mapping with the property that
\[
	\vartheta(z) \in F^{-1} \End(V)_{\Phi(z)} \quad \text{and} \quad
	\vartheta(z) - \theta_{\partial/\partial z} \in F^0 \End(V)_{\Phi(z)}
\]
for every $z \in \HH$. As a pullback, it transforms according to the rule
\[
	\vartheta(z + 2 \pi i) = T \vartheta(z) T^{-1}.
\]
Recall that $T = e^{2 \pi i(S + N)}$, where $S \in \End(V)$ is semisimple with real
eigenvalues contained in a half-open interval of length $1$. The expression $e^{-z(S+N)}
\vartheta(z) e^{z(S+N)}$ is again invariant under the substitution $z \mapsto z + 2 \pi i$,
and therefore
\[
	e^{-z(S+N)} \, \vartheta(z) \, e^{z(S+N)} = B(e^z)
\]
for a well-defined holomorphic mapping $B \colon \dst \to \End(V)$.

\newpar
We can use our integral estimate for $\vartheta$ to prove that $B$ extends
holomorphically across the origin. To do that, we first need to know how much
conjugation by the operator $e^{-z(S+N)}$ affects the norm of an endomorphism. Define
\[
	\delta(T) = \min \menge{\delta > 0}{\text{$T$ has two eigenvalues of the form
			$\lambda$ and $\lambda \cdot e^{2 \pi i \delta}$}};
\]
concretely, this is the minimal distance between consecutive eigenvalues of $T$ (on
the unit circle), divided by $2\pi$. In the extreme case where all eigenvalues of
$T$ are equal to each other, we have $\delta(T) = 1$. Also define
\[
	m(N) = \min \menge{m \in \NN}{N^m \neq 0}
\]
as the index of nilpotency of the operator $N = \log T_u$.

\begin{plem} \label{lem:conjugation}
	There is a constant $C > 0$ such that
	\begin{align*}
		\bigl\lVert e^{-z(S+N)} &A e^{z(S+N)} \bigr\rVert_{\Phi(-1)} \\
										&\leq C e^{(1 - \delta(T)) \abs{\Re z}} 
		\max \Bigl( \abs{\Re z}^{m(N) + 2 \sqrt{2r} C_0}, \abs{\Re z}^{-2 \sqrt{2r} C_0} \Bigr)
		\norm{A}_{\Phi(z)} 
	\end{align*}
	for every $A \in \End(V)$, uniformly on the horizontal strip
	\[
		\menge{z \in \HH}{\abs{\Im z} \leq \pi}.
	\]
	Here $C_0 = \frac{1}{2}\sqrt{\binom{r+1}{3}}$, and the exact value of the constant
	$C$ only depends on $r = \rk E$ and on the minimal polynomial of $T \in \GL(V)$.
\end{plem}

\begin{proof}
	In order to simplify the notation, we denote by $\norm{v} = \norm{v}_{\Phi(-1)}$
	the Hodge norm at the point $-1 \in \HH$. Let us first consider the dominant
	term $e^{-zS}$. Let $P_{\lambda} \colon V \to E_{\lambda}(T_s)$ denote the
	projection to the $\lambda$-eigenspace of $T_s$. Then
	\[
		S = \sum_{\alpha \in I} \alpha P_{e^{2 \pi i \alpha}}.
	\]
	Let $A \in \End(V)$ be an arbitrary endomorphism. We have
	\[
		e^{-zS} A e^{zS}
		= \sum_{\alpha, \beta} e^{(\alpha - \beta) \abs{\Re z}} e^{-i(\alpha - \beta)
		\Im z} P_{e^{2 \pi i \alpha}} A P_{e^{2 \pi i \beta}}
	\]
	and so the triangle inequality and submultiplicativity of the $L^2$-norm give
	\begin{equation} \label{eq:AdS-operator-norm}
		\bigl\lVert e^{-zS} A e^{zS} \bigr\rVert
		\leq e^{(1 - \delta) \abs{\Re z}} 
		\biggl( \sum_{\lambda} \norm{P_{\lambda}} \biggr)^2 \norm{A},
	\end{equation}
	where $\delta = \delta(T)$, due to the fact that $\alphamax - \alphamin \leq 1 -
	\delta$. The term $e^{-zN}$ is easier to analyze because $N$ is nilpotent. Set $m
	= m(N)$. Then
	\[
		e^{-zN} A e^{zN} = \sum_{i,j=0}^m \frac{(-1)^i z^{i+j}}{i! \, j!} N^i A N^j
	\]
	and therefore
	\[
		\bigl\lVert e^{-zN} A e^{zN} \bigr\rVert
		\leq \sum_{i, j=0}^m \frac{\abs{z}^{i+i}}{i! \, j!} \norm{N}^{i+j}
		\norm{A} = \sum_{k=0}^m \frac{\abs{2z}^k}{k!} \norm{N}^k \norm{A}.
	\]
	Now we consider $A \in \End(V)$ as a multi-valued flat section of the induced
	variation of Hodge structure on $\End(E)$. Remembering the improved bound for the
	Higgs field in \Cref{cor:Higgs-bound-End}, we get from
	\Cref{prop:polynomial-bound} the inequality
	\[
		\norm{A} \leq e^{2 \sqrt{2r} C_0 \pi} 
		\max \Bigl( \abs{\Re z}^{2 \sqrt{2r} C_0}, \abs{\Re z}^{-2 \sqrt{2r} C_0} \Bigr)
		\norm{A}_{\Phi(z)} 
	\]
	Putting everything together, we arrive at
	\[
		\bigl\lVert e^{-z(S+N)} A e^{z(S+N)} \bigr\rVert
		\leq C e^{(1 - \delta) \abs{\Re z}} 
		\max \Bigl( \abs{\Re z}^{m+2 \sqrt{2r} C_0}, \abs{\Re z}^{-2 \sqrt{2r} C_0} \Bigr)
		\norm{A}_{\Phi(z)} 
	\]
	with a constant $C > 0$ that only depends on the two integers $r = \rk E$
	and $m = m(N)$, and on the Hodge norms $\norm{P_{\lambda}}_{\Phi(-1)}$ and
	$\norm{N}_{\Phi(-1)}$. The latter quantities are in turn bounded, in
	\Cref{prop:N-norm}, by a constant whose value only depends on
	$r$ and on the minimal polynomial of $T \in \GL(V)$. This gives the desired
	result.
\end{proof}

\newpar
We are now ready to show that $B \colon \dst \to \End(V)$ extends holomorphically
across the origin. We choose
\[
	a = -2 + 2 \delta(T) > -2 \quad \text{and} \quad b = 2r \binom{r+1}{3} + 1,
\]
and then construct $\vartheta$ and $B$ according to the procedure described above.
Since $b = 8rC_0^2 + 1 \geq m(N) + 2 \sqrt{2r} C_0$, \Cref{lem:conjugation} gives
\[
	\bigl\lVert B(e^z) \bigr\rVert_{\Phi(-1)}^2
	\leq C^2 e^{-a \abs{\Re z}} 
	\max \Bigl( \abs{\Re z}^b, \abs{\Re z}^{-b} \Bigr)
	\norm{\vartheta}_{\Phi(z)}^2;
\]
the restriction on $\abs{\Im z}$ becomes irrelevant here, because $B(e^z)$ is
of course invariant under the substitution $z \mapsto z + 2 \pi i$.  Now $\abs{t} =
e^{-\abs{\Re z}}$ and $-\log \abs{t} = \abs{\Re z}$, and so we can rewrite
this inequality in the form
\[
	\min \Bigl( 1, (\log \abs{t})^{2b} \Bigr) \cdot \norm{B(t)}_{\Phi(-1)}^2 
	\leq C^2 h(\vartheta, \vartheta) \abs{t}^a (-\log \abs{t})^b.
\]
The $L^2$-estimate for $\vartheta$ in \Cref{prop:lifting} now gives us
\begin{equation} \label{eq:B-integral-bound}
	\int_{\dst} \min \Bigl( 1, (\log \abs{t})^{2b} \Bigr) 
		\cdot \norm{B(t)}_{\Phi(-1)}^2 \dmu \leq C',
\end{equation}
with a constant $C' > 0$ whose value has the same dependence on parameters as in
\Cref{lem:conjugation}. Since $B \colon \dst
\to \End(V)$ is holomorphic, this is enough to conclude that $B$
extends holomorphically across the origin.

\newpar
As an aside, we note that the bound on the integral also gives us some control over
the pointwise norm of $B(t)$. To simplify the calculations, we use the inequality
\[
	\min \Bigl( 1, (\log \abs{t})^{2b} \Bigr) \geq (1-\abs{t})^{2b},
\]
which follows from the fact that $\abs{\log x} \geq 1-x$ for $0 < x \leq
1$. The following result is a simple modification of the formula
for the Bergman kernel on the unit disk.

\begin{plem} \label{lem:Bergman}
	For any holomorphic function $f \colon \Delta \to \CC$, one has 
	\[
		\abs{f(t)}^2 \leq \frac{2^b (b+1)}{(1-\abs{t}^2)^{b+2}} \cdot 
		\frac{1}{\pi} \int_{\Delta} (1 - \abs{t})^b \abs{f}^2 \dmu.
	\]
\end{plem}

\begin{proof}
	Consider the power series expansion 
	\[
		f = \sum_{n=0}^{\infty} a_n t^n.
	\]
	An easy computation shows that
	\begin{align*}
		\frac{1}{\pi} \int_{\Delta} (1 - \abs{t})^b \abs{f}^2 \dmu 
		&= \sum_{n=0}^{\infty} \abs{a_n}^2 
		\frac{1}{\pi} \int_{\Delta} (1 - \abs{t})^b \abs{t}^{2n} \dmu \\
		&= \sum_{n=0}^{\infty} \abs{a_n}^2 \int_0^1 2 (1 - r)^b r^{2n+1} \mathit{dr}
		= \sum_{n=0}^{\infty} \abs{a_n}^2 \frac{2 (2n+1)! b!}{(2n+2+b)!}.
	\end{align*}
	We can now apply the Cauchy-Schwarz inequality to get
	\[
		\abs{f(t)}^2 \leq 
		\sum_{n=0}^{\infty} \abs{t}^{2n} \frac{(2n+2+b)!}{2 (2n+1)! b!} \cdot
		\sum_{n=0}^{\infty} \abs{a_n}^2 \frac{2 (2n+1)! b!}{(2n+2+b)!}.
	\]
	To estimate the first series, we set $x = \abs{t}^2$ and compute as follows:
	\begin{align*}
		\sum_{n=0}^{\infty} \abs{t}^{2n} \frac{(2n+2+b)!}{2 (2n+1)! b!} &= 
		\frac{1}{2 b!} \sum_{n=0}^{\infty} (2n+2)(2n+3) \dotsm (2n+2+b) x^n \\
		&\leq \frac{1}{2b!} \sum_{n=0}^{\infty} (2n+2)(2n+4) \dotsm (2n+2+2b) x^n \\
		&= \frac{2^b}{b!} \sum_{n=0}^{\infty} (n+1) \dotsm (n+1+b) x^n \\
		&= \frac{2^b}{b!} \left( \frac{d}{dx} \right)^{b+1} \sum_{n=0}^{\infty} x^n
		= \frac{2^b (b+1)}{(1-x)^{b+2}} 
		= \frac{2^b (b+1)}{(1-\abs{t}^2)^{b+2}}.
	\end{align*}
	The desired result follows.
\end{proof}

\newpar
If we apply this result to the integral bound for $B(t)$ in
\eqref{eq:B-integral-bound}, for example by choosing an orthonormal basis of
$\End(V)$ relative to the inner product coming from the Hodge structure $\Phi(-1)$,
and then writing $B(t)$ as a matrix, we get
\begin{equation} \label{eq:B-bound}
	\norm{B(t)}_{\Phi(-1)} \leq \frac{C}{(1 - \abs{t}^2)^{2b+2}},
\end{equation}
where $b = 2r \binom{r+1}{3} + 1$, and where $C > 0$ is a constant that again depends
only on $r = \rk E$ and on the minimal polynomial of $T \in \GL(V)$. 
The interesting thing is that $B(t)$ still behaves rather
nicely near the outer boundary of the unit disk, even though all we have is a bound
on its $L^2$-norm.

\subsection{Holomorphic extension of the untwisted period mapping}

\newpar
We return to our analysis of the untwisted period mapping 
\[
	\Psi_S(e^z) = e^{-z(S+N)} \Phi(z). 
\]
Recall that the differential of the mapping $z \mapsto \Psi_S(e^z)$ is given by
\[
	e^{-z(S+N)} \theta_{\partial/\partial z} e^{z(S+N)} - (S+N)
	\equiv B(e^z) - (S+N) \mod F^0 \End(V)_{\Psi_S(e^z)}
\]
Here we are using the fact that $B(e^z) - e^{-z(S+N)} \theta_{\partial/\partial z}
e^{z(S+N)}$ lies in $F^0 \End(V)_{\Psi_S(e^z)}$. The point is that
the operator on the right-hand side is holomorphic. 

\newpar
Now suppose that we have a holomorphic mapping
\[
	g \colon \HH \to \GL(V)
\]
that solves the ordinary differential equation
\begin{equation} \label{eq:ODE}
	g'(z) = \Bigl( B(e^z) - (S+N) \Bigr) g(z),
\end{equation}
subject to the initial condition $g(-1) = \id$.  A short calculation with derivatives
shows that the mapping $g(z)^{-1} \Psi_S(e^z)$ is then constant, which means that
\[
	\Psi_S(e^z) = g(z) \cdot \Psi_S(e^{-1}) = g(z) \cdot e^{(S+N)} \Phi(-1).
\]
The key to the remainder of the proof is that the differential equation in
\eqref{eq:ODE} has a \define{regular singular point} at $t=0$. 

\newpar
Let us first prove that \eqref{eq:ODE} has a solution with the desired
properties. For every $v \in V$, the differential equation
\[
	f'(z) = \Bigl( B(e^z) - (S+N) \Bigr) f(z)
\]
has a unique solution $f \colon \HH \to V$ that is holomorphic and satisfies the
initial condition $f(-1) = v$; this is Cauchy's theorem. Let $v_1, \dotsc, v_r \in V$
be a basis, and let $f_1, \dotsc, f_r \colon \HH \to V$ be the corresponding
solutions. By uniqueness, the values $f_1(z), \dotsc, f_r(z) \in V$ must be linearly
independent for every $z \in \HH$. We can therefore define a holomorphic mapping
\[
	g \colon \HH \to \GL(V)
\]
by requiring that $g(z) v_i = f_i(z)$ for $i = 1, \dotsc, r$. It follows that $g(-1)
= \id$, and so $g$ is the desired solution to \eqref{eq:ODE}. 

\newpar
Clearly, $g(z+2 \pi i) g(-1+2 \pi i)^{-1}$ is also a solution, and so by uniqueness, 
\[
	g(z + 2 \pi i) = g(z) g(-1 + 2 \pi i).
\]
Let $A \in \End(V)$ be the unique operator such that $g(-1 + 2 \pi i) = e^{2 \pi i
A}$ and such that all eigenvalues of $A$ have their real part contained in the
interval $[0,1)$. Then $g(z) e^{-zA}$ is invariant under $z \mapsto z + 2 \pi i$, and
so
\[
	g(z) = M(e^z) e^{zA}
\]
for a unique holomorphic mapping $M \colon \dst \to \GL(V)$. Since \eqref{eq:ODE}
has a regular singular point at $t = 0$, the basic theory of Fuchsian differential
equations implies that $M(t)$ is meromorphic at $t = 0$. 

\newpar
Let us briefly recall how this is proved. We use the shorthand $\norm{v} =
\norm{v}_{\Phi(-1)}$ for the Hodge norm in the Hodge structure $\Phi(-1)$. Since
$B \colon \Delta \to \End(V)$ is holomorphic, there is a constant $C > 0$
such that
\[
	\bignorm{B(e^z)-(S+N)} \leq C \quad \text{for $\Re z \leq -1$.}
\]
Let $z \in \HH$ be any point with $\Re z \leq -1$. Consider the auxiliary function
\[
	\varphi \colon [0,1] \to \RR, \quad
	\varphi(x) = \bignorm{g \bigl( -1+x(z+1) \bigr)}^2.
\]
The derivative of $\varphi$ is given by the formula
\[
	\varphi'(x) = 2 \Re \Bigl\langle (z+1) g' \bigl( -1 + x(z+1) \bigr), \,
		g \bigl( -1 + x(z+1) \bigr) \Bigr\rangle,
\]
and so the differential equation in \eqref{eq:ODE}, together with the
Cauchy-Schwarz inequality and the bound on $B(e^z) - (S+N)$, leads to
the simple inequality
\[
	\abs{\varphi'(x)} \leq 2 C \abs{z+1} \cdot \varphi(x).
\]
This is easily integrated, with the result that 
\[
	\norm{g(z)}^2 \leq \norm{g(-1)}^2 e^{2C \abs{z+1}} \leq r e^{2C \abs{z+1}},
\]
remembering that $g(-1) = \id$ and $r = \dim V$. Since the $L^2$-norm is
submultiplicative, we then get
\[
	\norm{M(e^z)} \leq \norm{g(z)} \cdot \norm{e^{-zA}} 
	\leq \sqrt{r} e^{C \abs{z+1}} \cdot e^{\abs{z} \norm{A}}
	\leq C' e^{d \abs{\Re z}}
\]
for some $C' > 0$ and some $d \in \NN$; note that we may assume that $\abs{\Im z}
\leq \pi$. But this says exactly that $M(t)$ has a pole of order at most $d$ at $t =
0$.

\newpar
We can now finish the proof of the first half of the nilpotent orbit theorem, still
under the assumption that the eigenvalues of $S$ lie in a fixed half-open interval
$I$ of length $1$. By construction, we have
\[
	\Psi_S(e^z) = g(z) \cdot e^{S+N} \Phi(-1) = M(e^z) \cdot e^{zA} \cdot e^{S+N} \Phi(-1)
\]
for every $z \in \HH$. The left-hand side is invariant under the substitution $z
\mapsto z + 2 \pi i$, and so the operator $e^{2 \pi i A}
\in \GL(V)$ has to leave the point $e^{S+N} \Phi(-1) \in \Dch$ fixed. It follows that
$A$ itself must leave the point fixed, and so we can erase the factor $e^{zA}$ from
the identity above. The result is that 
\[
	\Psi_S(t) = M(t) \cdot e^{S+N} \Phi(-1)
\]
for $t \in \dst$. Because $M(t)$ is meromorphic at $t = 0$, and $\Dch$ is projective,
this is enough to conclude that $\Psi_S$ extends holomorphically across the origin.

\newpar
Removing the restriction on the eigenvalues of $S$ is now an easy matter. Suppose
that $S \in \End(V)$ is \emph{any} operator such that $T_s = e^{2 \pi i S}$; of
course, $S$ is then automatically semisimple with real eigenvalues. Let $S(I) \in
\End(V)$ be the choice we were using above, with eigenvalues contained in the given
interval $I \subseteq \RR$. Then $S(I) - S$ has integer eigenvalues, and so 
\[
	\Psi_S(t) = e^{z(S(I) - S)} \Psi_{S(I)}(t) 
	= t^{S(I) - S} \Psi_{S(I)}(t)
\]
is obtained from the holomorphic maping $\Psi_{S(I)}(t)$ by acting by a semisimple
operator whose eigenvalues are powers of $t$. It follows that $\Psi_S(t)$ is also
meromorphic, and therefore extends holomorphically to $\Delta$ for the same reason as
before.

\newpar
In particular, we now have a well-defined limit $\Psi_S(0) \in \Dch$. Let us write
$F^{\bullet} \End(V)_{\Psi_S(0)}$ for the induced filtration on $\End(V)$. It remains
to show that
\[
	S + N \in F^{-1} \End(V)_{\Psi_S(0)}.
\]
Recall that the derivative of the holomorphic mapping $\Psi_S(e^z)$ is equal to
\[
	B(e^z) - (S+N) \mod F^0 \End(V)_{\Psi_S(e^z)},
\]
and that the operator $B(e^z)$ is holomorphic and belongs to $F^{-1}
\End(V)_{\Psi_S(e^z)}$. Since $\Psi_S \colon \Delta \to \Dch$ is holomorphic, the
chain rule shows that the derivative of $\Psi_S(e^z)$ goes to zero like $\abs{t} =
e^{-\abs{\Re z}}$. Passing to the limit, we get
\[
	B(0) - (S+N) \in F^0 \End(V)_{\Psi_S(0)},
\]
and therefore $S+N \in F^{-1} \End(V)_{\Psi_S(0)}$, as claimed. This completes the
proof of \Cref{thm:nilpotent-convergence-intro}, which contained the convergence
assertions in the nilpotent orbit theorem.

\newpar \label{par:canonical-extension}
The nilpotent orbit theorem has the following more geometric interpretation, in terms
of certain bundles on the unit disk. Recall that the operator $S \in \End(V)$ in
$T = e^{2 \pi i(S+N)}$ depended on a choice of half-open interval $I$ of length $1$.
Let $\shE_S$ denote the canonical extension \cite[Prop.~5.2]{Deligne-eq} of the
holomorphic vector bundle $\shE$, such that the connection extends to a logarithmic
connection $\nabla \colon \shE_S \to \Omega_{\Delta}^1(\log 0) \tensor \shE_S$ whose
residue $\Res_0 \nabla$ has eigenvalues in the same interval $I$. The bundle $\shE_S$ has a
unique trivialization
\[
	V \tensor_{\CC} \shO_{\Delta} \cong \shE_S
\]
such that the logarithmic connection becomes $\nabla(v \tensor 1) = (S+N)v \tensor
\dt$; therefore $S+N$ is exactly the Jordan decomposition of the residue $\Res_0
\nabla$. Under this trivialization, the Hodge bundle $F^p \shE$ goes to the subbundle
\[
	\menge{\bigl( t, \Psi_S^p(t) \bigr)}{t \in \dst} \subseteq \dst \times V
\]
of the trivial bundle; \Cref{thm:nilpotent-convergence-intro} is
therefore saying exactly that the Hodge bundles extend to holomorphic subbundles $F^p
\shE_S$ of the canonical extension. 

\newpar
Now if we introduce the quotient bundles
\[
	\shE_S^{p,q} = F^p \shE_S / F^{p+1} \shE_S,
\]
then the operator $B(t)$ that we constructed above gives a holomorphic extension of
the Higgs field to a morphism of holomorphic vector bundles
\[
	\shE_S^{p,q} \to \Omega_{\Delta}^1(\log 0) \tensor \shE_S^{p-1,q+1}
\]
whose residue at $t = 0$ is exactly the operator $S+N \in \End(V)$. This gives us a
sort of canonical extension of the Higgs bundle $\bigoplus \shE^{p,q}$ to a
logarithmic Higgs bundle $\bigoplus \shE_S^{p,q}$, on which the residue $\Res_0
\theta_{\partial/\partial t}$ of the Higgs field has eigenvalues in the interval $I$.
The fiber at $t=0$ is isomorphic to the graded vector space
\[
	\bigoplus_{p \in \ZZ} \Psi_S^p(0) / \Psi_S^{p+1}(0),
\]
and we have $\Res_0 \theta_{\partial/\partial t} = S+N$.

\subsection{The limiting Hodge filtration}

\newpar
From the untwisted period mapping, we get a sort of ``limit'' filtration $\Psi_S(0)
\in \Dch$, but it depends on the choice of $S \in \End(V)$, and therefore has
little intrinsic meaning. In this section, we introduce an additional filtration
$\Flim \in \Dch$, called the ``limiting Hodge filtration'', that only depends on the
period mapping itself. The limiting Hodge filtration is important in two ways: it
allows one to approximate the original period mapping by a nilpotent orbit; and it
shows up in the limiting mixed Hodge structure.

\newpar
For technical reasons, we again have to restrict to the case where the eigenvalues of
$S \in \End(V)$ are contained in a half-open interval of length $1$. 

\begin{pprop} \label{prop:Flim}
	As long as $\abs{\Im z}$ remains bounded, the limit
	\[
		\Flim = \lim_{\abs{\Re z} \to \infty} e^{-zN} \Phi(z)
		= \lim_{x \to \infty} e^{-xS} \Psi_S(0) \in \Dch
	\]
	exists and is called the \define{limiting Hodge filtration}. It satisfies
	\[
		N(\Flim^{\bullet}) \subseteq \Flim^{\bullet-1} \quad \text{and} \quad
		T_s(\Flim^{\bullet}) \subseteq \Flim^{\bullet}.
	\]
\end{pprop}

\newpar
From the formula for the untwisted period mapping, it is clear that
\[
	e^{-zN} \Phi(z) = e^{zS} \Psi_S(e^z).
\]
Since $\Psi_S \colon \Delta \to \Dch$ is holomorphic, there is a constant $C > 0$
such that 
\begin{equation} \label{eq:distance-Flim-1}
	\dDch \bigl( \Psi_S(t), \Psi_S(0) \bigr) \leq C \abs{t}
\end{equation}
for small values of $\abs{t}$. According to \eqref{eq:Dch-translation}, translation
by the operator $e^{zS} \in \GL(V)$ distorts distances in the compact dual $\Dch$ by
at most the operator norm of $\Ad e^{zS}$. We can estimate this using the following
variant of \Cref{lem:conjugation}. 

\begin{plem} \label{lem:Ad-S}
	Let $S \in \End(V)$ be a semisimple with real eigenvalues. Then
	\[
		\norm{\Ad e^{zS}} \leq C e^{(\alphamax - \alphamin) \abs{\Re z}},
	\]
	where $\alphamin$ and $\alphamax$ are the smallest and largest eigenvalue of $S$.
\end{plem}

\begin{proof}
Let $A \in \End(V)$ be an arbitrary endomorphism. We have
\[
	e^{zS} A e^{-zS} 
	= \sum_{\alpha, \beta} e^{(\beta - \alpha) \abs{\Re z}} e^{i \Im z(\alpha -
	\beta)} P_{\alpha} A P_{\beta},
\]
where $P_{\alpha} \colon V \to E_{\alpha}(S)$ are the projections to the eigenspaces
of $S$. Since the $L^2$-norm is submultiplicative, we get
\[
	\norm{e^{zS} A e^{-zS}} \leq e^{(\alphamax - \alphamin) \abs{\Re z}} 
	\biggl( \sum_{\alpha} \norm{P_{\alpha}} \biggr)^2 \norm{A},
\]
and therefore the desired upper bound on the operator norm of $\Ad e^{zS}$. 
\end{proof}

\newpar
Applying this lemma to \eqref{eq:distance-Flim-1}, we find that
\begin{equation} \label{eq:distance-Flim-2}
	\dDch \bigl( e^{-zN} \Phi(z), e^{zS} \Psi_S(0) \bigr) 
	\leq C' e^{-(1 - \rho) \abs{\Re z}}
\end{equation}
for some constant $C' > 0$, where $\rho = \alphamax - \alphamin < 1$ is the
difference between the largest and smallest eigenvalue of $S$. It remains to
understand the effect of the exponential factor $e^{-\abs{\Re z} S}$ on the
filtration $\Psi_S(0)$. 

\newpar
The following lemma does this in slightly greater generality; we will use it again
(in \Cref{par:lem-exponential-used}) when we prove the convergence of the rescaled
period mapping.

\begin{plem} \label{lem:exponential}
	Let $S \in \End(V)$ be a semisimple operator with real eigenvalues. For any
	filtration $F \in \Dch$, the limit
	\[
		\hat{F} = \lim_{x \to \infty} e^{xS} F
	\]
	exists in $\Dch$, and the filtration $\hat{F}$ is compatible with the eigenspace
	decomposition $V = \bigoplus_{\alpha} E_{\alpha}(S)$. Moreover, there is a constant $C
	\geq 0$ such that
	\[
		\dDch \bigl( \hat{F}, e^{xS} F \bigr) \leq C e^{-\delta x},
	\]
	where $\delta > 0$ is the smallest distance between consecutive eigenvalues of $S$.
\end{plem}

\begin{proof}
	Since a filtration is just a collection of subspaces, it suffices to prove that
	for any subspace $W \subseteq V$ of dimension $d$, the limit 
	\[
		\hat{W} = \lim_{x \to \infty} e^{xS} W 
	\]
	exists (in the Grassmannian of $d$-dimensional subspaces of $V$), and satisfies
	\[
		\hat{W} = \bigoplus_{\alpha \in \RR} E_{\alpha}(S) \cap \hat{W}.
	\]
	To make it clear what is going on, let us first do the case where $W = \CC v$ is
	one-dimensional. Write $v = \sum_{\alpha} v_{\alpha}$, where $S v_{\alpha} =
	\alpha v_{\alpha}$. Then
	\[
		e^{xS} v = \sum_{\alpha} e^{\alpha x} v_{\alpha}.
	\]
	Let $\beta \in \RR$ be the largest number such that $v_{\beta} \neq 0$. From
	\[
		e^{xS} (\CC v) = \CC \biggl( v_{\beta} + \sum_{\alpha < \beta} 
			e^{-(\beta - \alpha) x} v_{\alpha} \biggr),
	\]
	we see that $\lim_{x \to \infty} e^{xS}(\CC v)$ exists and equals $\CC v_{\beta}$.
	So the effect of the limit is to extract the component of $v$ for
	the largest possible eigenvalue. Moreover, the rate of convergence is evidently
	$e^{-(\beta - \beta') x}$, where $\beta' < \beta$ is the next largest eigenvalue of $S$.

	To deal with the general case, let $\alpha_1 < \alpha_2 < \dotsb < \alpha_n$ be the
	eigenvalues of $S$ in increasing order, and consider the filtration $G_{\bullet}$
	by increasing eigenvalues, with terms
	\[
		G_j = E_{\alpha_1}(S) \oplus \dotsb \oplus E_{\alpha_j}(S) \subseteq V.
	\]
	Set $W_1 = G_1 \cap W$; choose a subspace $W_2 \subseteq W$ such that
	$W_1 \oplus W_2 = G_2 \cap W$; and so on. Continuing in this way, we obtain a
	collection of subspaces $W_1, \dotsc, W_n \subseteq W$, possibly zero-dimensional,
	with the property that
	\[
		G_j \cap W = W_1 \oplus \dotsb \oplus W_j.
	\]
	By construction, any nonzero vector $v \in W_j$ must have a nontrivial component
	$v_{\alpha_j} \in E_{\alpha_j}(S)$, and as we have seen above, $e^{xS}(\CC v)$
	therefore converges to $\CC v_{\alpha_j}$ at a rate of $e^{-(\alpha_j -
	\alpha_{j-1}) x}$. It follows that the subspace $e^{xS} W_j$ converges, at the
	same rate, to a subspace $\hat{W}_j \subseteq E_{\alpha_j}(S)$, which consists of
	the $E_{\alpha_j}(S)$-components of all the vectors in $W_j$. Putting everything
	together, we find that 
	\[
		\lim_{x \to \infty} e^{xS} W = \bigoplus_{j=1}^n \hat{W}_j.
	\]	
	The rate of convergence is evidently $e^{-\delta x}$, where $\delta > 0$ is the
	smallest distance between consecutive eigenvalues of $S$.
\end{proof}

\newpar
	Here is an equivalent way for describing the limit filtration $\hat{F}$ in terms of
	the filtration $G_{\bullet}$ by increasing eigenvalues of $S$. Projecting $W
	\subseteq V$ to the subquotient $G_j/G_{j-1}$ yields the subspace
	\[
		(G_j \cap W + G_{j-1})/G_{j-1} \subseteq G_j / G_{j-1}.
	\]
	Since $G_j \cap W + G_{j-1} = W_j + G_{j-1}$, the subspace $\hat{W}_j \subseteq
	E_{\alpha}(S)$ that we used during the proof is exactly the preimage of the above
	subspace under the isomorphism $E_{\alpha_j}(S) \cong G_j/G_{j-1}$. So the effect
	of the limit
	\[
		\hat{F} = \lim_{x \to \infty} e^{xS} F
	\]
	is to make the filtration $F$ compatible with the eigenspace decomposition of $S$,
	by projecting to the subquotients of the filtration by increasing eigenvalues.

\newpar
The lemma shows that the limit
\[
	\Flim = \lim_{x \to \infty} e^{-x S} \Psi_S(0) \in \Dch
\]
exists, and that the resulting filtration satisfies $S(\Flim^{\bullet}) \subseteq
\Flim^{\bullet}$. On account of \eqref{eq:distance-Flim-2}, it follows that the
limit
\[
	\lim_{\abs{\Re z} \to \infty} e^{-zN} \Phi(z) 
	= \lim_{\abs{\Re z} \to \infty} e^{i \Im z \, S} e^{-\abs{\Re z} S} \Phi(z)
	= \Flim
\]
also exists, provided that $\abs{\Im z}$ remains bounded in the process. In fact, the
limiting Hodge filtration $\Flim$ is obtained from $\Psi_S(0)$ by projecting to the
subquotients of the filtration by \emph{decreasing} eigenvalues of $S$. This is due
to the minus sign in the exponential factor $e^{-\abs{\Re z} S}$. 

\newpar
We already know that $T_s(\Flim^{\bullet}) \subseteq \Flim^{\bullet}$.
To complete the proof of \Cref{prop:Flim}, we have to analyze how the
nilpotent operator $N$ acts on the limiting Hodge filtration $\Flim$. Recall that we
have 
\[
	(S + N) \Psi_S^p(0) \subseteq \Psi_S^{p-1}(0) \quad \text{for all $p \in \ZZ$.}
\]
From the definition of $\Flim$, we now get
\[
	N \Flim^p = (S+N) \Flim^p  = 
	\lim_{x \to \infty} e^{-xS} (S+N) \Psi_S^p(0)
	\subseteq \lim_{x \to \infty} e^{-xS} \Psi_S^{p-1}(0) = \Flim^{p-1},
\]
as desired.

\newpar
It turns out that there is another way to relate the limiting Hodge
filtration to untwisted period mappings. This is useful for the proof of the
nilpotent orbit theorem in higher dimensions, which we plan to discuss in
a sequel to this paper. The idea, which already occurs in \cite[\S8]{Schmid}, is to
go to a finite covering of $\dst$ of sufficiently high degree.

\newpar
For any $m \geq 1$, we can pull back $E$ along the finite covering space
\[
	\dst \to \dst, \quad t \mapsto t^m. 
\]
The result is another variation of Hodge structure on $\dst$. It has
the same space of multi-valued flat sections, but the monodromy operator changes to
$T^m$, and the period mapping changes to $\Phi(mz)$. This is convenient, because it
means that we do not have to introduce any new notation. Let $I \subseteq \RR$ be any
half-open interval of length $1$. Let $S_m \in \End(V)$ be the unique semisimple
operator with
\[
	T_s^m = e^{2 \pi i m S_m},
\]
such that $mS_m$ has eigenvalues in $I$; the eigenvalues of $S_m$ itself are
contained in $\frac{1}{m} I$. Because $T_s = e^{2 \pi i S}$, the 
eigenvalues of the operator $e^{2 \pi i(S - S_m)}$ are $m$-th roots of unity. 

\begin{plem} \label{lem:roots-of-unity}
	On each eigenspace of $T_s$, the operator $e^{2 \pi i(S - S_m)}$ acts as
	multiplication by an $m$-th root of unity; these roots of unity are distinct if
	$m \delta(T) \geq 1$.
\end{plem}

\begin{proof}
	Suppose that $I = [\alpha, \alpha+1)$; the remaining case is similar. List the
	eigenvalues of $S$ in increasing order as
	\[
		\alpha \leq \alpha_1 < \dotsb < \alpha_n < \alpha+1;
	\]
	the distance between adjacent eigenvalues is at least $\delta(T)$. Write $m
	\alpha_j = k_j + \beta_j$, where $k_j \in \ZZ$ and $\beta_j \in I$; then the
	eigenvalues of $mS_m$ are exactly $\beta_1, \dotsc, \beta_n$. Now
	\[
		k_j = \lfloor m \alpha_j - a \rfloor,
	\]
	and since $m \alpha_{j+1} - m \alpha_j \geq m \delta(T) \geq 1$, it follows that
	$k_{j+1} > k_j$; similarly, the fact that $m \alpha_r - m \alpha_1 < m$ implies
	that $k_n - k_1 < m$. The eigenvalues of $e^{2 \pi i(S - S_m)}$ are therefore the
	complex numbers $e^{2 \pi i k_j/m}$, which are distinct $m$-th roots of unity.
\end{proof}

\newpar
Now we consider a variant of the untwisted period mapping. The expression
\[
	e^{-mz(S_m + N)} \Phi(mz)
\]
is invariant under the substitution $z \mapsto z + 2 \pi i$, and so it again
descends to a holomorphic mapping
\[
	\Psi_m \colon \dst \to \Dch, \quad \Psi_m(e^z) = e^{-mz(S_m + N)} \Phi(mz).
\]
This in of course just the untwisted period mapping $\Psi_{mS_m}$ for the pullback of
$E$ along $t \mapsto t^m$. A brief calculation shows that
\[
	\Psi_m(e^{2 \pi i/m} \cdot t) = e^{2 \pi i(S - S_m)} \cdot \Psi_m(t).
\]
This identity has the following nice consequence.

\begin{pprop} \label{prop:nilpotent-orbit-theorem-covering}
	Let $m \in \NN$ be such that $m \delta(T) \geq 1$, and choose $S_m \in
	\End(V)$ as above. The holomorphic mapping
	\[
		\Psi_m \colon \dst \to \Dch, \quad \Psi_m(e^z) = e^{-mz(S_m + N)} \Phi(mz),
	\]
	extends holomorphically across the origin, and $\Psi_m(0) = \Flim$.
\end{pprop}

\begin{proof}
We already know from the proof of the nilpotent orbit theorem that $\Psi_m$ extends
holomorphically across the origin. It remains to show that the filtration $\Psi_m(0)
\in \Dch$ agrees with the limiting Hodge filtration $\Flim$. From the identity
\[
	\Psi_m(e^{2 \pi i/m} \cdot t) = e^{2\pi i (S - S_m)} \cdot \Psi_m(0),
\]
we conclude that the filtration $\Psi_m(0) \in \Dch$ is preserved by the operator
$e^{2 \pi i(S - S_m)}$. By \Cref{lem:roots-of-unity}, the eigenvalues of this
operator are $m$-th roots of unity, with a different root of unity $\zeta_{\lambda}$
on each eigenspace $E_{\lambda}(T_s)$. Consequently,
\[
	P_{\lambda} = \frac{1}{m} \sum_{k=0}^{m-1} 
	\bigl( \zeta_{\lambda}^{-1} e^{2 \pi i(S-S_m)} \bigr)^k,
\]
and one deduces easily that $T_s \cdot \Psi_m(0) = \Psi_m(0)$. For $\abs{t}$
sufficiently small, we have
\[
	\dDch \bigl( \Psi_m(t), \Psi_m(0) \bigr) \leq C \abs{t},
\]
which translates into
\[
	\dDch \Bigl( e^{-z(S_m + N)} \Phi(z), \Psi_m(0) \Bigr) \leq C e^{-\frac{1}{m}
	\abs{\Re z}}.
\]
After applying \eqref{eq:Dch-translation} and \eqref{eq:AdS-operator-norm} to control
the distortion caused by translating by the operator $e^{zS_m} \in \GL(V)$, we get
\[
	\dDch \Bigl( e^{-zN} \Phi(z), e^{zS_m} \Psi_m(0) \Bigr) \leq 
	C \biggl( \sum_{\lambda} \norm{P_{\lambda}}_{\Phi(-1)} \biggr)^2 
	e^{-\frac{1-\rho}{m} \abs{\Re z}}.
\]
Since $e^{zS_m} \Psi_m(0) = \Psi_m(0)$, we conclude that
\[
	\Psi_m(0) = \lim_{\abs{\Re z} \to \infty} e^{-zN} \Phi(z) = \Flim
\]
is indeed equal to the limiting Hodge filtration.
\end{proof}

\newpar
Here is a small example that shows the difference between the filration $\Psi_S(0)$ and
the limiting Hodge filtration $\Flim$. 

\begin{pexa} \label{ex:PsiS}
	Let $\alpha > 0$ be a real number. Using the notation from
	\Cref{ex:period-domain}, consider the polarized variation of Hodge structure $\Phi
	\colon \HH \to D$ with
	\[
		V_{\Phi(z)}^{1,0} = \CC \bigl( 1, e^{\alpha z} \bigr) \quad \text{and} \quad
		V_{\Phi(z)}^{0,1} = \CC \bigl( e^{\alpha \zb}, 1 \bigr).
	\]
	Here $N = 0$ and so $\Flim^1 = \CC(1,0)$, which is the Hodge filtration of a
	polarized Hodge structure. On the other hand, if we use 
	\[
		S = \begin{pmatrix} 0 & 0 \\ 0 & \alpha \end{pmatrix}, 
	\]
	then $\Psi_S^1(0) = \CC(1,1)$, and this is \emph{not} the Hodge filtration of a
	polarized Hodge structure.
\end{pexa}

\subsection{Effective estimates for the rate of convergence}

\newpar
What is missing from the above proof of the nilpotent orbit theorem are good
estimates for the rate of convergence of the untwisted period mapping. The purpose of
this section is to obtain such estimates, by using the curvature properties of the
Hodge metric and the maximum principle.

\newpar \label{par:ineffective-estimates}
Let us first see what kind of estimates we can derive from the fact that the
untwisted period mapping $\Psi_S$ extends across the origin.
Recall that the differential of the holomorphic mapping $\Psi_S(e^z)$ is equal to 
\[
	e^{-z(S+N)} \theta_{\partial/\partial z} e^{z(S+N)} - (S+N) \mod
	F^0 \End(V)_{\Psi_S(e^z)}.
\]
At the same time, for $\eps > 0$ sufficiently small, we can certainly find a holomorphic
mapping $g \colon \Delta_r \to \GL(V)$ with $g(0) = \id$ such that $\Psi_S(t) = g(t)
\cdot \Psi_S(0)$. By the chain rule, the differential of $\Psi_S(e^z)$ is therefore
also equal to
\[
	e^z \cdot g'(e^z) g(e^z)^{-1} \mod F^0 \End(V)_{\Psi_S(e^z)},
\]
provided that $\Re z < \log \eps$. Putting both things together, we get
\[
	\theta_{\partial/\partial z} - (S+N) \equiv 
	e^z \cdot e^{z(S+N)} g'(e^z) g(e^z)^{-1} e^{-z(S+N)}
	\mod F^0 \End(V)_{\Phi(z)},
\]
at least when $\Re z < \log \eps$. Arguing as in the proof of
\Cref{lem:conjugation} to estimate the effect of conjugating by $e^{z(S+N)}$, the
Hodge norm of the operator on the right-hand side is bounded, for $\abs{\Re z} \gg
0$, by a constant multiple of 
\[
	\abs{\Re z}^m e^{-\delta(T) \abs{\Re z}},
\]
where $m = m(N) + 2 \sqrt{2r} C_0$; the exact value of the constant depends of course
on $g(t)$ and hence on $\Psi_S(t)$. We remind the reader that $0 < \delta(T) \leq 1$ is 
the minimal distance between consecutive eigenvalues of $T$ on the unit circle,
divided by $2\pi$. If we denote by
\[
	N = \sum_{k \in \ZZ} N^{k,-k} \quad \text{and} \quad 
	S = \sum_{k \in \ZZ} S^{k,-k} \quad \text{and} \quad
	P_{\lambda} = \sum_{k \in \ZZ} P_{\lambda}^{k,-k}
\]
the Hodge decompositions, then it follows that the quantity
\[
	\norm{\theta_{\partial/\partial z} - N^{-1,1} - S^{-1,1}}_{\Phi(z)}^2
	+ \sum_{k \leq -2} \norm{N^{k,-k} + S^{k,-k}}_{\Phi(z)}^2
\]
is bounded, for $\abs{\Re z} \gg 0$, by a constant multiple of $\abs{\Re z}^{2m}
e^{-2\delta(T) \abs{\Re z}}$.

\newpar
We can improve the estimates from above by changing the interval containing the eigenvalues
of $S$. Recall that $T_s = e^{2 \pi i S}$, and that the eigenvalues of $S$ are
supposed to lie in a half-open interval of length $1$. Write the eigenvalues of $T_s$
in the form $\lambda_j = e^{2 \pi i \alpha_j}$, say with $0 \leq \alpha_1 < \dotsb <
\alpha_n < 1$. If we gradually slide the interval $[0,1)$ over to the right, we
obtain the following $n$ choices for the operator $S$, namely
\[
	S = \sum_{j=1}^n \alpha_j P_{\lambda_j} 
	+ (P_{\lambda_1} + \dotsb + P_{\lambda_k}),
\]
for $k = 0, 1, \dotsc, n-1$. If we apply the argument in the preceding paragraph to
each choice of $S$, and then take a suitable linear combination of the resulting
inequalities, we find that there are two constants $C > 0$ and $b \in \NN$, such that
\begin{align*}
	\norm{\theta_{\partial/\partial z} - N^{-1,1}}_{\Phi(z)}^2 + 
	\sum_{k \leq -2} \norm{N^{k,-k}}_{\Phi(z)}^2
	&\leq C \abs{\Re z}^{2m} e^{-2\delta(T) \abs{\Re z}} \\
	\sum_{k \leq -1} \norm{P_{\lambda}^{k,-k}}_{\Phi(z)}^2
	&\leq C \abs{\Re z}^{2m} e^{-2\delta(T) \abs{\Re z}}
\end{align*}
for $\abs{\Re z} \gg 0$. Here $m = m(N) + 2 \sqrt{2r} C_0$, but we cannot say anything
about the value of the constant $C$, or about how big $\abs{\Re z}$ has to be for the
inequality to hold; these are going to depend, in some unspecified way, on the
variation of Hodge structure $E$. The problem is that these estimates are not
\emph{effective}.

\begin{note}
	In fact, the argument above gives a somewhat better estimate for $P_{\lambda}$.
	Instead of $\delta(T)$, the optimal exponent is the minimal distance, on the unit
	circle, from $\lambda$ to the two immediately adjacent eigenvalues of $T$.
\end{note}

\newpar
We can make the above estimates effective -- and independent of the specific
variation of Hodge structure! -- by exploiting once more the curvature properties of
the Hodge metric.  The precise result we are going to prove is the following.

\begin{pthm} \label{thm:effective-estimates}
	Given $x < 0$, there are constants $C > 0$ and $m \in \NN$, such that
	\begin{align*}
		\norm{\theta_{\partial/\partial z} - N^{-1,1}}_{\Phi(z)}^2 + 
		\sum_{k \leq -2} \norm{N^{k,-k}}_{\Phi(z)}^2
		&\leq C \abs{\Re z}^{2b} e^{-2\delta(T) \abs{\Re z}} \\
		\sum_{k \leq -1} \norm{P_{\lambda}^{k,-k}}_{\Phi(z)}^2
		&\leq C \abs{\Re z}^{2b} e^{-2\delta(T) \abs{\Re z}}
	\end{align*}
	for every $z \in \HH$ with $\Re z \leq x$. The exact value of $C$ only
	depends on $x$, on $r = \rk E$, and on the minimal polynomial of $T \in \GL(V)$; the
	exact value of $b$ only depends on $r$.
\end{pthm}

\newpar
We need another basic fact about metrics with negative curvature on
holomorphic vector bundles. Let $E$ be a smooth vector bundle with a hermitian metric
$h$, defined on a domain $X \subseteq \CC$, and suppose that $E$
has the structure of a holomorphic vector bundle, given by a connection $d'' \colon
A^0(X, E) \to A^{0,1}(X, E)$ of type $(0,1)$. As before, we write the Chern
connection as $\delta' + d''$, where $\delta' \colon A^0(X, E) \to A^{1,0}(X, E)$,
and denote by
\[
	\Theta = (\delta' + d'')^2 \in A^{1,1} \bigl( X, \End(E) \bigr)
\]
the curvature operator of the metric.

\begin{plem} \label{lem:negative-curvature}
	Suppose that the hermitian metric $h$ has semi-negative curvature, in the sense
	that for every $u \in A^0(X, E)$, one has
	\[
		h \bigl( \Theta_{\partial/\partial z \wedge \partial/\partial \zb} \, u, u \bigr)
		\leq 0.
	\]
	Then for every nontrivial holomorphic section $u \in A^0(X, E)$ with $d'' u = 0$,
	the function $\log h(u,u)$ is subharmonic on $X$.
\end{plem}

\begin{proof}
	Since $\delta' + d''$ is a metric connection, we have $\dbar h(u,u) = h(u, \delta'
	u)$, hence
	\[
		\partial \dbar h(u,u) = h(\delta' u, \delta' u) + h(u, d'' \delta' u)
		= h(\delta' u, \delta' u) + h(u, \Theta u).
	\]
	After evaluating this on $\partial/\partial z \wedge \partial/\partial \zb$, we
	get
	\[
		\frac{\partial^2}{\partial z \partial \zb} h(u,u)
		= h \bigl( \delta'_{\partial/\partial z} u, \delta'_{\partial/\partial z} u
			\bigr) - h \bigl( \Theta_{\partial/\partial z \wedge \partial/\partial \zb} \,
		u, u \bigr) \geq h \bigl( \delta'_{\partial/\partial z} u,
		\delta'_{\partial/\partial z} u \bigr).
	\]
	At all points of $X$ where $h(u,u) > 0$, we now get $\Delta \log h(u,u) \geq 0$
	from the Cauchy-Schwarz inequality, by the same argument as in the proof of
	\Cref{lem:metric-subharmonic}. Since $\log h(u,u)$ is locally bounded from above,
	it follows that $\log h(u,u)$ is a well-defined subharmonic function with values
	in $[-\infty, \infty)$.
\end{proof}

\newpar
We return to our usual setting where $E$ is a polarized variation of Hodge
structure on $\dst$. The result above has the following implication for $\End(E)$.

\begin{pprop} \label{prop:twisted-subharmonic}
	Let $u \in A^0 \bigl( \dst, \End(E) \bigr)$ be a smooth section with the property
	that $d'' u \in A^0 \bigl( \dst, F^0 \End(E) \bigr)$. Write the Hodge
	decomposition of $u$ as
	\[
		u = \sum_{k \in \ZZ} u^{k,-k},
	\]
	where $u^{k,-k} \in A^0 \bigl( \dst, \End(E)^{k,-k} \bigr)$. Then for every $b \geq
	4r \binom{r+1}{3}$, the function
	\[
		\log \left( (-\log t)^{-b} \sum_{k \leq -1} h \bigl( u^{k,-k}, u^{k,-k} \bigr)
		\right)
	\]
	is subharmonic on $\dst$.
\end{pprop}

\begin{proof}
	We apply the result from above to the quotient bundle $\End(E)/F^0 \End(E)$, with
	the hermitian metric $h$ induced by the Hodge metric on $\End(E)$. According to
	\Cref{prop:curvature}, the curvature tensor of this metric satisfies
	\[
		h \bigl( \Theta_{\partial/\partial t \wedge \partial/\partial \tb} \, u, u \bigr)
		\leq 2h \bigl( \theta_{\partial/\partial t} u, \theta_{\partial/\partial t} u
		\bigr) \leq r \binom{r+1}{3} \frac{1}{\abs{t}^2 (\log \abs{t})^2} h(u, u),
	\]
	using the improved bound for the Higgs field of $\End(E)$ in
	\Cref{cor:Higgs-bound-End}. Just as in \Cref{par:hphi}, we now consider the
	modified hermitian metric
	\[
		\hphi = h \cdot e^{-\varphi} = h \cdot (-\log \abs{t})^{-b},
	\]
	for $b \in \NN$. Its curvature tensor satisfies
	\[
		\hphi \bigl( \Thetaphi_{\partial/\partial t \wedge \partial/\partial \tb} \, u, u \bigr)
		\leq \left( r \binom{r+1}{3} - \frac{b}{4} \right) \cdot 
		\frac{1}{\abs{t}^2 (\log \abs{t})^2} \hphi(u, u) \leq 0, 
	\]
	provided we choose $b \geq 4r \binom{r+1}{3}$. By assumption, $u$ gives a holomorphic
	section of the quotient bundle $\End(E)/F^0 \End(E)$, and so the result now
	follows from \Cref{lem:negative-curvature}.
\end{proof}

\newpar
We can now use the \define{maximum principle} to make the estimates in
\Cref{par:ineffective-estimates} effective. Since the operator $P_{\lambda} \in
\End(V)$ commutes with $T$, its defines a flat section of the bundle $\End(E)$ on
$\dst$. Define 
\[
	f = \sum_{k \leq -1} h \bigl( P_{\lambda}^{k,-k}, P_{\lambda}^{k,-k} \bigr),
\]
which is a smooth function on $\dst$. According to
\Cref{prop:twisted-subharmonic}, the function
\[
	\varphi = \log \bigl( f \cdot (-\log \abs{t})^{-b} \bigr)
\]
is subharmonic on $\dst$, where $b = 4r \binom{r+1}{3}$. We also know from
\Cref{par:ineffective-estimates} that, in a small neighborhood of the origin, $f$
is bounded from above by
\[
	C' (-\log \abs{t})^{2m(N) + 4 \sqrt{2r} C_0} \abs{t}^{2\delta(T)}.
\]
Since $b = 16 r C_0^2 \geq 2m(N) + 4 \sqrt{2r} C_0$, it follows that
\[
	\varphi - 2\delta(T) \log \abs{t} \leq \log C'
\]
for $\abs{t}$ sufficiently small. In particular, this means that the subharmonic
function $\varphi - 2\delta(T) \log \abs{t}$ is bounded from above on every compact
subset of $\Delta$. Now the maximum principle for subharmonic functions implies that
\[
	\varphi - 2\delta(T) \log \abs{t} \leq 
	\max_{\abs{t} = R} \Bigl( \varphi(t) - 2\delta(T) \log R \Bigr)
\]
for every $0 < R < 1$. Since $P_{\lambda} \in \End(V)$ is a flat section, we can
estimate the right-hand side with the help of \Cref{prop:polynomial-bound}.
After exponentiating, the conclusion is that
\[
	\sum_{k \leq -1} h \bigl( P_{\lambda}^{k,-k}, P_{\lambda}^{k,-k} \bigr)
	\leq C (-\log \abs{t})^{2b} \abs{t}^{2\delta(T)} 
	\quad \text{for $\abs{t} \leq R$,}
\]
where $b = 4r \binom{r+1}{3}$ and, assuming without loss of generality that $R >
e^{-1}$, 
\[
	C = R^{-2 \delta(T)} e^{4 \sqrt{2r} C_0 \pi} (-\log R)^{-2(b + 2 \sqrt{2r} C_0)}
	\norm{P_{\lambda}}_{\Phi(-1)}^2.
\]
Since \Cref{prop:N-norm} contains an upper bound on $\norm{P_{\lambda}}_{\Phi(-1)}$,
with a constant that only depends on $r = \dim V$ and on the minimal polynomial of $T
\in \GL(V)$, this gives us the first inequality in \Cref{thm:effective-estimates},
just with different notation. 

\newpar
To prove the other inequality, we apply the same argument to $t
\theta_{\partial/\partial t} - N$, viewed as a holomorphic section of the quotient
bundle $\End(E)/F^0 \End(E)$ on $\dst$. Everything works out because the Hodge norm
of $t \theta_{\partial/\partial t}$ is bounded from above by $C_0(-\log \abs{t})^{-1}$,
according to \Cref{cor:Higgs-bound}, and because $\norm{N}_{\Phi(-1)}$ is bounded by
virtue of \Cref{prop:N-norm}. This finishes the proof of the effective
estimates in \Cref{thm:effective-estimates}.

\subsection{Approximation by nilpotent orbits}

\newpar
In this section, we prove the second half of the nilpotent orbit theorem. Roughly
speaking, the result is that the original period mapping $\Phi(z)$ can be
approximated very well by the nilpotent orbit $\Phinil(z) = e^{zN} \Flim$. We fix an
arbitrary base point $o \in D$, and for simplicity, we denote by $\norm{v} = \norm{v}_o$ the
resulting norm on the vector space $V$. We use this norm to define the metric, and
hence the distance function $\dDch$, on the compact dual $\Dch$. 

\begin{pthm} \label{thm:approximation-nilpotent-orbit}
	There are constants $C > 0$, $x_0 < 0$, and $m \in \NN$ such that 
	\[
		\Phinil(z) = e^{z N} \Flim \in D  \quad \text{and} \quad
		d_D \bigl( \Phi(z), \Phinil(z) \bigr) \leq 
		C \abs{\Re z}^m e^{-\delta(T) \abs{\Re z}}
	\]
	for every $z \in \HH$ with $\Re z \leq x_0$. The constants $C, x_0$ only depend on
	the base point $o \in D$ and on the minimal polynomial of $T \in
	\GL(V)$; the integer $m$ only depends on $r = \rk E$.
\end{pthm}

\newpar
We first note that, due to the $G$-invariance of the metric on $D$, it is
enough to prove the theorem under the additional assumption that $\Phi(-1) = o$.
Here is why. Choose an element $g \in G$ such that $g \cdot \Phi(-1) = o$,
and consider the modified period mapping $\Phi'(z) = g \cdot \Phi(z)$. The monodromy
transformation changes to $T' = g T g^{-1}$, which has the same minimal polynomial;
the limiting Hodge filtration changes to $\Flim' = g \Flim$. If the theorem is known
for the modified period mapping $\Phi'$, then 
\[
	g \cdot e^{zN} \Flim = e^{zN'} \Flim' \in D \quad \text{for $\Re z \leq x_0$},
\]
with a constant $x_0$ that depends on $o \in D$ and on the minimal polynomial of
$T'$; but then clearly $e^{zN} \Flim \in D$ in the same range. In the same
way, the distance estimate for $\Phi'$ implies that for $\Phi$ itself.

\newpar
We assume from now on that $\Phi(-1) = o$.  The following lemma tells us how close a
point of $\Dch$ has to be to $\Phi(z)$ in order to belong to $D$. 

\begin{plem} \label{lem:D-criterion}
	There is a constant $\eps > 0$ such that
	\[
		\dDch \bigl( p, \Phi(z) \bigr) \leq \eps \cdot \abs{\Re z}^{-4C_0}
	\]
	implies both that $p \in D$ and that
	\[
		d_D \bigl( p, \Phi(z) \bigr) \leq 
		2 n e^{4C_0 \pi} \abs{\Re z}^{4C_0} \cdot \dDch \bigl( p, \Phi(z) \bigr).
	\]
	The exact value of $\eps$ only depends on the choice of base point $o \in D$.
\end{plem}

\begin{proof}
	Since $D$ is open in $\Dch$, there is a constant $\delta > 0$ such that 
	\[
		\menge{p \in \Dch}{\dDch(p, o) < \delta} \subseteq D;
	\]
	because $d_D$ and $\dDch$ are continuous, we can arrange moreover that
	\[
		d_D(p, o) \leq 2 \dDch(p, o)
	\]
	on the open ball in question. Since we are using the Hodge norm at the base point
	$o \in D$ to define the metric on $\Dch$, this second condition means that
	$\delta$ depends on the choice of $o \in D$. Fix a point $z \in \HH$ with
	$\Re z \leq -1$, and choose $g \in G$ such that $\Phi(z) = g \cdot \Phi(-1)$. For any $v
	\in V$, we have
	\[
		\norm{v}_{\Phi(z)} = \norm{v}_{g \cdot \Phi(-1)} 
		= \norm{g^{-1} v}_{\Phi(-1)}.
	\]
	According to \Cref{prop:polynomial-bound}, 
	\[
		e^{-2C_0 \pi} \abs{\Re z}^{-2C_0} \norm{v}_{\Phi(-1)} \leq 
		\norm{v}_{\Phi(z)} 
		\leq e^{2C_0 \pi} \abs{\Re z}^{2C_0} \norm{v}_{\Phi(-1)},
	\]
	where $C_0 = \frac{1}{2} \sqrt{\binom{r+1}{3}}$ and $r = \rk E$. Putting both
	things together, we get an upper bound on the operator norms of $g$ and $g^{-1}$,
	and therefore
	\[
		\max \bigl( \norm{g}_{\Phi(-1)}, \norm{g^{-1}}_{\Phi(-1)} \bigr) 
		\leq \sqrt{r} e^{2C_0 \pi} \abs{\Re z}^{2C_0}
	\]
	This gives us a bound on the operator norm of $\Ad g \colon \End(V) \to \End(V)$;
	indeed,
	\[
		\norm{gA g^{-1}}_{\Phi(-1)} \leq 
		\norm{g}_{\Phi(-1)} \norm{A}_{\Phi(-1)} \norm{g^{-1}}_{\Phi(-1)}
		\leq r e^{4 C_0 \pi} \abs{\Re z}^{4C_0} \norm{A}_{\Phi(-1)}.
	\]
	According to the discussion in \Cref{par:Adg}, for any $p \in \Dch$, we have
	\[
		\dDch \bigl( g^{-1} p, o \bigr)
		\leq r e^{4C_0 \pi} \abs{\Re z}^{4C_0} \cdot \dDch \bigl( p, \Phi(z) \bigr)
	\]
	Define $\eps = e^{-4C_0 \pi} \cdot \delta/r$. As long as $\dDch \bigl( p, \Phi(z)
	\bigr) < \eps \cdot \abs{\Re z}^{-4C_0}$, we can conclude from this that $g^{-1} p
	\in D$ and, therefore, $p \in D$; we also get the distance estimate
	\[
		d_D \bigl( p, \Phi(z) \bigr) = d_D \bigl( g^{-1} p, o \bigr) 
		\leq 2 \dDch \bigl( g^{-1} p, o \bigr)
		\leq 2 r e^{4C_0 \pi} \abs{\Re z}^{4C_0} \cdot \dDch \bigl( p, \Phi(z) \bigr).
	\]
	This is what we wanted to show.
\end{proof}

\newpar
Fix a point $z \in \HH$, and consider the curve
\[
	[0, \infty) \to \Dch, \quad x \mapsto e^{xN} \Phi(z-x).
\]
At $x = 0$, this starts at the point $\Phi(z) \in D$; as $x \to \infty$, it converges
to the point 
\[
	\lim_{x \to \infty} e^{x N} \Phi(z-x) = e^{zN} \Flim \in \Dch.
\]
The derivative at a given point $x > 0$ is easily seen to be 
\[
	N - e^{xN} \theta_{\partial/\partial z}(z - x) e^{-xN} \mod 
	F^0 \End(V)_{e^{xN} \Phi(z-x)}.
\]
We know from the effective estimates in \Cref{thm:effective-estimates} that 
\[
	\norm{\theta_{\partial/\partial z} - N^{-1,1}}_{\Phi(z-x)}^2 + 
	\sum_{k \leq -2} \norm{N^{k,-k}}_{\Phi(z-x)}^2
	\leq C^2 \bigl( \abs{\Re z} + x \bigr)^{2b} e^{-2 \delta (\abs{\Re z} + x)}.
\]
where $\delta = \delta(T)$, and where the value of $C$ depends on $r = \rk E$ and
on the minimal polynomial of $T \in \GL(V)$. Arguing as in the proof of
\Cref{lem:conjugation}, it follows that the length of the derivative, measured using
the metric on $\Dch$, is bounded from above by
\[
	C' \bigl( \abs{\Re z} + x \bigr)^{b + 2\sqrt{2r} C_0} e^{-\delta(\abs{\Re z} + x)} 
	\sum_{k=0}^{m(N)} \frac{(2x)^k}{k!} \norm{N}_{\Phi(-1)}^k.
\]
After integrating this expression over the interval $[0, \infty)$, we get a bound
for the distance between $\Phi(z)$ and $e^{zN} \Flim$ that looks like
\[
	\dDch \bigl( \Phi(z), e^{zN} \Flim \bigr)
	\leq C'' \abs{\Re z}^{b + 2\sqrt{2r} C_0} e^{-\delta \abs{\Re z}},
\]
with a constant $C'' > 0$ whose exact value is somewhat complicated, but depends only
on the two integers $m(N)$ and $r = \rk E$ and on the minimal polynomial of $T \in
\GL(V)$. 

\newpar
Now let us choose $x_0 < 0$ in such a way that $\Re z \leq x_0$ implies
\[
	C'' \abs{\Re z}^{b + 2\sqrt{2r} C_0} e^{-\delta \abs{\Re z}} \leq \eps \cdot
	\abs{\Re z}^{-4C_0},
\]
where $\eps$ is the constant from \Cref{lem:D-criterion}. We can then conclude
that $e^{zN} \Flim \in D$; we also get the distance estimate
\[
	d_D \bigl( \Phi(z), e^{z N} \Flim \bigr) \leq 
	2r e^{4 C_0 \pi} \abs{\Re z}^{4 C_0} \cdot 
	C'' \abs{\Re z}^{b + 2\sqrt{2r} C_0} e^{-\delta \abs{\Re z}}.
\]
Since all the constants on the right-hand side have the correct dependence on
parameters, this completes the proof of the second half of the nilpotent orbit theorem.

\newpar
We close this chapter with two remarks about the nilpotent orbit theorem:
\begin{enumerate}
	\item The nilpotent orbit theorem guarantees that $e^{zN} \Flim \in D$ for $\Re z
		\leq x_0$, where the constant $x_0 < 0$ only depends on $\rk E$ and on the
		minimal polynomial of $T$. One may wonder whether this actually holds for every
		$z \in \HH$. We do not know the answer to this question, and we were not able
		to find any relevant examples in the literature.
	\item One can easily prove a variant of the nilpotent orbit theorem for the
		filtration $\Psi_S(0)$. The statement is that there are constants $C >
		0$ and $m \in \NN$ such that
		\[
			e^{z(S+N)} \Psi_S(0) \in D  \quad \text{and} \quad
			d_D \bigl( \Phi(z), e^{z(S+N)} \Psi_S(0) \bigr) \leq 
			C \abs{\Re z}^m e^{-\delta(T) \abs{\Re z}},
		\]
		provided that $\abs{\Re z} \gg 0$. In fact, \Cref{lem:exponential} gives us the
		distance estimate
		\[
			\dDch \bigl( e^{-\abs{\Re z} \, S} \Psi_S(0), \Flim \bigr) \leq C
			e^{-\delta(T) \abs{\Re z}},
		\]
		and so the desired result follows from \Cref{thm:approximation-nilpotent-orbit}
		and its proof. 
\end{enumerate}

\section{Convergence of the rescaled period mapping}
\label{chap:rescaled}

\newpar
In this chapter, we use the nilpotent orbit theorem and its consequences to show
that the \define{rescaled period mapping}
\[
	\PhiSH \colon \HH \to D, \quad 
	\PhiSH(z) = e^{\half \log \abs{\Re z} \, H} e^{-\half(z-\zb) (S+N)} \Phi(z),
\]
has a well-defined limit in $D$ as $\abs{\Re z} \to \infty$. Recall from
\Cref{par:splitting} that the operator $H \in \End(V)$ is a splitting for the
monodromy weight filtration $W_{\bullet}$.

\begin{pthm} \label{thm:convergence}
	Let $H \in \End(V)$ be as in \Cref{prop:splitting}. Then the limit
	\[
		e^{-N} F_H = \lim_{\Re z \to -\infty} \PhiSH(z) \in D
	\]
	exists in the period domain. The resulting filtration $F_H \in \Dch$ satisfies
	\[
		T_s(F_H^{\bullet}) \subseteq F_H^{\bullet}, \quad
		H(F_H^{\bullet}) \subseteq F_H^{\bullet}, \quad 
		N(F_H^{\bullet}) \subseteq F_H^{\bullet-1}.
	\]
\end{pthm}

We remind the reader that the two operators $H$ and $i(S+N)$ both belong to the Lie
algebra $\glie$ of the real group $G = \Aut(V,Q)$, which means that 
\[
	e^{\half \log \abs{\Re z} \, H} e^{-\half(z-\zb) (S+N)} \in G.
\]
This is why the rescaled period mapping stays in $D$. Also note that the rescaled
period mapping depends both on $S$ and on $H$; we will see during the proof that the
filtration $F_H \in \Dch$ only depends on $H$, justifying the notation.

\subsection{Proof of convergence in the compact dual}

\newpar
Let us first rewrite everything in terms of $\Psi_S$. We have
\begin{align*}
	e^{\half \log \abs{\Re z} \, H} e^{-\half(z-\zb) (S+N)} \Phi(z) 
	&= e^{\half \log \abs{\Re z} \, H} e^{-\half(z-\zb) (S+N)} e^{z(S+N)} \Psi_S(e^z) \\
	&= e^{\half \log \abs{\Re z} \, H} e^{-\abs{\Re z} (S+N)} \Psi_S(e^z).
\end{align*}
The relation $[H,N] = -2N$ implies that 
\begin{equation} \label{eq:AdY-on-N}
	e^{\half \log x \, H} N e^{-\half \log x \, H} = e^{-\log x} \cdot N = \frac{N}{x},
\end{equation}
and therefore gives us the useful identity
\[
	e^{\half \log x \, H} e^{-x N} e^{-\half \log x \, H} = e^{-N}.
\]
Putting everything together, we arrive at
\begin{equation} \label{eq:Phi-rescaled}
	e^{\half \log \abs{\Re z} \, H} e^{-\half(z-\zb) (S+N)} \Phi(z) 
	= e^{-N} e^{\half \log \abs{\Re z} \, H} e^{-\abs{\Re z} \, S} \Psi_S(e^z),
\end{equation}
due to the fact that $[S,N] = [S,H] = 0$ (by \Cref{prop:splitting}).

\newpar
We know from the nilpotent orbit theorem that $\Psi_S(e^z)$ converges to $\Psi_S(0)$
at a rate proportional to $\abs{e^z} = e^{-\Re z}$. The effect of the two exponential
factors is controlled by \Cref{lem:exponential}. For the exponential factor
$e^{\half \log \abs{\Re z} \, H}$, the filtration by increasing eigenvalues of $H$ is
exactly the monodromy weight filtration $W_{\bullet}$;
moreover, the rate of convergence is $e^{-\half \log \abs{\Re z}} = \abs{\Re
z}^{-\half}$, since the eigenvalues of $H$ will generally be consecutive integers. For
the other exponential factor $e^{-\abs{\Re z} \, S} = e^{\abs{\Re z}(-S)}$, the relevant
filtration is by \emph{decreasing} eigenvalues of $S$ (because of the minus sign);
the rate of convergence is $e^{-\delta \abs{\Re z}}$, where $\delta > 0$ is the
minimal distance among consecutive eigenvalues of $S$.

\newpar
Let us now prove the convergence of the rescaled period mapping. Since $\Psi_S \colon
\Delta \to \Dch$ is holomorphic, there is a constant $C > 0$ such that 
\[
	\dDch \bigl( \Psi_S(e^z), \Psi_S(0) \bigr) \leq C e^{-\abs{\Re z}} 
	\quad \text{for $\abs{\Re z} \gg 0$.}
\]
According to \Cref{lem:Ad-S}, the operator norm of $\Ad e^{-\abs{\Re z} \, S}$ is
bounded by a constant multiple of $e^{(\alphamax - \alphamin) \abs{\Re z}}$. From
\eqref{eq:Dch-translation}, we therefore get
\[
	\dDch \Bigl( e^{-\abs{\Re z} \, S} \Psi_S(e^z), \, e^{-\abs{\Re z} \, S} \Psi_S(0) \Bigr)
	\leq C' e^{-\delta(T) \abs{\Re z}},
\]
using the fact that $(1 + \alphamin) - \alphamax \geq \delta(T)$ is greater or
equal to the minimal distance between adjacent eigenvalues of $T$. 

\newpar \label{par:lem-exponential-used}
Recall that the limiting Hodge filtration $\Flim \in \Dch$ satisfies
\[
	\Flim = \lim_{x \to \infty} e^{-x S} \Psi_S(0)
\]
and that \Cref{lem:exponential} gives us a distance estimate of the form
\[
	\dDch \Bigl( e^{-\abs{\Re z} \, S} \Psi_S(0), \Flim \Bigr) \leq C''
	e^{-\delta(T) \abs{\Re z}}.
\]
Because of the triangle inequality, we then get
\begin{equation} \label{eq:nilpotent-orbit-estimate}
	\dDch \Bigl( e^{-\abs{\Re z} \, S} \Psi_S(e^z), \Flim \Bigr)
	\leq (C' + C'') e^{-\delta(T) \abs{\Re z}}.
\end{equation}

\newpar \label{par:FH}
Next, we have to analyze the effect of the second exponential factor $e^{\half \log
\abs{\Re z} \, H}$. On the one hand, we have
\[
	\dDch \Bigl( e^{\half \log \abs{\Re z} \, H} e^{-\abs{\Re z} \, S} \Psi_S(t), \,
	e^{\half \log \abs{\Re z} \, H} \Flim \Bigr) 
	\leq C''' \abs{\Re z}^{2m(N)} e^{-\delta(T) \abs{\Re z}},
\]
due to the fact that the operator norm of $\Ad e^{\half \log \abs{\Re z} \, H}$ is bounded
by a constant multiple of $\abs{\Re z}^{2m}$ by \Cref{lem:Ad-S}; here $m(N)$
is the largest integer such that $N^m \neq 0$. On the other hand, the limit
\[
	F_H = \lim_{\abs{\Re z} \to \infty} e^{\half \log \abs{\Re z} \, H} \Flim \in \Dch
\]
exists by \Cref{lem:exponential}. Putting everything together, we find that the
limit
\[
	\lim_{\abs{\Re z} \to \infty} e^{\half \log \abs{\Re z} \, H} e^{-\abs{\Re z} \,
	S} \Psi_S(e^z) = F_H
\]
exists in $\Dch$. It follows that the Hodge filtrations of the rescaled period
mapping
\[
	\PhiSH(z) = e^{\half \log \abs{\Re z} \, H} e^{-\half(z-\zb) (S+N)} \Phi(z) \in D
\]
converge, in the compact dual $\Dch$, to the filtration $e^{-N} F_H$.

\newpar
The filtration $F_H$ can be described concretely as follows. First, we take the
filtration $\Psi_S(0)$ from the nilpotent orbit theorem and make it compatible with
the eigenspace decomposition of $S$, by projecting to the subquotients of the
filtration by \emph{decreasing} eigenvalues of $S$. By construction,
\[
	T_s \cdot \Flim^{\bullet} = \Flim^{\bullet}.
\]
Similarly, the filtration $F_H$ is obtained by starting from $\Flim$,
and making it compatible with the eigenspace decomposition of $H$ by projecting to
the subquotients of the monodromy weight filtration $W_{\bullet}$ (which is the
filtration by \emph{increasing} eigenvalues of $H$). This gives us the two relations
\[
	T_s \cdot F_H^{\bullet} = F_H^{\bullet} \quad \text{and} \quad
	H \cdot F_H^{\bullet} = F_H^{\bullet}.
\]
We already know from the nilpotent orbit theorem that $N \cdot \Flim^{\bullet} \subseteq
\Flim^{\bullet-1}$. From this, we can easily deduce that
\[
	N \cdot F_H^{\bullet} \subseteq F_H^{\bullet-1}.
\]
Indeed, the filtration $F_H \in \Dch$ was defined in such a way that
\[
	F_H^p = \lim_{x \to \infty} e^{x H} \Flim^p.
\]
From \eqref{eq:AdY-on-N}, we have $e^{-x H} N e^{x H} = e^{2x} N$, which gives
\[
	N \cdot F_H^p 
	= \lim_{x \to \infty} N e^{x H} \Flim^p
	= \lim_{x \to \infty} e^{x H} N \Flim^p \\
	\subseteq \lim_{x \to \infty} e^{x H} \Flim^{p-1} = F_H^{p-1}.
\]
This is what we wanted to prove.

\subsection{Proof that the limit belongs to the period domain}

\newpar
It remains to argue that $e^{-N} F_H \in D$, and hence that 
\[
	\PhiSH(z) = e^{\half \log \abs{\Re z} \, H} e^{-\half(z-\zb) (S+N)} \Phi(z)
\]
actually converges to a polarized Hodge structure on $V$. From the Hodge norm
estimates, we know that the corresponding inner products
\[
	\bigl\langle v, \, w \bigr\rangle_{\PhiSH(z)} =
	\Bigl\langle
		e^{\half(z-\zb) (S+N)} e^{-\half \log \abs{\Re z} \, H} v, \,
		e^{\half(z-\zb) (S+N)} e^{-\half \log \abs{\Re z} \, H} w 
	\Bigr\rangle_{\Phi(z)}
\]
remain \emph{bounded} as $\abs{\Re z} \to \infty$, for every $v,w \in V$.
Equivalently, if we fix a norm $\norm{-}$ on the vector space $V$, we have an upper
bound
\[
	\bigl\lVert v \bigr\rVert_{\PhiSH(z)}^2 = 
	\bigl\lVert e^{\half(z-\zb) (S+N)} e^{-\half \log \abs{\Re z} \, H} v 
	\bigr\rVert_{\Phi(z)}^2 \leq C \norm{v}^2
\]
for some constant $C > 0$. To prove the convergence, we need a lower bound. This can
easily be obtained with the help of the following trick.

\newpar
Fix a basis $v_1, \dotsc, v_r \in V$, and consider the $r \times r$-matrix $M(z)$ with
entries
\[
	\Bigl\langle
		e^{\half(z-\zb) (S+N)} e^{-\half \log \abs{\Re z} \, H} v_i, \,
		e^{\half(z-\zb) (S+N)} e^{-\half \log \abs{\Re z} \, H} v_j 
	\Bigr\rangle_{\Phi(z)}^2
\]
We need to bound the inverse matrix $M(z)^{-1}$, and since we know that all entries
of $M(z)$ stay bounded as $\abs{\Re z} \to \infty$, we only need to control the
function $\det M(z)$. But the determinant can be computed in a different way.
The wedge product
\[
	v_1 \wedge \dotsb \wedge v_r \in \det V
\]
is a multi-valued flat section of the variation of Hodge structure on $\det E$. Since
$\rk(\det E) = 1$, we have, with the appropriate notation,
\begin{align*}
	H(v_1 \wedge \dotsb \wedge v_r) &= N(v_1 \wedge \dotsb \wedge v_r) = 0, \\
	S(v_1 \wedge \dotsb \wedge v_r) &= \alpha \cdot v_1 \wedge \dotsb \wedge v_r)
\end{align*}
for some $\alpha \in \RR$. Therefore the exponential factors act trivially and
\[
	\det M(z) = \norm{v_1 \wedge \dotsb \wedge v_r}_{\Phi(z)}^2.
\]
But now the variation of Hodge structure on $\det E$ can only have a single Hodge
type $(p,q)$, with $p+q = kr$. Therefore the Hodge norm of $v_1 \wedge \dotsb \wedge
v_r$ is equal to
\[
	\norm{v_1 \wedge \dotsb \wedge v_r}_{\Phi(z)}^2 
	= (-1)^q Q(v_1 \wedge \dotsb \wedge v_r, v_1 \wedge \dotsb \wedge v_r),
\]
which is $(-1)^q$ times the determinant of the $r \times r$-matrix with entries
$Q(v_i,v_j)$. Anyway, the conclusion is that $\det M(z)$ is a nonzero constant. This
implies that the inverse matrix $M(z)^{-1}$ is also bounded, and hence that
\[
	\frac{1}{C} \norm{v}^2 \leq 
	\bigl\lVert e^{\half(z-\zb) (S+N)} e^{-\half \log \abs{\Re z} \, H} v 
	\bigr\rVert_{\Phi(z)}^2 \leq C \norm{v}^2
\]
for a suitable constant $C > 0$ and all $v \in V$.

\newpar
Now we get the result that we want from the following lemma. Pretty much the same
argument also appears in \cite[Lem~4.2.1]{Kashiwara}.

\begin{plem}
	Let $f \colon \NN \to D$ be a sequence of points of $D$ such that:
	\begin{enumerate}
		\item The limit $\lim_{m \to \infty} f(m)$ exists in $\Dch$.
		\item There is a constant $C > 0$ such that
			\[
				\frac{1}{C} \norm{v}^2 \leq \norm{v}_{f(m)}^2 \leq C \norm{v}^2
			\]
			for all vectors $v \in V$.
	\end{enumerate}
	Then $\lim_{m \to \infty} f(m) \in D$.
\end{plem}

\begin{proof}
	For each $m \in \NN$, we have the Hodge decomposition
	\[
		V = \bigoplus_{p+q=n} V_{f(m)}^{p,q}.
	\]
	After passing to a subsequence, we can assume that each limit
	\[
		W^{p,q} = \lim_{m \to \infty} V_{f(m)}^{p,q}
	\]
	exists (by compactness of the Grassmannian). We have to prove that
	\[
		V = \bigoplus_{p+q=n} W^{p,q},
	\]
	and that this Hodge structure of weight $n$ is polarized by the pairing $Q$. Since
	the different subspaces $W^{p,q}$ are obviously orthogonal to each other under
	$Q$, it suffices to show that $(-1)^q Q$ is positive definite on $W^{p,q}$. By hypothesis,
	\[
		(-1)^q Q(v,v) = \norm{v}_{f(m)}^2 \geq \frac{1}{C} \norm{v}^2
	\]
	for all $v \in V_{f(m)}^{p,q}$. After passing to the limit, the same is then
	true for $v \in W^{p,q}$, which means that $(-1)^q Q$ is positive definite. It is easy
	to deduce from this that
	\[
		W^{p,q} \cap W^{p',q'} = \{0\}
	\]
	whenever $(p,q) \neq (p',q')$, and so we do get a polarized Hodge structure.
\end{proof}

\newpar
We can use the effective estimates from the proof of the nilpotent orbit theorem to
get a better bound for the distance between 
\[
	\PhiSH(z) 
	= e^{\half \log \abs{\Re z} \, H} e^{-\half(z-\zb) (S+N)} \Phi(z) \in D
\]
and its limit $e^{-N} F_H \in D$. According to
\Cref{thm:approximation-nilpotent-orbit}, we have
\[
	d_D \bigl( \Phi(z), e^{zN} \Flim \bigr) \leq C \abs{\Re z}^m e^{-\delta(T)
	\abs{\Re z}},
\]
for every $z \in \HH$ with $\Re z \leq x_0$, with constants $C > 0$ and $x_0 < 0$
that are basically independent of the period mapping in question. Since the
exponential factors in the definition of $\PhiSH(z)$ belong to the real Lie
group $G$, it follows that
\[
	d_D \Bigl( \PhiSH(z), \, e^{-N} e^{\half \log \abs{\Re z} \, H} \Flim \Bigr)
	\leq C \abs{\Re z}^m e^{-\delta(T) \abs{\Re z}}.
\]
Here we used \eqref{eq:AdY-on-N} to simplify the second argument. Because $H$ has
integer eigenvalues, $e^{\half \log\abs{\Re z} \, H} \Flim$ converges to $F_H$ at a
rate proportional to $\abs{\Re z}^{-\half}$. This means that there is a constant $C' >
0$, whose value depends on $\Flim$, such that
\[
	d_D \bigl( \PhiSH(z), e^{-N} F_H \bigr) \leq C' \abs{\Re z}^{-\half}.
\]
Schmid's $\SL(2)$-orbit theorem gives much more precise information about the
convergence; we are going to prove a weak form of this result later, in
\Cref{chap:asymptotic}.

\section{Results about mixed Hodge structures}

\newpar
In this chapter, we review some basic facts about complex mixed Hodge structures. We
also discuss Deligne's functorial splitting for the weight filtration; we are going
to need it during the proof of the cheap $\SL(2)$-orbit theorem in
\Cref{chap:asymptotic}. Everything in this chapter is just (somewhat tedious) linear
algebra, but the results are important for the study of degenerating variations of
Hodge structure, especially in higher dimensions.

\subsection{Complex mixed Hodge structures}

\newpar
Unlike in the real case, we need \emph{three} filtrations to describe a mixed Hodge
structure, because the Hodge decomposition in a complex Hodge structure is not
determined by the Hodge filtration alone.

\begin{pdfn}
	A \define{mixed Hodge structure} on a finite-dimensional complex vector space $V$
	consists of an increasing filtration $W_{\bullet}$ with $W_n = 0$ for $n \ll 0$
	and $W_n = V$ for $n \gg 0$, and two decreasing filtrations
	$F^{\bullet}$ and $\Fb^{\bullet}$, such that each subquotient
	\[
		\gr_n^W = W_n/W_{n-1}
	\]
	has a Hodge structure of weight $n$, given by the two induced filtrations
	\begin{align*}
		F^{\bullet} \gr_n^W &= (F^{\bullet} \cap W_n + W_{n-1})/W_{n-1} \\
		\Fb^{\bullet} \gr_n^W &= (\Fb^{\bullet} \cap W_n + W_{n-1})/W_{n-1}.
	\end{align*}
	The filtration $W_{\bullet}$ is called the \define{weight filtration}.
\end{pdfn}

\newpar
To save space, we are going to use the shorthand notation
\[
	F^p W_n = F^p \cap W_n \quad \text{and} \quad
	\Fb^q W_n = \Fb^q \cap W_n
\]
from now on. The $(p,q)$-subspace in the Hodge decomposition of $\gr_n^W$ is 
\[
	F^p \gr_n^W \cap \, \Fb^q \gr_n^W 
	= \frac{(F^p W_n + W_{n-1}) \cap (\Fb^q W_n + W_{n-1})}{W_{n-1}}.
\]
In order to have a mixed Hodge structure on $V$, the direct sum of these subspaces
(over $p+q=n$) must equal $\gr_n^W$, which means concretely that
\[
	W_n = \sum_{p+q=n} (F^p W_n + W_{n-1}) \cap (\Fb^q W_n + W_{n-1}) 
\]
and that, whenever $p+q > n$, one has
\[
	(F^p W_n + W_{n-1}) \cap (\Fb^q W_n + W_{n-1}) = W_{n-1}.
\]
From this, one deduces, by induction on the length of the weight filtration, that 
\begin{equation} \label{eq:MHS-vanishing}
	F^p W_n \cap \Bigl( \Fb^q W_n + \Fb^{q-1} W_{n-1} +
	\Fb^{q-2} W_{n-2} + \dotsb \Bigr) = 0
\end{equation}
whenever $p + q > n$. Either way, this is a fairly complicated set of conditions. 

\begin{pexa}
A mixed Hodge structure is called \define{real} if $V = \VR \otimes_{\RR} \CC$ for an
$\RR$-vector space $\VR$, and if $\Fb = \sigma(F)$ and $W = \sigma(W)$, where $\sigma
\in \End_{\RR}(V)$ is the conjugation operator $\sigma(v \tensor z) = v \tensor \zb$.
Since the weight filtration is defined over $\RR$, each subquotient $\gr_n^W$ is a
real Hodge structure of weight $n$.
\end{pexa}

\newpar 
Here are some general operations on mixed Hodge structures:

\begin{enumerate}
	\item 
	Let $V$ be a mixed Hodge structure. For any $k \in \ZZ$, we obtain a new mixed
	Hodge structure $V(k)$ on the same underlying vector space by setting
	\[
		W_n V(k) = W_{n+2k} V, \quad
		F^p V(k) = F^{p-k} V, \quad
		\Fb^q V(k) = \Fb^{q-k} V
	\]
	for $n,p,q \in \ZZ$. This operation is called the \define{Tate twist}.
	\item
	Let $V$ be a mixed Hodge structure. The \define{dual} vector space $\Hom(V,
	\CC)$ inherits a mixed Hodge structure, with
	\begin{align*}
		W_n \Hom(V, \CC) &= \menge{f \colon V \to \CC}{f(W_{-n-1}) = 0}, \\
		F^p \Hom(V, \CC) &= \menge{f \colon V \to \CC}{f(F^{-p+1}) = 0}, \\
		\Fb^q \Hom(V, \CC) &= \menge{f \colon V \to \CC}{f(\Fb^{-q+1}) = 0}.
	\end{align*}
	It is easy to see that $\gr_n^W \Hom(V,\CC) \cong \Hom \bigl( \gr_{-n}^W
	V, \CC \bigr)$ are isomorphic Hodge structures of weight $n$.
	\item 
	Let $V_1$ and $V_2$ be mixed Hodge structures. Their \define{tensor product} $V_1 \tensor
	V_2$ is again a mixed Hodge structure, with
	\begin{align*}
		W_n(V_1 \tensor V_2) &= \sum_{n' + n'' = n} W_{n'} V_1 \tensor W_{n''} V_2, \\
		F^p(V_1 \tensor V_2) &= \sum_{p' + p'' = p} F^{p'} V_1 \tensor F^{p''} V_2, \\
		\Fb^q(V_1 \tensor V_2) &= \sum_{q' + q'' = q} \Fb^{q'} V_1 \tensor \Fb^{q''} V_2.
	\end{align*}
	As expected, one has an isomorphism of Hodge structures
	\[
		\gr_n^W(V_1 \tensor V_2) \cong \bigoplus_{n' + n'' = n}
		\gr_{n'}^W V_1 \otimes \gr_{n''}^W V_2.
	\]
	\item An important special case is the \define{endomorphism algebra} $\End(V)$ of the
	vector space $V$ underlying a mixed Hodge structure. It has a mixed Hodge structure with
	\begin{align*}
		W_n \End(V) &= \menge{f \colon V \to V}{f(W_{\bullet}) \subseteq W_{\bullet + n}}, \\
		F^p \End(V) &= \menge{f \colon V \to V}{f(F^{\bullet}) \subseteq F^{\bullet + p}}, \\
		\Fb^q \End(V) &= \menge{f \colon V \to V}{f(\Fb^{\bullet}) \subseteq \Fb^{\bullet + q}}.
	\end{align*}
	With these definitions, the multiplication map $\End(V) \tensor \End(V) \to
	\End(V)$ becomes a morphism of mixed Hodge structures.
\end{enumerate}

\subsection{Deligne's splitting construction}

\newpar
	A mixed Hodge structure is called \define{split} if it is a direct sum of Hodge
	structures of different weights, with the obvious weight filtration. For example,
	any $\sltwo$-Hodge structure is a split mixed Hodge structure. Being
	split is equivalent to having a decomposition
	\[
		V = \bigoplus_{i,j \in \ZZ} V^{i,j}
	\]
	with the property that
	\[
		W_n = \bigoplus_{i+j\leq n} V^{i,j}, \quad
		F^p = \bigoplus_{i \geq p,j } V^{i,j}, \quad
		\Fb^q = \bigoplus_{j \geq q, i} V^{i,j}.
	\]
	If $V$ is a mixed Hodge structure, then the associated graded
	\[
		\gr_{\bullet}^W = \bigoplus_{n \in \ZZ} \gr_n^W
	\]
	becomes in a natural way a split mixed Hodge structure.

\newpar \label{par:Ipq}
A general mixed Hodge structure is of course not split. Nevertheless, Deligne
\cite[Lem.~1.2.11]{Deligne} showed that there is a functorial decomposition 
\[
	V = \bigoplus_{i,j \in \ZZ} I^{i,j}
\]
into subspaces, with the property that
\begin{equation} \label{eq:Deligne-decomposition}
	W_n = \bigoplus_{i+j \leq n} I^{i,j} \qquad \text{and} \qquad
	F^p = \bigoplus_{i \geq p, j} I^{i,j},
\end{equation}
and such that $I^{i,j}$ maps isomorphically to the $(i,j)$-subspace in the Hodge
decomposition of $\gr_{i+j}^W$, under the projection from $W_{i+j}$ to
$\gr_{i+j}^W$. The subspaces in question are defined by the formula
\begin{equation} \label{eq:Deligne}
	I^{i,j} = F^i \cap W_{i+j} \cap 
	\bigl( \Fb^j \cap W_{i+j} + \Fb^{j-1} \cap W_{i+j-2} + \Fb^{j-2} \cap
	W_{i+j-3} + \dotsb \bigr).
\end{equation}
The decomposition is functorial in the following sense. Suppose that $f \colon V_1 \to
V_2$ is a morphism of mixed Hodge structures, which means that
\[
	f(W_n V_1) \subseteq W_n V_2, \quad f(F^p V_1) \subseteq F^p V_2, \quad
	f(\Fb^q V_1) \subseteq \Fb^q V_2
\]
for all $p,q,n \in \ZZ$. Then with the obvious notation, we have $f(I_1^{i,j})
\subseteq I_2^{i,j}$ for all $i,j \in \ZZ$. This is an immediate consequence of the
formula for $I^{i,j}$. 

\begin{pexa}
	When the mixed Hodge structure is split, one has $I^{i,j} = F^i \cap W_{i+j} \cap
	\Fb^j$ and
	\[
		\Fb^q = \bigoplus_{j \geq q, i} I^{i,j}.
	\]
	In the case of a non-split mixed Hodge structure, there is no simple formula for
	$\Fb$ in terms of the subspaces $I^{i,j}$.
\end{pexa}

\newpar
Let us try to understand Deligne's construction from a slightly different point of
view. To get a decomposition of $V$, all we need is a \define{splitting} $H \in
\End(V)$ for the weight filtration; by this we mean that $H$ is semisimple with integer
eigenvalues, and that $W_{n} = E_{n}(H) \oplus W_{n-1}$ for all $n \in
\ZZ$. Such a splitting induces an isomorphism
\[
	V = \bigoplus_{n \in \ZZ} E_{n}(H) \, \cong \, \bigoplus_{n \in \ZZ} \gr_{n}^W,
\]
and the desired decomposition of $V$ is then simply the image of the Hodge
decomposition in the direct sum of Hodge structures on the right-hand side. 

\newpar
Any splitting preserves the weight filtration, and acts on $\gr_n^W$ as multiplication
by $n$. This condition actually characterizes splittings, as the following lemma
shows.

\begin{plem} \label{lem:splittings}
	Let $V$ be a mixed Hodge structure.
	\begin{aenumerate}
	\item Suppose that $H \in W_0 \End(V)$ acts on each subquotient $\gr_n^W$ as
		multiplication by $n$. Then $H$ is a splitting of the weight filtration.
	\item If $H$ and $H'$ are two splittings of the weight filtration, then there is a
		unique element $g \in \GL(V)$ such that $g - \id \in W_{-1} \End(V)$ and $H' =
		g H g^{-1}$.
	\end{aenumerate}
\end{plem}

\begin{proof}
	We first prove (a). For any $n \in \ZZ$, we have $(H - n \id)(W_n) \subseteq
	W_{n-1}$ because of the assumptions on $H$. If $W_{n_0-1} = 0$ and $W_{n_1} = V$, we get
	\[
		\prod_{n = n_0}^{n_1} (H - n \id) = 0,
	\]
	and so $H$ is semisimple with integer eigenvalues. It also follows that $W_n =
	E_n(H) \oplus W_{n-1}$, and so $H$ is indeed a splitting.

	Next, we prove (b). There is a unique element $g \in \GL(V)$ that maps each
	eigenspace $E_n(H)$ into the eigenspace $E_n(H')$ and makes all the following diagrams
	commute:
	\[
		\begin{tikzcd}
			& E_n(H') \dar{\cong} \\
			E_n(H) \rar{\cong} \urar[bend left=20]{g} & \gr_n^W
		\end{tikzcd}
	\]
	Clearly then $H' = g H g^{-1}$, and also $g - \id \in W_{-1} \End(V)$, as claimed.
\end{proof}

\newpar
Now we return to Deligne's construction.  Using the mixed Hodge structure on
$\End(V)$, we can express the condition in \eqref{eq:Deligne-decomposition} as
\[
	H \in F^0 W_0 \End(V) = F^0 \End(V) \cap W_0 \End(V).
\]
Deligne's result then takes the following form.

\begin{pprop} \label{prop:Deligne}
	Let $V$ be a mixed Hodge structure. There is a unique splitting $H \in
	\End(V)$ for the weight filtration with the property that
	\[
		H \in F^0 W_0 \End(V) \cap \Bigl( \Fb^0 W_0 \End(V) + \Fb^{-1} W_{-2} \End(V) +
		\Fb^{-2} W_{-3} \End(V) + \dotsb \Bigr).
	\]
	Moreover, if $v \in F^p \cap W_{p+q} \cap \Fb^q$, then one has $Hv = (p+q) v$.
\end{pprop}

\begin{proof}
	We first prove existence. Consider the endomorphism $h$ of
	\[
		\gr_{\bullet}^W = \bigoplus_{n \in \ZZ} \gr_{n}^W
	\]
	that acts as multiplication by $n$ on the summand $\gr_{n}^W$. In terms of
	the mixed Hodge structure on $\End(V)$, this endomorphism is an element of
	$\gr_0^W \End(V)$, of Hodge type $(0,0)$. We start by lifting $h$ in an arbitrary
	way to an element $H \in F^0 W_0 \End(V)$; according to \Cref{lem:splittings}, this
	is a splitting of the weight filtration. We are now going to adjust this
	initial choice.

	We can also lift $h$ to an element $\Hb_0 \in \Fb^0 W_0 \End(V)$. Then $H - \Hb_0 \in
	W_{-1} \End(V)$, and we consider its image in $\gr_{-1}^W \End(V)$, which has a
	Hodge structure of weight $-1$. As such,
	\[
		\gr_{-1}^W \End(V) = F^0 \gr_{-1}^W \End(V) \, \oplus \, \Fb^0 \gr_{-1}^W \End(V),
	\]
	and so we can find $H_{-1} \in F^0 W_{-1} \End(V)$ and $\Hb_{-1}
	\in \Fb^0 W_{-1} \End(V)$ such that $H - \Hb_0$ and $H_{-1} + \Hb_{-1}$ have the
	same image in $\gr_{-1}^W \End(V)$. If we replace $H$ by $H - H_{-1}$, and
	$\Hb_0$ by $\Hb_0 + \Hb_{-1}$, we can arrange that $H - \Hb_0 \in W_{-2} \End(V)$.
	Next, we project into 
	\[
		\gr_{-2}^W \End(V) = F^0 \gr_{-2}^W \End(V) \, \oplus \,
		\Fb^{-1} \gr_{-2}^W \End(V).
	\]
	As before, we can find $H_{-2} \in F^0 W_{-2} \End(V)$ and $\Hb_{-2}
	\in \Fb^{-1} W_{-2} \End(V)$ such that $H - \Hb_0$ and $H_{-2} + \Hb_{-2}$ have
	the same image in $\gr_{-2}^W \End(V)$; after replacing $H$ by $H - \Hb_0$, we get
	$H - \Hb_0 - \Hb_{-2} \in W_{-3} \End(V)$. Continuing in this manner, we obtain
	\[
		H = \Hb_0 + \Hb_{-2} + \Hb_{-3} + \dotsb, 
	\]
	where $\Hb_k \in \Fb^{k+1} W_k \End(V)$ for $k \leq -2$. This is the desired
	splitting. It is clear from this argument that the splitting $H$ is optimal among
	all splittings that preserve $F$, in the sense that the subspace of $\End(V)$ in
	the definition is as small as possible.

	Uniqueness of $H$ follows in the same way. If we have two splittings $H,H' \in
	\End(V)$ with the stated properties, then their difference $H - H'$ is an element of
	\[
		F^0 W_{-1} \End(V) \cap \Bigl( \Fb^0 W_0 \End(V) + \Fb^{-1} W_{-2} \End(V) +
		\Fb^{-2} W_{-3} \End(V) + \dotsb \Bigr).
	\]
	But since $\End(V)$ is a mixed Hodge structure, this intersection is trivial by
	\eqref{eq:MHS-vanishing}. 

	Finally, let us suppose that we have a vector $v \in F^p \cap \Fb^q \cap W_n$,
	where $n = p+q$. Then 
	\[
		Hv - n v \in F^p \cap W_{n-1} \cap \Bigl( \Fb^q W_{n-1} + \Fb^{q-1} W_{n-2} +
		\Fb^{q-2} W_{n-3} + \dotsb \Bigr),
	\]
	and as the intersection on the right-hand side is again trivial, it follows that $Hv
	= nv$.
\end{proof}

\newpar
We are going to denote the splitting in \Cref{prop:Deligne} by the symbol $H(W, F,
\Fb)$ if we want to emphasize its dependence on the mixed Hodge structure. 

\newpar
It is now an easy matter to deduce Deligne's formula for the subspaces $I^{i,j}$ in
\eqref{eq:Deligne}. Let $n=p+q$, and suppose that we have a vector of Hodge type
$(p,q)$ in $\gr_n^W$. We can lift it to $v \in F^p W_n$ and also to $\vb \in \Fb^q
W_n$. Let $P_n \colon V \to E_n(H)$ denote the projection to the $n$-eigenspace of
$H$; since $H$ is semisimple with integer eigenvalues, this equals
\[
	P_n = \prod_{k \neq n} \frac{1}{n-k} (H - k \id),
\]
where the product runs over all eigenvalues of $H$ different from $n$. As $v - \vb
\in W_{n-1}$, we have $P_n v = P_n \vb$. Now $v \in F^p \cap W_n$, and therefore
\[
	P_n v \in F^p \cap W_n.
\]
On the other hand, we have $\vb \in \Fb^q \cap W_n$, and therefore
\[
	P_n \vb \in \Fb^q \cap W_n + \Fb^{q-1} \cap W_{n-2} + \Fb^{q-2} \cap W_{n-3} +
	\dotsb,
\]
because of the defining properties of $H$ in \Cref{prop:Deligne}. This proves that
$P_n v \in I^{p,q}$. The conclusion is that, in the decomposition of $V$ induced by
the splitting $H$, each subspace is contained in some $I^{p,q}$. 

\newpar
Conversely, let us show that $H$ acts as multiplication by $i+j$ on the subspace
$I^{i,j}$. Set $n = i+j$. For any vector $v \in I^{i,j}$, we have $v \in
F^i \cap W_n$, hence
\[
	Hv - n v \in F^i \cap W_{n-1}.
\]
At the same time, the formula in \eqref{eq:Deligne} and the properties of $H$ imply
that
\[
	Hv - nv \in \Fb^j W_{n-1} + \Fb^{j-1} W_{n-2} + \Fb^{j-2} W_{n-3} + \dotsb
\]
Since the intersection between these two subspaces is trivial by
\eqref{eq:MHS-vanishing}, we get $Hv = nv$, as claimed. This is enough to conclude
that, under the isomorphism
\[
	V = \bigoplus_{n \in \ZZ} E_n(H) \, \cong \, \bigoplus_{n \in \ZZ} \gr_n^W,
\]
the subspace of type $(i,j)$ in the Hodge decomposition on the right-hand side
corresponds exactly to the subspace $I^{i,j}$.

\newpar
Let us also check that Deligne's construction respects the basic operations on mixed
Hodge structures; this point is not covered very well in the literature
\cite[Prop.~1.9]{Morgan}.

\begin{pprop} \label{prop:Deligne-dual-tensor}
	The splitting in \Cref{prop:Deligne} is functorial; it is also compatible with
	duals and tensor products.
\end{pprop}

\begin{proof}
	Let us first consider tensor products. Suppose that $V_1$ and $V_2$ are two mixed
	Hodge structures, and set $V = V_1 \tensor V_2$. Denote the three splittings
	coming from \Cref{prop:Deligne} by $H_1 \in \End(V_1)$, $H_2 \in \End(V_2)$, and
	$H \in \End(V)$. Then the claim is
	that $H = H_1 \tensor \id + \id \tensor H_2$. Now the operator $H_1 \tensor \id +
	\id \tensor H_2 \in \End(V)$ acts on 
	\[
		\gr_n^W V \cong \bigoplus_{n'+n'' = n} \gr_{n'}^W V_1 \tensor \gr_{n''}^W V_2
	\]
	as multiplication by $n$; since it is also semisimple with integer eigenvalues, it
	is a splitting of the weight filtration on $V$. From the definition of the mixed
	Hodge structure on $V = V_1 \tensor V_2$, it is easy to see that both $H_1 \tensor
	\id$ and $\id \tensor H_2$ belong to 
	\[
		F^0 W_0 \End(V) \cap \Bigl( \Fb^0 W_0 \End(V) + \Fb^{-1} W_{-2} \End(V) +
		\Fb^{-2} W_{-3} \End(V) + \dotsb \Bigr).
	\]
	The result we want now follows from the uniqueness statement in
	\Cref{prop:Deligne}.

	The same kind of argument proves that if $V$ is a mixed Hodge structure, and if
	$H \in \End(V)$ denotes the splitting in \Cref{prop:Deligne}, then the splitting
	of the mixed Hodge structure on $\Hom(V, \CC)$ is given by the formula $f \mapsto
	-f \circ H$.

	Lastly, let us prove the functoriality of the splitting. Suppose that $f \colon
	V_1 \to V_2$ is a morphism of mixed Hodge structures. Let $H_1 \in \End(V_1)$ and
	$H_2 \in \End(V_2)$ denote the two splittings. By the above, the splitting of the
	mixed Hodge structure on $\Hom(V_1, V_2)$ is then given by $Hf = H_2 \circ f - f
	\circ H_1$. Since $f \in \Hom(V_1, V_2)$ is a morphism, it satisfies
	\[
		f \in F^0 \Hom(V_1, V_2) \cap \Fb^0 \Hom(V_1, V_2) \cap W_0 \Hom(V_1, V_2).
	\]
	By \Cref{prop:Deligne}, it follows that $Hf = 0$, which translates into $H_2 \circ
	f = f \circ H_1$. 
\end{proof}

\newpar
What about complex conjugation? If we swap the role of the two filtrations $F$ and
$\Fb$, we obtain another splitting $\Hb = H(W, \Fb, F)$, whose characteristic property
is that
\[
	\Hb \in \Fb^0 W_0 \End(V) \cap \Bigl( F^0 W_0 \End(V) + F^{-1} W_{-2} \End(V) +
	F^{-2} W_{-3} \End(V) + \dotsb \Bigr).
\]
The corresponding subspaces in the decomposition of $V$ are then
\[
	\Ib^{i,j} = \Fb^j \cap W_{i+j} \cap 
	\bigl( F^i \cap W_{i+j} + F^{i-1} \cap W_{i+j-2} + F^{i-2} \cap
	W_{i+j-3} + \dotsb \bigr).
\]
We have $H = \Hb$ if and only if $H \in \Fb^0 \End(V)$ if and only if the mixed Hodge
structure is split; we leave the proof of this assertion as an easy exercise.

\newpar
One can also understand the result in \Cref{prop:Deligne} by comparing the two
splittings $H$ and $\Hb$. We have $\Hb - H \in W_{-2} \End(V)$, and therefore
\begin{align*}
	\Hb - H \in & \bigl( F^{-1} W_{-2} \End(V) + F^{-2} W_{-3} \End(V) + \dotsb \bigr)
	\cap \\
	& \bigl( \Fb^{-1} W_{-2} \End(V) + \Fb^{-2} W_{-3} \End(V) + \dotsb \bigr).
\end{align*}
This element functions as a sort of ``extension class'' for the mixed Hodge structure
(because $V$ is split if and only if $\Hb - H = 0$). 

\newpar
Given a mixed Hodge structure, let us consider the subspace
\[
	R(W_{-2}) = \bigl( F^{-1} W_{-2} + F^{-2} W_{-3} + \dotsb \bigr) \cap
	\bigl( \Fb^{-1} W_{-2} + \Fb^{-2} W_{-3} + \dotsb \bigr).
\]
The next proposition describes its properties.

\begin{pprop} \label{prop:MHS-R}
	In any mixed Hodge structure, one has a decomposition
	\[
		W_{-1} = F^0 W_{-1} \oplus \Fb^0 W_{-1} \oplus R(W_{-2}).
	\]
	The subspace $R(W_{-2})$ is a sub-mixed Hodge structure, and
	\[
		R(W_{-2}) 
		= \bigoplus_{i, j \leq -1} I^{i,j}
		= \bigoplus_{i, j \leq -1} \Ib^{i,j}.
	\]
\end{pprop}

\begin{proof}
	The formula in \eqref{eq:Deligne} shows that $I^{i,j} \subseteq R(W_{-2})$ for
	$i,j \leq -1$; by symmetry, it follows that $\Ib^{i,j} \subseteq R(W_{-2})$ as
	well. Now remember that $I^{i,j}$ maps isomorphically to the subspace of type
	$(i,j)$ in the Hodge decomposition of $\gr_{i+j}^W$. It follows readily that
	\[
		W_n = F^0 W_n + \Fb^0 W_n + W_n \cap R(W_{-2})
	\]
	for every $n \leq -1$. Indeed, the subspace on the right-hand side maps onto
	$\gr_n^W$ under the projection $W_n \to \gr_n^W$, because its image contains
	the subspace of type $(i,j)$ in the Hodge decomposition as long as either $i \geq
	0$, or $j \geq 0$, or $i,j \leq -1$. At the same time, it
	is easy to see that the three subspaces are linearly independent. Indeed, if we
	have $v + \vb + w = 0$ with $v \in F^0 W_{-1}$, $\vb \in F^0 W_{-1}$, and $w \in
	R(W_{-2})$, then 
	\[
		v = -(\vb + w) \in F^0 W_{-1} \cap \bigl( \Fb^0 W_{-1} + \Fb^{-1} W_{-2} +
		\Fb^{-2} W_{-3} + \dotsb \bigr),
	\]
	and this intersection is trivial by \eqref{eq:MHS-vanishing}. By symmetry, we get $v =
	\vb = w = 0$. Since
	\[
		W_{-1} = \bigoplus_{i+j \leq -1} I^{i,j} 
		= F^0 W_{-1} \oplus \bigoplus_{\substack{i+j \leq -1 \\ i \leq -1}} I^{i,j}
		= \Fb^0 W_{-1} \oplus \bigoplus_{\substack{i+j \leq -1 \\ i \leq -1}} \Ib^{i,j},
	\]
	this also proves the nice identity
	\[
		R(W_{-2}) 
		= \bigoplus_{i, j \leq -1} I^{i,j}
		= \bigoplus_{i, j \leq -1} \Ib^{i,j}.
	\]
	Because of the formulas for $F$ and $W$ in \eqref{eq:Deligne-decomposition}, and a
	similar formula for $\Fb$ in terms of the subspaces $\Ib^{i,j}$, this is enough to
	conclude that $R(W_{-2})$ is a sub-mixed Hodge structure.
\end{proof}

\newpar
With this notation, we have $\Hb - H \in R \bigl( W_{-2} \End(V) \bigr)$. According
to \Cref{lem:splittings}, there is a unique element $g \in \GL(V)$ such that
\begin{equation} \label{eq:Hb-H}
	\Hb = g H g^{-1} \quad \text{and} \quad
	g - \id \in R \bigl( W_{-2} \End(V) \bigr).
\end{equation}
This gives another way to think about the characteristic property of Deligne's splitting $H$.

\newpar
In a Hodge structure of weight $n$, one has $V = F^0 \oplus \Fb^{n+1}$. It turns
out that a mixed Hodge structure still determines a canonical complement for the
subspace $F^0$. 

\begin{pprop} \label{prop:F0-complement}
	In any mixed Hodge structure, one has
	\[
		V = F^0 \oplus \sum_{n \in \ZZ} \Fb^{n+1} W_n.
	\]
	In terms of Deligne's decomposition, one has
	\[
		\sum_{n \in \ZZ} \Fb^{n+1} W_n = \bigoplus_{i \leq -1, j} I^{i,j}.
	\]
\end{pprop}

\begin{proof}
	From the formula for $I^{i,j}$ in \eqref{eq:Deligne}, it is clear that $I^{i,j}
	\subseteq \sum_n \Fb^{n+1} W_n$ as long as $i \leq -1$. Since we already know from
	\eqref{eq:Deligne-decomposition} that
	\[
		V = F^0 \oplus \bigoplus_{i \leq -1, j} I^{i,j},
	\]
	it is therefore enough to prove that $F^0 \cap \sum_n \Fb^{n+1} W_n = 0$. This is
	easily checked by projecting to the subquotients $\gr_n^W$ in descending order.
\end{proof}

\subsection{Mixed Hodge structures and hermitian pairings}

\newpar
Let us briefly discuss what happens when the vector space $V$ comes equipped with a
nondegenerate hermitian pairing $Q \colon V \tensor \Vb \to \CC$. The pairing induces
an isomorphism 
\[
	Q \colon V \to \Hom(\Vb, \CC), \quad v \mapsto Q(v, \argbl),
\]
between $V$ and the conjugate dual vector space. We make the assumption that 
\[
	Q \colon V \to \Hom(\Vb, \CC)(-n)
\]
is actually an isomorphism of mixed Hodge structures (for a certain integer $n \in \ZZ$).
Since the role of the two filtrations $F$ and $\Fb$ get swapped in the mixed Hodge
structure on $\Vb$, this amounts concretely to the following two conditions:
\begin{enumerate}
	\item We have $W_k = \menge{v \in V}{\text{$Q(v,x) = 0$ for all $x \in W_{2n-k-1}$}}$.
	\item We have $\Fb^q = \menge{v \in V}{\text{$Q(v,x) = 0$ for all $x \in
		F^{n-q+1}$}}$.
\end{enumerate}
They imply that, for every $k \in \ZZ$, the hermitian pairing $Q$ induces an
isomorphism of Hodge structures between $\gr_k^W$ and the conjugate dual of
$\gr_{2n-k}^W$ (with a Tate twist by $n$).

\newpar
Just as in \Cref{par:glie}, the presence of the pairing turns $\End(V)$ into a real
mixed Hodge structure. The real structure is given by
\[
	\glie = \menge{A \in \End(V)}{A^{\dagger} = -A},
\]
the Lie algebra of the real Lie group $G$. The ``complex conjugate'' of an
endomorphism $A \in \End(V)$ is again $\sigma(A) = -A^{\dagger}$, where the dagger
means the adjoint with respect to $Q$. The two conditions in the previous paragraph
are saying that $\sigma \bigl( W_k \End(V) \bigr) = W_k \End(V)$ and that $\Fb^q
\End(V) = \sigma \bigl( F^q \End(V) \bigr)$, and so $\End(V)$ is indeed a real mixed
Hodge structure.

\newpar
From \Cref{prop:Deligne}, we get a splitting $H \in \End(V)$ of the mixed Hodge
structure. By swapping the role of the two filtrations $F$ and $\Fb$, we get another
splitting $\Hb \in \End(V)$. 

\begin{plem}
	We have $\Hb = 2n \id -H^{\dagger}$, where the dagger means the adjoint with respect to
	$Q$.
\end{plem}

\begin{proof}
	Since the splitting in \Cref{prop:Deligne} is functorial, this follows from the
	fact that $Q \colon V \to \Hom(\Vb, \CC)(-n)$ is an isomorphism of mixed Hodge
	structures.
\end{proof}

\newpar
By averaging the two splittings $H$ and $\Hb = 2n \id -H^{\dagger}$, we can create a
splitting that is ``real'', meaning compatible with the hermitian pairing $Q$. 

\begin{pprop} \label{prop:real-splitting}
	The operator $H_{\RR} = \half(H - H^{\dagger}) \in \glie$ is the unique real
	splitting of the shifted weight filtration $W_{n + \bullet}$ with the property
	that
	\[
		H_{\RR} - (H - n \id) \in R \bigl( W_{-2} \End(V) \bigr).
	\]
\end{pprop}

\begin{proof}
	The operator $\half(H + \Hb) - n \id = \half(H - H^{\dagger}) \in \glie$ preserves
	the weight filtration, and acts on $\gr_{n+k}^W$ as multiplication by $k$; it is
	therefore a splitting of the shifted weight filtration $W_{n+\bullet}$. Recall
	from \Cref{prop:MHS-R} that $\Hb - H \in R \bigl( W_{-2} \End(V) \bigr)$.
	The proof of uniqueness is the same as in \Cref{prop:Deligne}.
\end{proof}

\section{The limiting mixed Hodge structure}

\newpar
The purpose of this chapter is show that the vector space $V$ of multi-valued flat
sections has a mixed Hodge structure on it; following Schmid, we shall call this the
\define{limiting mixed Hodge structure}. Before stating the result, we briefly recall
the definition of a (complex) mixed Hodge structure. 

\subsection{\boldmath The limiting $\sltwo$-Hodge structure}

\newpar
We are going to derive the existence of the limiting mixed Hodge structure from a more
precise result. Recall that, after having chosen a splitting $H \in \End(V)$ for the
monodromy weight filtration, we get a unique representation
\[
	\rho \colon \sltwo(\CC) \to \End(V)
\]
with the property that $\rho(\Hsl) = H$ and $\rho(\Ysl) = -N$. It turns out that this
representation is part of a polarized $\sltwo$-Hodge structure. The ``total'' Hodge
filtration of this $\sltwo$-Hodge structure is obtained from the rescaled period
mapping
\[
	\PhiSH \colon \HH \to D, \quad
	\PhiSH(z) = e^{\half \log \abs{\Re z} H} e^{-\half(z - \zb)(S+N)} \Phi(z).
\]
In \Cref{chap:rescaled}, we showed that $\PhiSH$ converges to a well-defined limit 
\[
	e^{-N} F_H = \lim_{\abs{\Re z} \to \infty} \PhiSH(z) \in D,
\]
which is the Hodge filtration of a polarized Hodge structure of weight $n$ on the
vector space $V$ (with polarization $Q$). The filtration $F_H$ has the property that
\begin{equation} \label{eq:conditions}
	H F_H^{\bullet} \subseteq F_H^{\bullet}, \qquad 
	T_s F_H^{\bullet} \subseteq F_H^{\bullet}, \qquad 
	N F_H^{\bullet} \subseteq F_H^{\bullet-1}.
\end{equation}
Since $e^{\Ysl} F_H$ is the Hodge filtration of a polarized Hodge structure of weight
$n$, the criterion in \Cref{thm:sltwo-filtration} shows that $F_H$ is the
total Hodge filtration of a polarized $\sltwo$-Hodge structure of weight $n$,
polarized by the hermitian pairing $Q$. Moreover, by
\Cref{cor:sltwo-endomorphism}, the semisimple part $T_s$ is necessarily an
endomorphism of this polarized $\sltwo$-Hodge structure.

\newpar
We observed earlier that the conjugate Hodge filtration is $e^{-\Ysl} \Fb_H$, where
\[
	\Fb_H^q = \menge{v \in V}{\text{$Q(v,x) = 0$ for all $x \in F^{n-q+1}$}}.
\]
It follows that the two ``total'' Hodge filtrations of the $\sltwo$-Hodge structure
are $F_H$ and $\Fb_H$, and so the Hodge structure on the weight space $V_k$ is given
by
\[
	V_k^{i,j} = V_k \cap F_H^i \cap \Fb_H^j.
\]
In particular, this says that $\Ysl(V_k^{i,j}) \subseteq V_{k-2}^{i-1,j-1}$ and
$\Xsl(V_k^{i,j}) \subseteq V_{k+2}^{i+1,j+1}$; especially the second inclusion is not
at all obvious from the conditions in \eqref{eq:conditions}.

\subsection{The limiting mixed Hodge structure}

\newpar
In terms of the limiting Hodge filtrations $\Flim$ and $\Fblim$ coming from the
nilpotent orbit theorem, one can describe $F_H$ and its conjugate as
\[
	F_H = \lim_{x \to \infty} e^{xH} \Flim \quad \text{and} \quad
	\Fb_H = \lim_{x \to \infty} e^{xH} \Fblim.
\]
The effect of the limit is that, under the isomorphism $E_k(H) \cong
W_k/W_{k-1}$, one has
\[
	F_H^p \cap E_k(H) \cong 
	\bigl( \Flim^p \cap W_k + W_{k-1} \bigr) / W_{k-1},
\]
and similarly for the conjugate filtration $\Fb_H$. This means that the passage from
$(\Flim, \Fblim)$ to $(F_H, \Fb_H)$ is exactly the passage from a mixed Hodge
structure to its associated graded.

\begin{pcor} \label{cor:limitingMHS}
	The three filtrations $W_{\bullet-n}(N)$, $\Flim$ and $\Fblim$ determine a mixed 
	Hodge structure on the vector space $V$. The associated graded of this mixed Hodge
	structure is a polarized $\sltwo$-Hodge structure of weight $n$. The operators
	\[
		N \colon V \to V(-1) \quad \text{and} \quad
		T_s \colon V \to V
	\]
	are morphisms of mixed Hodge structures.
\end{pcor}

\begin{proof}
	Set $W_k = W_k(N)$. The first assertion is clear because each weight space
	\[
		V_k = E_k(H) \cong W_k/W_{k-1}
	\]
	has a Hodge structure of weight $n+k$. The second assertion follows because we
	have $N(W_k) \subseteq W_{k-2}$ and $T_s(W_k) \subseteq W_k$ for all $k \in \ZZ$,
	and also $N(\Flim^p) \subseteq \Flim^{p-1}$ and $T_s(\Flim^p) \subseteq \Flim^p$
	for all $p \in \ZZ$.
\end{proof}

\subsection{A formula for the period mapping}

\newpar
It is sometimes useful to have a more explicit description of the period mapping. At
least on a small neighborhood of the origin, we can write down a fairly concrete
formula for $\Phi(z)$ with the help of the limiting mixed Hodge structure. For
variations of Hodge structure with unipotent monodromy, this result can be found in
\cite[(2.5)]{CK}.

\newpar
The nilpotent orbit theorem (in \Cref{thm:nilpotent-convergence-intro}) tell us that
\[
	\Psi_S(e^z) = e^{-z(S+N)} \Phi(z)
\]
converges to $\Psi_S(0) \in \Dch$ as $\abs{\Re z} \to \infty$. Since $\Psi_S(t)$ is
holomorphic, we should therefore be able to express $\Psi_S(t)$ in terms of
$\Psi_S(0)$, using functions that are holomorphic on a neighborhood of the origin. To
do this, recall the formula for the holomorphic tangent space
\[
	T_{\Psi_S(0)}^{1,0} \Dch \cong \End(V) / F^0 \End(V)_{\Psi_S(0)}.
\]
We can use the limiting mixed Hodge structure to find a subalgebra $\qlie
\subseteq \End(V)$ such that
\begin{equation} \label{eq:qlie-complement}
	\End(V) = \qlie \oplus F^0 \End(V)_{\Psi_S(0)}.
\end{equation}
Recall that $V$ has a mixed Hodge structure with Hodge filtration $\Flim$ and weight
filtration $W_{\bullet-n}$, and that the semisimple operator $T_s$ is an endomorphism
of this mixed Hodge structure. Consider the induced mixed Hodge structure on
$\End(V)$. According to the general result in \Cref{prop:F0-complement}, the subspace
\[
	\qlie = \sum_{n \in \ZZ} \Fblim^{n+1} W_n \End(V)
\]
is a vector space complement of $\Flim^0 \End(V)$, and hence
\begin{equation} \label{eq:qlie-complement-Flim}
	\End(V) = \qlie \oplus \Flim^0 \End(V).
\end{equation}
It is also clearly a subalgebra, and we have $N \in \qlie$ because $N \colon V \to
V(-1)$ is a morphism of mixed Hodge structures. Moreover, since $S$ is an
endomorphism of the mixed Hodge structure, the operator $\ad S$ preserves $\qlie$;
this gives us a decomposition
\[
	\qlie = \bigoplus_{\alpha \in \RR} \qlie_{\alpha},
\]
where $\qlie_{\alpha} = \qlie \cap E_{\alpha}(\ad S)$. This decomposition is
compatible with the Lie algebra structure on $\End(V)$, in the sense that
$\lbrack \qlie_{\alpha}, \qlie_{\beta} \rbrack \subseteq \qlie_{\alpha + \beta}$ by
the Jacobi identity.

\newpar
Now \eqref{eq:qlie-complement} follows from \eqref{eq:qlie-complement-Flim} by a
formal argument. Let $\alpha_1 > \dotsb > \alpha_m$ be the distinct
eigenvalues of $\ad S$ on $\End(V)$, in decreasing order. Setting $\lambda_k = e^{2
\pi i \alpha_k}$, we have 
\[
	\frac{E_{\alpha_1}(\ad S) \oplus \dotsb \oplus E_{\alpha_{k-1}}(\ad S) \oplus
	E_{\alpha_k}(\ad S)}{E_{\alpha_1}(\ad S) \oplus \dotsb \oplus E_{\alpha_{k-1}}(\ad S)}
	\cong E_{\lambda_k}(\Ad T_s),
\]
and under this isomorphism, the filtration induced by $\Psi_S(0)$ on the left-hand
side goes to the limiting Hodge filtration $\Flim$. We can then derive
\eqref{eq:qlie-complement} by induction on $k=1, \dotsc, m$.

\newpar
Since $T_{\Psi_S(0)}^{1,0} \Dch \cong \qlie$, the exponential mapping
\[
	\exp \colon \qlie \to \Dch, \quad A \mapsto e^A \cdot \Psi_S(0),
\]
restricts to a biholomorphism between a small neighborhood of $0 \in \qlie$ and a
small neighborhood of the point $\Psi_S(0) \in \Dch$. Consequently, there is a
constant $\eps > 0$, whose exact value depends on the period mapping under
consideration, and a unique holomorphic mapping
\[
	\Gammat \colon \Delta_{\eps} \to \qlie
\]
such that $\Gammat(0) = 0$ and $\Psi_S(t) = e^{\Gammat(t)} \cdot \Psi_S(0)$ for
$\abs{t} < \eps$. As long as $\Re z < \log \eps$, this gives us a (preliminary)
formula for the period mapping:
\begin{equation} \label{eq:Phi-formula}
	\Phi(z) = e^{z(S+N)} e^{\Gammat(e^z)} \cdot \Psi_S(0)
\end{equation}
Note that, unlike in \cite[(2.5)]{CK}, the relevant filtration here is \emph{not} the
limiting Hodge filtration $\Flim$, but the filtration $\Psi_S(0)$ coming from the
nilpotent orbit theorem. 

\newpar
Now we would like to derive from \eqref{eq:Phi-formula} a formula for $\Phi(z)$
that only involves the limiting Hodge filtration $\Flim$. The key is the following lemma.

\begin{plem}
	One has $\Psi_S(0) = e^A \Flim$ for a unique element $A \in \qlie$.
	Decomposing $A = \sum_{\alpha} A_{\alpha}$, with $A_{\alpha} \in \qlie_{\alpha}$,
	one has $A_{\alpha} = 0$ for $\alpha \leq 0$.
\end{plem}

\begin{proof}
	As before, the exponential mapping $A \mapsto e^A \cdot \Flim$ gives a biholomorphism
	between a neighborhood of $0 \in \qlie$ and a neighborhood of the point
	$\Flim \in \Dch$. Recall that $\Flim = \lim_{x \to \infty} e^{-xS} \Psi_S(0)$. For
	$x \gg 0$, it follows that there is an element $B(x) \in \qlie$
	such that $e^{-x S} \Psi_S(0) = e^{B(x)} \Flim$. Since $S$ preserves the filtration
	$\Flim$, we can rewrite this as
	\[
		\Psi_S(0) = e^{xS} e^{B(x)} e^{-xS} \cdot \Flim = e^A \cdot \Flim,
	\]
	where $A = e^{x \ad(S)} B(x) \in \qlie$ is necessarily constant. Now we decompose
	this constant endomorphism as $A = \sum_{\alpha} A_{\alpha}$, with $A_{\alpha} \in
	\qlie \cap E_{\alpha}(\ad S)$. Then
	\[
		e^{-xS} \Psi_S(0) = e^{-xS} e^A e^{xS} \cdot \Flim
		= \exp \left( \sum_{\alpha} e^{-\alpha x} A_{\alpha} \right) \cdot \Flim.
	\]
	As $x \to \infty$, this can only converge to $\Flim$ if $A_{\alpha} = 0$ for
	$\alpha \leq 0$.
\end{proof}

\newpar
Putting the lemma and \eqref{eq:Phi-formula} together, we find that
\[
	\Phi(z) = e^{z(S+N)} e^{\Gammat(e^z)} e^A \cdot \Flim
	= e^{z(S+N)} e^{\Gamma(e^z)} \cdot \Flim,
\]
where $\Gamma \colon \Delta_{\eps} \to \qlie$ is holomorphic and satisfies $\Gamma(0)
= A$. We can do even better by decomposing $\Gamma(t) = \sum_{\alpha}
\Gamma_{\alpha}(t)$; each individual function $\Gamma_{\alpha} \colon \Delta_{\eps}
\to \qlie_{\alpha}$ is holomorphic, and $\Gamma_{\alpha}(0) = 0$ for $\alpha \leq 0$.
Since $S$ preserves the filtration $\Flim$, we finally get
\[
	\Phi(z) = e^{zN} \cdot e^{zS} e^{\Gamma(e^z)} e^{-zS} \cdot \Flim
	= e^{zN} \cdot 
	\exp \left( \sum_{\alpha} e^{\alpha z} \, \Gamma_{\alpha}(e^z) \right) \cdot \Flim.
\]
To summarize, we have proved the following result.

\begin{pprop}
	There is a small positive number $\eps > 0$, and a collection of holomorphic
	functions $\Gamma_{\alpha} \colon \Delta_{\eps} \to \qlie_{\alpha}$ with
	$\Gamma_{\alpha}(0) = 0$ for $\alpha \leq 0$, such that the formula
	\[
		\Phi(z) = e^{zN} \cdot 
		\exp \left( \sum_{\alpha} e^{\alpha z} \, \Gamma_{\alpha}(e^z) \right) \cdot \Flim
	\]
	describes the period mapping on the half-plane $\Re z < \log \eps$. 
\end{pprop}

Note that each term in the sum vanishes as $\abs{\Re z} \to \infty$. The rate at
which this happens is controlled by the smallest positive eigenvalue of $\ad S$; but
this is the smallest distance among consecutive eigenvalues of $S$, which we had
earlier denoted by $\delta(T)$. This shows one more time that $e^{-zN} \Phi(z)$ approaches
its limit $\Flim$ at a rate proportional to $e^{-\delta(T) \abs{\Re z}}$.

\section{\boldmath Asymptotic behavior of the Hodge metric and the
\texorpdfstring{$\SL(2)$}{SL(2)}-orbit theorem}
\label{chap:asymptotic}

\newpar
In this chapter, we study the rate of convergence of the rescaled period mapping, and
prove a cheap version of Schmid's famous $\SL(2)$-orbit theorem
\cite[Thm.~5.13]{Schmid}. 

\newpar
On the vector space $V$, we have the limiting Hodge filtration $\Flim$, as well as
the conjugate filtration $\Fblim$, defined by 
\[
	\Fblim^q = \menge{v \in V}{\text{$Q(v,w) = 0$ for all $w \in \Flim^{n+1-q}$}}.
\]
These put a mixed Hodge structure on $V$, with weight filtration
$W_{\bullet-n} = W_{\bullet-n}(N)$; moreover, $N \colon V \to V(-1)$ is a morphism of
mixed Hodge structures. At the same time, $V$ also has a polarized $\sltwo$-Hodge
structure of weight $n$, whose two Hodge filtrations are $F_H$ and $\Fb_H$,
where
\[
	\Fb_H^q = \menge{v \in V}{\text{$Q(v,w) = 0$ for all $w \in F_H^{n+1-q}$}}.
\]

\newpar
We begin by comparing the two filtrations $\Flim$ and $F_H$.

\begin{plem} \label{lem:Flim-FH}
	There is an element $h \in \GL(V)$ with $h - \id \in W_{-1} \End(V)$ and $h N
	h^{-1} = N$, such that $\Flim = h F_H$.
\end{plem}

\begin{proof}
	Let $H(W_{\bullet-n}, \Flim, \Fblim)$ denote Deligne's splitting of
	the limiting mixed Hodge structure (as in \Cref{prop:Deligne}). The corresponding
	subspaces
	\[
		\Ilim^{i,j} = \Flim^i \cap W_{i+j-n} \cap 
		\bigl( \Fblim^j \cap W_{i+j-n} + \Fblim^{j-1} \cap W_{i+j-2-n} + \Fblim^{j-2} \cap
		W_{i+j-3-n} + \dotsb \bigr)
	\]
	decompose $V$ into a direct sum, in such a way that
	\[
		W_k = \bigoplus_{i+j \leq n+k} \Ilim^{i,j} \qquad \text{and} \qquad
		\Flim^p = \bigoplus_{i \geq p, j} \Ilim^{i,j};
	\]
	moreover, one has $N(\Ilim^{i,j}) \subseteq \Ilim^{i-1,j-1}$, due to the fact that
	$N \colon V \to V(-1)$ is a morphism of mixed Hodge structures. In terms of the
	splitting, $[H(W_{\bullet-n}, \Flim, \Fblim), N] = -2N$.

	Recall that in an $\sltwo$-Hodge structure of weight $n$, one has
	\[
		V_k^{i,j} = F_H^i \cap V_k \cap \Fb_H^j,
	\]
	whenever $i+j=n+k$. Under the projection from $W_{i+j-n}$ to $\gr_{i+j-n}^W$, the
	two subspaces $\Ilim^{i,j}$ and $V_{i+j-n}^{i,j}$ both map isomorphically to the
	$(i,j)$-subspace in the Hodge decomposition of $\gr_{i+j-n}^W$, and so there is a
	unique isomorphism $h^{i,j}$ making the following diagram commute:
	\[
		\begin{tikzcd}
			& V_{i+j-n}^{i,j} \dar{\cong} \dlar[swap,bend right=20]{h^{i,j}} \\
			\Ilim^{i,j} \rar{\cong} & \bigl( \gr_{i+j-n}^W \bigr)^{i,j}
		\end{tikzcd}
	\]
	Consequently, there is a unique automorphism $h \in \GL(V)$ such that $h
	\restr{V_{i+j-n}^{i,j}} = h^{i,j}$ for all $i,j \in \ZZ$. It follows that 
	\[
		h(F_H^p) = \bigoplus_{i \geq p} h \bigl( V_{i+j-n}^{i,j} \bigr) 
		= \bigoplus_{i \geq p} \Ilim^{i,j} = \Flim^p.
	\]
	By construction, $h$ preserves the weight filtration $W$ and acts as the identity
	on each $\gr_k^W$; therefore $h - \id \in W_{-1} \End(V)$. Moreover, the diagram
	\[
		\begin{tikzcd}[column sep=large]
			V_{i+j-n}^{i,j} \dar{N} \rar{h^{i,j}} & \Ilim^{i,j} \dar{N} \\
			V_{i+j-2-n}^{i-1,j-1} \rar{h^{i-1,j-1}} & \Ilim^{i-1,j-1}
		\end{tikzcd}
	\]
	is commutative, and this implies that $hN = Nh$. In terms of the two splittings,
	we have
	\[
		H(W_{\bullet-n}, \Fb, \Fblim) = h \bigl( H + n \id \bigr) h^{-1},
	\]
	which is consistent with \Cref{lem:splittings}.
\end{proof}

\newpar
Note that, in the lemma above, the splitting $H$ was arbitrary, subject only to the
conditions in \Cref{prop:splitting}. If we pick a splitting that is adapted to the
limiting mixed Hodge structure, we can get a much better result. Indeed, suppose we
use instead the splitting $H_{\RR}(W_{\bullet-n}, \Flim, \Fblim) \in \glie$
for the filtration $W_{\bullet}$, constructed in \Cref{prop:real-splitting}, which
we got by averaging the two splittings $H(W_{\bullet-n}, \Flim, \Fblim)$ and
$H(W_{\bullet-n}, \Fblim, \Flim)$. This also has all the properties required by
\Cref{prop:splitting}. Then
\[
	H_{\RR}(W_{\bullet-n}, \Flim, \Fblim) 
	- \bigl( H(W_{\bullet-n}, \Flim, \Fblim) - n \id \bigr) 
	\in R \bigl( W_{-2} \End(V) \bigr),
\]
and so the element $h \in \GL(V)$ in \Cref{lem:Flim-FH} even satisfies $h - \id \in
W_{-2} \End(V)$. 

\newpar
We return to our investigation of the two filtrations $\Flim$ and $F_H$.  Let us
write the element $h \in \GL(V)$ as a finite sum
\[
	h = \id + h_{-1} + h_{-2} + \dotsb,
\]
where each $h_{-k} \in E_{-k}(\ad H) \cap \ker(\ad N)$. Recall from \Cref{par:FH}
that
\[
	F_H = \lim_{x \to \infty} e^{\half \log x \, H} \Flim.
\]
We can analyze the convergence more precisely as follows. For $x > 0$, one has
\[
	e^{\half \log x \, H} \Flim
	= e^{\half \log x \, H} h F_H
	= e^{\half \log x \, H} h e^{-\half \log x \, H} F_H,
\]
due to the fact that $H$ preserves the filtration $F_H$. Now
\[
	e^{\half \log x \, H} h e^{-\half \log x \, H}
	= \id + \sum_{k=1}^{\infty} x^{-k/2} h_{-k},
\]
and since $[N,h_{-k}] = 0$, we finally obtain
\[
	e^{\half \log x \, H} e^{-x N} \Flim
	= e^{-N} e^{\half \log x \, H} \Flim
	= \left( \id + \sum_{k=1}^{\infty} x^{-k/2} h_{-k} \right) \cdot \Fsh,
\]
where $\Fsh = e^{-N} F_H \in D$. Note that the filtration on the right-hand side
belongs to $D$ when $x \gg 0$.  With the substitution $u = x^{-1/2}$, we can rewrite
this as
\[
	\left( \id + \sum_{k=1}^{\infty} u^k h_{-k} \right) \cdot \Fsh,
\]
as long as $u$ is sufficiently close to $0$. 

\newpar
To understand the behavior of the Hodge metric, we have to express this in terms of
the real group $G$. For simplicity, let us denote by
\[
	V = \bigoplus_{p+q=n} V^{p,q}
\]
the Hodge structure with Hodge filtration $\Fsh = e^{-N} F_H$, and by
\[
	\End(V) = \bigoplus_{j \in \ZZ} \End(V)^{j,-j}
\]
the induced Hodge structure of weight $0$ on $\End(V)$. Recall that this is an $\RR$-Hodge
structure; the Lie algebra $\glie$ of the real group $G$ gives the real structure. Setting
\[
	\mlie = \glie \cap \bigoplus_{j \neq 0} \End(V)^{j,-j}
\]
we then have $\End(V) = \mlie \oplus F^0 \End(V)$; compare
\eqref{eq:iso-tangent-spaces}. In concrete terms, this decomposition works as
follows. For an endomorphism $A \in \End(V)$, denote by $A_j \in \End(V)^{j,-j}$ the
components in the Hodge decomposition; then
\[
	\sum_{j < 0} \bigl( A_j - (A_j)^{\dagger} \bigr) \in \mlie
	\quad \text{and} \quad
	\sum_{j \geq 0} A_j + \sum_{j < 0} (A_j)^{\dagger} \in F^0 \End(V)
\]

\newpar
Recall that the subspace $\mlie \subseteq \End(V)$ maps isomorphically to the tangent
space of $D$ at the point $\Fsh$. The exponential mapping
\[
	\End(V) \to D, \quad A \mapsto e^A \cdot \Fsh,
\]
is real analytic, and therefore restricts to a real-analytic diffeomorphism between a
neighborhood of the origin in $\mlie$ and a neighborhood of the point $\Fsh \in D$,
it follows that there is a small positive number $\eps > 0$ and two unique real
analytic functions
\[
	B \colon (-\eps, \eps) \to \mlie \quad \text{and} \quad
	C \colon (-\eps, \eps) \to F^0 \End(V)
\]
with the property that $B(0) = C(0) = 0$ and 
\[
	\left( \id + \sum_{k=1}^{\infty} u^k h_{-k} \right) e^{C(u)} = e^{B(u)}.
\]
Let us expand $B$ and $C$ into convergent power series of the form
\[
	B(u) = \sum_{k=1}^{\infty} u^k B_{-k} \quad \text{and} \quad
	C(u) = \sum_{k=1}^{\infty} u^k C_{-k},
\]
where $B_{-k} \in \mlie$ and $C_{-k} \in F^0 \End(V)$. 

\newpar
From the power series expansion of
\[
	e^{B(u)} = \id + u B_{-1} + u^2 \Bigl( B_{-2} + \half B_{-1}^2 \Bigr) + \dotsb,
\]
we see that there are universal non-commutative polynomials $P_{-k}$ in the variables
$B_{-1}, B_{-2}, \dotsc$, such that
\[
	e^{B(u)} = \id + \sum_{k=1}^{\infty} u^k P_{-k}(B_{-1}, \dotsc, B_{-k}).
\]
The first few terms are easily computed to be
\begin{align*}
	P_{-1}(B_{-1}) &= B_{-1}, \\
	P_{-2}(B_{-1},B_{-2}) &= B_{-2} + \half B_{-1}^2 \\
	P_{-3}(B_{-1}, B_{-2}, B_{-3}) &= B_{-3} + \half B_{-2} B_{-1} + \half B_{-1}
	B_{-2} + \frac{1}{6} B_{-1}^3.
\end{align*}
In general, the coefficients of $k! P_{-k}$ are positive integers; moreover, $P_{-k}$
is homogeneous of degree $k$, if we declare that the degree of $B_{-j}$ is $j$. It
is easy to see that the variable $B_{-k}$ always appears in $P_{-k}$ with coefficient $1$.

\newpar
From the $\sltwo$-Hodge structure on $V$, the vector space $\End(V)$ inherits an
$\sltwo$-Hodge structure of weight $0$. In particular, the Lie algebra
$\sltwo(\CC)$ acts on $\End(V)$, and therefore induces a decomposition
\[
	\End(V) \cong \bigoplus_{m \in \NN} S_m \tensor_{\CC}
	\Hom_{\CC} \bigl( S_m, \End(V) \bigr)^{\sltwo(\CC)}.
\]
The summand with $S_m$ is called the \define{$\boldsymbol{S_m}$-isotypical
component}. According
to the discussion in \Cref{par:Casimir}, the isotypical components are sub-Hodge
structures of the Hodge structure on $\End(V)$. They are also stable under taking
adjoints; therefore the decomposition
\[
	\End(V) = \mlie \oplus F^0 \End(V)
\]
respects the isotypical components.

\newpar
Since each operator $h_{-k} \in E_{-k}(\ad H) \cap \ker(\ad N)$ is primitive, it
obviously belongs to the $S_k$-isotypical component of the representation. This has
the following consequence for the operators $B_{-k}, C_{-k} \in \End(V)$.

\begin{pprop} \label{prop:isotypical}
	For each $k \geq 1$, both $B_{-k}$ and $C_{-k}$ belong to the sum of the
	$S_j$-isotypical components with $j = k, k-2, k-4, \dotsc$.
\end{pprop}

\begin{proof}
	Expanding the relation 
	\begin{align*}
		\left( \id + \sum_{k=1}^{\infty} u^k P_{-k}(B_{-1}, \dotsc, B_{-k}) \right) & \\
		 = \left( \id + \sum_{k=1}^{\infty} u^k h_{-k} \right) 
		 &\left( \id + \sum_{k=1}^{\infty} u^k P_{-k}(C_{-1}, \dotsc, C_{-k}) \right) 
	\end{align*}
	from above now gives us the following series of identities:
	\[
		P_{-k}(B_{-1}, \dotsc, B_{-k}) 
		= h_{-k} + P_{-k}(C_{-1}, \dotsc, C_{-k}) + 
		\sum_{j=1}^{k-1} h_{-k+j} \cdot P_{-j}(C_{-1}, \dotsc, C_{-j})
	\]
	For $k = 1$, this reads $B_{-1} = h_{-1} + C_{-1}$; here $B_{-1} \in \mlie$ and
	$C_{-1} \in F^0 \End(V)$. Remembering that $\End(V) = \mlie \oplus F^0 \End(V)$,
	we see again that this equation uniquely determines $B_{-1}$ and $C_{-1}$. Since
	$h_{-1}$ belongs to the $S_1$-isotypical
	component of the representation, and since this is a sub-Hodge structure, it
	follows that $B_{-1}$ and $C_{-1}$ also belong to the $S_1$-isotypical component.

	We can now prove the assertion by induction on $k \geq 1$. Rearranging the
	identity from above, we obtain an identity of the form
	\[
		B_{-k} - C_{-k} = P(B_{-1}, \dotsc, B_{-k+1}, C_{-1}, \dotsc, C_{-k+1}, h_{-1},
		\dotsc, h_{-k})
	\]
	for a certain non-commutative polynomial $P$, homogeneous of degree $k$. 
	Now each term on the right-hand side is a product of operators of the form $B_{-j}$,
	$C_{-j}$ or $h_{-j}$, and by induction, each such operator is in the direct
	sum of the $S_{j-2i}$-isotypical components for $i \geq 0$. Since
	\[
		S_k \tensor S_{\ell} \cong \bigoplus_{i=0}^{\min(k,\ell)} S_{k+\ell-2i},
	\]
	and since the degree of each term is $-k$, this gives the result.
\end{proof}

\newpar
This has the following nice consequence for the nilpotent orbit $e^{zN} \Flim$. The
result is of course less precise than the $\SL(2)$-orbit theorem
\cite[Thm.~5.13]{Schmid}, but in return, the proof is \emph{much} easier, and we
believe that the version below is actually sufficient for all practical purposes.

\begin{pthm} \label{thm:SL2-cheap}
	There is a constant $\eps > 0$, whose value depends on $\Flim$ and $N$, and a
	real-analytic function $B \colon (-\eps, \eps) \to \glie$, such that
	\[
		\Phinil(x+iy) = e^{(x + iy) N} \Flim = e^{i y N} e^{-\half \log \abs{x} \, H}
		e^{B(\abs{x}^{-1/2})} \cdot e^{-N} F_H
	\]
	as long as $\abs{x} > 1/\eps^2$. In the power series expansion
	\[
		B(u) = \sum_{k=1}^{\infty} u^k B_{-k},
	\]
	each coefficient $B_{-k} \in \glie$ belongs to the direct sum of the
	$S_m$-isotypical components of the $\sltwo(\CC)$-representation with $m = k, k-2,
	k-4, \dotsc$. 
\end{pthm}

The point is that all the operators on the right-hand side of this formula belong
to the real group $G$. This makes it possible to describe the asymptotic behavior of
the Hodge metric or the Weil operator in any nilpotent orbit.

\begin{note}
Instead of an arbitrary splitting $H \in \End(V)$, we can use the splitting
\[
	H_{\RR}(W_{\bullet-n}, \Flim, \Fblim) \in \glie
\]
given by \Cref{prop:real-splitting}. Then $h - \id \in W_{-2} \End(V)$, and an
analysis of the proof shows that $B_{-1} = 0$. With this choice, the statement of
\Cref{thm:SL2-cheap} becomes closer to the full $\SL(2)$-orbit theorem in
\cite[Thm.~5.13]{Schmid}. 
\end{note}

\newpar
We can use \Cref{thm:SL2-cheap} to get more precise information about the
asymptotic behavior of the Hodge metric. The Hodge norm estimates tell us the exact
rate of growth of $\norm{v}_{\Phi(z)}^2$ for any $v \in V$, but so far, we have not
said anything about the behavior of the inner products $\inner{v}{w}_{\Phi(z)}$
when $v,w \in V$ are two different multi-valued flat sections.

\newpar
Let us denote by $\inner{v}{w}$ the inner product coming from the polarized Hodge
structure $\Fsh = e^{-N} F_H$. Set $z = x + iy$. According to
\Cref{thm:approximation-nilpotent-orbit}, we have
\[
	d_D \bigl( \Phi(x+iy), e^{(x+iy)N} \Flim \bigr) \leq 
	C \abs{x}^m e^{-\delta(T) \abs{x}}
\]
as long as $x \leq x_0$, with constants $C > 0$ and $x_0 < 0$
that are basically independent of the period mapping in question. Since the
exponential factors in the definition of $\PhiSH$ belong to the real Lie
group $G$, it follows that
\[
	d_D \Bigl( \PhiSH(x+iy), \, e^{\half \log \abs{x} \, H} e^{xN} \Flim \Bigr)
	\leq C \abs{x}^m e^{-\delta(T) \abs{x}}.
\]
We now restrict our attention to $\abs{x} > 1/\eps^2$, and set $u = \abs{x}^{-\half}$. Then
\[
	e^{\half \log \abs{x} \, H} e^{xN} \Flim = e^{B(u)} \Fsh,
\]
and so our distance estimate becomes
\[
	d_D \Bigl( \PhiSH(x+iy), \, e^{B(u)} \Fsh \Bigr)
	\leq C \abs{x}^m e^{-\delta(T) \abs{x}}.
\]
\Cref{lem:distance-inner-products} allows us to conclude that, as long as
$\abs{x}$ is sufficiently large,
\[
	\Bigl\lvert \inner{v}{w}_{\PhiSH(x+iy)} - \inner{v}{w}_{e^{B(u)} \Fsh}
	\Bigr\rvert \leq C' \norm{v}_{e^{B(u)} \Fsh} \norm{w}_{e^{B(u)} \Fsh} 
	\cdot \abs{x}^m e^{-\delta(T) \abs{x}}
\]
for some constant $C' > 0$ that is independent of $v,w \in V$. 

\newpar
Now suppose that $v \in E_k(H)$ and $w \in E_{\ell}(H)$ belong to possibly different weight
spaces, say with $k \leq \ell$. The behavior of the inner product
\[
	\inner{v}{w}_{e^{B(u)} \Fsh} = \inner{e^{-B(u)} v}{e^{-B(u)} w}
\]
is very easy to understand, given what we know about the coefficients of the series
$B(u)$. The whole expression is a convergent power series in $u$. Since different
eigenspaces of $H$ are orthogonal under the inner product $\inner{v}{w}$, the leading
term in this series is divisible by $u^{\ell-k}$. It follows that there is a constant
$C > 0$ such that for $\abs{\Re z} \gg 0$, one has
\begin{equation} \label{eq:inner-PhiSH}
	\bigl\lvert \inner{v}{w}_{\PhiSH(z)} \bigr\rvert \leq 
	C \norm{v} \norm{w} \cdot \abs{\Re z}^{-(\ell-k)/2}.
\end{equation}
From the definition of the rescaled period mapping in \eqref{eq:PhiSH-intro}, we get
\[
	\bigl\langle v, \, w \bigr\rangle_{\PhiSH(z)} =
	\abs{\Re z}^{-(k+\ell)/2} \cdot \bigl\langle 
		e^{\half(z-\zb) (S+N)} v, \, e^{\half(z-\zb) (S+N)} w \bigr\rangle_{\Phi(z)},
\]
and as long as $\Im z$ remains bounded, it follows that 
\begin{equation} \label{eq:inner-Phi}
	\bigl\lvert \inner{v}{w}_{\Phi(z)} \bigr\rvert \leq 
	C' \norm{v} \norm{w} \cdot \abs{\Re z}^k
\end{equation}
for a different constant $C' > 0$ and $\abs{\Re z} \gg 0$.

\newpar
We know from the Hodge norm estimates that $\norm{v}_{\Phi(z)}^2$ and
$\norm{w}_{\Phi(z)}^2$ are of order $\abs{\Re z}^k$ and $\abs{\Re z}^{\ell}$,
respectively. Consequently, \eqref{eq:inner-Phi} is equivalent to the statement that,
for $\abs{\Re z} \gg 0$, one has
\[
	\bigl\lvert \inner{v}{w}_{\Phi(z)} \bigr\rvert \leq 
	C'' \min \bigl( \norm{v}_{\Phi(z)}^2, \norm{w}_{\Phi(z)}^2 \bigr).
\]
Because of the triangle inequality, we deduce that for any two multi-valued flat
sections $v,w \in V$, there is a constant $C(v,w) > 0$ such that
\[
	\bigl\lvert \inner{v}{w}_{\Phi(z)} \bigr\rvert 
	\leq C(v,w) \cdot \min \Bigl( \norm{v}_{\Phi(z)}^2, \norm{w}_{\Phi(z)}^2 \Bigr)
\]
for $\abs{\Re z} \gg 0$. This finishes the proof of \Cref{thm:asymptotic-intro} from
the introduction.

%%%%%%%%%%%%%%%%%%%

\cleardoublepage
\phantomsection
\addcontentsline{toc}{chapter}{References}

\bibliographystyle{amsalpha}
\bibliography{bibliography}

\providecommand{\bysame}{\leavevmode\hbox to3em{\hrulefill}\thinspace}
\providecommand{\MR}{\relax\ifhmode\unskip\space\fi MR }
\providecommand{\MRhref}[2]{%
  \href{http://www.ams.org/mathscinet-getitem?mr=#1}{#2}
}
\providecommand{\href}[2]{#2}
\begin{thebibliography}{dCM05}

\bibitem[Ahl38]{Ahlfors:SchwarzLemma}
Lars~V. Ahlfors, \emph{An extension of {S}chwarz's lemma}, Trans. Amer. Math.
  Soc. \textbf{43} (1938), no.~3, 359--364. \MR{1501949}

\bibitem[Ber10]{Berndtsson:Things}
Bo~Berndtsson, \emph{An introduction to things {$\overline\partial$}}, Analytic
  and algebraic geometry, IAS/Park City Math. Ser., vol.~17, Amer. Math. Soc.,
  Providence, RI, 2010, pp.~7--76. \MR{2743815}

\bibitem[Cat08]{Cattani:MixedLefschetzTheorems}
Eduardo Cattani, \emph{Mixed {L}efschetz theorems and {H}odge-{R}iemann
  bilinear relations}, Int. Math. Res. Not. IMRN (2008), no.~10, Art. ID
  rnn025, 20. \MR{2429243}

\bibitem[CG75]{Cornalba+Griffiths:AnalyticCycles}
Maurizio Cornalba and Phillip Griffiths, \emph{Analytic cycles and vector
  bundles on non-compact algebraic varieties}, Invent. Math. \textbf{28}
  (1975), 1--106. \MR{367263}

\bibitem[CK89]{Cattani+Kaplan:Luminy}
Eduardo Cattani and Aroldo Kaplan, \emph{Degenerating variations of {H}odge
  structure}, Ast\'erisque (1989), no.~179-180, 9, 67--96, Actes du Colloque de
  Th{\'e}orie de Hodge (Luminy, 1987). \MR{1042802 (91k:32019)}

\bibitem[CKS86]{Cattani+Kaplan+Schmid:Degeneration}
Eduardo Cattani, Aroldo Kaplan, and Wilfried Schmid, \emph{Degeneration of
  {H}odge structures}, Ann. of Math. (2) \textbf{123} (1986), no.~3, 457--535.
  \MR{840721 (88a:32029)}

\bibitem[dCM05]{deCataldo+Migliorini:HodgeTheoryMaps}
Mark de~Cataldo and Luca Migliorini, \emph{The {H}odge theory of algebraic
  maps}, Ann. Sci. \'Ecole Norm. Sup. (4) \textbf{38} (2005), no.~5, 693--750.
  \MR{2195257 (2007a:14016)}

\bibitem[Del70]{Deligne:EquationsDifferentielles}
Pierre Deligne, \emph{\'{E}quations diff\'erentielles \`a points singuliers
  r\'eguliers}, Lecture Notes in Mathematics, Vol. 163, Springer-Verlag,
  Berlin, 1970. \MR{0417174 (54 \#5232)}

\bibitem[Del71]{Deligne:HodgeII}
\bysame, \emph{Th\'eorie de {H}odge. {II}}, Inst. Hautes \'Etudes Sci. Publ.
  Math. (1971), no.~40, 5--57. \MR{0498551 (58 \#16653a)}

\bibitem[Del84]{Deligne:PositiviteSignes}
\bysame, \emph{Positivit\'e: signes}, Personal note, dated Feb.~11, 1984.

\bibitem[Dem12]{Demailly:ComplexGeometry}
Jean-Pierre Demailly, \emph{Complex analytic and differential geometry}, 2012.

\bibitem[Den22]{Deng:NilpotentOrbit}
Ya~Deng, \emph{On the nilpotent orbit theorem of complex variation of {H}odge
  structures}, 2022, preprint
  \href{https://arxiv.org/abs/2203.04266}{\texttt{arXiv:2203.04266}}.

\bibitem[GS75]{Griffiths+Schmid:RecentDevelopments}
Phillip Griffiths and Wilfried Schmid, \emph{Recent developments in {H}odge
  theory: a discussion of techniques and results}, Discrete subgroups of {L}ie
  groups and applicatons to moduli ({I}nternat. {C}olloq., {B}ombay, 1973),
  1975, pp.~31--127. \MR{0419850}

\bibitem[Kas85]{Kashiwara:AsymptoticBehavior}
Masaki Kashiwara, \emph{The asymptotic behavior of a variation of polarized
  {H}odge structure}, Publ. Res. Inst. Math. Sci. \textbf{21} (1985), no.~4,
  853--875. \MR{817170 (87h:32049)}

\bibitem[Moc02]{Mochizuki:Nilpotent}
Takuro Mochizuki, \emph{Asymptotic behaviour of tame nilpotent harmonic bundles
  with trivial parabolic structure}, J. Differential Geom. \textbf{62} (2002),
  no.~3, 351--559. \MR{2005295}

\bibitem[Moc07]{Mochizuki:AsymptoticBehaviorI}
\bysame, \emph{Asymptotic behaviour of tame harmonic bundles and an application
  to pure twistor {$D$}-modules. {I}}, Mem. Amer. Math. Soc. \textbf{185}
  (2007), no.~869, xii+324. \MR{2281877}

\bibitem[Mor78]{Morgan:AlgebraicTopology}
John~W. Morgan, \emph{The algebraic topology of smooth algebraic varieties},
  Inst. Hautes \'{E}tudes Sci. Publ. Math. (1978), no.~48, 137--204.
  \MR{516917}

\bibitem[Sch73]{Schmid:VHS}
Wilfried Schmid, \emph{Variation of {H}odge structure: the singularities of the
  period mapping}, Invent. Math. \textbf{22} (1973), 211--319. \MR{0382272 (52
  \#3157)}

\bibitem[Sim88]{Simpson:ConstructingVHS}
Carlos~T. Simpson, \emph{Constructing variations of {H}odge structure using
  {Y}ang-{M}ills theory and applications to uniformization}, J. Amer. Math.
  Soc. \textbf{1} (1988), no.~4, 867--918. \MR{944577}

\bibitem[Sim90]{Simpson:HarmonicBundlesCurves}
\bysame, \emph{Harmonic bundles on noncompact curves}, J. Amer. Math. Soc.
  \textbf{3} (1990), no.~3, 713--770. \MR{1040197}

\bibitem[Zuc79]{Zucker:DegeneratingCoefficients}
Steven Zucker, \emph{Hodge theory with degenerating coefficients. {$L_{2}$}
  cohomology in the {P}oincar\'e metric}, Ann. of Math. (2) \textbf{109}
  (1979), no.~3, 415--476. \MR{534758 (81a:14002)}

\end{thebibliography}

\end{document}